\documentclass[a4paper,11pt]{amsart}
\usepackage{bbm}
\usepackage{amsfonts}
\usepackage{amsthm}
\usepackage{amsmath}
\usepackage{amssymb}
\usepackage{mathrsfs}
\usepackage{amscd}
\usepackage{mathdots}
\usepackage[all]{xy}
\usepackage{graphicx}
\usepackage{chngcntr}
\usepackage{fancyhdr}
\usepackage[top=1in,bottom=1in,left=1.25in,right=1.25in]{geometry}
\usepackage[pdftex]{hyperref}

\def\Q{\mathbb{Q}}
\def\Z{\mathbb{Z}}
\def\Ab{\mathbb{A}}
\def\B{\mathcal{B}}
\def\P{\mathcal{P}}

\def\E{\mathsf{E}}
\def\K{\mathsf{K}}

\def\ad{\mathrm{ad}}

\def\ra{\rightarrow}
\def\lra{\longrightarrow}
\def\ov{\overline}
\def\ovla{\overleftarrow}
\def\ovra{\overrightarrow}
\def\ul{\underline}
\def\wh{\widehat}
\def\wt{\widetilde}

\newcommand{\DF}[2]{{\displaystyle\frac{#1}{#2}}}

\newcommand{\I}{\mathbb{I}}

\newcommand{\W}{\mathbb{W}}
\newcommand{\INT}{\mathbb{Z}}

\newcommand{\RAT}{\mathbb{Q}}

\newcommand{\A}{\mathcal {A}}
\newcommand{\V}{\mathcal {V}}
\newcommand{\ES}{\mathscr{S}}

\newcommand{\dr}{\mathrm{dR}}
\newcommand{\cris}{\mathrm{cris}}

\def\El{\mathcal{E}}
\newcommand{\G}{\mathcal {G}}

\newcommand{\eM}{\mathrm {M}_0}
\newcommand{\eMw}{\mathrm{M}^w_0}
\newcommand{\eMc}{\mathrm{M}_0^\mathrm{c}}

\newcommand{\dkV}{V_k^\vee}

\newcommand{\IIsom}{\mathbf{Isom}}
\newcommand{\Isom}{\mathrm{Isom}}

\newcommand{\Aut}{\mathrm{Aut}}

\newcommand{\Hom}{\mathrm{Hom}}
\newcommand{\rdt}{\mathrm{rdt}}

\newcommand{\Sh}{\mathrm{Sh}}

\newcommand{\Hdr}{\mathrm{H}^1_{\mathrm{dR}}}

\newcommand{\GL}{\mathrm{GL}}
\newcommand{\GSp}{\mathrm{GSp}}

\def\Adm{\mathrm{Adm}}
\def\Aff{\mathrm{Aff}}
\def\loc{\mathrm{loc}}
\def\Sht{\mathrm{Sht}}

\newcommand{\mmu}{\{\mu\}}

\newcommand{\Admu}{\mathrm{Adm}(\{\mu\})}
\newcommand{\BGmu}{B(G,\{\mu\})}
\newcommand{\CGmu}{C(\mathcal{G},\{\mu\})}
\newcommand{\IW}{\widetilde{W}}

\newcommand{\Gr}{\mathrm{Gr}}
\newcommand{\Mloc}{\mathrm{M}^{\mathrm{loc}}_G}
\newcommand{\Mlo}{\mathrm{M}^{\mathrm{loc}}}
\newcommand{\Mlohat}{\widehat{\mathrm{M}}^{\mathrm{loc}}}

\DeclareMathOperator{\Spec}{\mathrm{Spec}}

\DeclareMathOperator{\gr}{\mathrm{gr}}
\DeclareMathOperator{\Spf}{\mathrm{Spf}}
\DeclareMathOperator{\Max}{\mathrm{Max}}

\newcommand{\nc}{\newcommand}

\def\makeop#1{\expandafter\def\csname#1\endcsname
	{\mathop{\rm #1}\nolimits}\ignorespaces}
\makeop{Hom}   \makeop{End}   \makeop{Aut}   \makeop{Isom}  \makeop{Pic} 
\makeop{Gal}   \makeop{ord}   \makeop{Char}  \makeop{Div}   \makeop{Lie} 
\makeop{PGL}   \makeop{Corr}  \makeop{PSL}   \makeop{sgn}   \makeop{Spf}
\makeop{Spec}  \makeop{Tr}    \makeop{Nr}    \makeop{Fr}    \makeop{disc}
\makeop{Proj}  \makeop{supp}  \makeop{ker}   \makeop{im}    \makeop{dom}
\makeop{coker} \makeop{Stab}  \makeop{SO}    \makeop{SL}    \makeop{SL}
\makeop{Cl}    \makeop{cond}  \makeop{Br}    \makeop{inv}   \makeop{rank}
\makeop{id}    \makeop{Fil}   \makeop{Frac}  \makeop{GL}    \makeop{SU}
\makeop{Nrd}   \makeop{Sp}    \makeop{Tr}    \makeop{Trd}   \makeop{diag}
\makeop{Res}   \makeop{ind}   \makeop{depth} \makeop{Tr}    \makeop{st}
\makeop{Ad}    \makeop{Int}   \makeop{tr}    \makeop{Sym}   \makeop{can}
\makeop{length}\makeop{SO}    \makeop{torsion} \makeop{GSp} \makeop{Ker}
\makeop{Adm}   \makeop{Mat}
\def\makebb#1{\expandafter\def
	\csname bb#1\endcsname{{\mathbb{#1}}}\ignorespaces}
\def\makebf#1{\expandafter\def\csname bf#1\endcsname{{\bf
			#1}}\ignorespaces} 
\def\makegr#1{\expandafter\def
	\csname gr#1\endcsname{{\mathfrak{#1}}}\ignorespaces}
\def\makescr#1{\expandafter\def
	\csname scr#1\endcsname{{\EuScript{#1}}}\ignorespaces}
\def\makecal#1{\expandafter\def\csname cal#1\endcsname{{\mathcal
			#1}}\ignorespaces} 

\def\doLetters#1{#1A #1B #1C #1D #1E #1F #1G #1H #1I #1J #1K #1L #1M
	#1N #1O #1P #1Q #1R #1S #1T #1U #1V #1W #1X #1Y #1Z}
\def\doletters#1{#1a #1b #1c #1d #1e #1f #1g #1h #1i #1j #1k #1l #1m
	#1n #1o #1p #1q #1r #1s #1t #1u #1v #1w #1x #1y #1z}
\doLetters\makebb   \doLetters\makecal  \doLetters\makebf
\doLetters\makescr 
\doletters\makebf   \doLetters\makegr   \doletters\makegr

    \def\setminus{\smallsetminus}
\def\Gm{{{\bbG}_{\rm m}}}   

\normalsize

\makeop{Bl}

\def\Spec{{\rm Spec}\,}

\def\Fpbar{\overline{\bbF}_p}

\def\Qp{{\bbQ}_p}

\def\Zp{{\bbZ}_p}
\def\Qbar{\overline{\bbQ}}
\def\Sh{{\rm Sh}}

\newcommand{\C}{\mathbb C}

\newcommand{\pr}{\indent }

\def\pr{{\rm pr}}

\renewcommand{\>}{\rangle} 

\newcommand{\isoto}{\stackrel{\sim}{\longrightarrow}}
\nc{\embed}{\hookrightarrow}

\newcommand{\ac}{algebraically closed }

\nc{\ol}{\overline}
\nc{\opp}{\mathrm{opp}}
\def\ul{\underline}
\def\onto{\twoheadrightarrow}

\def\wh{\widehat}

\makeop{Ram}
\makeop{Rep}

\def\E{\mathsf{E}}
\def\K{\mathsf{K}}

\makeatletter
\newenvironment{subeqn}{\refstepcounter{subsubsection}
$$}{\leqno{\rm(\thesubsubsection)}$$\global\@ignoretrue}
\makeatother

\normalsize \relpenalty=10000 \binoppenalty=10000

\begin{document}

\newtheorem{defn}[subsubsection]{Definition}
\newtheorem{theorem}[subsubsection]{Theorem}
\newtheorem{lemma}[subsubsection]{Lemma}
\newtheorem{proposition}[subsubsection]{Proposition}
\newtheorem{corollary}[subsubsection]{Corollary}
\theoremstyle{definition}
\newtheorem{definition}[subsubsection]{Definition}
\theoremstyle{definition}
\newtheorem{construction}[subsubsection]{Construction}
\theoremstyle{definition}
\newtheorem{notations}[subsubsection]{Notations}
\theoremstyle{definition}
\newtheorem{asp}[subsubsection]{Assumption}
\theoremstyle{definition}
\newtheorem{set}[subsubsection]{Setting}
\theoremstyle{remark}
\newtheorem{remark}[subsubsection]{Remark}
\theoremstyle{remark}
\newtheorem{example}[subsubsection]{Example}

\numberwithin{equation}{subsection}

\theoremstyle{plain}
\newtheorem{introth}{Theorem}
\renewcommand{\theintroth}{\Alph{introth}}

\title{EKOR strata for Shimura varieties with parahoric level structure}
\author{Xu Shen, Chia-Fu Yu, and Chao Zhang}
\address{Morningside Center of Mathematics\\
	Academy of Mathematics and Systems Science\\
	Chinese Academy of Sciences\\
	No. 55, Zhongguancun East Road\\
	Beijing 100190, China}\email{shen@math.ac.cn}
\address{Institute of Mathematics\\ Academia Sinica and NCTS\\Astronomy Mathematics Building\\ No. 1, Roosevelt Road Sec. 4\\ Taipei, Taiwan, 10617}\email{chiafu@math.sinica.edu.tw}
\address{Shing-Tung Yau Center of Southeast University\\Yifu Architecture Building, No. 2, Sipailou\\ Nanjing 210096, China}\email{zhangchao1217@gmail.com}


\begin{abstract}
	In this paper we study the geometry of reduction modulo $p$ of the Kisin-Pappas integral models for certain Shimura varieties of abelian type with parahoric level structure.
	We give some direct and geometric constructions for the EKOR strata on these Shimura varieties, using the theories of $G$-zips and mixed characteristic local $\G$-Shtukas. We establish several basic properties of these strata, including the smoothness, dimension formula, and closure relation. Moreover, we apply our results to the study of Newton strata and central leaves on these Shimura varieties.
\end{abstract}
\maketitle \setcounter{tocdepth}{2}
\tableofcontents

\

\

\

\section*[Introduction]{Introduction}

In this paper, we study the geometry of reduction modulo $p$ of the Kisin-Pappas integral models (\cite{Paroh}) for certain Shimura varieties of abelian type with parahoric level structure.
Our goal is to develop a geometric approach to define and study the EKOR (\emph{Ekedahl-Kottwitz-Oort-Rapoport}) stratifications on the reductions of these Shimura varieties, which were first introduced by He and Rapoport (\cite{He-Rap}) in a different setting.
\\

The reduction modulo $p$ of Shimura varieties admits very rich geometric structures. We refer to the excellent survey papers \cite{Ra, guide to red mod p, Hain} for some overview. In this paper, we focus on one important perspective of the mod $p$ geometry of Shimura varieties with parahoric level structure.
In \cite{He-Rap}, He and Rapoport have proposed a general guideline to study the geometry of reduction modulo $p$ of Shimura varieties with parahoric level structure. In particular, they postulated \emph{five basic axioms} (\cite{He-Rap} section 3) on the integral models of Shimura varieties. Assuming the verification of these axioms, in \cite{He-Rap} section 6,
He and Rapoport introduced the EKOR stratification on the special fibers of these integral models, based on the works of Lusztig and He on $G$-stable piece decompositions (\cite{Lusztig1, Lusztig, He}). This new stratification has the following key features:
\begin{itemize}
	\item for hyperspecial levels, by works of Viehmann (\cite{truncat level 1}), the EKOR stratification coincides with the EO (\emph{Ekedahl-Oort}) stratification; 
	\item for Iwahori levels, by works of Lusztig and He (\cite{Lusztig, He} and \cite[Corollary 6.2]{He-Rap}), it coincides with the KR (\emph{Kottwitz-Rapoport}) stratification;
	\item for a general parahoric level, the EKOR stratification is a \emph{refinement} of the KR stratification.
\end{itemize}
The finer structure of EKOR stratification makes it easier to be compared with other natural stratifications, e.g. the \emph{Newton} stratification. In fact, under their axioms He-Rapoport showed that for a general parahoric level, each Newton stratum contains a certain EKOR stratum (\cite{He-Rap} Theorem 6.18), while in general there is no KR stratum that is entirely contained in a given Newton stratum.
One can expect that the geometry of EKOR strata will lead to interesting arithmetic applications, see \cite{GK, WZ} for examples in the good reduction case.
\\

In \cite{He-Rap} section 7, the He-Rapoport axioms were verified for the Siegel modular varieties.  For PEL-type Shimura varieties associated to unramified groups of type $A$ and $C$ and to odd ramified unitary groups, these axioms were verified by He-Zhou in \cite{HZ}. For certain Shimura varieties of Hodge type associated to tamely ramified and residually split groups, in \cite{zhou isog parahoric} Zhou has verified these axioms for the Kisin-Pappas integral models constructed in \cite{Paroh}.
Hence in these cases we get the EKOR stratifications by \cite{He-Rap} section 6. For a hyperspecial or an Iwahori level, by the above works of Viehmann or Lusztig and He, some basic geometric properties of EKOR strata, like smoothness and quasi-affineness, are known by the corresponding geometric properties of EO or KR strata.  However, for a general parahoric level, even if one had verified the He-Rapoport axioms, the smoothness of EKOR strata was still unknown. In fact, the geometric meaning of EKOR types was quite mysterious, since the construction in \cite{He-Rap} is purely group theoretic.
\\

In this paper we would like to find a direct and geometric construction of these strata, for some concrete integral models.
More precisely, similar to \cite{zhou isog parahoric}, we also work with the concrete integral models constructed by Kisin and Pappas in \cite{Paroh} for certain Shimura varieties of abelian type. The main results of Kisin and Pappas are that, under certain conditions, there exist some local model diagrams\footnote{The existence of the local model diagram is in fact one of the He-Rapoport axioms, and
	of course, one hopes that eventually all the He-Rapoport axioms should be verified for the Kisin-Pappas integral models.}, 
which relate the integral models of Shimura varieties with the Pappas-Zhu local model schemes \cite{local model P-Z}. 
Roughly speaking, our construction of the EKOR stratification will be about certain refinement of the local model diagram in characteristic $p$. It can be viewed as a geometric realization of the group theoretical considerations in \cite{He-Rap} section 6. 
\\

To state our main results, we need some notations. Let $p>2$ be a fixed prime throughout the paper. Let $(G,X)$ be a Shimura datum of abelian type, $\K=K_pK^p\subset G(\Ab_f)$ an open compact subgroup with $K^p\subset G(\Ab_f^p)$ sufficiently small and $K_p\subset G(\Q_p)$ a \emph{parahoric} subgroup. We have the associated Shimura variety $\Sh_\K=\Sh_{\K}(G,X)$ over the reflex field $\E$. Let $v|p$ be a place of $\E$ and $E=\E_v$.
Let $x$ be a point of the Bruhat-Tits building $\B(G,\Q_p)$, with the attached Bruhat-Tits stabilizer group scheme $\G=\G_x$ and its neutral connected component $\G^\circ=\G_x^\circ$, such that $K_p=\G^\circ(\Z_p)$. Assume that $G$ splits over a tamely ramified extension of $\Q_p$ and $p\nmid |\pi_1(G^{\text{der}})|$.
We will consider the following cases:
\begin{itemize}
	\item $(G,X)$ is of Hodge type and $\G=\G^\circ$.
	\item $(G,X)$ is of abelian type such that $(G^\ad,X^\ad)$ has no factors of type $D^\mathbb{H}$ (we refer to \cite{varideshi} Table 2.3.8 for the precise meaning of type $D^\mathbb{H}$).
	\item $(G,X)$ is of abelian type such that $G$ is unramified over $\Q_p$ and $K_p$ is contained in some hyperspecial subgroup of $G(\Q_p)$.
\end{itemize}
Let $\ES_{\K}=\ES_{\K}(G,X)$ be the integral model over $O_E$ of the Shimura variety $\Sh_{\K}$ constructed by Kisin-Pappas in \cite{Paroh}. The above cases are exactly when we have the local model diagram by  Theorem 4.2.7 (for the first case) and Theorem 4.6.23 (for the last two cases) of \cite{Paroh}.  This is also why we restrict to Shimura data in the above cases (note that each case may overlap with the others, and the first case is used to deduce the other cases in \cite{Paroh} 4.6).
We are interested in the geometry of the special fiber $\ES_{\K,0}$ over $k=\ov{\mathbb{F}}_p$. When $K_p$ is hyperspecial (thus $G$ is unramified), $\ES_{\K,0}$ is smooth. In the general case, by \cite[Corollary 0.3]{Paroh}, $\ES_{\K,0}$ is reduced, and the strict henselizations of the local rings on $\ES_{\K,0}$ have irreducible components which are normal and Cohen-Macaulay.
We will simply write $\ES_0$ for the special fiber when the level $\K$ is fixed. 
\\

We have in fact two approaches for the constructions of EKOR strata on $\ES_0$:
\begin{itemize}
	\item 
	a local construction which uses the theory of $G$-zips due to Moonen-Wedhorn \cite{disinv} and Pink-Wedhorn-Ziegler \cite{zipdata, zipaddi}. Here ``local'' means that we first work on a fixed KR stratum; then we let the KR stratum vary;
	\item a global construction which uses the theory of local $\G$-Shtukas, generalizing some constructions of Xiao-Zhu in \cite{XiaoZhu} (see also \cite{Zhu2}) in the case of good reductions. Here ``global'' means that we work directly on the whole special fiber (up to perfection).
\end{itemize}
To explain the ideas, we assume that $(G,X)$ is of Hodge type for simplicity. Let $G=G_{\Q_p}$ and $\{\mu\}$ be the attached Hodge cocharacter of $G$. Kottwitz and Rapoport defined a  finite subset \[\Admu\] of the Iwahori Weyl group of $G$, \emph{the $\{\mu\}$-admissible set},  cf. \cite{KR, guide to red mod p}. Writing $K=K_p$, from $\Admu$ we get finite sets $\Admu_K$ and ${}^K\Admu$, which will parametrize the types of the KR stratification and EKOR stratification of level $K$ respectively. We have the relations \[\Admu\supset {}^K\Admu \twoheadrightarrow \Admu_K.\] We have a partial order $\leq_K$ on $\Admu_K$, induced from the Bruhat order on the associated affine Weyl group.  On the finite set ${}^K\Admu$, we have also a partial order $\leq_{K,\sigma}$, introduced by He in \cite{mini leng double coset} section 4.
See \ref{Iwahori Weyl} for more details on these sets $\Admu, \Admu_K$ and ${}^K\Admu$. 
\\

The starting point of our local construction is the observation that the EKOR stratification (constructed in \cite{He-Rap} section 6) is a refinement of the KR stratification, and for a KR type $w\in \Admu_K$,
there is always an attached algebraic zip datum $\mathcal{Z}_w$ (see \ref{zip in EKOR}) in the sense of Pink-Wedhorn-Ziegler. Let $\G_0=\G\otimes_{\Z_p}\ov{\mathbb{F}}_p$ and $\G_0^{\rdt}$ be its reductive quotient. Let $M^{\loc}$ be the special fiber over $k=\ov{\mathbb{F}}_p$ of the Pappas-Zhu local model scheme attached to the triple $(G,\{\mu\}, K)$ constructed in \cite{local model P-Z}. Then by construction $\G_0$ acts on $M^{\loc}$ and the underling topological space $|[\G_0\backslash M^{\loc}]|$ of the quotient stack $[\G_0\backslash M^{\loc}]$ is homeomorphic to $\Admu_K$ (for which the topology is defined by its partial order $\leq_K$).
Under the above assumptions, the existence of the local model diagram gives us a morphism of algebraic stacks
\[\lambda_K: \ES_0\ra [\G_0\backslash M^{\loc}],\]such that the fibers are the KR strata of $\ES_0$. We identify the finite Weyl group $W_K$ attached to $K$ with the Weyl group of $\G_0^{\rdt}$. There exists an explicit set $J_w$ (see \ref{zip in EKOR}) of simple reflections in $W_K$ defined from $w$ and $K$.
After choosing a suitable Siegel embedding, we can construct a $\G_0^{\rdt}$-zip of type $J_w$ on the KR stratum $\ES_0^w$ attached to $w$. Let $\G_0^{\rdt}$-$Zip_{J_w}$ be the algebraic stack over $k$ of $\G_0^{\rdt}$-zips of type $J_w$.  By \cite{zipaddi} we have an isomorphism of algebraic stacks $\G_0^{\rdt}$-$Zip_{J_w}\simeq [E_{\mathcal{Z}_w}\backslash\mathcal{G}_0^{\mathrm{rdt}}]$, where $E_{\mathcal{Z}_w}$ is the zip group attached to $\mathcal{Z}_w$, see Definition \ref{D: alg zip data}.
Thus we get a morphism of algebraic stacks \[\zeta_w:\ES_0^w\rightarrow [E_{\mathcal{Z}_w}\backslash\mathcal{G}_0^{\mathrm{rdt}}]\] whose fibers are precisely the EKOR strata in $\ES_0^w$. 

Let $\pi_K: {}^K\Admu\ra \Admu_K$ be the natural surjection, then the local construction gives us a geometrization of $\pi_K^{-1}(w)$, as we have the following bijection
\[\pi_K^{-1}(w)\simeq {}^{J_w}W_K\simeq |[E_{\mathcal{Z}_w}\backslash\mathcal{G}_0^{\mathrm{rdt}}]|,\]
where ${}^{J_w}W_K\subset W_K$ is the set of minimal length representatives for $W_{J_w}\backslash W_K$.
The strategy of the construction is similar to \cite{EOZ}, but making more systematical use of the local model diagram. We refer to subsections \ref{subsec--pointwise constr} and \ref{subsection EO in KR} for details of the construction. 

The main property of the morphism $\zeta_w$ is the following result.
\begin{introth}[Theorem \ref{sm of zeta}]
	The morphism $\zeta_w$ is smooth.
\end{introth}
As a consequence, we can translate many geometric properties of the quotient stack $[E_{\mathcal{Z}_w}\backslash\mathcal{G}_0^{\mathrm{rdt}}]$ to those of $\ES_0^w$. In particular, each EKOR stratum is a locally closed smooth subvariety of $\ES_0^w$, and the closure of an EKOR stratum in $\ES_0^w$ is a union of EKOR strata. Moreover, we can also describe the dimension of an EKOR  stratum (if non-empty) once we know the dimension of the KR stratum containing it.
\\

The disadvantage of the local construction is that the type $J_w$ varies on different KR strata. The theory developed in \cite{zipdata, zipaddi} is not enough to put these $\G_0^{\rdt}$-zips of different types uniformly together. In particular, when we want to show the closure relation of EKOR strata defined locally as above, we will meet a serious problem. To overcome these difficulties,
in section \ref{section global} we adapt some ideas of Xiao and Zhu in \cite{XiaoZhu}. One of the key observations is that by the works of Scholze (\cite[Corollary 21.6.10]{SW}) and He-Pappas-Rapoport (\cite[Theorem 2.15]{He-Pap-Rap}), the (perfection of the) special fiber of the Pappas-Zhu local model admits an embedding into the Witt vector affine flag varieties $\Gr_\G$ (\cite{zhu-aff gras in mixed char,BS}). 
So we introduce the notions of local $(\G,\mu)$-Shtukas and their truncations in level 1, cf. Definition \ref{D: shtuka} and subsection \ref{subsection restricted loc sht}, generalizing those in \cite{XiaoZhu} in the unramified case. Roughly, a local $(\G,\mu)$-Shtuka over a perfect ring $R$ is a $\G$-torsor $\El$ over $W(R)$, together with an isomorphism $\beta: \sigma^\ast\El[\frac{1}{p}]\stackrel{\sim}{\ra} \El[\frac{1}{p}]$ over $W(R)[\frac{1}{p}]$, such that the relative position between $\sigma^\ast\El$ and $\El$ at any geometric point of $\Spec\,R$ lies in $\Admu_K$.  Here $\sigma: W(R)\ra W(R)$ is the Frobenius.
Let $\Aff^{pf}_k$ be the category of perfect $k$-algebras.
We need to pass to the world of perfect algebraic geometry as in \cite{zhu-aff gras in mixed char} and \cite{XiaoZhu}. 

Consider the prestack $\Sht_{\mu,K}^{\loc}$ of local $(\G,\mu)$-Shtukas, which can be described as (cf. Lemma \ref{L: pretacks sht}) \[\Sht_{\mu,K}^{\loc}\simeq \Big[\frac{M^{\loc,\infty}}{Ad_\sigma L^+\G}\Big],\]where $M^{\loc,\infty}\subset LG$ is the pre-image of $M^{\loc}\subset \Gr_\G=LG/L^+\G$, and the quotient means that we take the $\sigma$-conjugation action of $L^+\G$ on $M^{\loc,\infty}$. Consider the reductive 1-truncation group $L^{1-\rdt}\G$, i.e. for any $R\in \Aff^{pf}_k$, $L^{1-\rdt}\G(R)=\G_0^{\rdt}(R)$. Let $L^+\G^{(1)-\rdt}:=\ker (L^+\G\ra L^{1-\rdt}\G)$
and \[M^{\loc,(1)-\rdt}\subset LG/L^+\G^{(1)-\rdt}\] be the image of $M^{\loc,\infty}\subset LG$ under the projection $LG\ra LG/L^+\G^{(1)-\rdt}$. 
Fix any integer $m\geq 2$,  we have the algebraic stack of $(m,1)$-restricted local $(\G,\mu)$-Shtukas
\[\Sht_{\mu,K}^{\loc(m,1)}= \Big[\frac{M^{\loc,(1)-\rdt}}{Ad_\sigma L^m\G}\Big]. \]Here $L^m\G$ is the $m$-truncation group, i.e. for any $R\in \Aff^{pf}_k$, $L^m\G(R)=\G(W_m(R))$.
There are natural perfectly smooth morphisms \[\Sht_{\mu,K}^{\loc}\ra \Sht_{\mu,K}^{\loc(m,1)}\stackrel{\pi_K}{\lra} [\G_0\backslash M^{\loc}].\] Recall that
we have a homeomorphism of topological spaces $\Admu_K\simeq |[\G_0\backslash M^{\loc}]|$.
By the works of Lusztig and He,  we have a homeomorphism of topological spaces (see Lemma \ref{L: top restricted sht})
\[{}^K\Admu\simeq |\Sht_{\mu,K}^{\loc(m,1)}|,\]
where the topology on ${}^K\Admu$ is defined by the partial order $\leq_{K,\sigma}$.
By the works of Hamacher-Kim \cite{Ham-Kim} and Pappas \cite{Pappas Ober}, we have a universal local $(\G,\mu)$-Shtukas over $\ES_{0}^{pf}$, the perfection of the special fiber of our Shimura variety. Taking its $(m,1)$-restriction, we get a morphism of algebraic stacks
\[\upsilon_K: \ES_{0}^{pf}\ra  \Sht_{\mu,K}^{\loc(m,1)},\]
which lifts the morphism of algebraic stacks $\lambda_K: \ES_{0}^{pf}\ra [\G_0\backslash M^{\loc}]$ induced by the local model digram. 
\begin{introth}[Theorem \ref{T: perf smooth}]
	The morphism $\upsilon_K$ is perfectly smooth.
\end{introth}
The relation between the local and global constructions is as follows.
Recall that we have $\pi_K: \Sht_{\mu,K}^{\loc(m,1)}\ra [\G_0\backslash M^{\loc}]$.
For any $w\in |[\G_0\backslash M^{\loc}]|$, there is a natural perfectly smooth morphism \[\pi_K^{-1}(w)\ra [E_{\mathcal{Z}_w}\backslash\mathcal{G}_0^{\mathrm{rdt}}]^{pf}\] which induces a homeomorphism of underlying topological spaces, see Proposition \ref{P: G zip and shtukas}. Here $[E_{\mathcal{Z}_w}\backslash\mathcal{G}_0^{\mathrm{rdt}}]^{pf}$ is the perfection of the algebraic stack $[E_{\mathcal{Z}_w}\backslash\mathcal{G}_0^{\mathrm{rdt}}]$. The morphism $\upsilon_K$ thus interpolates the morphisms $\zeta_w^{pf}$ when $w\in \Adm(\{\mu\})_K$ varies.
\\

The fibers of $\upsilon_K$ give the EKOR strata on $\ES_{0}^{pf}$. Since the perfection does not change underlying topological spaces, we get the closure relation of EKOR strata on $\ES_{0}$.
\\


We collect some of the main results for EKOR strata as follows. 
\begin{introth}\label{Thm C}
	\begin{enumerate}
		\item {\rm (Theorem \ref{thm--first properties}, Theorem \ref{T: EKOR abelian} (1), Corollary \ref{C: closure relation})} 
		We have the EKOR stratification
		\[\ES_0=\coprod_{x\in {}^K\Admu}\ES_0^x.\]
		For each $x\in {}^K\Admu$, the corresponding EKOR stratum $\ES_0^x$ is a locally closed subvariety which is non-empty, smooth, equi-dimensional of dimension $\ell(x)$. Moreover, we have the closure relation
		\[\ov{\ES_0^x}=\coprod_{x'\leq_{K,\sigma}x}\ES_0^{x'}.\]
		
		\item {\rm (Theorem \ref{thm--quasi affine}, Theorem \ref{T: EKOR abelian} (2))} If $K$ is Iwahori, then any KR stratum is quasi-affine. In general, if the He-Rapoport axiom 4 (c) holds, any EKOR stratum is quasi-affine. In particular, if $G_{\RAT_p}$ is residually split, then any EKOR stratum is quasi-affine.
	\end{enumerate}
\end{introth}
In (1), the non-emptiness is deduced by results of  several other authors, see for example \cite{KMS, Yu, zhou isog parahoric}. We refer to Corollary \ref{coro--non-emp and closure} for more details. The smoothness and dimension formula are conjectures of He and Rapoport in \cite{He-Rap}, based on their axioms together with some different observations (see also our discussions after Theorem \ref{intro thm--etale}). It is an interesting question to investigate the singularities of the closure of an EKOR stratum. If the level $K$ is Iwahori, then it is known that each closure of KR stratum is normal and Cohen-Macaulay, see \cite[Corollary 4.2.12]{Paroh} and \cite[Theorem 1.2]{local model P-Z}. We don't know whether this is true for a general parahoric level.

Quasi-affineness of EKOR strata is also conjectured by He and Rapoport in \cite{He-Rap}. Statement (2) says that this conjecture holds as long as the He-Rapoport axiom 4 (c) holds. It is reduced to quasi-affineness of Iwahori KR strata using Theorem \ref{intro thm--etale}. Noting that axiom 4 (c) has been verified by Zhou if $G_{\RAT_p}$ is residually split \cite{zhou isog parahoric}, EKOR strata are quasi-affine in this case. Quasi-affineness of Iwahori KR strata  in the Siegel case was known by the work of G\"ortz-Yu \cite{Iwahoric Siegel-Gor-Yu}.

It is somehow surprising that quasi-affineness of Iwahori KR strata is proved using techniques to study EKOR strata: KR strata are defined using the local model diagram, so it sounds reasonable to expect that one could study Iwahori KR strata simply using the local model diagram. On the contrary, our proof of the quasi-affineness of Iwahori KR strata  does use techniques from the study of EKOR strata, and in particular it involves the morphism $\zeta_w$ which contains the geometric information that can not be seen from the local model diagram. We refer to the proof of Theorem \ref{thm--quasi affine} for more details.
\\

We can define various ordinary loci and superspecial loci using EKOR strata. More precisely, 
\begin{enumerate}
	\item we can do this for each KR stratum. Namely, for $w\in \Admu_K$, in $\ES_0^w$,
	\begin{enumerate}
		\item there is a unique EKOR stratum, namely $\ES_0^{{}^Kw_K}$ with ${}^Kw_K$ as in \ref{ordi rep of KR}. This stratum is open dense in $\ES_0^w$ by the smoothness of the map $\zeta_w$, and is called the $w$\emph{-ordinary locus};
		\item there is a unique EKOR stratum, namely $\ES_0^{x_w}$.  Here $x_w$ is as in \ref{subsubsection two sections}. This stratum is closed in $\ES_0^w$, and is called the $w$-\emph{superspecial locus};
	\end{enumerate}
	\item  we can also do this globally in $\ES_0$. Namely,
	\begin{enumerate}
		\item the \emph{ordinary locus} $\ES_0^{\mathrm{ord}}$ is the union of EKOR strata $\ES_0^x$ with $\ell(x)=\mathrm{dim}(\ES_0)$. It is clear by definition that $$\ES_0^{\mathrm{ord}}=\coprod_{w\in \Admu_K,\atop\ell({}^Kw_K)=\mathrm{dim}(\ES_0)}\ES_0^{{}^Kw_K},$$ and it is an open dense smooth subscheme in $\ES_0$ (note that the density follows from the smoothness of the maps $\zeta_w$ for all $w\in \Admu_K$). By a result of He-Nie in \cite{mu-ordi of SHV} (see the following Proposition \ref{P: max EKOR}) we can rewrite the above disjoint union as
		\[\ES_0^{\mathrm{ord}}=\coprod_{\mu'\in W_0(\ul{\mu}), t^{\mu'}\in {}^K\IW}\ES_0^{t^{\mu'}}.\] We refer to Proposition \ref{P: max EKOR} for the precise meaning of the notation $W_0(\ul{\mu})$,
		\item the \emph{superspecial locus} is the EKOR strata $\ES_0^\tau$, where $\tau$ is as in \ref{tau}. $\ES_0^\tau$ is the unique closed EKOR stratum in $\ES_0$, and it is of dimension 0. Moreover, $\ES_0^\tau$ is a central leaf (see \cite{Kim} for some discussions on central leaves in the reduction of Hodge-type Shimura varieties with parahoric level structure). 
	\end{enumerate}
\end{enumerate}

Let $I\subset K$ be an Iwahori subgroup. Fix the prime to $p$ level $K^p$ and we simply write $\ES_I=\ES_{IK^p}$ and $\ES_K=\ES_{KK^p}$.
Then by \cite{zhou isog parahoric} section 7, we have a morphism between special fibers \[\pi_{I,K}: \ES_{I,0}\ra \ES_{K,0}.\]
\begin{introth}[Theorem \ref{thm--etale forg level}, Theorem \ref{T: EKOR abelian} (3)] \label{intro thm--etale}
	For $x\in {}^K\Admu$ viewed as an element of ${}^I\Admu=\Admu$, the morphism \[\pi_{I,K}^x:\ES_{I,0}^x\rightarrow \ES_{K,0}^x\] induced by $\pi_{I,K}$ is finite \'{e}tale. If in addition the He-Rapoport axiom 4 (c) is satisfied, $\pi_{I,K}^x$ is a finite \'{e}tale covering.
\end{introth}

The \'{e}taleness of $\pi_{I,K}^x$ is also a conjecture of He and Rapoport in \cite{He-Rap}. They prove that $\pi_{I,K}^x$ is finite and surjective assuming their five axioms, and deduce the dimension formula from it. If $\pi_{I,K}^x$ is, in addition, \'{e}tale, then $\pi_{I,K}^x:\ES_{I,0}^x\rightarrow \ES_{K,0}^x$ becomes a finite \'{e}tale covering, and hence the smoothness of $\ES_{K,0}^x$ follows directly from that of $\ES_{I,0}^x$. We only need to prove the \'{e}taleness of $\pi_{I,K}^x$, as the finiteness is actually a consequence of the axioms (without axiom 4 (c)) together with the \'{e}taleness of $\pi_{I,K}^x$. 
\\

We can apply the results of EKOR strata to study Newton strata and central leaves for reductions of Shimura varieties with parahoric level. By \cite{class of F-isocrys}, using the universal $F$-isocrystal with $G$-structure in the Hodge type case, we can define the Newton stratification on $\ES_0$. Then we can extend to the case of abelian type as in \cite{stra abelian-good redu}. Similarly we can define central leaves in this setting.
In \cite{He-Rap} subsection 6.6, He and Rapoport proved that each Newton stratum contains a certain EKOR stratum in their setting, in particular for Shimura varieties with integral models satisfying their axioms.
In \cite{fully H-N decomp} G\"ortz-He-Nie introduced the notion of fully Hodge-Newton decomposable pairs $(G,\{\mu\})$, and they made a deep study for these pairs. In particular, under the assumption that the He-Rapoport axioms were verified,  they proved that for Hodge-Newton decomposable Shimura varieties, each Newton stratum is a union of certain EKOR strata. The methods to prove the above mentioned results in \cite{He-Rap} and \cite{fully H-N decomp} are group theoretic.
With our geometric constructions of EKOR strata at hand, these results of \cite{He-Rap} and \cite{fully H-N decomp} become unconditional for the Kisin-Pappas integral models.
\begin{introth}
	\begin{enumerate}
		\item {\rm (Corollary \ref{coro--first properties}, Corollary \ref{C: straight EKOR})} Each Newton stratum contains an EKOR stratum $\ES_0^x$ such that $x$ is $\sigma$-straight. Moreover, $\ES_0^x$ is a central leaf of dimension $\langle \rho, \nu(b)\rangle$. Here $\rho$ is the half sum of positive roots, and $\nu(b)$ is the Newton cocharacter of the Newton stratum. 
		\item {\rm (Theorem \ref{EKOR vs NP})} Assume that the attached pair $(G,\{\mu\})$ is fully Hodge-Newton decomposable. Then
		\begin{enumerate}
			\item each Newton stratum of $\ES_{0}$ is a union of EKOR strata;
			
			\item each EKOR stratum in a non-basic Newton stratum is an adjoint central leaf, and it is open and closed in the Newton stratum, in particular, non-basic Newton strata are smooth;
			
			\item the basic Newton stratum is a union of certain Deligne-Lusztig varieties.
		\end{enumerate}
	\end{enumerate}
\end{introth}

We refer the reader to subsections \ref{subsection ADLV} and \ref{subsection ADLV global} for geometric constructions of EKOR strata for affine Deligen-Lusztig varieties, and to Propositions \ref{P: ADLV and EKOR} and \ref{P: uniformization EKOR} for the relation between local and global EKOR strata.
\\

In section \ref{section example} we discuss the case of Siegel modular varieties and study the example of Siegel threefolds in details. This is also an example of fully Hodge-Newton decomposable Shimura varieties.
\\

We briefly describe the structure of this article. 
In the first section we recollect some facts about $G$-zips, the Iwahori Weyl group and some related group theoretic sets, which will be used later. In section \ref{section loc mod}, we review some constructions of the Pappas-Zhu local models and the Kisin-Pappas integral models of Shimura varieties of abelian type, which are the objects to be studied in this paper. In section \ref{section EKOR loc}, we construct and study the EKOR stratification for Shimura varieties of Hodge type by a local method.
More precisely, we construct a $\G_0^{\rdt}$-zip and thus an EO stratification on each KR stratum. In section \ref{section global}, we give some global constructions of the EKOR strata by adapting and generalizing some ideas of Xiao-Zhu in \cite{XiaoZhu}. More precisely, we will introduce the notions of local $(\G,\mu)$-Shtukas and their truncations in level 1. We study their moduli and apply them to our Shimura varieties.
In section \ref{section abelian}, we extend our constructions to the abelian type case. We also apply the results of EKOR strata to the study of Newton strata and central leaves for these Shimura varieties. In section \ref{section example}, we discuss the Siegel case and investigate the example of $\GSp_4$ in details. Finally, in the appendix (Proposition \ref{PEL.8}) we verify He-Rapoport's axiom 4 (c) for Shimura varieties of
PEL-type, which improves some of our main
results (Theorem \ref{Thm C} (2) and Theorem \ref{intro thm--etale}) for the PEL-type case.\\
\\ 
\textbf{Acknowledgments.} 
The key ideas for the global constructions using moduli stacks of restricted local Shtukas in section \ref{section global} were inspired after the first author talked with Xinwen Zhu and attended George Pappas's talk in the Oberwolfach workshop ``Arithmetic of Shimura varieties'' in January 2019.  We would like to thank them sincerely.
We also thank the organizers of this workshop: Laurent Fargues, Ulrich G\"ortz, Eva Viehmann, and Torsten Wedhorn, for the invitation (for X. S. and C. Z.)  so that this discovery was made possible. 
The authors thank
Ulrich G\"ortz, Xuhua He, Mark Kisin, George Pappas, Michael Rapoport, Peter Scholze, Liang Xiao, Rong Zhou, Xinwen Zhu for their inspiring works on which the present article is partially based.
We also thank the referees for careful reading and suggestions on the improvements of exposition. 
The first author was partially supported by the Chinese Academy of Sciences grants 50Y64198900, 29Y64200900, the Recruitment Program of Global Experts of China, and the NSFC grants No. 11631009 and No. 11688101. The second author was partially supported by the MoST grants 107-2115-M-001-001-MY2 and 109-2115-M-001-002-MY3.

\newpage

\section[Group theoretic preparations]{Recollections of some group theoretic results}\label{section group}

In this section, we review and collect some known facts about $G$-zips, the Iwahori Weyl group and some related group theoretic sets, which will be used later. Fix a prime $p$.

\subsection[Algebraic zip data]{Algebraic zip data}\label{subsection-zip data}

For a linear algebraic group $H$ over a field, we denote its unipotent radical by $R_uH$. For an element $h\in H$, let $\ov{h}$ be its image in $H/R_uH$.
\begin{definition}(\cite[Definition 1.1]{zipdata})\label{D: alg zip data}
	An algebraic zip datum is a tuple $\mathcal {Z}=(G, P, Q, \sigma')$ consisting of a (connected) reductive group $G/\mathbb{F}_p$, together with parabolic subgroups $P$ and $Q$ defined over a finite extension $k/\mathbb{F}_p$, and an isogeny $\sigma':P/R_uP\rightarrow Q/R_uQ$. The group $$E_{\mathcal {Z}}:=\{(p,q)\in P\times Q\mid \sigma'(\overline{p})=\overline{q}\}$$ is called the \emph{zip group} attached to $\mathcal {Z}$; it acts on $G_k$ through the map $((p,q),g)\mapsto pgq^{-1}$.
\end{definition}

By abuse of notation, we still denote by $E_{\mathcal {Z}}$ the base change to $\ov{k}$ of the zip group attached to $\mathcal {Z}$.
We are interested in the decomposition of $G_{\ov{k}}$ into $E_\mathcal{Z}$-orbits. To describe it, we fix a Borel subgroup $B$ of $G_{\overline{k}}$, a maximal torus $T\subset B$ and an element $g\in G(\overline{k})$ such that $B\subset Q_{\ov{k}}$, ${}^gB\subset P_{\ov{k}}$, $\sigma'(\overline{{}^gB})=\overline{B}$ and $\sigma'(\overline{{}^gT})=\overline{T}$. Here $\overline{B}$ is the image of $B$ in $Q_{\ov{k}}/R_uQ_{\ov{k}}$, and similarly for the other objects. Let $W$ be the Weyl group of $G_{\ov{k}}$ with respect to $T$ and $\texttt{S}\subset W$ be the set of simple reflections corresponding to $B$. Let $J\subset \texttt{S}$ be the type of $P$, $W_J$ be the subgroup of $W$ generated by $J$, and ${}^JW$ be the set of minimal length representatives for $W_J\backslash W$. For $w\in {}^JW\subset W=N_G(T)/T$, let $\dot{w}\in N_G(T)$ be a representative which maps to $w$.
We set $$G^w:=E_{\mathcal {Z}}\cdot gB\dot{w}B.$$

\begin{theorem}{\rm(\cite[Theorem 1.3, Proposition 7.1, Proposition 7.3]{zipdata})} \label{thm on zip orbits} The subsets $G^w$ for $w\in {}^JW$ form a pairwise disjoint decomposition of $G_{\overline{k}}$ into locally closed smooth subvarieties. The dimension of $G^w$ is $\mathrm{dim}\ P+\ell(w)$. If the differential of $\sigma'$ at $1$ vanishes, $G^w$ is a single orbit of $E_\mathcal{Z}$.
\end{theorem}
\begin{remark}
	Using the choice of $(B,T,g)$, we can identify $P/R_uP$ (resp. $Q/R_uQ$) with a Levi subgroup $L_P$ (resp. $L_Q$) of $P$ (resp. $Q$), and view $\sigma'$ as an isogeny $L_P\rightarrow L_Q$. We can rewrite $E_\mathcal{Z}$ as $$E_{\mathcal {Z}}=\{(u_1l,\ u_2\sigma'(l))\mid u_1\in R_uP,\ u_2\in R_uQ,\ l\in L_P \}.$$
\end{remark}
\begin{remark}\label{group--ord and supspe}
	The closure of $G^w$ is a union of $G^{w'}$s as described in \cite[Theorem 1.4]{zipdata}. In particular, there is a unique open dense stratum (given by the unique maximal element in ${}^JW$), called the ordinary locus; and there is a unique closed stratum (given by $\mathrm{id}\in {}^JW$), called the superspecial locus.
\end{remark}

We will need to construct morphisms from a scheme $S$ over $k$ to the quotient stack $[E_\mathcal{Z}\backslash G_k]$. We refer the reader to \cite{stackL} for some basics about algebraic stacks.
A morphism \[f:S\rightarrow [E_\mathcal{Z}\backslash G_k]\] is, by definition, an $E_\mathcal{Z}$-torsor $\mathcal{E}$ over $S$ together with an $E_\mathcal{Z}$-equivariant morphism $f^\#:\mathcal{E}\rightarrow G_k$. However, it is not always obvious how to construct $(\mathcal{E},f^\#)$ directly from certain structures (e.g. Dieudonn\'{e} modules) on $S$. The following notation will be used in this paper.

\begin{definition}(\cite[Definition 3.1]{zipaddi})\label{def--G-zip}
	A $G$-zip of type $J$ over $S$ is a tuple $(I,I_P,I_Q,\iota')$, where 
	\begin{itemize}
		\item $I$ is a right $G$-torsor over $S$, 
		\item $I_P\subset I$ is a right $P$-torsor over $S$, 
		\item $I_Q\subset I$ is a right $Q$-torsor over $S$, and 
		\item $\iota':I_P/R_uP\rightarrow I_Q/R_uQ$ is a $P/R_uP$-equivariant morphism over $S$, i.e. we have $\iota'(xp)=\iota'(x)\sigma'(p)$, for all $x\in I_P/R_uP$ and $p\in P/R_uP$.
	\end{itemize}
\end{definition}
Let $G$-$Zip_J$ be the stack of $G$-zips of type $J$ over $k$. 
By \cite[Theorem 12.7]{zipdata} and \cite[Theorem 1.5]{zipaddi}, we have an isomorphism of algebraic stacks \[G\text{-}Zip_J\simeq [E_\mathcal{Z}\backslash G_k],\]and they are smooth algebraic stacks of dimension 0 over $k$.
Thus
to give a morphism \[S\rightarrow [E_\mathcal{Z}\backslash G_k]\] is equivalent to give a $G$-zip of type $J$ over $S$. We will only explain how to construct a morphism $S\rightarrow [E_\mathcal{Z}\backslash G_k]$ from a $G$-zip $(I,I_P,I_Q,\iota')$. Let $\mathcal{E}$ be the cartesian product \[\xymatrix{\mathcal{\mathcal{E}}\ar[d]\ar[rr]& &I_Q\ar[d]\\
	I_P\ar[r]&I_P/R_uP\ar[r]^{\iota'}& I_Q/R_uQ.}\]
It is an $E_{\mathcal{Z}}$-torsor over $S$. There is a morphism $f^\#:\mathcal{E}\rightarrow G_k$ as follows. For $t=(t_1,t_2)\in \mathcal{E}(S)$ with $t_1\in I_P(S)$ and $t_2\in I_Q(S)$ such that $\iota'(t_1R_uP)=t_2R_uQ$, there is a unique $g\in G(S)$ such that $t_1g=t_2$. We set $f^\#(t):=g$. One checks easily that $f^\#$ is $E_{\mathcal{Z}}$-equivariant. We will write $f^\#$ as the composition \begin{subeqn}\label{map to Grdt}\xymatrix@C=2.7cm{\mathcal{E}\ar[r]^{(t_1,t_2)\mapsto (t_1,t_2)}& I\times I\ar[r]^{(f_1,f_2)\mapsto d(f_1,f_2)}& G_k}.\end{subeqn}
Here $d(f_1,f_2)$ is the unique element in $G_k$ which takes $f_1$ to $f_2$.
\begin{remark}\label{remark on zip-1}
	Let $\sigma'$ be of the form $i \circ \sigma$, where $\sigma:P/R_uP\rightarrow P^{(p)}/R_uP^{(p)}$ is the relative Frobenius, and $i:P^{(p)}/R_uP^{(p)}\rightarrow Q/R_uQ$ is an isomorphism of group varieties. It is clear that the differential of $\sigma'$ at $1$ vanishes in this case, and hence Theorem \ref{thm on zip orbits} applies. Moreover, a $G$-zip $(I,I_P,I_Q,\iota')$ is equivalent to the tuple $(I,I_P,I_Q,\iota)$, where $\iota$ is the induced $P^{(p)}/R_uP^{(p)}$-equivariant isomorphism \[(I_P/R_uP)^{(p)}\rightarrow I_Q/R_uQ.\]
\end{remark}

Let $\omega_0$ be the element of maximal length in $W$, and $\sigma: W\ra W$ be the Frobenius map. Set
$K:={}^{\omega_0}\sigma(J)$. Here we write ${}^g\!J$ for
$gJg^{-1}$. Let \[x\in {}^K\!W^{\sigma(J)}\] be the element of
minimal length in $W_K\omega_0W_{\sigma(J)}$. Then $x$ is the
unique element of maximal length in ${}^K\!W^{\sigma(J)}$ (see
\cite{VW} 5.2). There is a partial order $\preceq$ on ${}^JW$,
defined by $w'\preceq w$ if and only if there exists $y\in W_J$, such that
\[yw'x\sigma(y^{-1})x^{-1}\leq w\] (see \cite[Definition 5.8]{VW}). Here $\leq$ is the Bruhat order (see A.2 of
\cite{VW} for the definition). As usual, the partial order $\preceq$ makes
${}^JW$ into a topological space (see \cite[Proposition 2.1]{zipaddi} for example). On the other hand, for an algebraic stack $X$, we have its underlying topological space $|X|$ (see \cite{stackL} section 5). By \cite[Proposition 2.2]{zipaddi}, the above
Theorem \ref{thm on zip orbits} actually tells us that we have a homeomorphism of topological spaces
\[|[E_\mathcal{Z}\backslash G_{\ov{k}}]|\simeq {}^JW.\]

\subsection[The Iwahori Weyl group]{The Iwahori Weyl groups}\label{Iwahori Weyl}

Our notations in this subsection will be different from those in the previous one, namely, $G$ will be a reductive group over $\mathbb{Q}_p$, and $\mathcal{G}$ will be a parahoric group scheme over $\INT_p$. Let $K=\G(\Z_p)\subset G(\Q_p)$ be the associated parahoric subgroup.
We also set $k:=\overline{\mathbb{F}}_p$, $\breve{\INT}_p:=W(k)$, $\breve{\mathbb{Q}}_p:=\breve{\INT}_p[\frac{1}{p}]$, $\breve{K}:=\mathcal {G}(\breve{\mathbb{Z}}_p)$, and $\breve{G}:=G(\breve{\mathbb{Q}}_p)$. Let $\sigma: \breve{\INT}_p\ra \breve{\INT}_p$ be the Frobenius map, which induces Frobenius maps on $\breve{\mathbb{Q}}_p, \breve{K}, \breve{G}$ and related objects.

Let $\mathcal{B}$ be the Bruhat-Tits building of $G_{\breve{\Q}_p}$ and let $T\subset G$ be a maximal torus which after extension of scalars is contained in a Borel subgroup $B$ of $G_{\breve{\mathbb{Q}}_p}$. The split part of $T$ defines an apartment of $\mathcal{B}$. Let $\mathfrak{a}$ be an alcove of $\mathcal{B}$ inside this apartment. 
Let $I\subset \mathcal {G}(\mathbb{Z}_p)$ be an Iwahori subgroup,  such that $\breve{I}$ is the Iwahori subgroup of $G(\breve{\mathbb{Q}}_p)$ fixing the alcove $\mathfrak{a}$. Let $N(T)$ be the normalizer of $T$. The \emph{Iwahori Weyl group} is given by \[\widetilde{W}=N(T)(\breve{\mathbb{Q}}_p)/(T(\breve{\mathbb{Q}}_p)\cap \breve{I}),\] and the relative Weyl group of $G_{\breve{\mathbb{Q}}_p}$ is given by \[W_0=N(T)(\breve{\mathbb{Q}}_p)/T(\breve{\mathbb{Q}}_p).\] Choosing a special vertex in the alcove corresponding to $\breve{I}$ which will be fixed once and for all, we have \[\widetilde{W}\cong X_*(T)_{\Gamma_0}\rtimes W_0\] and \[\widetilde{W}\cong W_{a}\rtimes \pi_1(G)_{\Gamma_0},\] where $\Gamma_0=\mathrm{Gal}(\overline{\mathbb{Q}}_p/\breve{\mathbb{Q}}_p)$ and $W_a$ is the affine Weyl group, which is a Coxeter group. One can then use the Bruhat partial order (resp. length function) on $W_a$ to define a partial order (resp. length function) on $\widetilde{W}$. To be more precise, for $w=(w_a,t)\in \widetilde{W}$ with $w_a\in W_a$ and $t\in \pi_1(G)_{\Gamma_0}$, we set $\ell(w)=\ell(w_a)$, and $w=(w_a,t)\leq w'=(w'_a,t')$ if an only if $w_a\leq w'_a$ and $t=t'$.

We recall the definition of $\{\mu\}$-admissible sets. For $\mu'\in X_\ast(T)_{\Gamma_0}$, we write $t^{\mu'}$ when viewed as an element in $\widetilde{W}$. Let $\mu: \mathbb{G}_{m,\ov{\Q}_p}\ra G_{\ov{\Q}_p}$ be a cocharacter and $\{\mu\}$ be the associated conjugacy class.
The attached $\{\mu\}$-\emph{admissible set} is the following finite subset of $\IW$ introduced by Kottwitz and Rapoport (cf. \cite{KR, guide to red mod p, PRS}): 
\begin{subeqn}\label{admi mu}\mathrm{Adm}(\{\mu\})=\{w\in \widetilde{W}\mid w\leq t^{x(\underline{\mu})}\text{ for some }x\in W_0\}.\end{subeqn}
Here $\underline{\mu}$ is the image in $X_\ast(T)_{\Gamma_0}$ of a dominant representative in the conjugacy class $\{\mu\}$. We will be interested in some related sets taking into account the information of the parahoric subgroup $K$.
Let \[W_K=(N(T)(\breve{\mathbb{Q}}_p)\cap \breve{K})/(T(\breve{\mathbb{Q}}_p)\cap \breve{I}),\] then by \cite[Proposition 12]{HR} we have an isomorphism 
\[W_K\simeq W_{\mathcal{G}_0^{\mathrm{rdt}}},\]where $\mathcal{G}_0=\mathcal{G}\otimes_{\mathbb{Z}_p} k$,  $\mathcal{G}_0^{\mathrm{rdt}}$ is its reductive quotient, and $W_{\mathcal{G}_0^{\mathrm{rdt}}}$ is the associated Weyl group.
Let ${}^K \IW\subset \IW$ be the subset of minimal length representatives for $W_K\backslash \IW$. We set\footnote{We learned the notation ${}^K\Admu$ from \cite{GHR}.}
\begin{itemize}
	\item $\Admu^K=W_K\Admu W_K, \text{ \ a subset of }\IW;$
	
	\item $\Admu_K=W_K\backslash\Admu^K/W_K, \text{ \ a subset of }W_K\backslash\IW/W_K;$
	
	\item ${}^K\Admu=\Admu^K\cap {}^K \IW, \text{ \ a subset of }{}^K\IW.$
	
\end{itemize}
The natural map ${}^K\IW\ra W_K\backslash\IW/W_K$ induces a surjection \[{}^K\Admu\twoheadrightarrow\Admu_K.\]
We have the following important result due to He (\cite{K-R on union of aff D-L} ) and Haines-He (\cite{HH}):
\begin{theorem}{\rm(\cite[Theorem 6.1]{K-R on union of aff D-L}, \cite[Proposition 5.1]{HH})}\label{T: EKOR set}
	We have \[{}^K\Admu=\Admu\cap {}^K \IW.\] 
\end{theorem}

\subsubsection[The partial order $\leq_{K,\sigma}$ on ${}^K\Admu$]{The partial order $\leq_{K,\sigma}$ on ${}^K\Admu$}\label{partial order} There is a partial order $\leq_{K,\sigma}$ on ${}^K\IW$. It is defined by \[x_1\leq_{K,\sigma} x_2\quad \mathrm{if\;  there\; exists}\; y\in W_K\; \mathrm{such\; that}\quad yx_1\sigma(y)^{-1}\leq x_2.\] Here $\leq$ is the Bruhat order. By \cite[4.7]{mini leng double coset}, $\leq_{K,\sigma}$ is indeed a partial order. By \cite[Remark 6.14]{He-Rap}, $\leq_{K,\sigma}$ is finer than the Bruhat order in general. We will then restrict the partial order $\leq_{K,\sigma}$ on ${}^K\IW$ to ${}^K\Admu$.
\begin{proposition}{\rm(\cite[Proposition 2.1]{mu-ordi of SHV})}\label{P: max EKOR}
	The maximal elements in ${}^K\Admu$ with respect to the partial order $\leq_{K,\sigma}$ are $t^{\mu'}$, where $\mu'$ runs over elements in the $W_0$-orbit of $\ul{\mu}$ with $t^{\mu'}\in {}^K\IW$.
\end{proposition}

\subsubsection[The element $\tau$]{The element $\tau$}\label{tau} There is a distinguished element $\tau_{\{\mu\}}\in \IW$ attached to the conjugacy class $\{\mu\}$. In terms of the isomorphism $\widetilde{W}\cong W_{a}\rtimes \pi_1(G)_{\Gamma_0}$, it corresponds to the element \[(\mathrm{id},\mu^{\#})\in W_{a}\rtimes \pi_1(G)_{\Gamma_0},\] where $\mu^{\#}$ is the common image of $\mu\in \{\mu\}$ in $\pi_1(G)_{\Gamma_0}$. It is then clear that $\tau_{\{\mu\}}$ is of length 0, and lies in ${}^K\Admu$. It is the minimal element in ${}^K\Admu$ with respect to the above partial order $\leq_{K,\sigma}$.

The conjugacy class $\{\mu\}$ will be fixed in our practical applications, and hence we will simply write $\tau$ for $\tau_{\{\mu\}}$.

\subsubsection[${}_K\widetilde{W}^K$ and ${}^K\widetilde{W}_K$]{${}^K \IW^K, {}_K\widetilde{W}^K$ and ${}^K\widetilde{W}_K$}\label{ordi rep of KR} 

Let ${}^K \IW^K\subset \IW$ be the subset of minimal length representatives for $W_K\backslash \IW/W_K$.
Since each double coset $W_KwW_K$ admits a unique element of minimal length, we will identify $\Admu_K$ as a subset of ${}^K \IW^K$, which is again denoted by $\Admu_K$. For $w\in \Admu_K$, by $\ell(w)$ we mean that we view $w\in{}^K \IW^K$ and  $\ell(w)$ is its length.  Sometimes we will also write $x_w\in {}^K \IW^K$ for the corresponding $w\in \Admu_K$ to distinguish the notation.

A certain subset ${}_K\widetilde{W}^K\subset \widetilde{W}$ is introduced in \cite{Richarz 1} to describe the dimension and closure of Schubert cells. For $w\in \widetilde{W}$, let $w^K$  be the unique element of minimal length in $wW_K$. Then by \cite[Lemma 1.6]{Richarz 1}, there is a unique element ${}_Kw^K$ of maximal length in $\{(vw)^K\mid v\in W_K\}$. We set \[{}_K\widetilde{W}^K:=\{{}_Kw^K\mid w\in \widetilde{W}\}.\] It seems that a slightly different subset ${}^K\widetilde{W}_K\subset \widetilde{W}$ is more convenient for our purpose. For $w\in \widetilde{W}$, let ${}^Kw$ be the unique element of minimal length in $W_Kw$, there is a unique element ${}^Kw_K$ of maximal length in $\{{}^K(wv)\mid v\in W_K\}$.
We set \[{}^K\widetilde{W}_K:=\{{}^Kw_K\mid w\in \widetilde{W}\}.\] The natural projection $\widetilde{W}\rightarrow W_K\backslash \widetilde{W}/W_K$ induces bijections \[{}_K\widetilde{W}^K\stackrel{\sim}{\longrightarrow} W_K\backslash \widetilde{W}/W_K\stackrel{\sim}{\longleftarrow} {}^K\widetilde{W}_K.\]

The following statement should be well known to experts. It follows essentially from a result of Howlett (see e.g. \cite[Theorem 4.18]{fini dim alg and qt gp}).
\begin{lemma}\label{a lemma about length}
	For $w\in \widetilde{W}$, we have $\ell({}_Kw^K)=\ell({}^Kw_K)$.
\end{lemma}
\begin{proof}
	To simplify notations, we assume that $w$ is the unique element of minimal length in $W_KwW_K$, i.e. $w\in {}^K \IW^K$. Let $W_J=w^{-1}W_Kw\cap W_K$, and ${}^JW_K\subset W_K$ be the set of minimal length representatives of $W_J\backslash W_K$. For $v\in W_K$, we have a unique decomposition $v=v_J\cdot {}^Jv$ with $v_J\in W_J$ and ${}^Jv\in {}^JW_K$, and hence $$wv=wv_J\cdot {}^Jv=(wv_Jw^{-1})\cdot w\cdot{}^Jv.$$
	For $x\in {}^JW_K$, we have, by \cite[Theorem 4.18]{fini dim alg and qt gp} (and its proof), that $wx\in {}^K\widetilde{W}$ and $\ell(wx)=\ell(w)+\ell(x)$. In particular, noting that $wv_Jw^{-1}\in W_K$, we have ${}^K(wv)=w\cdot {}^Jv$, and hence ${}^Kw_K=wx_0$ with $\ell({}^Kw_K)=\ell(w)+\ell(x_0)$. Here $x_0$ is the unique element of maximal length in ${}^JW_K$.

	There is a similar description for ${}_Kw^K$. Namely, let $W_I=wW_Kw^{-1}\cap W_K$, and $W_K^I\subset W_K$ be the set of minimal length representatives of $W_K/W_I$, we have ${}_Kw^K=y_0w$ and $\ell({}_Kw^K)=\ell(w)+\ell(y_0)$, where $y_0$ is the maximal element in $W_K^I$. One sees immediately that $\ell({}_Kw^K)=\ell({}^Kw_K)$, as $\ell(x_0)=\ell(y_0)$.
\end{proof}

\subsubsection[]{}\label{subsubsection two sections}
By definition we have inclusion \[{}^K\widetilde{W}_K\subset {}^K\widetilde{W}.\] For any $w\in \Admu_K$, we get a unique representative of $w$ in ${}^K\widetilde{W}_K$, which we denote by ${}^Kw_K$. Then \[{}^Kw_K\in \Admu^K\cap {}^K \IW={}^K\Admu.\] In this way we get a section $w\mapsto {}^Kw_K$  of the surjection ${}^K\Admu \twoheadrightarrow \Admu_K$.
On the other hand, we have also the inclusion \[{}^K \IW^K\subset {}^K\widetilde{W}.\]
Recall that for $w\in \Admu_K$, we get a unique representative $x_w\in {}^K \IW^K\subset {}^K\widetilde{W}$. Then \[x_w\in \Admu^K\cap {}^K \IW={}^K\Admu.\]
In this way we get another section $w\mapsto x_w$ of the surjection ${}^K\Admu \twoheadrightarrow \Admu_K$.
We note that ${}^Kw_K$ and $x_w$ are the maximal and minimal elements respectively in $W_KwW_K\cap {}^K\wt{W}$, the fiber at $w$ of the surjection ${}^K\Admu\twoheadrightarrow\Admu_K$.


\subsubsection[The Newton map]{The Newton map}\label{nu on aff Weyl}  We identify $X_*(T)_{\Gamma_0,\Q}=X_*(T)^{\Gamma_0}_\RAT$. Let $X_*(T)_{\Gamma_0,\RAT}^+$ be the set of dominant elements in $X_*(T)_{\Gamma_0,\RAT}$ defined by the positive relative roots of $G_{\breve{\Q}_p}$ corresponding to $B$ (cf. \cite{geo and homplg aff DL} 1.7). The action of $\sigma$ on $X_*(T)_{\Gamma_0,\RAT}/W_0$ transfers to an action on $X_*(T)_{\Gamma_0,\RAT}^+$ (the so called L-action), and hence gives a subset $(X_*(T)_{\Gamma_0,\RAT}^+)^{\langle\sigma \rangle}$.

There is a map (cf. \cite{geo and homplg aff DL}) \[\nu:\IW\rightarrow (X_*(T)_{\Gamma_0,\RAT}^+)^{\langle\sigma \rangle}\] as follows. For $w\in \IW$, there exists $n\in \mathbb{N}$ such that $\sigma^n$ acts trivially on $\IW$ and that \[\lambda=w\sigma(w)\cdots \sigma^{n-1}(w)\in X_*(T)_{\Gamma_0}.\] The element \[\frac{1}{n}\lambda\in X_*(T)_{\Gamma_0,\Q}\] is independent of the choice of $n$. The map $\nu:\IW\rightarrow (X_*(T)_{\Gamma_0,\RAT}^+)^{\langle\sigma \rangle}$ is then given by setting $\nu(w)$ to be the unique dominant element in the $W_0$-orbit of $\frac{1}{n}\lambda$. It is called the Newton map. By construction, it is constant on each $\sigma$-conjugacy class of $\IW$.

\subsubsection[$\sigma$-straight elements]{$\sigma$-straight elements}\label{recall sigm straight}Let us recall definition and basic properties of $\sigma$-straight elements in $\IW$. An element $w\in \IW$ is called $\sigma$-straight if \[\ell(w\sigma(w)\sigma^2(w)\cdots \sigma^{m-1}(w))=m\ell(w)\] for all $m\in \mathbb{N}$. By \cite[Lemma 1.1]{min leng priprint}, it is equivalent to \[\ell(w)=\langle\nu(w),2\rho\rangle,\] where $\nu(w)$ is as in \ref{nu on aff Weyl}, and $\rho$ is the half sum of positive roots.

We denote by \[\Admu_{\sigma\text{-str}}\subset \Admu\] the subset of $\sigma$-straight elements. We also set $${}^K\Admu_{\sigma\text{-str}}:=\Admu_{\sigma\text{-str}}\cap {}^K\IW\subset {}^K\Admu,$$where the last inclusion is by Theorem \ref{T: EKOR set}.
One checks easily by definition that the element $\tau$ is $\sigma$-straight, so $\tau\in {}^K\Admu_{\sigma\text{-str}}$.

\subsection[$\CGmu$ and related quotients]{$\CGmu$ and related quotients}\label{subsubsec-C(G,mu)}
We continue the notations in the last subsection.
\subsubsection[The set $\BGmu$]{The set $\BGmu$}Let $B(G)$ be the set of $\sigma$-conjugacy classes in $\breve{G}$.
Kottwitz constructed two maps in \cite{isocys with addi}, namely, the Newton map $$\nu:B(G)\rightarrow (X_*(T)_{\Gamma_0,\RAT}^+)^{\langle\sigma \rangle}$$ and the Kottwitz map $$\kappa:B(G)\rightarrow \pi_1(G)_{\Gamma}.$$ Here $\Gamma:=\mathrm{Gal}(\overline{\RAT}_p/\RAT_p)$. 
An element $[b]\in B(G)$ is uniquely determined by its images $\nu([b])$ and $\kappa([b])$. The relation between the Newton maps on $B(G)$ and $\IW$ respectively is as follows.  A $\sigma$-conjugacy class of $\IW$ is called straight if it contains a $\sigma$-straight element.
Let $B(\IW)_{\sigma\text{-str}}$ be the set of straight $\sigma$-conjugacy classes of $\IW$.
By \cite[Theorem 3.3]{geo and homplg aff DL}, the map \[\Psi: B(\IW)_{\sigma\text{-str}}\rightarrow B(G)\] induced by the inclusion $N(T)(\breve{\mathbb{Q}}_p)\subset \breve{G}$ is bijective, and compatible with the Newton maps on both sides.

There is a partial order $\leq$ on $X_*(T)_\RAT$ defined as follows. For $v_1,v_2\in X_*(T)_\RAT$, $v_1\leq v_2$ if and only if $v_2-v_1$ is a non-negative $\RAT$-sum of positive relative coroots. One can the define a partial order on $B(G)$ as follows:
\[[b_1]\leq [b_2]\text{ if and only if }\kappa([b_1])=\kappa([b_2])\text{ and }\nu([b_1])\leq\nu([b_2]).\]One then define
\[\BGmu:=\{[b]\in B(G)\mid \kappa([b])=\mu^{\natural}, \nu([b])\leq \overline{\mu}\}.\]
Here $\mu^{\natural}$ is the common image of $\mu\in \mmu$ in $\pi_1(G)_{\Gamma}$, and $\overline{\mu}$ is the Galois average of a dominant representative of the image of an element of $\mmu$ in $X_*(T)_{\Gamma_0,\RAT}$ with respect to the L-action of $\sigma$ on $X_*(T)_{\Gamma_0,\RAT}^+$. By \cite[Proposition 4.1]{K-R on union of aff D-L}, the above $\Psi$ maps the image of $\Admu_{\sigma\text{-str}}$ in $B(\IW)_{\sigma\text{-str}}$ bijectively to $\BGmu$.

\subsubsection{The set $\CGmu$}\label{subsubsection CGmu}
We consider the (right) action of $\breve{K}\times \breve{K}$ on $\breve{G}$ given by $(g,(k_1,k_2))\mapsto k_1^{-1}gk_2$. Let $\breve{K}_{\sigma}\subset \breve{K}\times \breve{K}$ be the graph of the Frobenius map and set $C(\mathcal{G})= \breve{G}/\breve{K}_{\sigma}$. We have a natural surjection $C(\mathcal{G})\ra B(G)$. The subset \[\breve{K}\Admu \breve{K}\subset \breve{G}\] is stable with respect to the action of $\breve{K}_{\sigma}$. We denote by \[\CGmu=\breve{K}\Admu \breve{K}/\breve{K}_{\sigma}\] the set of its orbits. Then $\CGmu\subset C(\mathcal{G})$. 

Let $\breve{K}_1$ be the pro-unipotent radical of $\breve{K}$, then $\breve{K}_{\sigma}(\breve{K}_1\times \breve{K}_1)\subset \breve{K}\times \breve{K}$ is a subgroup. We define $$\mathrm{EKOR}(K,\mmu):=\breve{K}\Admu \breve{K}/\breve{K}_{\sigma}(\breve{K}_1\times \breve{K}_1),\quad\mathrm{KR}(K,\mmu):=\breve{K}\Admu \breve{K}/\breve{K}\times \breve{K}.$$
We then have natural maps \begin{subeqn}\label{proj from cent}\xymatrix{B(G,\{\mu\})&\CGmu\ar[r]\ar[l] &\mathrm{EKOR}(K,\mmu)\ar[r] & \mathrm{KR}(K,\mmu),}\end{subeqn}where all the arrows are surjections.

By the Bruhat-Tits decomposition, we see that the natural inclusion  $N(T)(\breve{\mathbb{Q}}_p)\rightarrow \breve{G}$ induces a bijection \[\Admu_K\stackrel{\sim}{\longrightarrow}\mathrm{KR}(K,\mmu).\] To identify  ${}^K\Admu$ and $\mathrm{EKOR}(K,\mmu)$, we need the following results due to Lusztig and He.

\begin{theorem}{\rm(\cite[Theorem 6.1]{He-Rap})} \label{He's result--decompo} Notations as above, we have
	\begin{enumerate}
		\item for any $x\in {}^K\widetilde{W}$, $\breve{K}_{\sigma}(\breve{K}_1x\breve{K}_1)=\breve{K}_{\sigma}(\breve{I}x\breve{I})$;
		\item $\breve{G}=\coprod_{x\in {}^K\widetilde{W}} \breve{K}_{\sigma}(\breve{K}_1x\breve{K}_1)=\coprod_{x\in {}^K\widetilde{W}} \breve{K}_{\sigma}(\breve{I}x\breve{I})$.
	\end{enumerate}
	In particular, the inclusion  $N(T)(\breve{\mathbb{Q}}_p)\rightarrow \breve{G}$ induces a bijection \[{}^K\Admu\stackrel{\sim}{\longrightarrow}\mathrm{EKOR}(K,\mmu).\]
\end{theorem}
We will identify ${}^K\Admu$ with $\mathrm{EKOR}(K,\mmu)$ and $\Admu_K$ with $\mathrm{KR}(K,\mmu)$ from now on.

Following \cite[Definition 13.1.1]{aff DL in aff flag}, an element $w\in \widetilde{W}$ is said to be \emph{fundamental} if $\breve{I}\dot{w}\breve{I}$ consists of a single $\breve{I}$-$\sigma$-conjugacy class. To explain relations between fundamental elements, $\sigma$-straight elements and $\CGmu$ (resp. $\BGmu$), we need the following results. Recall the set \[{}^K\Admu_{\sigma\text{-}\mathrm{str}}={}^K\Admu\cap \Admu_{\sigma\text{-str}}.\] The $\Psi$ above induces a map \[{}^K\Admu_{\sigma\text{-}\mathrm{str}} \longrightarrow\BGmu.\]
\begin{theorem}{\rm(\cite[Proposition 4.5]{geo and homplg aff DL}, \cite[Theorem 6.17]{He-Rap})} \label{He's result--straight vs fundmtl} Notations as above, we have
	\begin{enumerate}
		\item $\sigma$-straight elements are fundamental,
		\item the map ${}^K\Admu_{\sigma\text{-}\mathrm{str}} \longrightarrow\BGmu$ is surjective.
	\end{enumerate}
\end{theorem}
If $x\in {}^K\widetilde{W}$ is $\sigma$-straight, then we have  $\breve{K}_{\sigma}(\breve{K}_1x\breve{K}_1)=\breve{K}_{\sigma}(\breve{I}x\breve{I})$ by Theorem \ref{He's result--decompo} (1), and $\breve{I}x\breve{I}=\breve{K}_{\sigma}\cdot x$ by Theorem \ref{He's result--straight vs fundmtl} (1). In particular, we have \[\breve{K}_{\sigma}(\breve{K}_1x\breve{K}_1)=\breve{K}_{\sigma}\cdot x.\] This means that the natural surjection \[\CGmu\twoheadrightarrow{}^K\Admu\] admits a natural section (explicitly given as above) when restricting to the subset \[{}^K\Admu_{\sigma\text{-}\mathrm{str}}\subset{}^K\Admu.\] Combined with Theorem \ref{He's result--straight vs fundmtl} (2), we have the following commutative diagram:

\[\xymatrix{ & &{}^K\Admu\\
	{}^K\Admu_{\sigma\text{-str}}\ar@{^(->}[r]\ar@{^(->}[urr]\ar@{->>}[drr]&\CGmu\ar@{->>}[ur]\ar@{->>}[dr]\\
	& & \BGmu.}\]

\subsubsection[]{}\label{zip in EKOR}
Now we make the link between the theory of algebraic zip data in subsection \ref{subsection-zip data}, the theory of Iwahori Weyl groups \ref{Iwahori Weyl} and $\CGmu$.
We need some constructions which are slightly modified version of those in the proof of \cite[Theorem 6.1]{He-Rap}. Let $\widetilde{\mathbb{S}}$ be the set of simple reflections in $\IW$ and $J_K\subset \widetilde{\mathbb{S}}$ be the set of simple reflections in $W_K$. We have $\sigma(J_K)=J_K$ by construction.
Let $w\in {}^K\IW^K$, i.e., $w$ is of shortest length in $W_KwW_K$. We set $\sigma'=\sigma\circ \mathrm{Ad}(w)$.

The map $\breve{K}\rightarrow \breve{K}w\breve{K}$, $k\mapsto wk$ induces a bijection \[\breve{K}/(\breve{K}\cap w^{-1}\breve{K}w)_{\sigma'}\stackrel{\sim}{\longrightarrow} \breve{K}w\breve{K}/\breve{K}_\sigma.\] We remind the readers that, by our conventions at the beginning of \ref{subsubsec-C(G,mu)}, the $\breve{K}_\sigma$-action is given by \[g\cdot k=k^{-1}g\sigma(k),\] where $g\in \breve{K}w\breve{K}$ and $k\in \breve{K}$. The $(\breve{K}\cap w\breve{K}w^{-1})_{\sigma'}$-action is defined in the same way. 

Let $\mathcal{G}_0=\mathcal {G}\otimes \ov{\mathbb{F}}_p$ and $\mathcal{G}_0^{\mathrm{rdt}}$ be the maximal reductive quotient of $\mathcal{G}_0\otimes \ov{\mathbb{F}}_p$. Then $\mathcal{G}_0^{\mathrm{rdt}}$ is a reductive group defined over $\mathbb{F}_p$. Let $\overline{B}$ be the image in $\mathcal{G}_0^{\mathrm{rdt}}$ of $\breve{I}$ and $\overline{T}$ be a maximal torus of $\overline{B}$. Let \[J_w=J_K\cap \mathrm{Ad}(w^{-1})(J_K),\] which is a subset of simple reflections in the Weyl group of $\mathcal{G}_0^{\mathrm{rdt}}$ (with respect to $(\overline{B},\overline{T})$). Let $\overline{L}_{J_w}\subset \mathcal{G}_0^{\mathrm{rdt}}$ (resp. $\overline{P}_{J_w}\subset \mathcal{G}_0^{\mathrm{rdt}}$) be the standard Levi subgroup (resp. parabolic subgroup) of type $J_w$, and $\overline{L}_{\sigma'(J_w)}\subset \mathcal{G}_0^{\mathrm{rdt}}$ (resp. $\overline{P}_{\sigma'(J_w)}\subset \mathcal{G}_0^{\mathrm{rdt}}$) be the standard Levi subgroup (resp. parabolic subgroup) of type $\sigma'(J_w)$. Then we have a natural isogeny $\overline{L}_{J_w}\rightarrow \overline{L}_{\sigma'(J_w)}$ which is again denoted by $\sigma'$. The tuple
$$\mathcal {Z}_w:=(\mathcal{G}_0^{\mathrm{rdt}}, \overline{P}_{J_w}, \overline{P}_{\sigma'(J_w)}, \sigma':\overline{L}_{J_w}\rightarrow \overline{L}_{\sigma'(J_w)})$$
is an algebraic zip datum. Let \[E_{\mathcal {Z}_w}=(\overline{L}_{J_w})_{\sigma'}(U_{J_w}\times U_{\sigma'(J_w)}),\] where $U_{J_w}$ (resp. $U_{\sigma'(J_w)}$) is the unipotent radical of $\ov{P}_{J_w}$ (resp. $\ov{P}_{\sigma'(J_w)}$). It has a left action on $\mathcal{G}_0^{\mathrm{rdt}}$. Moreover, we have a natural map \begin{subeqn}\label{bij for fixed w}f_w:\breve{K}w\breve{K}/\breve{K}_\sigma\rightarrow E_{\mathcal {Z}_w}\backslash\mathcal{G}_0^{\mathrm{rdt}}, \ \ \ \ \ k_1wk_2\mapsto \overline{\sigma(k_2)k_1}\end{subeqn} which induces a bijection \[\breve{K}w\breve{K}/\breve{K}_\sigma(\breve{K}_1\times\breve{K}_1)\stackrel{\sim}{\longrightarrow} E_{\mathcal {Z}_w}\backslash \mathcal{G}_0^{\mathrm{rdt}}.\] The induced map will also be denoted by $f_w$. Let $w\in \Admu_K$ and view it as an element of ${}^K\IW^K$. Then $\breve{K}w\breve{K}/\breve{K}_\sigma(\breve{K}_1\times\breve{K}_1)$ is identical to the fiber of the surjection \[{}^K\Admu=\mathrm{EKOR}(K,\mmu) \twoheadrightarrow \mathrm{KR}(K,\mmu) = \Admu_K \] at $w$. On the other hand,
noting that $\sigma'=\sigma\circ \mathrm{Ad}(w)=\mathrm{Ad}(\sigma(w))\circ\sigma$, by Theorem \ref{thm on zip orbits} and Remark \ref{remark on zip-1}, the underlying space of $E_{\mathcal {Z}_w}\backslash\mathcal{G}_0^{\mathrm{rdt}}$ is the finite set \[{}^{J_w}W_K.\] One also sees easily that it is precisely the set ${}^JW_K$ in the proof of Lemma \ref{a lemma about length}.
Note that \[w({}^{J_w}W_K)=W_KwW_K\cap {}^K\wt{W}\]
by \cite{Lusztig1} 2.1 (b) (see also the proof of \cite[Theorem 6.1]{He-Rap}).

\section[Local models and Shimura varieties]{Local models and Shimura varieties}\label{section loc mod}

In this section, we review the Pappas-Zhu local models \cite{local model P-Z} and the Kisin-Pappas integral models \cite{Paroh} for certain Shimura varieties of abelian type. We refer to \cite{RZ, Goertz1, Goertz2, PR1, PR2, PR3, PRS} for more information on the theory of local models.

\subsection[Local models]{Local models}\label{subsection loc models}

\subsubsection[The loop groups]{Loop groups and affine flag varieties} Let $k=\overline{k}$ be an algebraically closed field.
Let $\mathcal{G}$ be an affine group scheme of finite type over $k[[t]]$. We set  $L\mathcal{G}$ and $L^+\mathcal{G}$ to be the group functors on the category of $k$-algebras given by $$LG(R):=\mathcal{G}(R((t)))=G(R((t)))\text{\ \ \ \ and \ \ \  }L^+\mathcal{G}(R):=\mathcal{G}(R[[t]])$$
respectively. Then $L^+\mathcal{G}$ is represented by an affine group scheme over $k$, and $LG$ is represented by an ind-affine ind-group scheme over $k$. Moreover, if $\mathcal{G}$ is smooth, then $L^+\mathcal{G}$ is reduced.

The notations in this part are different from other sections. We will switch back to our original notations from  \ref{subsubsec local mod-cons and prop}. Let $G/k((t))$ be a reductive group, and $\mathcal{G}/k[[t]]$ be the parahoric model of $G$ corresponding to a facet $\mathfrak{a}$ in the enlarged Bruhat-Tits building of $G$. The \emph{affine flag variety} $\Gr_\mathcal{G}$ is the fpqc-sheaf associated to the functor on the category of $k$-algebras given by \[R\mapsto LG(R)/L^+\mathcal{G}(R).\] It has a distinguished base point $e_0$ given by the identity section of $G$.

\begin{remark}\label{loop in mixed char+}
	If $\mathcal{G}$ is an affine group scheme of finite type over $W(k)$, we can also define, by abuse of notations, $L^+\mathcal{G}$ to be the group functor on the category of $k$-algebras given by $L^+\mathcal{G}(R):=\mathcal{G}(W(R))$. Like in the equi-characteristic case, it is represented by an affine group scheme over $k$ which is reduced if $\mathcal{G}$ is smooth. Moreover, the functor on the category of perfect $k$-algebras given by $LG(R):=\mathcal{G}(W(R)[1/p])=G(W(R)[1/p])$ is represented by an ind perfect scheme.
	We can consider the Witt vector affine flag variety $\Gr_\G=\Gr_\mathcal{G}^W$ which is given by \[\Gr_\mathcal{G}=LG/L^+\G\] on the category of perfect $k$-algebras. By \cite{zhu-aff gras in mixed char} and \cite{BS}, this is an ind-proper perfect scheme over $k$. We will work with these affine flag varieties in section \ref{section global}.
\end{remark}

\subsubsection[Schubert cells and Schubert varieties]{Affine Schubert cells and affine Schubert varieties}

As in  \ref{Iwahori Weyl}, we have the Iwahori Weyl group $\widetilde{W}$ with Bruhat order $\leq$, the finite Coxeter group $W_{\mathfrak{a}}$ (denoted by $W_K$ with $K=\G(k[[t]])$ in \ref{Iwahori Weyl} in the mixed characteristic setting) isomorphic to the Weyl group of $\mathcal{G}_0^\mathrm{rdt}$, the reductive quotient of $\G_0=\G\otimes k$, and the subset ${}_\mathfrak{a}\widetilde{W}^{\mathfrak{a}}\subset \widetilde{W}$ in bijection with $W_\mathfrak{a}\backslash\widetilde{W}/W_\mathfrak{a}$. For $w\in \widetilde{W}$, the orbit map \[L^+\mathcal{G}\rightarrow \Gr_\mathcal{G}, \quad g\mapsto g\cdot \dot{w}e_0\] factors through $\mathcal{G}_0$.  The $\mathcal{G}_0$-orbit of $\dot{w}e_0$, denoted by $\Gr_w=\Gr_{\mathcal{G},w}$, is a connected smooth and locally closed subscheme of $\Gr_\mathcal{G}$. Let $\Gr_w^{-}$ be its closure, which is a reduced closed subscheme of $\Gr_\mathcal{G}$. The variety $\Gr_w$ (resp. $\Gr_w^{-}$) is called the \emph{affine Schubert cell} (resp. affine Schubert variety) attached to $w$. By the Bruhat-Tits decomposition, we have a set-theoretically disjoint union of locally closed subsets
$$\Gr_\mathcal{G}=\coprod_{v\in W_\mathfrak{a} \backslash\widetilde{W}/W_\mathfrak{a}}\Gr_v=\coprod_{w\in {}_\mathfrak{a}\widetilde{W}^\mathfrak{a}}\Gr_w.$$

\begin{theorem}{\rm(\cite[Proposition 2.2]{Richarz 2})} \label{thm of richarz} For $w\in {}_\mathfrak{a}\widetilde{W}^\mathfrak{a}$, we have
	\begin{enumerate}
		\item $\mathrm{dim}\  \Gr_w=\ell(w)$;
		\item $\Gr_w^{-}$ is proper, and $\Gr_w^{-}=\coprod\limits_{v\in {}_\mathfrak{a}\widetilde{W}^\mathfrak{a},v\leq w}\Gr_v$.
	\end{enumerate}
\end{theorem}
\begin{remark}\label{length KR-equichar}
	As in \ref{Iwahori Weyl}, we also have ${}^\mathfrak{a}\widetilde{W}_{\mathfrak{a}}\subset \widetilde{W}$, and by Lemma \ref{a lemma about length}, $\mathrm{dim}\  \Gr_w=\ell({}^\mathfrak{a}w_{\mathfrak{a}})$ for all $w\in {}_\mathfrak{a}\widetilde{W}^\mathfrak{a}$.
\end{remark}

We consider the $\mathcal{G}_0$-action on $\Gr_w$.

\begin{lemma}\label{smooth stab}
	For $x\in \Gr_w(k)$, its stabilizer $\mathcal{G}_{0,x}$ is smooth.
\end{lemma}
\begin{proof}
	The $\mathcal{G}_0$-action on $\Gr_w$ is transitive, so we can take $x=we_0$. The $\mathcal{G}_0$-action on $\Gr_w$ is induced by the $L^+\mathcal{G}$-action, and the stabilizer of $x$ with respect to this $L^+\mathcal{G}$-action is $$L^+\mathcal{G}_w:=L^+\mathcal{G}\cap(\dot{w}\cdot L^+\mathcal{G}\cdot\dot{w}^{-1}).$$

	Let $\mathfrak{b}=\mathfrak{a}\cup w\mathfrak{a}$, where $w\mathfrak{a}$ is the $w$-translation of $\mathfrak{a}$, and $\mathcal{G}_\mathfrak{b}$ be the Bruhat-Tits group scheme attached to $\mathfrak{b}$. Then $L^+\mathcal{G}_w=L^+(\mathcal{G}_\mathfrak{b})$, and the homomorphism $\mathcal{G}_\mathfrak{b}\rightarrow \mathcal{G}$ induces a commutative diagram
	\[\xymatrix{
		L^+(\mathcal{G}_\mathfrak{b})\ar@{->>}[d]\ar[r]&L^+\mathcal{G}\ar@{->>}[d]\\
		\mathcal{G}_{\mathfrak{b},0}\ar[r]&\mathcal{G}_0.}\]
	In particular, $\mathcal{G}_{0,x}$ is a quotient of $\mathcal{G}_{\mathfrak{b},0}$, and hence smooth.
\end{proof}

\subsubsection[Constructions and properties of local models]{Constructions and properties of local models}\label{subsubsec local mod-cons and prop}
We refer the readers to \cite{PRS} for a general exposition on the theory of local models.
Here we will briefly recall the construction of the Pappas-Zhu local models in \cite{local model P-Z}. For a reductive group $G$ over $\mathbb{Q}_p$ and a prime number $p>2$, we denote by $\mathcal {B}(G,\mathbb{Q}_p)$ the (extended) Bruhat-Tits building of $G$. For $x\in \mathcal {B}(G,\mathbb{Q}_p)$, we write $\mathcal{G}$ for the parahoric group scheme attached to $x$, i.e. the connected stabilizer of $x$. It is then a linear algebraic groups over $\mathbb{Z}_p$ with generic fiber $G$. We will assume from now on that 
\begin{subeqn}\label{general hypo}
	G \text{ splits over a tamely ramified extension and } p\nmid |\pi_1(G^{\mathrm{der}})|.
\end{subeqn}

In \cite[\S 3]{local model P-Z}, there is a construction of a smooth affine group scheme $\underline{\mathcal{G}}$ over $\INT_p[u]$
which specializes to the parahoric group scheme $\mathcal{G}$ via the base
change $\INT_p[u]\rightarrow \INT_p$ given by $u\rightarrow p$ (see \cite[\S 4]{local model P-Z}), and such that $\underline{G}:=\underline{\mathcal{G}}|_{\INT_p[u,u^{-1}]}$ is
reductive. There is a corresponding ind-projective ind-scheme (the global affine
Grassmannian) \[\mathrm{Gr}_{\underline{\mathcal{G}},\mathbb{A}^1}\rightarrow \mathbb{A}^1=\Spec \INT_p[u]\] (see \cite[\S 6]{local model P-Z}). The base change $\mathrm{Gr}_{\underline{\mathcal{G}},\mathbb{A}^1}\times_{\mathbb{A}^1} \Spec \RAT_p$ given by $u\rightarrow p$ can be identified with
the affine Grassmannian $\mathrm{Gr}_{G}$ of $G$ over $\RAT_p$. (Recall that $\mathrm{Gr}_{G}$ represents the
fpqc sheaf associated to the quotient $R\mapsto G(R((t)))/G(R[[t]])$; the identification is
via $t = u-p$.)

Let $\mu:\mathbb{G}_{m,\overline{\RAT}_p}\longrightarrow G_{\overline{\RAT}_p}$ be a minuscule cocharacter, and $E/\mathbb{Q}_p$ be the local reflex field, i.e.
the field of definition of the conjugacy class $\{\mu\}$. The cocharacter $\mu$ defines a projective homogeneous space $G_{\overline{\RAT}_p}/P_{\mu^{-1}}$.
Here, $P_{\nu}$ denotes the parabolic subgroup that corresponds to the cocharacter $\nu$, i.e. the parabolic subgroup whose Lie algebra is the subset $\mathrm{Lie}(G_{\overline{\RAT}_p})$ consisting of elements of  non-negative weights with respect to $\nu$. Since the conjugacy class $\{\mu\}$ is defined over $E$, we can see that this homogeneous space has a canonical model $X_{\mu}$ defined over $E$. Since $\mu$ is minuscule, the corresponding Schubert cell $\Gr_{G_{\overline{\RAT}_p},\mu}$, whose set of $\overline{\RAT}_p$-points is by definition
\[G(\overline{\RAT}_p[[t]])\mu(t)G(\overline{\RAT}_p[[t]])/G(\overline{\RAT}_p[[t]]),\] is closed in the affine Grassmannian $\Gr_{G_{\overline{\RAT}_p}}$. Moreover, $\Gr_{G_{\overline{\RAT}_p},\mu}$ is defined over $E$ and can be $G_E$-equivariantly identified with $X_{\mu}$.

The \emph{Pappas-Zhu local model} $\mathrm{M}^\mathrm{loc}_{G,\mathcal{G},\{\mu\}}$ is the Zariski closure
of $X_\mu\subset \mathrm{Gr}_{G,E}$ in \[\mathrm{Gr}_{\underline{\mathcal{G}},\mathbb{A}^1}\times_{\mathbb{A}^1}\Spec O_E,\] where the base change $\INT_p[u]\rightarrow O_E$
is given by $u\mapsto p$. We will simply write $\Mloc$ for $\mathrm{M}^\mathrm{loc}_{G,\mathcal{G},\{\mu\}}$ when $\mathcal{G}$ and $\{\mu\}$ are fixed. By its construction, $\mathrm{M}^\mathrm{loc}_{G}$ is a projective flat scheme over $O_E$ which admits an action of the group scheme $\mathcal{G}_{O_E}$.

\begin{theorem}\label{the local model}{\rm (\cite[Theorem 9.1]{local model P-Z}, \cite[Corollary 2.1.3]{Paroh})} Under the assumptions in (\ref{general hypo}), the scheme $\Mloc$ is normal. The geometric special fiber of $\Mloc$ is
	reduced and admits a stratification with locally closed smooth strata; the closure of
	each stratum is normal and Cohen-Macaulay.
	Under the above assumptions, the base change $\Mloc\otimes_{O_E}O_L$
	is normal, for every finite extension $L/E$.
\end{theorem}

\begin{remark}\label{remark--Goertz PEL}
	In the PEL type case (A) and (C), G\"ortz proved in \cite{Goertz1, Goertz2}  that if the group $G$ is unramified over $\Q_p$, then the naive local models defined by Rapoport-Zink in \cite{RZ} are flat over $O_E$, thus they coincide with the Pappas-Zhu local models above. In particular in these cases there exists a moduli interpretation for $\Mlo_G$.
\end{remark}

We will need the geometry of the geometric special fiber of $\Mloc$, denoted by $\eM$ for simplicity. Let $\underline{G}$ and $\underline{\mathcal{G}}$ be as at the beginning of  \ref{subsubsec local mod-cons and prop}. We set $G'=\underline{G}\times k((u))$, and $\mathcal{G}'=\underline{\mathcal{G}}\times k[[u]]$. Here both base-changes are induced by reduction mod $p$. Noting that $\mathcal{G}'_0\cong \mathcal{G}_{0,k}$, we view $W_K$ as the Weyl group of $\mathcal{G}'_0$. By \cite[\S 9.2.2]{Paroh}, the construction of $\underline{G}$ induces an isomorphism between the Iwahori Weyl groups $\widetilde{W}$ of $G_{\breve{\RAT}_p}$ and $\widetilde{W}'$ of $G'$, and hence we will view $\Admu\subset \widetilde{W}$ as an admissible subset of $\widetilde{W}'$.

We set \[\mathcal{S}_{\mathcal{G}',\mu}=\bigcup_{w\in \Admu_K}\Gr^{-}_{\mathcal{G}',w}\] with the reduced-induced scheme structure. It is a closed subvariety in $\Gr_{\mathcal{G}'}$. By \cite[Theorem 9.3]{local model P-Z}, \[\mathcal{S}_{\mathcal{G}',\mu}=\eM\] as closed subschemes of $\Gr_{\mathcal{G}'}$. In view of this identification, we write $\eM^w:=\Gr_{\mathcal{G}',w}$ and $\eM^{w,-}:=\Gr^{-}_{\mathcal{G}',w}$ for $w\in \Admu_K$. In particular, combined with Theorem \ref{thm of richarz}, Remark \ref{length KR-equichar} and Lemma \ref{smooth stab},  we have the following.
\begin{corollary}\label{local model-collect}
	There is a set-theoretically disjoint union of locally closed subsets
	$$\eM=\coprod_{w\in \Admu_K}\eM^w.$$ Here $\Admu_K$ is as in (\ref{admi mu}). Moreover,
	\begin{enumerate}
		\item $\eM^{w,-}=\coprod\limits_{v\in \Admu_K,v\leq w}\eM^v$;
		
		\item each $\eM^w$ consists of a single $\mathcal{G}_0$-orbit, and the stabilizer of each closed point is smooth;
		\item $\mathrm{dim}\  \eM^w=\ell({}^Kw_K)$;
	\end{enumerate}
\end{corollary}
\begin{remark}
	In \cite{He-Pap-Rap} 2.6, there is a modified version of the Pappas-Zhu local models for which we can remove the condition $p\nmid |\pi_1(G^{\mathrm{der}})|$. Moreover, in the abelian type case,  by \cite[Theorem 2.15]{He-Pap-Rap}, the associated small $v$-sheaves (in the sense of Scholze) satisfy Scholze's conjecture in \cite{SW} (Conjecture 21.4.1). Therefore, up to a mild modification as in \cite{He-Pap-Rap} there exist embeddings of (the perfection of) the special fibers $\eM$ of the above local models into the Witt vector affine flag varieties $\Gr_\mathcal{G}^W$, cf. the above Remark \ref{loop in mixed char+}. We will come back to this point of view in section \ref{section global}.
\end{remark}

\subsection[Integral models of Hodge type]{Integral models for Shimura varieties of Hodge type}\label{subsection integral Hodge}

Let $(G,X)$ be a Shimura datum of Hodge type and $p>2$. For $x\in \mathcal {B}(G,\mathbb{Q}_p)$ (the extended Bruhat-Tits building of $G_{\Q_p}$ over $\Q_p$), we write $\G=\mathcal{G}_x$ for the stabilizer of $x$ and $\mathcal{G}^\circ$ for its connected component of the identity. Then $\mathcal{G}^\circ$ is the parahoric group scheme attached to $x$, which is an integral model of $G$ over $\Z_p$.  We will assume
\begin{subeqn}\label{general hypo--refined}
	G_{\mathbb{Q}_p} \text{ splits over a tamely ramified extension, } p\nmid |\pi_1(G_{\mathbb{Q}_p}^{\mathrm{der}})|\, \text{and } \mathcal{G}=\mathcal{G}^\circ.
\end{subeqn}

Let $\E=E(G,X)$ be the reflex field of $(G,X)$, and $v$ be a place of $\E$ over $p$ that will be fixed once and for all. Let $E=\E_v$ and $O_E$ be its ring of integers. The residue field of $O_E$ will be denoted by $\kappa$ as usual. Fixing $K_p:=\mathcal{G}(\mathbb{Z}_p)$, for $K^p\subset G(\mathbb{A}_f^p)$ small enough, we set $\K:=K_pK^p$, and we are interested in certain integral models of $\Sh_\K(G,X)_E$.


By \cite[4.1.5,  4.1.6]{Paroh}, there is a symplectic embedding \[i:(G,X)\hookrightarrow (\mathrm{GSp}(V,\psi),S^\pm)\] such that there is a $\mathbb{Z}$-lattice $V_{\INT}\subset V$ such that $V_{\INT}\subset V_{\INT}^\vee$, and that the base-change to $\mathbb{Z}_p$ of $\widetilde{\mathcal{G}}$, the Zariski closure of $G$ in $\mathrm{GL}(V_{\INT_{(p)}})$, is $\mathcal{G}$. For here and after, we simply write $\GSp=\GSp(V,\psi)$ and $V_R=V_{\INT}\otimes R$ for an algebra $R$. Let $g=\frac{1}{2}\mathrm{dim}V$, $H=H_pH^p$ where $H_p$ is the subgroup of $\mathrm{GSp}(\mathbb{Q}_p)$ leaving $V_{\INT_p}$ stable, and $H^p$ is a compact open subgroup of $\mathrm{GSp}(\mathbb{A}_f^p)$ containing $K^p$, leaving $V_{\widehat{\mathbb{Z}}^p}$ stable and small enough. Let $\ES_{H}(\GSp,S^\pm)$ be the moduli scheme over $\Z_{(p)}$ of isomorphism classes of certain triples $(\A,\lambda, \varepsilon^p)$ consisting of a $g$-dimensional abelian scheme $\A$ equipped with a polarization $\lambda: \A\ra \A^t$ of degree $|V_{\INT}^\vee/V_{\INT}|$ and a level structure $\varepsilon^p$; for more details see \cite{CIMK} 2.3.3.
The above symplectic embedding induces an embedding \[i:\Sh_\K(G,X)_E\hookrightarrow\ES_{H}(\GSp,S^\pm)_{O_E},\] and we write $\ES^-_K(G,X)$ for the closure of $\Sh_\K(G,X)_E$ and \[\ES_\K(G,X)\] for the normalization of $\ES^-_\K(G,X)$. In particular we have morphisms
\[\ES_\K(G,X)\ra \ES^-_\K(G,X)\subset \ES_{H}(\GSp,S^\pm)_{O_E}.\]
In particular we get a ``universal'' abelian scheme $\A\ra \ES_\K(G,X)$ by pulling back the universal abelian scheme over $\ES_{H}(\GSp,S^\pm)_{O_E}$ under the composition of the above morphisms.


\subsubsection[An alternative construction of local models]{An alternative construction of local models}\label{subsubsec-alter local model}We will follow \cite{Paroh} 4.1.5. Notations and assumptions as before, the Shimura datum $(G,X)$ gives a $G(\overline{\mathbb{Q}}_p)$-conjugacy class $\{\mu\}$ of cocharacters. Its induced $\mathrm{GL}(V_{\mathbb{Z}_p})(\overline{\mathbb{Q}}_p)$-conjugacy class contains a cocharacter $\mu'$ defined over $\mathbb{Z}_p$. Let $P$ (resp. $P'$) be the parabolic subgroup of $G_{\overline{\mathbb{Q}}_p}$ (resp. $\mathrm{GL}(V_{\mathbb{Z}_p})$) with non-negative weights with respect to $\mu'$ (resp. $\mu$). We denote by $\mathrm{M}_{G,X}^{\mathrm{loc}}$ the closure of the $G$-orbit of $y$ in $\mathrm{GL}(V_{\mathbb{Z}_p})/P'$, where $y$ is the point corresponding to $P$. Then $\mathrm{M}_{G,X}^{\mathrm{loc}}$ is defined over $O_E$ and equipped with an action of $\mathcal {G}$. By \cite[Corollary 2.3.16]{Paroh}, it is the Pappas-Zhu local model attached to $(G_{\mathbb{Q}_p},\{\mu\},\G)$.

To explain properties of $\ES_\K(G,X)$ and its relations with $\mathrm{M}_{G,X}^{\mathrm{loc}}$, we need the following data. By \cite[Lemma 1.3.2]{CIMK}, there is a tensor $s\in V_{\mathbb{Z}_{(p)}}^\otimes$ defining $\mathcal{G}$.  Consider \[\V=\Hdr(\mathcal {A}/\ES_\K(G,X)), \quad \V_E=\Hdr(\mathcal {A}/\Sh_\K(G,X)_E).\] Let  $\V^1\subset \V$ be the Hodge filtration, and $s_{\mathrm{dR},E}\in \V_E^\otimes$ be the section induced by $s$.
\begin{theorem}\label{result P-Z and K-P}
	Notations an assumptions as above, we have the following.
	
	\begin{enumerate}
		\item $\mathrm{M}_{G,X}^{\mathrm{loc}}$ is normal with reduced geometric special fiber. Moreover, $\mathrm{M}_{G,X}^{\mathrm{loc}}$ depends only on the pair $(G_{\mathbb{Q}_p}, \{\mu\})$.
		
		\item {\rm (\cite[Proposition 4.2.6]{Paroh})} The tensor $s_{\mathrm{dR},E}$ extends to a tensor $s_{\mathrm{dR}}\in \V^\otimes$. The $\ES_\K(G,X)$-scheme $\pi: \widetilde{\ES}_\K(G,X)\rightarrow \ES_\K(G,X)$ which classifies isomorphisms $f:V_{\mathbb{Z}_p}^\vee\rightarrow \V$ mapping $s$ to $s_{\mathrm{dR}}$ is a $\mathcal {G}$-torsor.
		
		\item {\rm (\cite[Theorem 4.2.7]{Paroh})} The morphism $q:\widetilde{\ES}_\K(G,X)\rightarrow \mathrm{GL}(V_{\mathbb{Z}_p})/P'$, $f\mapsto f^{-1}(\V^1)$ factors through $\mathrm{M}_{G,X}^{\mathrm{loc}}$. Moreover, it is smooth. In particular we get the following local model diagram
		\[\xymatrix{&\widetilde{\ES}_\K(G,X)\ar[ld]_\pi\ar[rd]^q&\\
			\ES_\K(G,X) & &	\mathrm{M}_{G,X}^{\mathrm{loc}}.
		}\]
	\end{enumerate}
\end{theorem}

\begin{remark}\label{remark--good siegel ebd}
	For $(V_{\INT}, \psi)$ as above, we set $V'_{\INT}:=(V_{\INT}\oplus V_{\INT}^\vee)^4$ and $V':=V'_{\INT}\otimes \RAT$. By Zarhin's trick, there is a perfect alternating form $\psi'$ on $V'_{\INT}$, such that the (faithful) representation of $G$ on $V'$ factors through $\GSp(V',\psi')$, and hence induces a closed immersion $\mathcal{G}\rightarrow \GSp(V'_{\INT_p},\psi')$. One can then use this $(V'_{\INT}, \psi')$ as $(V_{\INT}, \psi)$, and the integral model thus obtained also satisfies all the properties in the above theorem.
\end{remark}

\subsubsection[]{}\label{general integral models Hodge type}
Now we consider the general case that $\G^\circ\subset \G$ not necessarily equal. Let $K_p^\circ=\G^\circ(\Z_p)\subset$ $ K_p=\G(\Z_p)$ be the associated open compact subgroups of $G(\Q_p)$. Let $K^p\subset G(\Ab_f^p)$ be a sufficiently small open compact subgroup. Set $\K=K_pK^p$, and $\K^\circ=K_p^\circ K^p$. We get the natural  projection \[\Sh_{\K^\circ}(G,X)\ra \Sh_\K(G,X).\] Denote by $\ES_{\K^\circ}(G,X)$ the normalization of $\ES_\K(G,X)$ in $\Sh_{\K^\circ}(G,X)$. Since $K^p$ is sufficiently small, by \cite[Proposition 4.3.7]{Paroh}, the covering \[\ES_{\K^\circ}(G,X)\ra \ES_\K(G,X)\] is \'etale, and it splits over an unramified extension of $O_E$. Let $\widetilde{\ES}_{\K^\circ}(G,X)\ra \ES_{\K^\circ}(G,X)$ be the pullback of the $\G$-torsor over $\ES_\K(G,X)$. By abuse of notation we still denote this morphism by $\pi$, and the composition $\widetilde{\ES}_{\K^\circ}(G,X)\ra\widetilde{\ES}_\K(G,X)\ra 	\mathrm{M}_{G,X}^{\mathrm{loc}}$ by $q$. Thus we get a 
diagram
\[\xymatrix{&\widetilde{\ES}_{\K^\circ}(G,X)\ar[ld]_\pi\ar[rd]^q&\\
	\ES_{\K^\circ}(G,X) & &	\mathrm{M}_{G,X}^{\mathrm{loc}},
}\]
where $\pi$ is a $\G$-torsor, and $q$ is $\G$-equivariant and smooth of relative dimension $\dim G$. As conjectured in \cite{Paroh} 4.3.10, this $\G$-torsor $\widetilde{\ES}_{\K^\circ}(G,X)\ra \ES_{\K^\circ}(G,X)$ should have a reduction to a $\G^\circ$-torsor.

\subsection[Deformations of $p$-divisible groups with crystalline tensors]{Deformations of $p$-divisible groups with crystalline tensors}\label{subsection deformation}
To understand the local geometric structures of $\ES_\K=\ES_\K(G,X)$ or $\mathrm{M}_{G,X}^{\mathrm{loc}}$, we need to study the deformation theory of $p$-divisible groups with crystalline tensors. In this subsection we will mainly follow \cite{Paroh} section 3. As previously we assume $p>2$. Let $k$ be either a finite extension of $\mathbb{F}_p$ or $\ov{\mathbb{F}}_p$.

Let $R$ be a complete local ring with residue field $k$ and maximal ideal $\mathfrak{m}$. We set \[\W(R):=W(k)\oplus \W(\mathfrak{m})\subset W(R),\] where $\W(\mathfrak{m})$ consists of Witt vectors $(w_i)_{i\geq 1}$ such that $w_i\in \mathfrak{m}$ and $(w_i)_{i\geq 1}$ goes to 0 $\mathfrak{m}$-adically. Let $I_R:=\mathrm{ker}(\W(R)\rightarrow R)$. We denote by $\sigma$ the Frobenius endomorphism on $\W(R)$, and $\sigma_1:I_R\rightarrow \W(R)$ the inverse of the Verschiebung $\upsilon$.

\begin{definition}{\rm (\cite[Definition 1]{disp pdiv})}\label{def--dieu disp}
	A Dieudonn\'{e} display over $R$ is a tuple $(M,M_1,\varphi,\varphi_1)$, where
	\begin{itemize}
		\item $M$ is a finite free $\W(R)$-module;
		\item $M_1\subset M$ is a $\W(R)$-submodule such that $I_RM\subset M_1\subset M$ and $M/M_1$ is a projective $R$-module;
		\item $\varphi:M\rightarrow M$ is a $\sigma$-linear map;
		\item $\varphi_1:M_1\rightarrow M$ is a $\sigma$-linear map whose image generates $M$ as a $\W(R)$-module, and which satisfies $\varphi_1(\upsilon(w)m)=w\varphi(m)$, for all $w\in \W(R)$ and $m\in M.$
	\end{itemize}
\end{definition}

It is constructed in \cite{disp pdiv} a functor from the category of $p$-divisible groups over $R$ to the category of Dieudonn\'{e} displays over $R$. Moreover, by the main results there, this functor induces an equivalence of categories.

We are particularly interested in cases when $W(R)$, and hence $\W(R)$, is $p$-torsion free. This holds when $R$ is $p$-torsion free, or $pR=0$ and $R$ is reduced. In this case, one can recover $(M,M_1,\varphi,\varphi_1)$ from $(M,M_1,\varphi_1)$ by taking $\varphi(m):=\varphi_1(\upsilon(1)m)$. We will also call the tuple $(M,M_1,\varphi_1)$ satisfying related conditions in Definition \ref{def--dieu disp} a Dieudonn\'{e} display.

Let $\widetilde{M}_1$ be the image of the homomorphism
$$\sigma^*(i):\sigma^*M_1:=\W(R)\otimes_{\sigma,\W(R)}M_1\rightarrow \sigma^*M=\W(R)\otimes_{\sigma,\W(R)}M$$
induced by the inclusion $i:M_1\rightarrow M$.
\begin{lemma}{\rm (\cite[Lemma 3.1.5]{Paroh})}\label{lemma--linear phi} Suppose that $W(R)$ is $p$-torsion free. Then
	\begin{enumerate}
		\item for a Dieudonn\'{e} display $(M,M_1,\varphi_1)$ over $R$, the linearization of $\varphi_1$ factors as a composition $\sigma^*M_1\rightarrow \widetilde{M}_1\stackrel{\Psi}{\rightarrow}M$ with $\Psi$ an isomorphism;
		\item given an isomorphism $\Psi:\widetilde{M}_1\rightarrow M$, there is a unique Dieudonn\'{e} display $(M,M_1,\varphi_1)$ over $R$ which produces $(M,M_1,\Psi)$ via the construction in {\rm (1)}.
	\end{enumerate}
\end{lemma}
We will also call a triple $(M,M_1,\Psi)$ as above a Dieudonn\'{e} display.

\subsubsection[Versal deformations]{Versal deformations} Let $\mathscr{G}_0$ be a $p$-divisible group over $k$, and $(N,N_1,\phi,\phi_1)$ be the attached Dieudonn\'{e} display. By Lemma \ref{lemma--linear phi} (and using similar notations), this is given by an isomorphism $\Psi_0:\widetilde{N}_1\stackrel{\sim}{\longrightarrow} N$. We will describe a certain versal deformation of $\mathscr{G}_0$ constructed in \cite{Paroh} using Dieudonn\'{e} displays. The Hodge filtration on $N\otimes k$ gives a parabolic subgroup $P_0\subset \GL(N\otimes k)$. We will fix a parabolic subgroup $P\subset \GL(N)$ lifting $P_0$. Let $\Mlo=\GL(N)/P$, and \[\Mlohat=\Spf R\] be the completion of $\Mlo$ along the image of identity in $\GL(N\otimes k)$. In particular, $R$ is a power series ring over $W(k)$.

Let $M=N\otimes_{W(k)} \W(R)$ and $\overline{M}_1\subset M/I_RM$ be the filtration corresponding to the parabolic \[\mathfrak{g}P\mathfrak{g}^{-1}\subset \GL(N_R),\] where $\mathfrak{g}\in (\GL(N)/P)(R)$ is the universal point. Let $M_1$ be the preimage of $\overline{M}_1$ in $M$, $\widetilde{M}_1$ be as in Lemma \ref{lemma--linear phi}, and $\Psi: \widetilde{M}_1\stackrel{\sim}{\longrightarrow} M$ be an isomorphism which reduces to $\Psi_0$ modulo $\W(\mathfrak{m})$. The triple \[(M,M_1,\Psi)\] gives a Dieudonn\'{e} display, and hence a $p$-divisible group $\mathscr{G}$ over $R$ which deforms $\mathscr{G}_0$.

Let $\mathfrak{a}_R=\mathfrak{m}^2+pR\subset R$. By \cite[Lemma 3.1.9]{Paroh}, there is a natural identification
\begin{subeqn}\label{a cano iso}\xymatrix{\widetilde{M}_1\otimes_{\W(R)}\W(R/\mathfrak{a}_R)\ar[r]^{\cong}& \widetilde{N}_1\otimes_{W(k)}\W(R/\mathfrak{a}_R)}
\end{subeqn}
making the following diagram commutative.
$$\xymatrix{
	\widetilde{M}_1\otimes_{\W(R)}\W(R/\mathfrak{a}_R)\ar[d]^{\cong}\ar[r]&\sigma^*\big(M\otimes_{\W(R)}\W(R/\mathfrak{a}_R)\big)\ar@{=}[d]\\
	\widetilde{N}_1\otimes_{W(k)}\W(R/\mathfrak{a}_R)\ar[r]&\sigma^*(N)\otimes_{W(k)}\W(R/\mathfrak{a}_R).}$$
As in \cite[3.1.11]{Paroh}, $\Psi$ is said to be \emph{constant modulo} $\mathfrak{a}_R$ if the composition
$$\widetilde{N}_1\otimes_{W(k)}\W(R/\mathfrak{a}_R)\cong \widetilde{M}_1\otimes_{\W(R)}\W(R/\mathfrak{a}_R)\stackrel{\Psi}{\longrightarrow}M\otimes_{\W(R)}\W(R/\mathfrak{a}_R)\cong N\otimes_{W(k)}\W(R/\mathfrak{a}_R)$$is equal to $\Phi_0\otimes 1$. A useful class of versal deformations is as follows.
\begin{lemma}{\rm (\cite[Lemma 3.1.12]{Paroh})}
	If $\Psi$ is constant modulo $\mathfrak{a}_R$ then the deformation $\mathscr{G}$ is versal.
\end{lemma}

\subsubsection[Deformations with crystalline tensors]{Deformations with crystalline tensors} 
Let $N\otimes k\supset \mathrm{Fil}^1(N\otimes k)$ be the Hodge filtration. By \cite[4C, 4G]{zipaddi}, one has a natural descending filtration on $N^\otimes\otimes k$, and in particular a subspace $\mathrm{Fil}^0\subset N^\otimes\otimes k$.
Now we assume that there is a tensor $s_{\mathrm{cris}}\in N^\otimes$ which is $\phi$-invariant and its image in $N^\otimes\otimes k$ lies in $\mathrm{Fil}^0$. A deformation theory is developed in \cite[3.2]{Paroh} in a more general setting, but to apply to current problems, we simply assume that there is an isomorphism \[f:V_{\INT_p}\otimes W(k)\rightarrow N\] mapping $s$ to $s_{\mathrm{cris}}$. Here $V_{\INT_p}$ and $s$ are as in Theorem \ref{result P-Z and K-P}. We will fix the isomorphism $f$ in discussions here. In particular, we have \[\Mloc\subset (\GL(N)/P)_{O_E}\] which is normal with reduced fibers.

Let $\Spf\,R_G$ be the completion of $\Mloc$ along the image of identity in $\GL(N\otimes k)$. Then $R_G$ is a quotient of $R_E:=R\otimes_{W(k)} O_E$. We set \[M_{R_E}:=M\otimes_{\W(R)}\W(R_E), \quad M_{R_G}:=M\otimes_{\W(R)}\W(R_G),\] and $\widetilde{M}_{R_G,1}\subset M_{R_G}$ be constructed the same way as $\widetilde{M}_{1}$. Note that $R_G$ is $p$-torsion free, since $\Mloc$ is normal with reduced fibers and $\Spf\,R_G$ is the completion of $\Mloc$ along a closed point. In particular,  $W(R_G)$ is $p$-torsion free. Then we have, by \cite[3.1.6]{Paroh}, that \[\widetilde{M}_{R_G,1}=\widetilde{M}_{1}\otimes_{\W(R)}\W(R_G).\]
\begin{lemma}{\rm \cite[Corollary 3.2.11]{Paroh}}\label{lemma--torsor M-1} The tensor $s$ also lies in $\widetilde{M}_{R_G,1}^\otimes$. Moreover, $$\mathcal{T}:=\Isom_{\W(R_G)}\big((\widetilde{M}_{R_G,1},s),(M_{R_G},s)\big)$$ is a trivial $\mathcal{G}$-torsor.
\end{lemma}

Let $\mathfrak{a}_E:=\mathfrak{m}_{R_E}^2+\pi R_E$, where $\pi\in O_E$ is a uniformizer, then $R_E/\mathfrak{a}_E\cong R/\mathfrak{a}_R$. By \cite[3.2.12]{Paroh}, there is an isomorphism $\Psi:\widetilde{M}_{R_E,1}\rightarrow M_{R_E}$ which is constant modulo $\mathfrak{a}_E$ and such that $$\Psi_{R_G}:=\Psi\otimes_{\W(R_E)} \W(R_G):\widetilde{M}_{R_G,1}\rightarrow M_{R_G}$$ respects the tensor $s$.

As in the previous case, $(M_{R_E}, M_{R_E,1}, \Psi)$ gives a Dieudonn\'{e} display, and hence a $p$-divisible group $\mathscr{G}$ over $R_E$ deforming $\mathscr{G}_0$. Moreover, by the previous lemma, any deformations thus obtained are isomorphic, and hence the deformation $\mathscr{G}/R_E$ is versal. For a finite extension $L/E$, an $O_L$-point of $\Spf R_E$ factors through the subspace $\Spf R_G$ if the deformation $\mathscr{G}_{O_L}$ is compatible with the tensor $s$. We refer to \cite[Proposition 3.2.17]{Paroh} for the precise statement. If $\mathscr{G}_0$ is the $p$-divisible group attached to a point $x\in \ES_\K(k)$, the $p$-divisible group $\A_{\widehat{O}_x}[p^\infty]$ is a deformation of $\mathscr{G}_0$, and hence, by versality of $\mathscr{G}$, is given by a morphism $\Spf \widehat{O}_x\rightarrow \Spf R_E$. By \cite[Proposition 4.2.2]{Paroh}, any such morphism induces an isomorphism \[\Spf \widehat{O}_x\stackrel{\sim}{\longrightarrow} \Spf R_G.\]

\subsubsection[Some variations]{Some variations}\label{subsub variations}
Recall that we defined $M$ to be $N\otimes\W(R)$, and we used this in a direct way almost everywhere in previous descriptions. In particular, the Dieudonn\'{e} display $(M_{R_E}, M_{R_E,1}, \Psi)$, and hence \[(M_{R_G}, M_{R_G,1},\Psi_{R_G})\] which determines the $p$-divisible group  $\A_{\widehat{O}_x}[p^\infty]$, is based on this identification. However, it is sometimes more convenient to use it in a less direct way.

We start with the universal element of $(\GL(V_{\INT_p})/P)(R)$ restricted to $\Spf R_G$. It lifts to an element in $\mathfrak{g}\in \GL(V_{\INT_p})(R_G)$ such that $\mathfrak{g}\equiv \mathrm{id}\ \mathrm{mod}\ \mathfrak{m}_{R_E}$, and then to an element $\widetilde{\mathfrak{g}}\in \GL(V_{\INT_p})(\W(R_G))$. Let $N^1\subset N$ be the filtration corresponding to $P$. Identifying $N\otimes \W(R_G)$ and $M_{R_G}$ as before, then \[\mathfrak{g}\cdot (N^1\otimes R_G)\subset N\otimes R_G\] is the Hodge filtration of $M_{R_G}\otimes R_G$. Let $(N\otimes\W(R_G))_1$ be the preimage of $N^1\otimes R_G$ in $N\otimes\W(R_G)$, then \[N_1\otimes\W(R_G)\subset (N\otimes\W(R_G))_1\] and \[\widetilde{\mathfrak{g}}\cdot (N\otimes\W(R_G))_1=M_{R_G,1}.\] We have a natural identification \[\widetilde{N}_1\otimes\W(R_G)\cong (\widetilde{N\otimes\W(R_G)})_1,\] and  $\sigma(\widetilde{\mathfrak{g}})$ induces an isomorphism from $\widetilde{N}_1\otimes\W(R_G)$ to $\widetilde{M}_{R_G,1}$, denoted by $\sigma(\widetilde{\mathfrak{g}})_{\flat}$.
Now the Dieudonn\'{e} display $(M_{R_G}, M_{R_G,1},\Psi_{R_G})$ becomes
\begin{subeqn}\label{new discp for versal display}\big(N\otimes \W(R_G),(N\otimes \W(R_G))_1, \Psi'_{R_G}\big), \text{\ \ \ where\ \ \ } \Psi'_{R_G}:=\widetilde{\mathfrak{g}}^{-1}\circ\Psi_{R_G}\circ \sigma(\widetilde{\mathfrak{g}})_{\flat}.
\end{subeqn}
\begin{remark}\label{key remark}
We would like to describe $\Psi'_{R_G}$ restricted to some interesting sub-formal-schemes of $\Spf R_G$.

\begin{enumerate}
	\item One can not assume the lifting $\mathfrak{g}$ (and hence $\widetilde{\mathfrak{g}}$) to be an element in $\mathcal{G}$ in general, as we don't have smooth coverings $\mathcal{G}_{O_E}\rightarrow \Mloc$. For a sub-formal-scheme $\Spf A\subset \Spf R_G$ such that $\W(A)$ is $p$-torsion free, we can describe the restriction of $(M_{R_G},M_{R_G,1},\Psi_{R_G})$ to $A$ in the same way as in \ref{subsub variations}, and get a triple $\big(N\otimes \W(A),(N\otimes \W(A))_1, \Psi'_{A}\big)$. Here $(N\otimes \W(A))_1$ is the preimage of $N^1\otimes A$ in $N\otimes \W(A)$, and \[\Psi'_{A}=\widetilde{\mathfrak{g}_A}^{-1}\circ\Psi_{R_G}\circ \sigma(\widetilde{\mathfrak{g}_A})_{\flat}\] with $\mathfrak{g}_A\in \GL(V_{\INT_p})(A)$ a lifting of the tautological morphism $\Spf A\rightarrow \Mloc$, and $\widetilde{\mathfrak{g}_A}\in \GL(V_{\INT_p})(\W(A))$ a lifting of $\mathfrak{g}_A$. If $\Spf A$ is contained in the KR stratum of $\Spf R_G$ containing the closed point, i.e. if the morphism $\Spf A\subset \Spf R_G\rightarrow \Mloc$ factors through $\eM^w\subset \Mloc$, then we can choose $\mathfrak{g}_A\in \mathcal{G}(A)$ and $\widetilde{\mathfrak{g}_A}\in \mathcal{G}(\W(A))$ for $\Psi_{A}$.
	
	\item Let $\Spf A\subset \Spf R_G$ be a sub-formal-scheme such that $\W(A)$ is $p$-torsion free. Noting that $$\sigma(\widetilde{\mathfrak{g}_A})_{\flat}=\mathrm{id} \in\GL(V_{\INT_p})(\W(A/\mathfrak{a}_EA))$$ by the canonical identification (\ref{a cano iso}), we have $$\Psi'_{A}=\widetilde{\mathfrak{g}_A}^{-1}\circ(\Psi_0\otimes 1)\in \GL(V_{\INT_p})(\W(A/\mathfrak{a}_EA))$$ as $\Psi$ is assumed to be constant mod $\mathfrak{a}_E$. If $(A,m_A)$ is the complete local ring at $x$ of the KR stratum containing it, then over $A$, we can choose $\widetilde{\mathfrak{g}_A}\in\mathcal{G}(\W(A))$ and find that  $$\Psi'_{A}=\widetilde{\mathfrak{g}_A}^{-1}\circ(\Psi_0\otimes 1)\in\mathcal{G}(\W(A/m^2_A)).$$
\end{enumerate}

\end{remark}


\subsubsection[The global crystalline tensor and related torsors]{The global crystalline tensor and related torsors}\label{subsubsection global cystalline tensor} 
We recall the global crystalline tensor constructed in \cite{Ham-Kim}. For a ring $R$, one can define displays over $R$ as in Definition \ref{def--dieu disp}, but working with $W(R)$ and setting $I_R=\mathrm{ker}(W(R)\rightarrow R)$. If $R$ is $p$-adically complete, Lau \cite[Proposition 2.1]{smooth truncated display} constructed a natural functor $\mathcal{M}$ from the category of $p$-divisible groups over $R$ to that of displays over $R$. If $R$ is, in addition, a complete local ring, then we have a canonical isomorphism \[\mathcal{M}\cong\mathbb{M}\otimes_{\W(R)}W(R).\] Here $\mathbb{M}$ is the Dieudonn\'{e} display attached to a $p$-divisible group $X$, and $\mathcal{M}=\mathcal{M}(X)$ is the attached display. By descent theory of displays in \cite[\S1.3]{disp formal pdiv}, one can define displays over $p$-adic formal schemes. In particular, we have a natural functor form the category of $p$-divisible groups over a $p$-adic formal scheme $\mathcal{S}$ to that of displays over $\mathcal{S}$.

Let $\ES_\K^\wedge$ be the $p$-adic completion of the integral model $\ES_\K/O_{\breve{E}}$, and $(\underline{M}, \underline{M}_1, \underline{\Psi})$ be the display attached to the $p$-divisible group $\A[p^\infty]\mid_{\ES_\K^\wedge}$. Here \[\underline{\Psi}:\widetilde{\underline{M}}_1\rightarrow \underline{M}\] is an isomorphism constructed in a similar way as in Lemma \ref{lemma--linear phi}.

\begin{proposition}{\rm (\cite[Proposition 3.3.1]{Ham-Kim})}\label{global cris} There is a tensor $\underline{s}_{\cris}\in \underline{M}^\otimes$ which is also in $\underline{M}_1^\otimes$, such that $\underline{s}_{\cris}$ is $\underline{\Psi}$-invariant, and that for each $z\in \ES_\K(k)$, its restriction to $\Spf R_G$ is $s_{\cris}$.
\end{proposition}

Let $R$ be $p$-adic ring, $I_R=\mathrm{Im}(V)$ as above, then by \cite[Proposition 3]{disp formal pdiv}, $W(R)$ is $p$-adic and $I_R$-adic. So an element $a\in W(R)$ is a unit if and only if its image in $R$ is a unit. Now for $\mathfrak{n}\in \Max W(R)$, its image in $R$ is a proper ideal, and hence a maximal ideal. In particular, the map $\Max R\rightarrow \Max W(R), \mathfrak{m}\mapsto \widetilde{\mathfrak{m}}$, where $\widetilde{\mathfrak{m}}$ is the pre-image in $W(R)$ of $\mathfrak{m}$, is a bijection.

\begin{lemma}\label{a lemma of flatness}
	Let $R$ be a $p$-adic Noetherian ring which is reduced and $X$ be a scheme over $W(R)$ of finite presentation. Then $X$ is flat over $W(R)$ if $X\otimes_{W(R)}W(R_{\mathfrak{m}}^{\wedge})$ is flat over $W(R_{\mathfrak{m}}^{\wedge})$ for all $\mathfrak{m}\in \Max R$, where $R_{\mathfrak{m}}^{\wedge}$ is the $\mathfrak{m}$-adic completion of $R_{\mathfrak{m}}$.
\end{lemma}
\begin{proof}
	Let us denote by $A^h$ the henselization of a ring $A$. It suffices to show that $X\otimes_{W(R)}W(R)^h_{\widetilde{\mathfrak{m}}}$ is flat over $W(R)^h_{\widetilde{\mathfrak{m}}}$ for all $\mathfrak{m}\in \Max R$. Let $\{\mathfrak{p}_1,\cdots,\mathfrak{p}_n\}$ be the minimal primes of $R$, then we have $R\hookrightarrow \prod_iR_i$ and hence $W(R)\hookrightarrow \prod_iW(R_i)$, where $R_i:=R/\mathfrak{p}_i$. Let $\mathfrak{m}_i\in \Max R_i$ be the image of $\mathfrak{m}$ in $R_i$, we have \[W(R)_{\widetilde{\mathfrak{m}}}\hookrightarrow \prod_iW(R_i)_{\widetilde{\mathfrak{m}}_i}.\] For each $i$, we have a natural homomorphism $f_i:W(R_i)_{\widetilde{\mathfrak{m}}_i}\rightarrow W(R^\wedge_{i,\mathfrak{m}_i})$ by the universality of localizations. We claim that $f_i$ is injective. The lemma follows from this injectivity: we will then have injections
	\[W(R)^h_{\widetilde{\mathfrak{m}}}\hookrightarrow \prod_iW(R_i)^h_{\widetilde{\mathfrak{m}}_i}\hookrightarrow \prod_iW(R^\wedge_{i,\mathfrak{m}_i}).\]
	By \cite[Theorem 4.2.8 (a)]{crit flat}, flatness over $W(R)^h_{\widetilde{\mathfrak{m}}}$ reduces to that over $\prod_iW(R^\wedge_{i,\mathfrak{m}_i})$ which is clear by our assumption.

	Now we prove the injectivity of $f_i$. To simplify notations, we remove the subscript $i$. Noting that $R$ is an integral domain, we see that $R$ is either an $\mathbb{F}_p$-algebra or $p$-torsion free.
	
	\begin{enumerate}
		\item 

	 If $R$ is an $\mathbb{F}_p$-algebra, then $W(R)$ and $W(R^\wedge_{\mathfrak{m}})$ are integral domains with $W(R)\subset$ $W(R^\wedge_{\mathfrak{m}})$. The induced homomorphism $W(R)_{\widetilde{\mathfrak{m}}}\subset W(R^\wedge_{\mathfrak{m}})$ is clearly injective.

	\item If $R$ is $p$-torsion free, then we have a commutative diagram of injective homomorphisms
	$$\xymatrix{
		W(R)\ar[r]\ar[d] &\prod\limits_{i=0}^\infty R\ar[d]\\
		W(R^\wedge_{\mathfrak{m}})\ar[r] &\prod\limits_{i=0}^\infty R^\wedge_{\mathfrak{m}}.}$$
	Here the horizontal maps are given by the Witt polynomials. Elements in $W(R)-\widetilde{\mathfrak{m}}$ are of the form $r=(r_0, r_1,\cdots)$ with $r_0\in R-\mathfrak{m}$. Its image in $\prod_{i=0}^\infty R$ is of the form \[(r_0, r_0^p+pr'_1, \cdots, r_0^{p^n}+pr'_n, \cdots).\] Noting that $p\in \mathfrak{m}$, $r$ is not a zero-divisor in $\prod_{i=0}^\infty R$, and hence $W(R)_{\widetilde{\mathfrak{m}}}\subset W(R^\wedge_{\mathfrak{m}})$.
		\end{enumerate}
\end{proof}

Let $\ES_\K/O_{\breve{E}}$, and $(\underline{M}, \underline{M}_1, \underline{\Psi})$ be as before. We will simply work with its base-change to $R$, where $\Spf R\subset \ES_\K^\wedge$ is open affine.
\begin{corollary}\label{coro--cris torsor}Notations as above, we have
	\begin{enumerate}
		\item the scheme $\widetilde{\ES}_{\K,\cris}^\wedge:=\IIsom_{W(R)}\big((V^\vee_{\INT_p},s)\otimes W(R), (\underline{M}, \underline{s}_\cris)\big)$ is a $\mathcal{G}$-torsor over $W(R)$;
		\item 
		the scheme $\widetilde{\ES}'^\wedge_{\K,\cris}:=\IIsom_{W(R)}\big((V^\vee_{\INT_p},s)\otimes W(R), (\widetilde{\underline{M}}_1, \underline{s}_\cris)\big)$ is a $\mathcal{G}$-torsor over $W(R)$.
	\end{enumerate}
\end{corollary}
\begin{proof}
	Let $R^\wedge$ be the complete local ring at a closed point of $\Spf R$. Then $\widetilde{\ES}_{\K,\cris}^\wedge\otimes W(R^\wedge)$ is a $\mathcal{G}$-torsor by Proposition \ref{global cris}, and hence statement (1) follows from Lemma \ref{a lemma of flatness}. By Lemma \ref{lemma--torsor M-1}, $\widetilde{\ES}'^\wedge_{\K,\cris}\otimes W(R^\wedge)$ is a $\mathcal{G}$-torsor, and hence statement (2) follows the same way.
\end{proof}

\subsection[Integral models of abelian type]{Integral models for Shimura varieties of abelian type}\label{subsection integral abelian}

Recall that a Shimura datum $(G, X)$ is said to be of abelian type, if there is a Shimura datum of Hodge type $(G_1, X_1)$ and a central isogeny $G_1^{\mathrm{der}}\rightarrow G^{\mathrm{der}}$ which induces an isomorphism of adjoint Shimura data $(G_1^{\mathrm{ad}},X_1^{\mathrm{ad}})\stackrel{\sim}{\rightarrow} (G^{\mathrm{ad}},X^{\mathrm{ad}})$. 

Let $(G, X)$ be a Shimura datum of abelian type and fix a choice of a Hodge type datum $(G_1, X_1)$ as above. As before $p>2$.

\subsubsection[]{}
Let $x\in \B(G,\Q_p)$ and $x_1\in \B(G_1,\Q_p)$ such that $x^{\ad}=x_1^\ad\in \B(G^\ad,\Q_p)$. We denote the model of $G$ (resp. $G_1$) over $\Z_p$ defined as the stabilizer of $x$ (resp. $x_1$) by $\G$ (resp. $\G_1$), with connected model $\G^\circ$ (resp. $\G_1^\circ$). We have group schemes $G_{\Z_{(p)}}$ (resp. $G_{1,\Z_{(p)}}$) and $G_{\Z_{(p)}}^\circ$ (resp. $G_{1,\Z_{(p)}}^\circ$) over $\Z_{(p)}$ corresponding to $\G$ (resp. $\G_1$) and $\G^\circ$ (resp. $\G_1^\circ$).
Write $K_p=\G(\Z_p)$, $K_p^\circ=\G^\circ(\Z_p)$, $K_{1,p}=\G_1(\Z_p)$, and $K_{1,p}^\circ=\G_1^\circ(\Z_p)$.

Let $G^{\ad\circ}_{\Z_{(p)}}$ be the parahoric model of $G^\ad_{\Q_p}$ over $\Z_{(p)}$ defined by $x^\ad\in \B(G^\ad,\Q_p)$. Similarly we have $G^{\ad\circ}_{1,\Z_{(p)}}$.
Then \[G^{\ad\circ}_{\Z_{(p)}}=G^{\ad\circ}_{1,\Z_{(p)}}\] as $G^\ad_{\Q_p}=G^\ad_{1,\Q_p}$ and $x^{\ad}=x_1^\ad$ by assumption.  Set $G^\ad_{\Z_{(p)}}=G_{\Z_{(p)}}/Z_{\Z_{(p)}}$, where $Z_{\Z_{(p)}}$ is the closure of the center $Z_G$ of $G$ in $G_{\Z_{(p)}}$. Similarly, we have $G^\ad_{1,\Z_{(p)}}=G_{1,\Z_{(p)}}/Z_{1,\Z_{(p)}}$.
In general, $G^{\ad\circ}_{\Z_{(p)}}$ is not equal to the neutral component $(G^\ad_{\Z_{(p)}})^\circ$ of $G^\ad_{\Z_{(p)}}$. However, suppose that either the center $Z_G$ is connected or that $Z_{G^\mathrm{der}}$ has rank prime to $p$, then by \cite[Lemma 4.6.2 (2)]{Paroh}, $G^{\ad\circ}_{\Z_{(p)}}= (G^\ad_{\Z_{(p)}})^\circ$. In what follows, we will assume that either $Z_{G_1}$ is connected or that $Z_{G^\mathrm{der}_1}$ has rank prime to $p$. Then by the above discussion, we have \[G^{\ad\circ}_{\Z_{(p)}}=G^{\ad\circ}_{1,\Z_{(p)}}=(G^\ad_{1,\Z_{(p)}})^\circ,\] which is the neutral component of $G^\ad_{1,\Z_{(p)}}$.

\subsubsection[]{}\label{subsubsection connected adjoint}
Recall (\cite{varideshi}) that the group \[\mathscr{A}(G)=G(\Ab_f)/Z_G(\Q)^-\ast_{G(\Q)_+/Z_G(\Q)}G^\ad(\Q)^+\] acts on $\Sh(G,X)=\varprojlim_{K}\Sh_{K}(G,X)$. Set \[\mathscr{A}(G)^\circ=G(\Q)^-_+/Z_G(\Q)^-\ast_{G(\Q)_+/Z_G(\Q)}G^\ad(\Q)^+,\] which depends only on $G^{\mathrm{der}}$ and not on $G$. Similarly for $G_1$ we have $\mathscr{A}(G_{1})$ and $\mathscr{A}(G_{1})^\circ$. 

Using the integral models $G^\circ=G_{\Z_{(p)}}^\circ$ and $G^{\mathrm{ad}\circ}=G^{\mathrm{ad}\circ}_{\Z_{(p)}}$ we introduce $$\mathscr{A}(G_{\Z_{(p)}})=G(\mathbb{A}_f^p)/Z_G(\Z_{(p)})^-*_{G^\circ(\Z_{(p)})_+/Z_G(\Z_{(p)})}G^{\mathrm{ad}\circ}(\Z_{(p)})^+$$
and
$$\mathscr{A}(G_{\Z_{(p)}})^\circ=G(\Z_{(p)})_+^-/Z_G(\Z_{(p)})^-*_{G^\circ(\Z_{(p)})_+/Z_G(\Z_{(p)})}G^{\mathrm{ad}\circ}(\Z_{(p)})^+.$$
Similarly for $G_1$ we have $\mathscr{A}(G_{1,\Z_{(p)}})$ and $\mathscr{A}(G_{1,\Z_{(p)}})^\circ$.

\subsubsection[]{}\label{subsubsection Hodge to abelian}
Let $\Sh_{K_p^\circ}(G,X)=\varprojlim_{K^p}\Sh_{K_p^\circ K^p}(G,X)$ be the limit with the $G(\Ab_f^p)$-action. We want to construct an integral model \[\ES_{K_p^\circ}(G,X)\] of $\Sh_{K_p^\circ}(G,X)$ together with the $G(\Ab_f^p)$-action, such that for any $K^p\subset G(\Ab_f^p)$, \[\ES_{\K^\circ}(G,X):=\ES_{K_p^\circ}(G,X)/K^p\] defines an integral model of $\Sh_{\K^\circ}(G,X)$ with $\K^\circ=K_p^\circ K^p$. 

Consider the Hodge type datum $(G_1,X_1)$ and the associated integral models $\ES_{K^\circ_{1,p}K_1^p}(G_1,X_1)$. Set \[\ES_{K^\circ_{1,p}}(G_1,X_1)=\varprojlim_{K_1^p}\ES_{K^\circ_{1,p}K_1^p}(G_1,X_1).\] This is an integral model of $\Sh_{K_{1,p}^\circ}(G_1,X_1)=\varprojlim_{K^p_1}\Sh_{K_{1,p}^\circ K^p_1}(G_1,X_1)$ by subsection \ref{subsection integral Hodge}.

Let $X^+_1\subset X_1$ be a connected component. For $\K_1=K_{1,p}^\circ K_1^p$, let \[\Sh_{\K_1^\circ}(G_1,X_1)^+\subset \Sh_{\K_1^\circ}(G_1,X_1)\] be the geometrically connected component which is the image of $X^+_1\times 1$. Then the scheme $\Sh_{K_1^\circ}(G_1,X_1)^+=\varprojlim_{K^p_1}\Sh_{K_{1,p}^\circ K^p_1}(G_1,X_1)^+$ is defined over $\E_1^p$, where $\E_1$ is the reflex field of $(G_1,X_1)$, and $\E_1^p$ is the maximal extension of $\E_1$ which is unramified at $p$. Let $O_{(p)}$ be the localization at $(p)$ of the ring of integers of $\E_1^p$. We write \[\ES_{\K_1^\circ}(G_1,X_1)^+\] for the closure of $\Sh_{\K_1^\circ}(G_1,X_1)^+$ in $\ES_{\K_1^\circ}(G_1,X_1)\otimes O_{(p)}$, and set \[\ES_{K_{1,p}^\circ}(G_1,X_1)^+:=\varprojlim_{K_1^p}\ES_{\K_1^\circ}(G_1,X_1)^+,\] which is an integral model of $\Sh_{K_1^\circ}(G_1,X_1)^+$ over $O_{(p)}$.
The $G_1^{\mathrm{ad}}(\Z_{(p)})^+$-action on $\Sh_{K_{1,p}^\circ}(G_1,X_1)^+$ extends to $\ES_{K_{1,p}^\circ}(G_1,X_1)^+$, which (after converting to a right action) induces an action of $\mathscr{A}(G_{1,\Z_{(p)}})^\circ$ on $\ES_{K_{1,p}}(G_1,X_1)^+$. 

By \cite[Lemma 4.6.10]{Paroh}  we have an injection
\[\mathscr{A}(G_{1,\Z_{(p)}})^\circ\backslash \mathscr{A}(G_{\Z_{(p)}})\hookrightarrow \mathscr{A}(G_1)^\circ\backslash \mathscr{A}(G)/K^\circ_p. \]
Let $J\subset G(\Q_p)$ denote a set which maps bijectively to a set of coset representatives for the image of $\mathscr{A}(G_{\Z_{(p)}})$ in $\mathscr{A}(G)^\circ\backslash \mathscr{A}(G)/K^\circ_p$. Let $\E'=\E(G,X)\E(G_1,X_1)$.
By \cite[Corollary 4.6.15]{Paroh}, the $O_{\E^{'p},(v)}=O_{\E^{'p}}\otimes O_{(v)}$-scheme
\[\ES_{K^\circ_p}(G,X)=\Big[[\ES_{K_{1,p}^\circ}(G_1,X_1)^+\times\mathscr{A}(G_{\Z_{(p)}})]/\mathscr{A}(G_{1,\Z_{(p)}})^\circ \Big]^{|J|} \] has a natural structure of a $O_{(v)}'=O_{\E'}\otimes O_{(v)}$-scheme with $G(\Ab_f^p)$-action, and is a model of $\Sh_{K_p^\circ}(G,X)$. Moreover, if we denote $E'=\E'_{v'}$ for any $v'|v|p$ place of $\E'$, by \cite[Corollary 4.6.18]{Paroh}, there is a diagram of $O_{E'}$-schemes with $G(\Ab_f^p)$-action
\[\xymatrix{&\widetilde{\ES}^\ad_{K^\circ_p}(G,X)\ar[ld]_\pi\ar[rd]^q&\\
	\ES_{K^\circ_p}(G,X) & &	\mathrm{M}_{G_1,X_1}^{\mathrm{loc}},
}\]
where $G(\Ab_f^p)$ acts trivially on $\mathrm{M}_{G_1,X_1}^{\mathrm{loc}}$,  $\pi$ is a $G^\ad_{1,\Z_p}$-torsor, and $q$ is $G^\ad_{1,\Z_p}$-equivariant. Any sufficiently small open compact $K^p\subset G(\Ab_f^p)$ acts freely on $\widetilde{\ES}^\ad_{K^\circ_p}(G,X)$, and the morphism $\widetilde{\ES}^\ad_{K^\circ_p}(G,X)/K^p\ra \mathrm{M}_{G_1,X_1}^{\mathrm{loc}},$ is smooth of relative dimension $\dim G^\ad_1$.

\subsubsection[]{}\label{subsubsection choice of Hodge datum}
Suppose that $G$ splits over a tamely ramified extension of $\Q_p$. Then by \cite[Lemma 4.6.22]{Paroh}, we can choose the Hodge type datum $(G_1,X_1)$ as above such that it satisfies the following conditions:
\begin{enumerate}
	\item $\pi_1(G^{\mathrm{der}}_1)$ is a $2$-group, and is trivial if $(G^\ad, X^\ad)$ has no factors of type $D^\mathbb{H}$.
	\item $G_1$ splits over a tamely ramified extension of $\Q_p$.
	\item Let $\E=E(G,X), \E_1=E(G_1,X_1)$ and $\E'=\E\cdot \E_1$, then any primes $v|p$ of $\E$ splits completely in $\E'$.
	\item $Z_{G_1}$ is a torus.
	\item $X_\ast(G^{\mathrm{ab}}_1)_{\Gamma_0}$ is torsion free, where $\Gamma_0\subset \Gamma$ is the inertia subgroup.
\end{enumerate}	

\subsubsection[]{}

Let $(G,X)$ be as above such that $G$ splits over a tamely ramified extension of $\Q_p$. Then we can choose a Hodge type datum $(G_1,X_1)$ as in \ref{subsubsection choice of Hodge datum}, such that we get a $G(\Ab_f^p)$-equivariant $O_{E}$-scheme
$\ES_{K^\circ_p}(G,X)$ as in \ref{subsubsection Hodge to abelian}. Moreover, we have

\begin{theorem}{\rm (\cite[Theorem 4.6.23]{Paroh})}\label{T: KP abelian}
	\begin{enumerate}
		\item $\ES_{K^\circ_p}(G,X)$  is \'etale locally isomorphic to $\mathrm{M}_{G_1,X_1}^{\mathrm{loc}}$; if $p \nmid |\pi_1(G^{\mathrm{der}})|$, then $\ES_{K^\circ_p}(G,X)$  is \'etale locally isomorphic to $\mathrm{M}_{G,X}^{\mathrm{loc}}$.
		\item For any discrete valuation ring $R$ of mixed characteristic $(0,p)$, the map \[\ES_{K^\circ_p}(G,X)(R)\ra \ES_{K^\circ_p}(G,X)(R[1/p])\] is a bijection.
		
		\item If $(G^\ad, X^\ad)$ has no factors of type $D^\mathbb{H}$, then $(G_1,X_1)$  can be chosen so that there exists a diagram of $O_{E}$-schemes with $G(\Ab_f^p)$-action
		\[\xymatrix{&\widetilde{\ES}^\ad_{K^\circ_p}(G,X)\ar[ld]_\pi\ar[rd]^q&\\
			\ES_{K^\circ_p}(G,X) & &	\mathrm{M}_{G_1,X_1}^{\mathrm{loc}},
		}\]where $G(\Ab_f^p)$ acts trivially on $\mathrm{M}_{G_1,X_1}^{\mathrm{loc}}$, $\pi$ is a $G^{\ad\circ}_{\Z_p}$-torsor, $q$ is $G^{\ad\circ}_{\Z_p}$-equivariant, and for any sufficiently small open compact $K^p\subset G(\Ab_f^p)$ acts freely on $\widetilde{\ES}^\ad_{K^\circ_p}(G,X)$, and the morphism \[\widetilde{\ES}^\ad_{K^\circ_p}(G,X)/K^p\ra \mathrm{M}_{G_1,X_1}^{\mathrm{loc}}\] is smooth of relative dimension $\dim G^\ad$.
		\item If $G$ is unramified over $\Q_p$, and there exists $x'\in \B(G,\Q_p)$ with $\G_{x^{'\ad}}=\G^\circ_{x^{\ad}}$, then we can choose $(G_1,X_1)$ such that the above construction applies with $x'$ in place of $x$ and gives rise to an $O_{E}$-scheme $\ES_{K^\circ_p}(G,X)$ satisfying the conclusion of (3) above.
	\end{enumerate}
\end{theorem}

We note that in the proof of (3) of the above theorem, $(G_1,X_1)$ is choosen to satisfy all the conditions in \ref{subsubsection choice of Hodge datum} together with a compatible parahoric subgroup $K_{1,p}^\circ\subset G_1(\Q_p)$ such that $K_{1,p}^\circ=K_{1,p}$. In the diagram, the related group is \[G_{\Z_p}^{\ad\circ}=G_{1,\Z_p}^{\ad\circ}=(G_{1,\Z_p}^{\ad})^\circ=G_{1,\Z_p}^{\ad},\]  as in the diagram in \ref{subsubsection Hodge to abelian}. Here the first two equalities follow from \ref{subsubsection connected adjoint}, and the last equality follows from \cite[Proposition 1.1.4]{Paroh}, see also \ref{subsubsection connected hodge adjoint}.
In particular, we can apply the above theorem to a Hodge type datum $(G,X)$ with $\G^\circ\subsetneq \G$ not equal as in \ref{general integral models Hodge type}. In this case, we need to assume that $(G^\ad, X^\ad)$ has no factors of type $D^\mathbb{H}$: we choose another Hodge type datum $(G_1,X_1)$ as above with $\G_1^\circ=\G_1$ to get the $G^{\ad\circ}_{\Z_p}$-local model diagram for $(G,X)$.

\section[EKOR Hodge type]{EKOR strata of  Hodge type: local constructions}\label{section EKOR loc}

In this section, we will construct and study the EKOR stratification for Shimura varieties of Hodge type using the theory of $G$-zips. This will be a local construction, in the sense that we will construct EKOR strata by constructing an EO stratification on each KR stratum. In the next section we will give some global constructions for the whole special fiber.

We will assume $p>2$ throughout the rest of the paper.
For technical reasons, we will always assume condition (\ref{general hypo--refined}). Let $(G,X)$ be a Shimura datum of Hodge type, and $\ES_\K=\ES_\K(G,X)$ be the integral model introduced in \ref{subsection integral Hodge} with $\K=K_pK^p, K_p=\G(\Z_p)$ and $\G=\G^\circ$. Let $\ES_{K_p}=\varprojlim_{K^p}\ES_{K_pK^p}$.
We are interested in the special fibers $\ES_{\K,0}$ of $\ES_\K$, and $\ES_{K_p,0}$ of $\ES_{K_p}$ respectively, over $k=\ov{\mathbb{F}}_p$. We sometimes denote them simply by $\ES_0$ when the level is clear.
We will simply write $K=K_p$. We keep the notations as in subsection \ref{subsection integral Hodge}. In particular, we have the $\Z$-lattice $V_\Z\subset V$, the tensor $s\in V_{\Z_{(p)}}^\otimes$ defining $\G$ and the conjugacy class of cocharacter $\{\mu\}$.

\subsection[Definitions of some stratifications]{Definitions of some stratifications}\label{subsection stratifications hodge}
Notations as in \ref{subsubsec-C(G,mu)}, recall that we have the set $\CGmu$ which is a subset of $C(\mathcal{G})$. As before, let $k=\ov{\mathbb{F}}_p$. Consider the set of $k$-valued points $\ES_{\K,0}(k)=\ES_\K(k)$.
We will construct a map \[\Upsilon_\K:\ES_\K(k)\rightarrow C(\mathcal {G},\{\mu\}).\]  Let $x\in \ES_\K(k)$ and $\widetilde{x}\in \ES_\K(O_F)$ be a lifting of $x$, where $F$ is a finite extension of $\breve{\RAT}_p$. The starting point is the following Theorem \ref{kisin mod}, for which we need to fix some notations.

We set $\mathfrak{S}:=\breve{\INT}_p[[u]]$, with a Frobenius $\varphi$ which is the usual Frobenius on $\breve{\INT}_p$ and $u\mapsto u^p$ on indeterminate. Let $E(u)$ be the Eisenstein polynomial for a fixed uniformizer $\pi\in F$. Then $E(u)\in \mathfrak{S}$. Consider the Galois group
$\Gamma_F=\mathrm{Gal}(\overline{\RAT}_p/F)$. Let $\mathrm{Rep}_{\Gamma_F}^{\mathrm{cris}\circ}$ be the category of $\Gamma_F$-stable $\INT_p$-lattices spanning a crystalline $\Gamma_F$-representation, and $\mathrm{Mod}^{\varphi}_{/\mathfrak{S}}$ be the category of finite free $\mathfrak{S}$-modules $\mathfrak{M}$ with a Frobenius semi-linear isomorphism $$1\otimes \varphi:\varphi^*(\mathfrak{M})[1/E(u)]\rightarrow \mathfrak{M}[1/E(u)].$$
Note that the maximal ideal of $\mathfrak{S}$ is $(p,u)$.
\begin{theorem}{\rm(\cite[Theorem 3.3.2]{Paroh})} \label{kisin mod} There is a fully faithful tensor functor $$\mathfrak{M}:\mathrm{Rep}_{\Gamma_F}^{\mathrm{cris}\circ}\rightarrow\mathrm{Mod}^{\varphi}_{/\mathfrak{S}}$$which is compatible with formation of symmetric and exterior powers, and exact when restricting to $D^\times:=\Spec \mathfrak{S}\backslash \{(p,u)\}$. Moreover,
	
	\begin{enumerate}
		\item for a $p$-divisible group $\mathscr{G}$ over $O_F$, set $L=T_p\mathscr{G}^\vee:=\Hom_{{\INT}_p}(T_p\mathscr{G}, \INT_p)$ and $\mathfrak{M}:=\mathfrak{M}(L)$, then there 
		are canonical isomorphisms
		\[\mathbb{D}(\mathscr{G})(O_F)\simeq O_F\otimes_{\mathfrak{S}}\varphi^\ast(\mathfrak{M}), \quad D_{\mathrm{dR}}(L\otimes\Q_p)\simeq F\otimes_{\mathfrak{S}}\varphi^\ast(\mathfrak{M})\] such that the induced
		injection \[\mathbb{D}(\mathscr{G})(O_F)\rightarrow D_{\mathrm{dR}}(L\otimes\Q_p)\] is compatible with filtrations;
		\item for $\mathscr{G}_0:=\mathscr{G}\otimes k$, $\mathbb{D}(\mathscr{G}_0)(\breve{\INT}_p)$ is canonically identified with  $\varphi^*(\mathfrak{M}/u\mathfrak{M})$.
	\end{enumerate}
\end{theorem}

We consider the $p$-divisible group $\mathscr{G}=\A_{\widetilde{x}}[p^\infty]$ over $O_F$. Notations as in the above theorem, by \cite[4.1.7]{Paroh}, the tensor $s\in V_{\INT_p}^\otimes$ induces a tensor $s_{\text{\'{e}t}, \widetilde{x}}\in L^\otimes$ which is $\Gamma_F$-invariant. Applying the theorem, $s_{\text{\'{e}t},\wt{x}}$ induces a tensor $\wt{s}_{\wt{x}}\in \mathfrak{M}(L)^\otimes$, and thus a tensor $s_{\mathrm{cris},x} \in D_x^\otimes$, which is $\varphi_x$-invariant. Here $D_x:=\mathbb{D}(\mathscr{G}_0)(\breve{\INT}_p)$, and $\varphi_x$ is the Frobenius. We remark that, by \cite[Proposition 1.3.9 (2)]{LRKisin}, $s_{\mathrm{cris},x} \in D_x^\otimes$ is independent of choices of $\widetilde{x}$. Now by \cite[Corollary 3.3.6]{Paroh}, $$I_x:=\IIsom_{\breve{\mathbb{Z}}_p}((V^\vee_{\mathbb{Z}_{(p)}}\otimes \breve{\mathbb{Z}}_p,s),(D_x,s_{\mathrm{cris},x}))$$
is a trivial $\mathcal {G}$-torsor. Thus we can take a $t\in I_x(\breve{\mathbb{Z}}_p)$. Then the pull back of $\varphi_x$ via \[t: (V^\vee_{\mathbb{Z}_{(p)}}\otimes \breve{\mathbb{Z}}_p,s)\ra(D_x,s_{\mathrm{cris},x})\] is of the form \[(\mathrm{id}\otimes \sigma)\circ g_{x,t},\] for some $g_{x,t}\in G(\breve{\mathbb{Q}}_p)$. The image of $g_{x,t}$ in $C(\mathcal {G})$, denote by $g_x\in C(\mathcal {G})$, is independent of $t$. In particular, \[x\mapsto g_x\] gives a well defined map \[\Upsilon_\K:\ES_\K(k)\rightarrow C(\mathcal{G})\] whose fibers are called \emph{central leaves}.

Let us check that $\Upsilon_\K$ factors through $\CGmu$. It suffices to prove the following statement.

\begin{lemma}\label{L:frob}
	The element $g_{x,t}\in G(\breve{\mathbb{Q}}_p)$ as above is of the form $\breve{K}\Admu\breve{K}$.
\end{lemma}
\begin{proof}
	By the Bruhat-Tits decomposition, we know that $g_{x,t}\in \breve{K}w\breve{K}$, for some $w\in {}^K\IW^K$. We view $g_{x,t}$ as a homomorphism $V_k^\vee\ra V_k^\vee$.
	The Hodge filtration $t^{-1}(\mathrm{Ker}(\varphi_x))=\mathrm{Ker}(g_{x,t})$ gives, by the local model diagram, an element in $$\Admu_K=W_K\backslash\Admu^K/W_K.$$In particular, $w\in \Admu^K$.
\end{proof}

We will also write $\Upsilon_\K$ for the map $\ES_\K(k)\rightarrow \CGmu$. Composing $\Upsilon_\K$ with the natural maps in \ref{proj from cent}, we get maps
\begin{itemize}
	\item $\delta_\K:\ES_\K(k)\rightarrow \BGmu$ whose fibers are called \emph{Newton strata};
	
	\item $\lambda_\K:\ES_\K(k)\rightarrow \Admu_K$ whose fibers are called \emph{Kottwitz-Rapoport strata}, or KR strata for short;
	
	\item $\upsilon_\K:\ES_\K(k)\rightarrow {}^K\Admu$ whose fibers are called \emph{Ekedahl-Kottwitz-Oort-Rapoport strata}, or EKOR strata for short.
\end{itemize}
Recall that by the local diagram in Theorem \ref{result P-Z and K-P} we can also define the KR strata (see \cite[Corollary 4.2.12]{Paroh}). Now we explain that the two constructions for KR strata are compatible.
Let the notations be as in Lemma \ref{L:frob}. For $x\in \mathscr{S}_\K(k)$, a section $t\in I_x(\breve{\mathbb{Z}}_p)$ induces an element $g_{x,t}\in \breve{K}\Admu\breve{K}$. Let $M_0=\mathrm{M}_{G,X}^{\mathrm{loc}}\otimes_{O_E} k$. Assigning to  $g_{x,t}$ the point in $M_0(k)$ corresponding to the filtration $\mathrm{Ker}(g_{x,t})\subset V^\vee_{ k}$ induces a bijection \[\breve{K}\backslash \breve{K}\Admu\breve{K}/\breve{K}\stackrel{\simeq}{\rightarrow}\mathcal{G}_0(k)\backslash M_0(k).\] This identifies fibers of $\lambda_K$ with the KR strata defined using the local model diagram as in Theorem \ref{result P-Z and K-P}.

Combined with the diagram for group theoretic data obtained in subsection \ref{subsubsec-C(G,mu)}, we have the following commutative diagram.
\[\xymatrix{ &\Admu_K &{}^K\Admu\ar@{->>}[l] \\
	\ES_\K(k)\ar[ur]^{\lambda_\K}\ar[rr]^{\Upsilon_\K} \ar[urr]^{\upsilon_\K} \ar[drr]^{\delta_\K}
	& &\CGmu\ar@{->>}[u]\ar@{->>}[d]
	&\ {}^K\Admu_{\sigma\text{-str}}\ar@{_(->}[l] \ar@{_(->}[ul]\ar@{->>}[dl]\\
	& &B(G, \{\mu\}) }\]
By construction, it is clear that when the prime to $p$ level $K^p$ varies, the maps $\Upsilon_\K, \delta_\K,\upsilon_\K$ and $\lambda_\K$ are compatible for the natural projections.

We will need geometric structures on the above sets:
\begin{itemize}
	
	\item One can use the $F$-crystal attached to the universal abelian scheme $\A\ra \ES_0$ together with the $G$-additional structure to show that the fibers of $\Upsilon_\K$ and $\delta_\K$ are the sets of $k$-valued points of locally closed reduced subschemes of $\ES_{0}$, see for example \cite[Corollary 3.3.8]{Ham-Kim}. We also call these locally closed reduced subschemes central leaves and Newton strata respectively. 
	\item The local model diagram in Theorem \ref{result P-Z and K-P} shows that the fibers of $\lambda_\K$ are the sets of $k$-valued points of locally closed reduced subschemes of $\ES_{0}$, which we call equally the KR (Kottwitz-Rapoport) strata. 
	\item Our main task in this section is to show that there exists also geometric structures on the fibers of $\upsilon_\K$, cf. Theorem \ref{sm of zeta}.
\end{itemize}

As for the non-emptiness of the above various strata, we list the following results.
\begin{enumerate}
	\item In the hyperspecial case, Viehmann-Wedhorn proved the non-emptiness of Newton strata in the PEL type case (\cite{VW}), and Dong-Uk Lee proved the non-emptiness of Newton strata in the Hodge type case \cite{Newton non-epty}.
	\item Kisin-Madapusi Pera-Shin (\cite{KMS}) and C.-F. Yu (\cite{Yu}) proved in the Hodge type case for $G$ quasi-split over $\Q_p$, that the Newton strata are non-empty.
	\item R. Zhou proved in \cite[Proposition 8.2]{zhou isog parahoric} that all KR strata are non-empty by applying the key input from \cite{KMS} that the basic Newton strata are non-empty (see also \cite{Yu}), and then all Newton strata are non-empty in this case.
	\item Based on these results (especially the non-emptiness of KR strata in \cite{zhou isog parahoric}), we will show later all EKOR strata are non-empty, cf. Corollary \ref{coro--non-emp and closure}.
\end{enumerate}

\subsection[He-Rapoport axioms]{He-Rapoport axioms}\label{He-Rap axioms}We turn to a different (abstract) setting.
He and Rapoport introduced five axioms in \cite{He-Rap} for integral models of Shimura varieties with parahoric levels to study stratifications in the special fibers. We will not recall all the axioms here, but just introduce some of them which will be used in this paper. Let $(G,X)$ be a Shimura datum. Assume that the for each parahoric subgroup $K_p\subset G(\Q_p)$, we already have an integral model $\ES_{\K}=\ES_\K(G,X)$ for the associated Shimura variety $\Sh_\K=\Sh_\K(G,X)$ over $O_E$, the ring of integers of a local reflex field at $p$. Here $\K:=K_pK^p$ with $K^p$ small enough.

\subsubsection[Axiom 1]{Axiom 1}\label{ax 1} Let $K_p\supset K_p'$ be another parahoric subgroup, and $\K':=K'_pK^p$. Then there is a natural morphism \[\pi_{\K',\K}:\ES_{\K'}\rightarrow \ES_{\K}\] which is proper surjective and extends the natural morphism on the generic fibers.

Let $\ES_{\K',0}$ and $\ES_{\K,0}$ be their special fibers respectively, axiom 1 implies that the induced morphism $\pi_{\K',\K,0}:\ES_{\K',0}\rightarrow \ES_{\K,0}$ is proper surjective.

\subsubsection[Axiom 4 (c)]{Axiom 4 (c)}\label{ax 4-c} Let $\ES_{\K,0}^c\subset \ES_{\K,0}$ and $\ES_{\K',0}^{c'}\subset \ES_{\K',0}$ be central leaves with \[\pi_{\K,\K',0}(\ES_{\K,0}^c)\subset \ES_{\K',0}^{c'}.\] Then the induced morphism $\pi_{\K,\K',0}^c:\ES_{\K,0}^c\rightarrow \ES_{\K',0}^{c'}$ is surjective with finite fibers.
\subsubsection[Axiom 5]{Axiom 5}\label{ax 5} Let $I\subset G(\Q_p)$ be an Iwahori subgroup, and $\tau$ be as in \ref{tau}, then the intersection of the KR stratum $\ES_{IK^p,0}^\tau$ with each geometrically connected component of $\ES_{IK^p,0}$ is non empty.

\subsubsection[Known cases]{Known cases}\label{ax checked}
\begin{itemize}
	\item For Siegel modular varieties the axioms are verified by He-Rapoport in \cite{He-Rap}.
	\item For PEL-type Shimura varieties associated to unramified groups of type $A$ and $C$ and to odd ramified unitary groups, the axioms are verified by He-Zhou in \cite{HZ}.
	
	\item For Shimura varieties of Hodge type satisfying condition (\ref{general hypo--refined}), it has been checked by R. Zhou, see \cite{zhou isog parahoric}, that all the five axioms  except for the surjectivity in axiom 4 (c) hold for the Kisin-Pappas integral models as in \ref{subsection integral Hodge}. If $G_{\RAT_p}$ is in addition residually split, then all the five axioms hold. 
\end{itemize}

\subsection[Pointwise constructions]{$\mathcal{G}_0^\mathrm{rdt}$-zips attached to points}\label{subsec--pointwise constr}
We come back to the setting as that at the beginning of this section.
We start with some constructions generalizing those in \cite[3.2.6]{EOZ}. For $w\in \IW$, we simply write the same symbol for a representative in $G(\breve{\RAT}_p)$. Let $\mathfrak{a}$ be the facet corresponding to $\mathcal{G}$, we set \[\mathfrak{b}:=\mathfrak{a}\cup w\mathfrak{a} \quad \textrm{and} \quad \mathfrak{c}:=\mathfrak{a}\cup \sigma(w)^{-1}\mathfrak{a}.\] Here $w\mathfrak{a}$ (resp. $\sigma(w)^{-1}\mathfrak{a}$) is the translation by $w$ (resp. $\sigma(w)^{-1}$) of $\mathfrak{a}$. Let $\mathcal{G}_{\mathfrak{b}}$ and $\mathcal{G}_{\mathfrak{c}}$ be the corresponding Bruhat-Tits group schemes respectively, then working with reduced $k$-algebras, we see by Remark  \ref{loop in mixed char+} that $\sigma'=\sigma\circ \mathrm{Ad}(w)$ induces a homomorphism \[L^+\mathcal{G}_{\mathfrak{b}}\rightarrow L^+\mathcal{G}_{\mathfrak{c}}.\] In particular, we have an isogeny of smooth group varieties over $k$ \[\mathcal{G}_{\mathfrak{b},0}\rightarrow \mathcal{G}_{\mathfrak{c},0},\] which is again denoted by $\sigma'$. Recall that we have homomorphisms
\[\mathcal{G}_{\mathfrak{b},0}\ra \G_{\mathfrak{a},0}, \quad \mathcal{G}_{\mathfrak{c},0}\ra \G_{\mathfrak{a},0}. \]
We will also set $\mathcal{G}_{0,w}$ (resp. $\mathcal{G}_{0,\sigma(w)^{-1}}$) to be the image of $\mathcal{G}_{\mathfrak{b},0}$ (resp. $\mathcal{G}_{\mathfrak{c},0}$) in $\mathcal{G}_{0}=\G_{\mathfrak{a},0}$. It is a quotient of $\mathcal{G}_{\mathfrak{b},0}$ (resp. $\mathcal{G}_{\mathfrak{c},0}$), and hence is a smooth group scheme over $k$.

Before moving into detailed constructions, we would like to say a few words about the main ideas. Our strategy here is the same as that in \cite[\S3.2]{EOZ}, and we will also use $F$-zips introduced by Moonen and Wedhorn (\cite{disinv}) freely. The definition of an $F$-zip will not be recalled, as it will be used in explicit forms here. Recall for each $w\in\Admu_K$, we have constructed an algebraic zip datum \[\mathcal {Z}_w=(\mathcal{G}_0^{\mathrm{rdt}}, \overline{P}_{J_w}, \overline{P}_{\sigma'(J_w)}, \sigma':\overline{L}_{J_w}\rightarrow \overline{L}_{\sigma'(J_w)})\] in \ref{zip in EKOR}.
We will construct first an $F$-zip from the element $w\in \Admu_K$, and use it as a standard $F$-zip attached to $w$. Properties of this standard $F$-zip will also be studied. There are also $F$-zips attached to points of $\ES_0$. We will then show that one gets desired $\mathcal{G}_0^\mathrm{rdt}$-zips by ``comparing'' the $F$-zips at points with the standard ones.

\subsubsection[A standard $F$-zip]{A standard $F$-zip}
For $w\in \Admu_K$, we will always fix a representative of it in $G(\breve{\Q}_p)\subset\GL(V)(\breve{\RAT}_p)$ without changing notations. Then $w$ has a left action on $V_{\breve{\RAT}_p}$ which induces,
 by taking inverse dual definition, an endomorphism of $V_{\breve{\INT}_p}^\vee$ which is again denoted by $w$. Consider the reduction modulo $p$   of $V_{\breve{\INT}_p}^\vee$, the $k$-vector space $V_k^\vee$.
Let $w_k$ be the induced map on $V_k^\vee$, we have an $F$-zip structure $(C^\bullet_w,D^w_\bullet, \varphi_\bullet^w)$ on $V_k^\vee$ given by
\begin{itemize}
	\item $C_w^0:=\dkV\supset C^1_w:=\mathrm{Ker}(w_k)\supset C^2_w:=0;$

	\item $D_{-1}^w:=0\subset D_0^w:=\mathrm{Im}((\sigma(w))_k)\subset D_1^w:=\dkV;$

	\item $\varphi_0^w:C^0_w/C^1_w\rightarrow D_0^w$ and $\varphi_1^w:C^1_w\rightarrow D_1^w/D_0^w$, which are $\sigma$-linear isomorphisms whose linearizations are induced by $(\sigma(w))_k$ and the inverse of the isomorphism $(\frac{p}{\sigma(w)})\otimes k: D_1^w/D_0^w\rightarrow C_w^{1,\sigma}$ respectively.
\end{itemize}
We remark that the filtration $C^\bullet_w$ (resp. $D_\bullet^w$) depends only on the image of $w$ (resp. $\sigma(w)$) in $G(\breve{\RAT}_p)/\breve{K}$ (resp. $\breve{K}\backslash G(\breve{\RAT}_p)$).
\begin{lemma}\label{lemma--zip datum attached to w}
	Let $\Aut(C^\bullet_w,s)$ be the group scheme of automorphisms on $V_k^\vee$ respecting both the filtration $C^\bullet_w$ and the tensor $s\in V^\otimes_k$, similarly for $\Aut(D_\bullet^w,s)$. We have the following.
	\begin{enumerate}
		\item $\Aut(C^\bullet_w,s)=\mathcal{G}_{0,w}$ and $\Aut(D_\bullet^w,s)=\mathcal{G}_{0,\sigma(w)^{-1}}$.
		\item Let $\mathcal{G}_{0,w}^L$ be the image of $\mathcal{G}_{0,w}$ in $\Aut(\oplus_i\gr_i(C^\bullet_w))$ and similarly for $\mathcal{G}_{0,\sigma(w)^{-1}}^L$, then $\sigma'$ induces an isomorphism \[\mathcal{G}_{0,w}^{L,(p)}\rightarrow \mathcal{G}_{0,\sigma(w)^{-1}}^L\] which is again denoted by $\sigma'$.
		\item The isomorphism \[\oplus_i\gr_i(C^{\bullet,(p)}_w)\rightarrow \oplus_i\gr_i(D_\bullet^w)\] induced by $\varphi_0^w\oplus \varphi_1^w$ is equivariant with respect to $\sigma'$.
	\end{enumerate}
\end{lemma}
\begin{proof}
	For (1), it is more convenient to use equal-characteristic results. Let $G'/k((t))$ and $\mathcal{G}'/k[[t]]$ be as in subsection \ref{subsection loc models} (after Remark \ref{remark--Goertz PEL}). In particular, $w$ is identified with an element $w'$ in the Iwahori Weyl group of $G'$. Let $\mathfrak{a}'$ be the facet corresponding to $\mathcal{G}'$, we set, as before \[\mathfrak{b}':=\mathfrak{a}'\cup w'\mathfrak{a}' \quad \textrm{and} \quad \mathfrak{c}':=\mathfrak{a}'\cup \sigma(w')^{-1}\mathfrak{a},\]
	and get Bruhat-Tits group schemes $\mathcal{G}'_{\mathfrak{b}'}$ and $\mathcal{G}'_{\mathfrak{c}'}$. The homomorphisms
	$\mathcal{G}_{\mathfrak{b},0}\ra \G_{0}$ and $\mathcal{G}_{\mathfrak{c},0}\ra \G_{0}$ are identified with $\mathcal{G}'_{\mathfrak{b}',0}\ra \G'_{0}=\G_0$ and $\mathcal{G}'_{\mathfrak{c}',0}\ra \G'_{0}=\G_0$, so $\mathcal{G}_{0,w}$ and $\mathcal{G}_{0,\sigma(w)^{-1}}$ are identified with $\mathcal{G}'_{0,w'}$ and $\mathcal{G}'_{0,\sigma(w')^{-1}}$ respectively. Let \[\mathscr{C}:=(\cdots\subset tV_{k[[t]]}^\vee\subset \frac{t}{w'}V_{k[[t]]}^\vee\subset V_{k[[t]]}^\vee\subset\cdots)\]  and \[ \mathscr{D}:=(\cdots\subset tV_{k[[t]]}^\vee\subset \sigma(w')V_{k[[t]]}^\vee\subset V_{k[[t]]}^\vee\subset\cdots)\]
	be the periodic lattice chain corresponding to $C_w^\bullet=C_{w'}^\bullet$ and $D^w_\bullet=D^{w'}_\bullet$ respectively. To simplify notations, we will write $\mathcal{H}$ for $\GL(V_{k[[t]]})$, $\breve{K}'^0$ for $\mathcal{H}(k[[t]])$ and $P^0$ (resp. $Q^0$) for the stabilizer of $C_w^\bullet$ (resp. $D^w_\bullet$) in $\GL(V_k)$. The Bruhat-Tits group scheme corresponding to $\breve{K}'^0\cap w'\breve{K}'^0 w'^{-1}$ is then $\GL(\mathscr{C})$ which is parahoric. We have a commutative diagram
	$$\xymatrix{
		L^+\mathcal{G}'w'L^+\mathcal{G}'/L^+\mathcal{G}'\ar[r]^{\cong}\ar[d]&L^+\mathcal{G}'/L^+\mathcal{G}'_{\mathfrak{b}'}\ar[r]^{\cong}\ar[d]&\mathcal{G}_0/\mathcal{G}_{0,w}\ar[d]\\
		L^+\mathcal{H}w'L^+\mathcal{H}/L^+\mathcal{H}\ar[r]^{\cong}&L^+H/L^+\GL(\mathscr{C})\ar[r]^{\cong}&\GL(V_k)/P^0.}$$
	By the arguments in the proof of \cite[Proposition 8.1]{local model P-Z} (more precisely, in the middle of page 207), the induced morphism of affine Grassmannians $\Gr_{\mathcal{G}'}\rightarrow \Gr_{\GL(k[[t]])}$ is a closed immersion, and hence the left vertical arrow in the above diagram is an immersion. But then the right vertical arrow will also be an immersion, and hence the first identification in (1) follows.
	
	The periodic lattice chain corresponding to $C_w^{\bullet,(p)}$ is
	\[(\cdots\subset tV_{k[[t]]}^\vee\subset \DF{t}{\sigma(w')}V_{k[[t]]}^\vee\subset V_{k[[t]]}^\vee\subset\cdots).\] Noting that the left action of $\sigma(w')$ maps the lattice chain of $C_w^{\bullet,(p)}$ to that of $D^w_\bullet$, we have, as in the previous case, a commutative diagram
	$$\xymatrix{
		L^+\mathcal{G}'\sigma(w')^{-1}L^+\mathcal{G}'/L^+\mathcal{G}'\ar[r]^{\cong}\ar@{_(->}[d]&L^+\mathcal{G}'/L^+\mathcal{G}'_{\mathfrak{c}'}\ar[r]^{\cong}\ar@{_(->}[d]&\mathcal{G}_0/\mathcal{G}_{0,\sigma(w)^{-1}}\ar@{_(->}[d]\\
		L^+\mathcal{H}\sigma(w')^{-1}L^+\mathcal{H}/L^+\mathcal{H}\ar[r]^{\cong}&L^+\mathcal{H}/L^+\GL(\mathscr{D})\ar[r]^{\cong}&\GL(V_k)/Q^0.}$$
	By the same reason, the second identification in (1) follows.
	
	For (2) and (3), both the mixed characteristic method and the equi characteristic one work. We will use the mixed one. Let $L^0=\Aut(\oplus_i\gr_i(C^{\bullet}_w))$ (resp. $L'^0=\Aut(\oplus_i\gr_i(D_\bullet^w))$) be the Levi quotient attached to $P^0$ (resp. $Q^0$). The action of $\sigma(w)$ induces an isomorphism from the parahoric subgroup (of $\GL(V)(\breve{\RAT}_p)$) of $C_w^{\bullet,(p)}$ to that of $D^w_\bullet$, and hence induces an isomorphism $L^{0,(p)}\stackrel{\cong}{\longrightarrow}L'^0$. Moreover, it induces an commutative diagram $$\xymatrix{
		\mathcal{G}_{\mathfrak{b},0}^{(p)}\ar[r]\ar[d]&\mathcal{G}_{\mathfrak{c},0}\ar[d]\\
		L^{0,(p)}\ar[r]&L'^0,}$$
	where the vertical arrows factor through $\mathcal{G}_{0,w}^{(p)}$ and $\mathcal{G}_{0,\sigma(w)^{-1}}$ respectively. In particular, (2) follows.
	
	The isomorphism $\oplus_i\gr_i(C^{\bullet,(p)}_w)\rightarrow \oplus_i\gr_i(D_\bullet^w)$ induced by $\varphi_0^w\oplus \varphi_1^w$ is also induced by the action of $\sigma(w)$, and hence is equivariant with respect to the isomorphism $L^{0,(p)}\rightarrow L'^0$. In particular, (3) follows.
	
\end{proof}


The natural homomorphisms $\mathcal{G}_{0,w}\rightarrow \overline{P}_{J_w}$ and $\mathcal{G}_{0,\sigma(w)^{-1}}\rightarrow \overline{P}_{\sigma'(J_w)}$ are compatible with the isomorphisms $\sigma'$s. The tuple $\widetilde{\mathcal {Z}}_w:=(\mathcal{G}_0, \mathcal{G}_{0,w}, \mathcal{G}_{0,\sigma(w)^{-1}} , \sigma')$ could be viewed as a generalization or ``lifting'' of the zip datum $\mathcal {Z}_w=(\mathcal{G}_0^{\mathrm{rdt}}, \overline{P}_{J_w}, \overline{P}_{\sigma'(J_w)}, \sigma')$. It is natural and necessary to work with structures related to $\widetilde{\mathcal {Z}}_w$.

\subsubsection[$F$-zips attached to points]{$F$-zips attached to points}\label{global F-zip} Fix a sufficiently small $K^p$ and write $\ES_0=\ES_{\K,0}$ with $\K=KK^p$.
Let $\A\ra \ES_{0}$ be the universal abelian scheme and consider the vector bundle $\V_0=H^1_{\dr}(\A/\ES_0)$ over $\ES_0$.
We have an $F$-zip structure $(C^\bullet, D_\bullet, \varphi_\bullet)$ on $\V_0$, given by
\begin{itemize}
	\item $C^0:=\V_0\supseteq C^1:=\V^1_0\supseteq C^2:=0$, which is the Hodge filtration;

	\item $D_{-1}:=0\subset D_0:=\mathrm{Im}(F)\subset D_1:=\V_0$, which is the conjugate filtration;

	\item $\varphi_0:C^0/C^1\rightarrow D_0$ and $\varphi_1:C^1\rightarrow D_1/D_0$, which are induced by Frobenius and the inverse of Verschiebung respectively.
\end{itemize}

Now we will fix a $w\in\Admu_K$ and a $k$-point $x$ in the KR-stratum $\ES_0^w$, and state the main results of our pointwise constructions. We will denote by $(C^\bullet_x, D_{\bullet,x}, \varphi_{\bullet,x})$ the pull back to $x$ of $(C^\bullet, D_\bullet, \varphi_\bullet)$. In the following we will talk about torsors over $k$, which are actually trivial since $k=\Fpbar$.
\begin{proposition}\label{prop--lift G-zip at point}
	Let $\widetilde{\I}_x$ be the $\mathcal{G}_0$-torsor $\IIsom_k((V_k^\vee,s), (\V_{0,x},s_{\dr}))$ over $k$.
	\begin{enumerate}
		\item Let $\widetilde{\I}_{x,+}:=\IIsom_k((C^\bullet_w,s), (C^\bullet_x,s_{\dr}))\subset \widetilde{\I}_x$, then it is a $\mathcal{G}_{0,w}$-torsor over $k$.
		\item Let $\widetilde{\I}_{x,-}:=\IIsom_k((D_\bullet^w,s), (D_{\bullet,x},s_{\dr}))\subset \widetilde{\I}_x$, then it is a $\mathcal{G}_{0,\sigma(w)^{-1}}$-torsor over $k$.
		\item Let $\mathcal{G}_{0,w}^U =\mathrm{Ker}(\mathcal{G}_{0,w}\twoheadrightarrow \mathcal{G}_{0,w}^L)$ and similarly for $\mathcal{G}_{0,\sigma(w)^{-1}}^U$. The Dieudonn\'{e} module structure on $\V_x$ induces an isomorphism \[\widetilde{\iota}:\widetilde{\I}_{x,+}^{(p)}/\mathcal{G}_{0,w}^{U,(p)}\rightarrow \widetilde{\I}_{x,-}/\mathcal{G}_{0,\sigma(w)^{-1}}^U\] which is equivariant with respect to the isomorphism $\sigma':\mathcal{G}_{0,w}^{L,(p)}\rightarrow \mathcal{G}_{0,\sigma(w)^{-1}}^L$.
	\end{enumerate}
\end{proposition}
\begin{proof}
	Let $D_x$ be the Dieudonn\'{e} module of $\A_x[p^{\infty}]$ and $\widetilde{\widetilde{\I}}_x$ be the (trivial) $\mathcal{G}$-torsor given by \[\widetilde{\widetilde{\I}}_x:=\IIsom_{\breve{\INT}_p}((V_{\breve{\INT}_p}^\vee,s), (D_x,s_{\mathrm{cris}})).\] We fix a section $t\in \widetilde{\widetilde{\I}}_x(\breve{\INT}_p)$ such that $t^*\varphi=\sigma\circ gw$ for some $g\in \mathcal{G}(\breve{\INT}_p)$. Let $\overline{t}$ be the image of $t$ in $\widetilde{\I}_x$, one sees immediately that $$\widetilde{\I}_{x,+}=\mathcal{G}_{0,w}\cdot \overline{t}, \quad \text{ and }\quad \widetilde{\I}_{x,-}=\mathcal{G}_{0,\sigma(w)^{-1}}\cdot \overline{g\cdot t}.$$In particular, (1) and (2) hold.

	Let $\widetilde{\I}_{x,+}^L$ be the image of $\widetilde{\I}_{x,+}$ in $\IIsom_k(\oplus_i\gr_i (C^\bullet_w), \oplus_i\gr_i (C^\bullet_x))$, then $\widetilde{\I}_{x,+}^L=\widetilde{\I}_{x,+}/\mathcal{G}_{0,w}^{U}$. Similarly, we have $\widetilde{\I}_{x,+}^{L,(p)}$ and $\widetilde{\I}_{x,-}^{L}$, as well as identifications $\widetilde{\I}_{x,+}^{L,(p)}=\widetilde{\I}_{x,+}^{(p)}/\mathcal{G}_{0,w}^{U,(p)}$ and $\widetilde{\I}_{x,-}^{L}=\widetilde{\I}_{x,-}/\mathcal{G}_{0,\sigma(w)^{-1}}^U$. Now we can define a morphism \[\widetilde{\iota}:\widetilde{\I}_{x,+}^{L,(p)}\rightarrow \IIsom_k(\oplus_i\gr_i (D_\bullet^w), \oplus_i\gr_i (D_{\bullet,x}))\] as follows. It is induced by mapping $f:\oplus_i\gr_i (C^{\bullet,(p)}_w)\stackrel{\cong}{\longrightarrow}\oplus_i\gr_i (C^{\bullet,(p)}_x)$ to the composition
	\begin{subeqn}\label{def of lift iota}
		\widetilde{\iota}(f):\oplus_i\gr_i (D_\bullet^w)\stackrel{\alpha}{\longrightarrow}\oplus_i\gr_i (C^{\bullet,(p)}_w)\stackrel{f}{\longrightarrow}\oplus_i\gr_i (C^{\bullet,(p)}_x)\stackrel{\oplus_i\varphi^\mathrm{lin}_\bullet}{\longrightarrow}\oplus_i\gr_i (D_{\bullet,x}).
	\end{subeqn}
	Here $\alpha$ is the inverse of the linearization of $\oplus_i\varphi_{w,\bullet}$, and $\varphi^\mathrm{lin}_\bullet$ is the linearization of $\varphi_\bullet$.

	Noting that $\widetilde{\iota}$ is equivariant with respect to $\sigma':\mathcal{G}_{0,w}^{L,(p)}\rightarrow \mathcal{G}_{0,\sigma(w)^{-1}}^L$ as $\alpha$ is so by Lemma \ref{lemma--zip datum attached to w} (3), so to prove statement (3) here, it suffices to check that $\widetilde{\iota}$ factors through $\widetilde{\I}_{x,-}^{L}$. To see this, one can simply take $f$ to be the map induced by $\overline{t}^{(p)}$, then it is clear that $\widetilde{\iota}(f)\in \widetilde{\I}_{x,-}^{L}$.
\end{proof}
Recall that in \ref{zip in EKOR} we have constructed an algebraic zip datum
\[\mathcal {Z}_w=(\mathcal{G}_0^{\mathrm{rdt}}, \overline{P}_{J_w}, \overline{P}_{\sigma'(J_w)}, \sigma':\overline{L}_{J_w}\rightarrow \overline{L}_{\sigma'(J_w)}).\]
In this paper, we are mainly interested in structures related to $\mathcal {Z}_w$. One can pass easily from the tuple $(\widetilde{\I}_x,\widetilde{\I}_{x,+}, \widetilde{\I}_{x,-},\widetilde{\iota})$ constructed in Proposition \ref{prop--lift G-zip at point} to a certain structure related to $\mathcal {Z}_w$. More precisely, we take \begin{subeqn}\label{G-zip at point}\I_x:=\widetilde{\I}_x\times^{\mathcal{G}_0}\mathcal{G}_0^{\mathrm{rdt}},\ \  \I_{x,+}:=\widetilde{\I}_{x,+}\times^{\mathcal{G}_{0,w}}\overline{P}_{J_w},\ \  \I_{x,-}:=\widetilde{\I}_{x,-}\times^{\mathcal{G}_{0,\sigma(w)^{-1}}}\overline{P}_{\sigma'(J_w)}\end{subeqn}
and \[\iota:\I_{x,+}^{(p)}/U_{J_w}^{(p)}\rightarrow \I_{x,-}/U_{\sigma'(J_w)}\] be the isomorphism induced by $\widetilde{\iota}$. The following statement is straightforward.
\begin{corollary}\label{corollary--G-zip at point}
	The tuple $(\I_x,\I_{x,+}, \I_{x,-},\iota)$ is a $\mathcal{G}_0^{\mathrm{rdt}}$-zip of type $J_w$ over $k$.
\end{corollary}
\begin{remark}\label{remark--point g}
	Notations as in the previous proposition, let $\overline{t}^{\mathrm{rdt}}$ (resp. $\overline{g}^{\mathrm{rdt}}$) be the image of $\overline{t}$ (resp. $\overline{g}$) in $\I_x$ (resp. $\G_0^{\mathrm{rdt}}$). Let $E_{\mathcal{Z}_w}$ be the zip group attached to $\mathcal{Z}_w$, and $\mathbb{E}_x$ be the $E_{\mathcal{Z}_w}$-torsor attached to $(\I_x,\I_{x,+}, \I_{x,-},\iota)$ (see e.g. our discussions after Definition \ref{def--G-zip}). Then \[(\overline{t}^{\mathrm{rdt}},\overline{\sigma(g)}^{\mathrm{rdt}} \overline{t}^{\mathrm{rdt}})\in \mathbb{E}_x(k),\] and its image in $\G_0^{\mathrm{rdt}}$ is $\overline{\sigma(g)}^{\mathrm{rdt}}$.
\end{remark}

\subsection[The EO stratification in a KR stratum]{The EO stratification in a KR stratum}\label{subsection EO in KR}

\subsubsection[The conjugate local model]{The conjugate local model}\label{subsection conjugate loc mod} Let $G'/k((t))$ and $\mathcal{G}'/k[[t]]$ be as in subsection \ref{subsection loc models} (after Remark \ref{remark--Goertz PEL}). In particular, $w$ (resp. $\Admu$) is identified with an element (resp. a subset) in the Iwahori Weyl group of $G'$, which will be denoted by the same symbol. Let $$\eM^\vee:=\bigcup_{w\in\Admu}L^+\G'\backslash L^+\G' w L^+\G',\text{\ \ \ and\ \ \ } \eMc:=\bigcup_{w\in\Admu}L^+\G' \sigma(w)^{-1} L^+\G'/L^+\G'.$$
As in the right quotient case (cf. Corollary \ref{local model-collect}), $\eM^\vee$ is a reduced closed subscheme of $L^+\G'\backslash LG'$ of dimension $\mathrm{dim}(\eM)$. Noting that $\eMc$, with the reduced scheme structure, is the image of $\eM^\vee$ under the homeomorphism $L^+\G'\backslash LG'\rightarrow LG'/L^+\G', x\mapsto \sigma(x)^{-1}$, so it is also of dimension $\mathrm{dim}(\eM)$. Moreover, as in the proof of Lemma \ref{lemma--zip datum attached to w}, $\eMc$ is a subscheme of $\Gr(V_{k})$. The scheme $\eMc$ will be called the \emph{conjugate local model}.


We come back to notations introduced in \ref{global F-zip}. We start with the local model diagram \[\xymatrix{
	&\wt{\ES}_0\ar[ld]_\pi\ar[rd]^q&\\
	\ES_0& & \eM,
}\] which is obtained from the local model in Theorem \ref{result P-Z and K-P} by taking the special fiber.
We will look at a different map \[q^c:\widetilde{\ES}_0\rightarrow \Gr(V_{k}),\quad f\mapsto f^{-1}(D_0\subset \V_0).\] One sees easily that $q^c$ factors through $\eMc$ at the level of $k$-points, and hence factors through $\eMc$, as $\widetilde{\ES}_0$ is reduced. The induced morphism $\widetilde{\ES}_0\rightarrow \eMc$ will still be denoted by $q^c$, and the diagram \[\xymatrix{
	&\wt{\ES}_0\ar[ld]_\pi\ar[rd]^{q^c}&\\
	\ES_0& & \eMc
}\] will be called the \emph{conjugate local model diagram}. One of our main tasks in this subsection is to study basic properties of the conjugate local model diagram.

\begin{theorem}\label{conj local model diag}
	The morphism $q^c: \widetilde{\ES}_{0}\rightarrow \eMc$ is smooth.
\end{theorem}
\begin{proof}
	We will show that, for each closed point $x\in \ES_0(k)$, denoting by $O_{\ES_0,x}^\wedge$ the attached complete local ring, there is a section $t:\Spec O_{\ES_0,x}^\wedge \rightarrow\widetilde{\ES}_{0}$ such that the composition $\Spec O_{\ES_0,x}^\wedge \stackrel{t}{\rightarrow}\widetilde{\ES}_{0}\stackrel{q^c}{\rightarrow} \eMc$ induces an isomorphism of complete local rings. 
	
	Let $R_G$ and $R$ be as in the discussions after Lemma \ref{lemma--torsor M-1}, and $R_{G,0}$ and $R_0$ be their reduction modulo $p$. Then there is an isomorphism $ O_{\ES_0,x}^\wedge\cong R_{G,0}$ such that $\A[p^\infty]|_{O_{\ES_0,x}^\wedge}\cong \mathcal{B}|_{R_{G,0}}$, where $\mathcal{B}$ is the versal $p$-divisible group over $R_E$ determined by the Dieudonn\'{e} display $(M_{R_E}, M_{R_E,1}, \Psi)$ as in loc. cit. Let $\mathrm{C}^1\subset M_{R_0}$ and $\mathrm{D}_0\subset M_{R_0}$ be the Hodge filtration and conjugate filtration respectively. We fix an isomorphism $t:V^\vee_k\rightarrow M_{R_0,x}$ respecting tensors, and view it as an isomorphism over $R_0$ by base-change. The filtration $t^{-1}(\mathrm{D}_0)\subset V^\vee_{R_0}$ induces a morphism $\Spec R_0\rightarrow \Gr(V_k)$. 
	
	Let $m_{R_0}$ be the maximal ideal of $R_0$, then by Remark \ref{key remark} (2), we have $$\big(t^{-1}(\mathrm{D}_0)\subset V^\vee_{R_0}\big)=\big((\mathfrak{u}\cdot t^{-1}(\mathrm{D}_{0,x}))\subset V^\vee_{R_0}\big)$$ over $R_0/m_{R_0}^2$. So $\Spec R_0\rightarrow \Gr(V_k)$  is an isomorphism at the level of complete local rings, and hence the induced morphism $\Spec O_{\ES_0,x}^\wedge \stackrel{t}{\rightarrow}\widetilde{\ES}_{0}\stackrel{q^c}{\rightarrow} \eMc$ is an isomorphism onto its image after taking completion. Noting that $\ES_0$ and $\eMc$ have the same dimension, the above morphism induces an isomorphism of complete local rings.
\end{proof}

Recall our notations in \ref{zip in EKOR}. Let $\mathcal{G}_0=\mathcal {G}\otimes k$ and $\mathcal{G}_0^{\rdt}$ be the maximal reductive quotient of $\mathcal{G}_0\otimes k$. Then it is a reductive group defined over $\mathbb{F}_p$. Let $\overline{B}$ be the image in $\mathcal{G}_0^{\rdt}$ of $\breve{I}$ and $\overline{T}$ be a maximal torus of $\overline{B}$. For $w\in {}^K\IW^K$ which is a minimal length representative of a member in $\Admu_K$, we set \[\sigma'=\sigma\circ \mathrm{Ad}(w),\quad \textrm{and}\quad J_w=J_K\cap \mathrm{Ad}(w^{-1})(J_K),\] where $J_K\subset \IW$ is the set of simple reflections in $W_K$. Let $\overline{L}_{J_w}\subset \mathcal{G}_0^{\rdt}$ (resp. $\overline{P}_{J_w}\subset \mathcal{G}_0^{\rdt}$) be the standard Levi subgroup (resp. parabolic subgroup) of type $J_w$, and $\overline{L}_{\sigma'(J_w)}\subset \mathcal{G}_0^{\rdt}$ (resp. $\overline{P}_{\sigma'(J_w)}\subset \mathcal{G}_0^{\rdt}$) be the standard Levi subgroup (resp. parabolic subgroup) of type $\sigma'(J_w)$. Then we have a natural isogeny $\overline{L}_{J_w}\rightarrow \overline{L}_{\sigma'(J_w)}$ which is again denoted by $\sigma'$. The tuple $$\mathcal {Z}_w:=(\mathcal{G}_0^{\rdt}, \overline{P}_{J_w}, \overline{P}_{\sigma'(J_w)}, \sigma':\overline{L}_{J_w}\rightarrow \overline{L}_{\sigma'(J_w)})$$ is an algebraic zip datum. Set \[E_{\mathcal {Z}_w}=(\overline{L}_{J_w})_\sigma(U_{J_w}\times U_{\sigma'(J_w)}),\] where $U_{J_w}$ (resp. $U_{\sigma'(J_w)}$) is the unipotent radical of $\ov{P}_{J_w}$ (resp. $\ov{P}_{\sigma'(J_w)}$). It has a left action on $\mathcal{G}_0^{\rdt}$, and hence induces a quotient stack $[E_{\mathcal {Z}_w}\backslash\mathcal{G}_0^{\mathrm{rdt}}]$.

Since we will work over $k$ and with a fixed KR stratum $\ES_0^w$, sometimes we will simply write $L_{J_w}$ (resp. $L_{\sigma'({J_w})}$, $P_{J_w}$, $P_{\sigma'(J_w)}$) for $\overline{L}_{J_w}$ (resp. $\overline{L}_{\sigma'(J_w)}$, $\overline{P}_{J_w}$, $\overline{P}_{\sigma'(J_w)}$). We will construct a morphism \[\zeta_w:\ES_0^w\rightarrow [E_{\mathcal {Z}_w}\backslash\mathcal{G}_0^{\mathrm{rdt}}].\] As explained in  \ref{subsection-zip data}, it is equivalent to construct a $\mathcal{G}_0^{\mathrm{rdt}}$-zip of type $J_w$ over $\ES_0^w$.
We will also assume, from now on, that $\ES_0^w\neq\emptyset$, as otherwise, $\zeta$ exists automatically. In subsection \ref{subsec--pointwise constr} Corollary \ref{corollary--G-zip at point}, for each $x\in \ES_0^w(k)$, we have constructed a $\mathcal{G}_0^{\mathrm{rdt}}$-zip of type $J_w$. Here we will need a family version.

To construct $\mathcal{G}_0^{\mathrm{rdt}}$-zips, we will slightly change the notations and write $\wt{\I}=\wt{\ES}_0$.
We start with the local model diagram \[\xymatrix{
	&\widetilde{\I}\ar[ld]_\pi \ar[rd]^q &\\ 
	\ES_0 & &\eM,
}\] 
where $\pi$ is a $\G_0$-torsor and $q$ is $\G_0$-equivariant.
The $\mathcal{G}_0$-orbit $\eMw\subset \eM$ induces a $\mathcal{G}_0$-torsor \[\widetilde{\I}^w:=q^{-1}(\eMw)=\widetilde{\I}|_{\ES_0^w}\] over $\ES_0^w$. 
We have the induced diagram
\[\xymatrix{
	&\widetilde{\I}^w\ar[ld]_\pi \ar[rd]^q &\\ 
	\ES_0^w & &\eM^w.
}\] 
We get a $\mathcal{G}_0^{\mathrm{rdt}}$-torsor over $\ES_0^w$ by simply taking
\begin{subeqn}\label{G0rdt-torsor}\I^w:=\widetilde{\I}^w\times^{\mathcal{G}_0}\mathcal{G}_0^{\mathrm{rdt}}.\end{subeqn}
We remark that $\I^w$ is the pull back to $\ES_0^w$ of the $\mathcal{G}_0^{\mathrm{rdt}}$-torsor $\I:=\widetilde{\I}\times^{\mathcal{G}_0}\mathcal{G}_0^{\mathrm{rdt}}\ra \ES_0$.

Noting that $w$ is an element in $\eMw(k)$, we consider $q^{-1}(w)$, the fiber of $w$ in $\widetilde{\I}^w$. We also consider the variation of the structure defined in Proposition \ref{prop--lift G-zip at point} (1), i.e. we set $$\widetilde{\I}^w_{+}:=\IIsom_{\ES_0^w}((C^\bullet_w,s), (C^\bullet,s_{\dr}))\subset \widetilde{\I}^w$$

\begin{lemma}\label{lemma--global lift zip +}
	We have $q^{-1}(w)=\widetilde{\I}^w_{+}$ which is a $\mathcal{G}_{0,w}$-torsor over $\ES_0^w$. Here $\mathcal{G}_{0,w}\subset \mathcal{G}_0$ is the stabilizer of $w$ as before.
\end{lemma}
\begin{proof}
	The natural morphism $\widetilde{\I}^w_{+}\rightarrow q^{-1}(w)$ is a closed immersion, as they are both closed subschemes of $\widetilde{\I}^w$. For $z\in \ES_0^w(k)$, we set $\widetilde{\I}^w_{+,z}$ and $q^{-1}(w)_z$ to be its fibers in $\widetilde{\I}^w_{+}$ and $q^{-1}(w)$ respectively. It is easy to see that $\widetilde{\I}^w_{+,z}(k)=q^{-1}(w)_z(k)$, and hence $\widetilde{\I}^w_{+}= q^{-1}(w)$ as $q^{-1}(w)$ is a smooth variety.

	The $\mathcal{G}_{0,w}$-action on $\widetilde{\I}^w_{+}$ is by definition faithful and free, so to see that it is a $\mathcal{G}_{0,w}$-torsor over $\ES_0^w$, we only need to check the flatness of $\pi:\widetilde{\I}^w_{+}\rightarrow \ES_0^w$. Noting that they are both smooth varieties, we only need to check that $$\mathrm{dim}(\widetilde{\I}^w_{+})=\mathrm{dim}(\ES_0^w)+\mathrm{dim}(\widetilde{\I}^w_{+,z}),\ \ \ \ \text{ for all }z\in \ES_0^w(k).$$
	We have $$\mathrm{dim}(\ES_0^w)=\mathrm{dim}(\eMw),\  \mathrm{dim}(\widetilde{\I}^w_{+,z})=\mathrm{dim}(\mathcal{G}_{0,w})\ \text{ and }\ \mathrm{dim}(\widetilde{\I}^w_{+})=\mathrm{dim}(\eMw)+\mathrm{dim}(\mathcal{G}_{0,w}),$$ where the second equality follows from Proposition \ref{prop--lift G-zip at point} (1),
	so the above equality is clear.
\end{proof}

The group $\mathcal{G}_{0,w}$ is smooth, and hence its image in $\mathcal{G}_0^{\mathrm{rdt}}$ lies in $P_{J_w}$, as it is so for $k$-points. We get as in (\ref{G-zip at point}) a $P_{J_w}$-torsor by taking
\begin{subeqn}\label{PJ-torsor}\I^w_+:=q^{-1}(w)\times^{\mathcal{G}_{0,w}}P_{J_w}=\widetilde{\I}^w_+\times^{\mathcal{G}_{0,w}}P_{J_w}.\end{subeqn}

\begin{remark}\label{remark--Grdt zip 1}Let us consider the composition \[\widetilde{\I}^w\rightarrow \eMw=\mathcal{G}_0/\mathcal{G}_{0,w}\twoheadrightarrow \mathcal{G}_0^{\mathrm{rdt}}/P_{J_w}.\] It is equivariant with respect to the action of $\mathcal{G}_0$ on $\widetilde{\I}^w$ and that of $\mathcal{G}_0^{\mathrm{rdt}}$ on $\mathcal{G}_0^{\mathrm{rdt}}/P_{J_w}$, and hence induces a $\mathcal{G}_0^{\mathrm{rdt}}$-equivariant morphism \[q_{\#}:\I^w=\widetilde{\I}^w\times^{\mathcal{G}_0}\mathcal{G}_0^{\mathrm{rdt}}\rightarrow\mathcal{G}_0^{\mathrm{rdt}}/P_{J_w}.\] One checks immediately that $\I^w_+=q_{\#}^{-1}(\overline{w})$, where $\overline{w}$ is the image of $w$ in $\mathcal{G}_0^{\mathrm{rdt}}/P_{J_w}$.
\end{remark}
Similarly, we can consider the conjugate local model diagram \[\xymatrix{
	&\widetilde{\I}\ar[ld]_\pi \ar[rd]^{q^c} &\\ 
	\ES_0 & &\eMc.
}\]   Let $\eM^{c,w}$ be the $\mathcal{G}_0$-orbit of the point $\varpi$ corresponding to the filtration $D_\bullet^w$. Then we have $\widetilde{\I}^w=q^{c,-1}(\eM^{c,w})$. We also have $q^{c,-1}(\varpi)$, the fiber of $\varpi$ in $\widetilde{\I}^w$ and \[\widetilde{\I}^w_{-}:=\IIsom_{\ES_0^w}((D_\bullet^w,s), (D_\bullet,s_{\dr}))\subset \widetilde{\I}^w.\] Based on Theorem \ref{conj local model diag} and Proposition \ref{prop--lift G-zip at point} (2), the same proof of Lemma \ref{lemma--global lift zip +} implies the following.
\begin{lemma}\label{lemma--global lift zip -}
	We have $q^{c,-1}(\varpi)=\widetilde{\I}^w_{-}$ which is a $\mathcal{G}_{0,\sigma(w)^{-1}}$-torsor over $\ES_0^w$. Here $\mathcal{G}_{0,\sigma(w)^{-1}}\subset \mathcal{G}_0$ is as in the second paragraph of \ref{subsec--pointwise constr}.
\end{lemma}
We get, as above, a $P_{\sigma'(J_w)}$-torsor by taking
\begin{subeqn}\label{PsigJ-torsor}\I^w_-:=q^{c,-1}(\varpi)\times^{\mathcal{G}_{0,\sigma(w)^{-1}}}P_{\sigma'(J_w)}=\widetilde{\I}^w_-\times^{\mathcal{G}_{0,\sigma(w)^{-1}}}P_{\sigma'(J_w)}.\end{subeqn}
Moreover, the composition \[\widetilde{\I}^w\rightarrow \eM^{c,w}=\mathcal{G}_0/\mathcal{G}_{0,\sigma(w)^{-1}}\twoheadrightarrow \mathcal{G}_0^{\mathrm{rdt}}/P_{\sigma'(J_w)}\] induces a $\mathcal{G}_0^{\mathrm{rdt}}$-equivariant morphism \[q^c_{\#}:\I^w=\widetilde{\I}^w\times^{\mathcal{G}_0}\mathcal{G}_0^{\mathrm{rdt}}\rightarrow\mathcal{G}_0^{\mathrm{rdt}}/P_{\sigma'(J_w)},\] and we have $\I^w_-=q_{\#}^{c,-1}(\overline{\varpi})$, where $\overline{\varpi}$ is the image of $\varpi$ in $\mathcal{G}_0^{\mathrm{rdt}}/P_{\sigma'(J_w)}$.

Notations as in Proposition \ref{prop--lift G-zip at point}, we consider the map $$\widetilde{\iota}:\widetilde{\I}_{+}^{w,(p)}/\mathcal{G}_{0,w}^{U,(p)}\rightarrow \IIsom_{\ES_0^w}(\oplus_i\gr_i (D_\bullet^w), \oplus_i\gr_i (D_{\bullet}))$$ defined as in (\ref{def of lift iota}).
\begin{lemma}\label{lemma--global lift zip}
	$\widetilde{\iota}$ induces an isomorphism $\widetilde{\I}_{+}^{w,(p)}/\mathcal{G}_{0,w}^{U,(p)}\rightarrow \widetilde{\I}_{-}^w/\mathcal{G}_{0,\sigma(w)^{-1}}^U$ which is equivariant with respect to the isomorphism $\sigma':\mathcal{G}_{0,w}^{L,(p)}\rightarrow \mathcal{G}_{0,\sigma(w)^{-1}}^L$.
\end{lemma}
\begin{proof}
	By Proposition \ref{prop--lift G-zip at point} (3), the induced map (by $\widetilde{\iota}$) on the sets of $k$-points factors through $\widetilde{\I}_{-}^w/\mathcal{G}_{0,\sigma(w)^{-1}}^U\subset \IIsom_{\ES_0^w}(\oplus_i\gr_i (D_\bullet^w), \oplus_i\gr_i (D_{\bullet}))$, so does $\widetilde{\iota}$. By the same reason as in the proof of [loc. cit], $\widetilde{\iota}$ is equivariant with respect to the isomorphism $\sigma'$, and hence is an isomorphism.
\end{proof}
As in the pointwise case, the following statement is straightforward.
\begin{corollary}\label{corollary--G-zip on KR}
	Let $\iota:\I_{+}^{w,(p)}/U_{J_w}^{(p)}\rightarrow \I_{-}^w/U_{\sigma'(J_w)}$ be the isomorphism induced by $\widetilde{\iota}$. Then the tuple $(\I^w,\I_{+}^{w}, \I_{-}^w,\iota)$ is a $\mathcal{G}_0^{\mathrm{rdt}}$-zip of type $J_w$.
\end{corollary}

The tuple $(\I^w,\I_{+}^{w}, \I_{-}^w,\iota)$ then induces a morphism of stacks \[\zeta_w:\ES_0^w\rightarrow [E_{\mathcal {Z}_w}\backslash\mathcal{G}_0^{\mathrm{rdt}}],\] whose fibers are precisely the EKOR strata in $\ES_0^w$ by our previous discussions in \ref{subsection stratifications hodge} and \ref{zip in EKOR}.
\begin{theorem}\label{sm of zeta}
	The morphism $\zeta_w$ is smooth.
\end{theorem}
\begin{proof}
Let $\mathbb{E}^w$ be formed by the following cartesian diagram:
\[\xymatrix{\mathbb{E}^w\ar[d]\ar[rrr] & & & \I_{-}^w\ar[d]\\
	\I_{+}^{w}\ar[r]&\I_{+}^{w,(p)}\ar[r]&\I_{+}^{w,(p)}/U_{J_w}^{(p)}\ar[r]^{\iota}&\I_{-}^w/U_{\sigma'(J_w)}},\]
where the first map of the bottom line is the relative Frobenius map. Then $\mathbb{E}^w$
 is an $E_{\mathcal{Z}_w}$-torsor over $\ES_0^w$. In particular, $\mathbb{E}^w$ is smooth over $k$. The morphism $\zeta_w$ is equivalent to an $E_{\mathcal{Z}_w}$-equivariant morphism \[\zeta_w^{\#}:\mathbb{E}^w\rightarrow \mathcal{G}_0^{\mathrm{rdt}},\] and the smoothness of $\zeta_w$ is equivalent to that of $\zeta_w^{\#}$, which is, furthermore, equivalent to the surjectivity of the induced map on tangent spaces  for all points of $\mathbb{E}^w(k)$. Noting that $\zeta_w^{\#}$ is $E_{\mathcal{Z}_w}$-equivariant, we only need to show that for any $y\in \ES_0^w(k)$ with some lifting $z\in \mathbb{E}^w(k)$, the induced map of tangent spaces \[T_z\mathbb{E}^w\rightarrow T_{\zeta_w^{\#}(z)}\mathcal{G}_0^{\mathrm{rdt}}\] is surjective.

Let $(A,\mathfrak{m})$ be the complete local ring of $\ES_0^w$ at $y$. The $E_{\mathcal{Z}_w}$-torsor $\mathbb{E}^w$ becomes trivial on $\Spec A$, and hence $\mathbb{E}^w_A\cong E_{\mathcal{Z}_w}\times_k \Spec A$ once we fix a section of it. Let $D_x$ be the Dieudonn\'{e} module of the $p$-divisible group at $x$, and $t: V^\vee_{\breve{\INT}_p}\rightarrow D_x$ be an isomorphism respecting tensors such that the linearization of Frobenius is of the form $g\circ \sigma(w)$, where $g\in \mathcal{G}(\breve{\Z}_p)$ is certain element. We denote by $\overline{t}^{\rdt}$ (resp. $\overline{g}^\rdt$) the image of $t$ (resp. $g$) in $\I$ (resp. $\mathcal{G}_0^\rdt$). By our choice, $z:=(\overline{t}^{\rdt},\overline{g}^\rdt\cdot\overline{t}^\rdt)$ is in $\mathbb{E}^w(k)$, and its image in $\mathcal{G}^{\mathrm{rdt}}_0(k)$ is $\overline{g}^\rdt$.


By Remark \ref{key remark} (2), viewing $z$ as a section over $A$, the morphism \[\zeta_w^{\#}:\mathbb{E}^w_{A/\mathfrak{m}^2}\cong E_{\mathcal{Z}_w}\times_k \Spec A/\mathfrak{m}^2\rightarrow \mathcal{G}^{\mathrm{rdt}}_0\] is given by $(h,\overline{\mathfrak{g}})\mapsto h\cdot (\overline{\mathfrak{g}}^\rdt \overline{g}^\rdt),$
here $\overline{\mathfrak{g}}$ is the image in $\mathcal{G}(A/\mathfrak{m}^2)$ of the element $\widetilde{\mathfrak{g}}^{-1}$ in [loc. cit], and $\overline{\mathfrak{g}}^\rdt$ is the image of $\overline{\mathfrak{g}}$ in $\mathcal{G}^\rdt(A/\mathfrak{m}^2)$. Let $U_{J_w}^-\subset \mathcal{G}_0^\rdt$ be the unipotent radical of the opposite parabolic subgroup of $P_{J_w}$. The natural morphism $\eMw\rightarrow \mathcal{G}_0^\rdt/P_{J_w}$ is a smooth covering, and hence the induced map on tangent spaces at identities is surjective. In particular, identifying $U_{J_w}^-$ as an open subscheme of $\mathcal{G}_0^\rdt/P_{J_w}$ containing the (image of) identity, $\overline{\mathfrak{g}}^\rdt$ induces a surjection $\mathfrak{m}/\mathfrak{m}^2\rightarrow \mathrm{Lie}(U_{J_w}^-)$. Noting that $\sigma':L_{J_w}\rightarrow L_{\sigma'(J_w)}$ is trivial on elements of the form $\mathrm{id}+\mathrm{Lie}(L_{J_w})$, by the last paragraph of the proof of \cite[Theorem 4.1.2]{EOZ}, the induced map on tangent spaces $T_z\mathbb{E}^w\rightarrow T_{\overline{g}^\rdt}\mathcal{G}_0^{\mathrm{rdt}}$ is surjective.
\end{proof}
Now we can state many properties of EKOR strata. Recall that we have the surjection ${}^K\Admu\twoheadrightarrow \Admu_K$.
\begin{theorem}\label{thm--first properties}
	For $x\in {}^K\Admu$, the corresponding EKOR stratum $\ES_0^x$ is a locally closed subscheme of $\ES_0$, which is smooth of dimension $\ell(x)$. Moreover, for $w\in \Admu_K$ (cf. \ref{subsubsection two sections})
	\begin{enumerate}
		\item there is a unique EKOR stratum, namely $\ES_0^{{}^Kw_K}$ with ${}^Kw_K$ as in \ref{ordi rep of KR}, in $\ES_0^w$ which is open dense. This is called the $w$-ordinary locus,
		\item there is a unique EKOR stratum, namely $\ES_0^{x_w}$, in $\ES_0^w$ which is of dimension $\ell(w)$. Here $x_w$ is just $w$ but viewed as an element of ${}^K\Admu$. This stratum is closed in $\ES_0^w$, and is called the $w$-superspecial locus.
	\end{enumerate}
\end{theorem}
\begin{proof}
	For $x\in {}^K\Admu$, the non-emptiness of the EKOR stratum $\ES_0^x$ follows from Corollary \ref{coro--non-emp and closure} (1), and the smoothness follows from Theorem \ref{sm of zeta}. To compute the dimension, we have, by Theorem \ref{sm of zeta} again, $\mathrm{dim}(\ES_0^x)=\mathrm{dim}(\ES_0^{w})-\mathrm{codim}(\mathcal{G}_0^{\mathrm{rdt},\overline{x}})$, where $\overline{x}\in {}^{J_w}W_K$ is the element corresponding to $x$. We also have $\mathrm{dim}(\ES_0^{w})=\ell({}^Kw_K)$ by Corollary \ref{local model-collect} (3), and $\mathrm{codim}(\mathcal{G}_0^{\mathrm{rdt},x})=\mathrm{dim}(U_w)-\ell(\overline{x})$ by Theorem \ref{thm on zip orbits}. Here $U_w$ is the unipotent radical of the standard parabolic $P_w\subset\mathcal{G}_0^{\mathrm{rdt}}$ of type $J_w$. Noting that $\mathrm{dim}(U_w)=\ell(y_0)$ for $y_0$ the longest element in ${}^JW_K$, we have $$\mathrm{dim}(\ES_0^x)=\ell({}^Kw_K)-\ell(y_0)+\ell(\overline{x})=\ell(w)+\ell(\overline{x})=\ell(x).$$
	Statements (1) and (2) follow from the above results.
\end{proof}
\begin{remark}
	\begin{enumerate}
		\item 
		We can also define the \emph{ordinary locus} $\ES_0^{\mathrm{ord}}$ to be the union of EKOR strata $\ES_0^x$s with $\ell(x)=\mathrm{dim}(\ES_0)$. It is clear by definition that $$\ES_0^{\mathrm{ord}}=\coprod_{w\in \Admu_K,\atop\ell({}^Kw_K)=\mathrm{dim}(\ES_0)}\ES_0^{{}^Kw_K},$$ and it is open in $\ES_0$. Moreover it is dense in $\ES_0$ by the smoothness of the morphisms $\zeta_w, w\in\Admu_K$ (Theorem \ref{sm of zeta}). By Proposition \ref{P: max EKOR}, the elements in the index set of the above disjoint union are of the form $t^{\mu'}$, where $\mu'$ runs over elements in the $W_0$-orbit of $\ul{\mu}$ with $t^{\mu'}\in {}^K\IW$.
		\item We can define the \emph{superspecial locus} to be the EKOR strata $\ES_0^\tau$, where $\tau$ is as in \ref{tau}. Then
		$\ES_0^\tau$ is the unique closed EKOR stratum in $\ES_0$, and it is of dimension 0.
	\end{enumerate}
\end{remark}

\begin{corollary}\label{coro--first properties}
	Let $\nu$ be the Newton map as in \ref{nu on aff Weyl}, and $x\in {}^K\Admu$ be $\sigma$-straight, then
	\begin{enumerate}
		\item the EKOR stratum $\ES_0^x$ is a central leaf. In particular, it is contained in the Newton stratum $\ES_0^{\nu(x)}$, and hence each Newton stratum contains an EKOR stratum;
		\item the dimension of $\ES_0^x$ is $\langle \nu(x),\rho\rangle$, where $\rho$ is the half sum of positive roots. In particular, central leaves given by $\sigma$-straight elements in a fixed Newton stratum are of the same dimension.
	\end{enumerate}
\end{corollary}
\begin{proof}
	Statement (1) follows directly from the commutative diagram before \ref{He-Rap axioms}. 
	
	We have $\mathrm{dim}(\ES_0^x)=\ell(x)$ by Theorem \ref{thm--first properties}, and $\ell(x)=\langle \nu(x),\rho\rangle$ by \ref{recall sigm straight}. In particular, (2) follows.
\end{proof}
\begin{remark}
	Statement (2) recovers W. Kim's formula for central leaves attached to $\sigma$-straight elements, cf. \cite[Corollary 5.3.1]{Kim}.
\end{remark}
\subsection[Change of parahoric and further properties]{Change of parahorics and further properties} It is sometimes helpful to look at the relations between EKOR strata for different parahoric subgroups. We will need sometimes the axioms (especially axiom 4 (c)) introduced by He and Rapoport (see subsection \ref{He-Rap axioms}). We remind the readers that, as we have remarked in \ref{ax checked}, in the current setting, by the work \cite{zhou isog parahoric} of Zhou, all the axioms except for surjectivity in axiom 4 (c) hold under our assumption (\ref{general hypo--refined}), and if in addition $G_{\RAT_p}$ is residually split, axiom 4 (c) also holds.

We start with the following result of He and Rapoport.
\begin{proposition}{\rm(\cite[Proposition 6.6]{He-Rap})}\label{change para--group}
	For $w\in \IW$, there is a subset \[\Sigma_K(w)\subset W_KwW_K\cap {}^K\IW\] such that \[\breve{K}_\sigma(\breve{I}w\breve{I})=\coprod_{x\in \Sigma_K(w)}\breve{K}_\sigma(\breve{I}x\breve{I}).\] Moreover, if $w\in {}^K\IW$, then $\Sigma_K(w)=\{w\}$.
\end{proposition}
As a consequence, one has the following. Fix a sufficiently small $K^p$ and let $\K=KK^p$ $\K'=K'K^p$ for parahoric subgroups $K'\subset K$. Assume that $\ES_{\K}$ and $\ES_{\K'}$ are constructed by the same Siegel embedding, then by \cite[Theorem 7.1]{zhou isog parahoric}, there exists a morphism $\pi_{\K',\K}: \ES_{\K'}\ra \ES_{\K}$ satisfying He-Rapoport's axiom 1.
\begin{proposition}{\rm(\cite[Proposition 6.11]{He-Rap}, \cite{zhou isog parahoric})}\label{change para--EKOR}
	Let $K'\subset K$ be standard parahoric subgroups with induced morphism $\pi_{\K',\K}:\ES_{\K',0}\rightarrow \ES_{\K,0}$. Then
	\begin{enumerate}
		\item for $x\in  {}^{K'}\Admu$, we have
		\[\pi_{\K',\K}(\ES_{\K',0}^x)\subset \coprod_{y\in \Sigma_K(x)}\ES_{\K,0}^y.\] In particular, for $x\in {}^K\Admu$ viewed as an element in ${}^{K'}\Admu$, we have $\pi_{\K',\K}(\ES_{\K',0}^x)\subset \ES_{\K,0}^x$.
		\item if He-Rapoport's axiom 4 (c) holds, we have equalities in the above statement.
	\end{enumerate}
\end{proposition}
As applications, one deduces the following.
\begin{corollary}\label{coro--non-emp and closure}
	For $\ES_{\K,0}$ and $x\in {}^K\Admu$, we have
	\begin{enumerate}
		\item the EKOR stratum $\ES_{\K,0}^x$ is non-empty;
		\item {\rm(\cite[Theorem 6.15]{He-Rap})} if He-Rapoport's axiom 4 (c) holds, the closure of $\ES_{\K,0}^x$ in $\ES_{\K,0}$ is $\coprod_{y\leq_{K,\sigma} x}\ES_{\K,0}^y$. Here $\leq_{K,\sigma}$ is as in \ref{partial order}.
	\end{enumerate}
\end{corollary}
\begin{proof}
	We only need to check (1). Since the prime to $p$ level $K^p$ is fixed, here and in the following, to simplify the notations we write $\ES_{K,0}$ for $\ES_{\K,0}$, and similarly for the Iwahori level.
	Viewing $x$ as an element of ${}^I\Admu=\Admu$, we have, by Proposition \ref{change para--EKOR} (1), that $\pi_{I,K}(\ES_{I,0}^x)\subset \ES_{K,0}^x$. As we have remarked in \ref{ax checked}, non-emptiness of $\ES_{I,0}^x$ has been proved by Zhou in \cite{zhou isog parahoric}, and hence $\ES_{K,0}^x$ is non-empty.
\end{proof}

To go further, we need compatibility of local model diagrams. More precisely, we will use the following commutative diagram. We use the above conventions. Let $\mathcal{I}$ be a parahoric model of $G$ corresponding to $I$.

\[\xymatrix@C=1.8cm@R=0.6cm{
	&\widetilde{\ES}_{I,0}\ar[dl]_{\pi_I}\ar[dr]^{q_I}\ar[d]   \\
	\ES_{I,0}\ar@{}[ddr] |{\Box} \ar[d]_{\pi_{I,K}}  & \widetilde{\ES}_{I,0}\times^{\mathcal{I}_0}\mathcal{G}_0\ar[l]\ar[d]^{\cong}& \mathrm{M}_{I,0}\ar[d]^{q_{I,K}}\\
	\ES_{K,0} & \pi_{I,K}^*\widetilde{\ES}_{K,0}\ar[ul]\ar[d]\ar[r] & \mathrm{M}_{K,0}\\
	&    \widetilde{\ES}_{K,0} \ar[ur]_{q_K}\ar[ul]^{\pi_K}       }\]
where the map $\widetilde{\ES}_{I,0}\ra \widetilde{\ES}_{I,0}\times^{\mathcal{I}_0}\mathcal{G}_0$ is induced by the natural map $\mathcal{I}_0\ra \mathcal{G}_0$.
\begin{lemma}\label{lemma--etale forg level}
	For $x\in {}^K\Admu$ viewed as an element of $\Admu$, $q_{I,K}$ induces a morphism $\mathrm{M}_{I,0}^x\rightarrow \mathrm{M}_{K,0}$ which is a finite \'{e}tale covering onto its image.
\end{lemma}
\begin{proof}
	Let $\mathcal{G}'$ (resp. $\mathcal{I}'$) be the parahoric (resp. Iwahori) group scheme over $k[[t]]$ attached to $\breve{K}$ (resp. $\breve{I}$). We have
	\begin{equation*}
	\begin{split}
	\mathrm{M}_{I,0}^x&=L^+\mathcal{I}'xL^+\mathcal{I}'/L^+\mathcal{I}'\cong L^+\mathcal{I}'/(L^+\mathcal{I}'\cap xL^+\mathcal{I}'x^{-1}), \ \ \ \text{ and }\\
	\mathrm{M}_{K,0}^x&:=L^+\mathcal{I}'xL^+\mathcal{G}'/L^+\mathcal{G}'\cong L^+\mathcal{I}'/(L^+\mathcal{I}'\cap xL^+\mathcal{G}'x^{-1})\subset L^+\mathcal{G}'xL^+\mathcal{G}'/L^+\mathcal{G}'=\mathrm{M}_{K,0}^{w_x}.
	\end{split}
	\end{equation*}
	By the same proof as in Lemma \ref{smooth stab}, the quotient maps $\mathcal{I}'_0\rightarrow \mathrm{M}_{I,0}^x$ and $\mathcal{I}'_0\rightarrow \mathrm{M}_{K,0}^x$ are both smooth coverings. In particular, $q_{I,K}^x:\mathrm{M}_{I,0}^x\rightarrow \mathrm{M}_{K,0}^x$ is a smooth covering. We also have, by \cite[Proposition 2.8 (ii)]{Richarz 1}, that $\mathrm{dim}(\mathrm{M}_{I,0}^x)=\ell(x)=\mathrm{dim}(\mathrm{M}_{K,0}^x)$, and hence $q_{I,K}^x$ is an \'{e}tale covering.

	To see that $q_{I,K}^x$ is finite, we only need to check that it is proper. Let $\mathrm{M}_{I,0}^{x,-}$ and  $\mathrm{M}_{K,0}^{x,-}$ be the closures of $\mathrm{M}_{I,0}^x$ and $\mathrm{M}_{K,0}^x$ respectively, endowed with the reduced subscheme structures. Then $q_{I,K}$ induces a proper morphism $$q_{I,K}^{x,-}:\mathrm{M}_{I,0}^{x,-}\longrightarrow q_{I,K}(\mathrm{M}_{I,0}^{x,-})=\mathrm{M}_{K,0}^{x,-}.$$ Noting that $\mathrm{M}_{I,0}^{x,-}=\coprod_{x'\leq x}\mathrm{M}_{I,0}^{x'}$ and $q_{I,K}(\mathrm{M}_{I,0}^{x'})$ is either $\mathrm{M}_{K,0}^x$ or disjoint with it, we have
	$$q_{I,K}(\mathrm{M}_{I,0}^{x'})\cap \mathrm{M}_{K,0}^x=\emptyset, \ \ \ \ \text{ for any } x'\leq x \text{ with }x'\neq x$$
	since $\dim\,q_{I,K}(\mathrm{M}_{I,0}^{x'})=\ell(x')<\dim\, \mathrm{M}_{K,0}^x=\ell(x)$. In particular, $q_{I,K}^x$ is the base-change to $\mathrm{M}_{K,0}^x$ of $q_{I,K}^{x,-}$, and hence is proper.
\end{proof}

\begin{theorem}\label{thm--etale forg level}For $x\in {}^K\Admu$ viewed as an element of $\Admu$, the morphism $\pi_{I,K}^x:\ES_{I,0}^x\rightarrow \ES_{K,0}^x$ induced by $\pi_{I,K}$ is finite \'{e}tale. If in addition He-Rapoport's axiom 4 (c) is satisfied, $\pi_{I,K}^x$ is a finite \'{e}tale covering.
\end{theorem}
\begin{proof}
For a point $z\in \ES_{I,0}^x(k)$, we fix an identification $V^\vee_{\breve{\INT}_p}\cong \mathbb{D}(\A_z[p^\infty])$ respecting tensors and consider its image $t\in \widetilde{\ES}_{I,0}(k)$. Let $O_{I,z}$ (resp. $O^x_{I,z}$) be the complete local ring at $z$ of $\ES_{I,0}$ (resp. $\ES_{I,0}^x$), and $R_0$ (resp. $R_{I,0}$) be the special fiber of the ring $R_E$ (resp. $R_G$) in the discussions after Lemma \ref{lemma--torsor M-1}, then $O_{I,z}\cong R_{I,0}$, and it induces an isomorphism \[\A[p^\infty]|_{O_{I,z}}\cong \mathcal{B}_I|_{R_{I,0}},\] where $\mathcal{B}_I$ is the (special fiber of the) versal $p$-divisible group as in loc. cit. Let $\V^1\subset \V$ be the Hodge filtration of $\A/\ES_{I,0}$, and $\mathfrak{u}$ be restriction to $\mathrm{M}_{I,0}^x$ of the universal element of $\Gr(V_k)(R_0)$ at the point corresponding to $t^{-1}(\V^1_x)\subset V^\vee_k$, then $\mathfrak{u}\cdot t\in \widetilde{\ES}_{I,0}(O_{I,z}^x)$, and it induces, by pulling back the Hodge filtration, a morphism $\Spec O_{I,z}^x\rightarrow \mathrm{M}_{I,0}^x$. By our choices as in loc. cit, $\Spec O_{I,z}^x\rightarrow \mathrm{M}_{I,0}^x$ induces a bijection of tangent spaces, and hence is an isomorphism at the level of complete local rings.

For $z':=\pi_{I,K}(z)\in \ES_{K,0}^x$, let $w_x$ the KR type of $x$ (equivalently, $z'$) and $t'\in \widetilde{\ES}_{K,0}(k)$ be the image of $t$. With $\mathfrak{u}'$ defined similarly, we have $\mathfrak{u}'\cdot t'\in \widetilde{\ES}_{K,0}(O_{K,z'}^{w_x})$, and it induces a morphism $\Spec O_{K,z'}^{w_x}\rightarrow \mathrm{M}_{K,0}^{w_x}$ which is an isomorphism at the level of complete local rings.

Let $O_{K,z'}^x$ be the complete local ring at $z'$ of the EKOR stratum, then we have a commutative diagram
\[\xymatrix{& \Spec O_{I,z}^x\ar[r]^{\mathfrak{u}\cdot t}\ar[d]^{\pi_{I,K}}& \mathrm{M}_{I,0}^{x}\ar[d]^{q_{I,K}}\\
	\Spec O_{K,z'}^x\ar@{^(->}[r]& \Spec O_{K,z'}^{w_x}\ar[r]^{\mathfrak{u}'\cdot t'}& \mathrm{M}_{K,0}^{w_x}.}\]

The morphism $\Spec O_{I,z}^x\rightarrow \Spec O_{K,z'}^{w_x}$ is a closed immersion by Lemma \ref{lemma--etale forg level}, and it factors through $\Spec O_{K,z'}^x$. Noting that $\Spec O_{I,z}^x$ and $\Spec O_{K,z'}^x$ are of the same dimension, the morphism $\pi_{I,K}:\Spec O_{I,z}^x\rightarrow \Spec O_{K,z'}^{x}$ is an isomorphism.
\end{proof}
\begin{remark}
	Notations as above, and let $\mathrm{M}_{K,0}^{x}\subset \mathrm{M}_{K,0}^{w_x}$ be as in Lemma \ref{lemma--etale forg level}. One can prove, without using smoothness or dimension formula, that $(\mathfrak{u}'\cdot t')^{-1}(\mathrm{M}_{K,0}^{x})=\Spec O_{K,z'}^x$.
\end{remark}


\subsubsection[Quasi-affineness of Iwahori KR]{Quasi-affineness of EKOR strata}
Let us recall some facts about quasi-affineness and ampleness. The following lemma can be found easily in text books and hence and we omit the proofs and references.
\begin{lemma}\label{prepare--quasiaffine}
	\begin{enumerate}
		\item Let $X$ be a scheme, then $X$ is quasi-affine if and only if $O_X$ is ample.
		\item Let $X$ and $Y$ be schemes that are separated and quasi-compact, $\mathcal{L}$ be a line bundle on $Y$.
		\begin{enumerate}
			\item Let $f:X\rightarrow Y$ be a composition of immersions and finite morphisms. If $\mathcal{L}$ is ample, then $f^*\mathcal{L}$ is ample.
			\item Let $g:X\rightarrow Y$ be faithfully flat,  then $\mathcal{L}$ is ample if $g^*\mathcal{L}$ is ample.
		\end{enumerate}
	\end{enumerate}
\end{lemma}
\begin{theorem}\label{thm--quasi affine}
	Every KR stratum in $\ES_{I,0}$ is quasi-affine. If in addition axiom 4 (c) is satisfied, every EKOR stratum in $\ES_{K,0}$ is quasi-affine.
\end{theorem}
\begin{proof}
	Notations as before, for $x\in {}^K\Admu$ viewed as an element in $\Admu$, the morphism $\ES_{I,0}^x\rightarrow \ES_{K,0}^x$ is a finite \'{e}tale covering if axiom 4 (c) is satisfied, and hence by Lemma \ref{prepare--quasiaffine}, it suffices to show that $\ES_{I,0}^w$ is quasi-affine for all $w\in \Admu$.

	We work with a symplectic embedding $\mathcal{I}\rightarrow \GSp(V_{\INT_p},\psi)$ as in Remark \ref{remark--good siegel ebd}, the induced morphism \[f:\ES_{I,0}\rightarrow \ES_{H,0}(\GSp,S^\pm)\] is finite. Let $\V^1\subset \V$ be the Hodge filtration attached to the universal abelian scheme on $\ES_{H,0}(\GSp,S^\pm)$, then $\omega:=\mathrm{det}(\V^1)$ is ample, and hence $f^*\omega\cong \mathrm{det}(\V^1_{\ES_{I,0}})$ is ample. Let $P^0$ be the stabilizer in $\GL(V_k)$ attached to the filtration $C^\bullet_w$, it is well known that $\omega$ comes from a character of $P^0$, denoted by $\eta$. The induced character of $\mathcal{I}_{0,w}$ factors through the quotient $\mathcal{I}_{0,w}^\rdt=\mathcal{I}_{0}^\rdt$ which is a torus (over $k$ and defined over $\mathbb{F}_p$) denoted by $T$.

	The $\mathcal{I}_0$-zip on $\ES_{I,0}^w$ induces a morphism \[\zeta_I:\ES_{I,0}^w\rightarrow [T_{\sigma'}\backslash T],\] and $f^*\omega|_{\ES_{I,0}^w}$ is the pullback via $\zeta_I$ of the line bundle $\mathcal{L}_\eta$ on $[T_{\sigma'}\backslash T]$ induced by the character $\eta:T\rightarrow \mathbb{G}_m$. Noting that $f^*\omega|_{\ES_{I,0}^w}$ is ample, to prove that $\ES_{I,0}^w$ is quasi-affine, it suffices to show that $\mathcal{L}_\eta\in \mathrm{Pic}([T_{\sigma'}\backslash T])$ is of finite order.

	The action of $T$ on itself is induced by the homomorphism $h:T\rightarrow T$, $x\mapsto \sigma'(x)/x$. Let $\mathrm{H}:=\mathrm{Ker}(h)$, we have an isomorphism of stacks $[T_{\sigma'}\backslash T]\cong [\mathrm{H} \backslash \mathrm{Spec}\ k]$. So it suffices to show that $\mathrm{Pic}([\mathrm{H} \backslash \mathrm{Spec}\ k])$ is finite.

	
	Recall that $\sigma'=\sigma\circ \mathrm{Ad}(w)$, under the invertible substitution $x=w^{-1}yw$, $h:T_{\sigma'}\rightarrow T$ becomes $y\mapsto \sigma(y)/(w^{-1}yw)$. We compute its tangent map at the identity. Let $m:T\times T\rightarrow T$ be the multiplication map, and $i_w:T\rightarrow T$ be the isomorphism induced by $y\mapsto (w^{-1}yw)^{-1}$. Then $$dh|_e=d\big(m\circ (\sigma,i_w)\big)=dm|_{(e,e)}\circ d\big(\sigma,i_w\big)|_e=d\sigma|_e+di_w|_e.$$
	Noting that $i_w$ is an isomorphism, so $di_w|_e$ is invertible, while $d\sigma|_e=0$, so $dh|_e$ is invertible. In particular, $\mathrm{H}=\mathrm{Ker}(h)$ is finite \'{e}tale.
	
	The (trivial) action of $\mathrm{H}$ on $\mathrm{Spec}\ k$ induces the trivial action of $\mathrm{H}(k)$ on $k^\times=O(\mathrm{Spec}\ k)^\times$. By the second exact sequence of \cite[Theorem 2.2.1]{mu-ordi Hasse}, $\mathrm{Pic}([\mathrm{H} \backslash \mathrm{Spec}\ k])\cong H_{\mathrm{alg}}^1(\mathrm{H}, k^\times)$, where $H_{\mathrm{alg}}^1(\mathrm{H}, k^\times)\subset H^1(\mathrm{H}(k), k^\times)$ is the subgroup of classes of cocycles induced by a morphism $\mathrm{H}\rightarrow \mathbb{G}_{m,k}$. Noting that $H^1(\mathrm{H}(k), k^\times)=\mathrm{Hom}(\mathrm{H}(k),k^\times)$ as we are using the trivial action, we have $H_{\mathrm{alg}}^1(\mathrm{H}, k^\times)\cong X^*(\mathrm{H})$. So $\mathrm{Pic}([\mathrm{H} \backslash \mathrm{Spec}\ k])\cong X^*(\mathrm{H})$, and hence is finite.
	
\end{proof}
\begin{remark}
	The element $\mathcal{L}_\eta\in\mathrm{Pic}([E_{\mathcal{Z}_w}\backslash\mathcal{G}^{\mathrm{rdt}}_0])$ is not of finite order in general, since otherwise
	the whole KR stratum (which is the whole special fiber if $K$ is hyperspecial) would be quasi-affine, which is absurd.
\end{remark}

\subsubsection{}
We can also talk about the relations between KR strata for different parahoric subgroups. Consider the projection morphism $\pi_{I,K}: \ES_{I,0}\ra\ES_{K,0}$.
Let $w\in \Admu_K$ which we view as an element of ${}^K\wt{W}^K$. Then \[\pi_{I,K}^{-1}(\ES_{K,0}^w)=\coprod_{x\in W_KwW_K\cap\Admu}\ES_{I,0}^x. \]
Thus we get an induced morphism
\[ \coprod_{x\in W_KwW_K\cap\Admu}\ES_{I,0}^x \lra \ES_{K,0}^w= \coprod_{y\in W_KwW_K\cap {}^K\wt{W}}\ES_{K,0}^y.\]
Note that by Theorem \ref{T: EKOR set}, $W_KwW_K\cap {}^K\wt{W}\subset {}^K\Admu=\Admu \cap {}^K\wt{W}$, thus
\[ W_KwW_K\cap {}^K\wt{W}\subset W_KwW_K\cap\Admu. \]
By Proposition \ref{change para--EKOR}, we have
\begin{itemize}
	\item 
for $x\in W_KwW_K\cap {}^K\wt{W}$,  $\pi_{I,K}(\ES_{I,0}^x)\subset \ES_{K,0}^x$; 
\item for $x\in W_KwW_K\cap\Admu \setminus W_KwW_K\cap{}^K\wt{W}$,
$\pi_{I,K}(\ES_{I,0}^x)\subset \coprod_{y\in \Sigma_K(w)} \ES_{K,0}^y$. 
\end{itemize}
If axiom 4 (c) holds, then both inclusions above are in fact equalities.

\subsection{Affine Deligne-Lusztig varieties}\label{subsection ADLV}
In this subsection, we discuss the local counterparts of our previous constructions and some local-global compatibilities of our constructions.
\subsubsection[]{}
We turn to a general local setting to discuss affine Deligne-Lusztig varieties.
Let $(G, [b], \{\mu\})$ be a triple where $G$ is a connected reductive group over $\Q_p$, $\{\mu\}$ is a conjugacy class of cocharacters $\mu: \mathbb{G}_{m,\ov{\Q}_p}\ra G_{\ov{\Q}_p}$, $[b]$ is a $\sigma$-conjugacy class such that $[b] \in B(G,\{\mu\})$.
Fix a representative $b\in G(\breve{\Q}_p)$ of $[b]$ and a parahoric subgroup $K\subset G(\Q_p)$. Let $\breve{K}\subset G(\breve{\Q}_p)$ be the associated parahoric subgroup.
We get the associated affine Deligne-Lusztig set \[X(\mu,b)_K=\{g\in G(\breve{\Q}_p)|\, g^{-1}b\sigma(g)\in \bigcup_{w\in \mathrm{Adm}(\{\mu\})_K}\breve{K}w\breve{K}\}/\breve{K},\]
which admits a geometric structure as a perfect scheme over $k=\ov{\mathbb{F}}_p$ by \cite{zhu-aff gras in mixed char} and \cite{BS}. See also the following \ref{subsection ADLV global}.  On the other hand, for any $w\in \Admu_K$, consider
\[X_w(b)_K=\{g\in G(\breve{\Q}_p)|\, g^{-1}b\sigma(g)\in \breve{K}w\breve{K}\}/\breve{K}.\]Then we have the KR decomposition
\[X(\mu,b)_K=\coprod_{w\in\Admu_K}X_w(b)_K. \]
We remark that in the above decomposition it can happen that some $X_w(b)_K$ is empty. We view $X(\mu,b)_K$ as a perfect scheme over $k$. Let $M^\loc=\bigcup_{w\in \mathrm{Adm}^K(\{\mu\})}\breve{K}w\breve{K}/\breve{K}$ which we view as a perfect scheme over $k$, cf. the following \ref{subsection loc sht}.
By construction, we have a tautological (``local model'') diagram (cf. \cite{guide to red mod p} p. 299)
\[\xymatrix{
	&\wt{X}(\mu,b)_K\ar[ld]_\pi\ar[rd]^q&\\
	X(\mu,b)_K& & M^\loc.
}\]Therefore, for any $X_w(b)_K\neq \emptyset$, we can view it as a locally closed perfect subscheme of $X(\mu,b)_K$. 

Following \cite{coxeter type} 3.4 and \cite{fully H-N decomp} 1.4, for $x\in {}^K\wt{W}$ we set
\[X_{K,x}(b)=\{g\in G(\breve{\Q}_p)|\, g^{-1}b\sigma(g)\in \breve{K}\cdot_\sigma\breve{I}x\breve{I}\}/\breve{K}.\]Then we have the EKOR decomposition 
\[ X(\mu,b)_K=\coprod_{x\in {}^K\Admu}X_{K,x}(b),\]
which is finer than the above KR decomposition. 

\subsubsection[]{}\label{subsubsec loc zeta}
Now we come back to the setting as in the beginning of this section.
Let $(G,X)$ be a Shimura datum of Hodge type. After fixing an element $[b]\in B(G,\{\mu\})$ and a representative $b\in G(\breve{\Q}_p)$ of $[b]$, we have the affine Deligne-Lusztig variety $X(\mu,b)_K$ as above.
Our construction in subsection \ref{subsection EO in KR} has a local analogue for $X(\mu,b)_K$, so that for each non empty KR stratum $X_w(b)_K$, we have a $\G_0^\rdt$-zip on it and thus a morphism of algebraic stacks
\[\zeta_w: X_w(b)_K\ra [E_{\mathcal{Z}_w}\backslash \G_0^\rdt]^{pf}.\]
Here $[E_{\mathcal{Z}_w}\backslash \G_0^\rdt]^{pf}$ is the perfection of the quotient stack $[E_{\mathcal{Z}_w}\backslash \G_0^\rdt]$, see \cite{XiaoZhu} Appendix (and also the following \ref{subsection app Shim}).
This gives us a geometric meaning for the above EKOR decomposition.
In the case that $K$ is hyperspecial and $X(\mu,b)_K$ is isomorphic to the perfection of the special fiber of a formal Rapoport-Zink space,  there is a related construction in 
\cite[Theorem 7.1]{Sh}.

Fix a sufficiently small $K^p$ and write $\ES_0=\ES_{\K,0}$ with $\K=KK^p$. Let $\ES_0^b$ be the Newton stratum attached to $[b]$.
Now the links between the local and global KR and EKOR strata are as follows.
\begin{proposition}\label{P: ADLV and EKOR}
	We have:
	\begin{enumerate}
		\item  $X_w(b)_K\neq \emptyset \Leftrightarrow \ES_0^w(k)\cap \ES_0^b(k)\neq \emptyset$.
		\item  $X_{K,x}(b)\neq \emptyset\Leftrightarrow \ES_0^x(k)\cap \ES_0^b(k)\neq \emptyset$.
	\end{enumerate}
\end{proposition}
\begin{proof}
	For (1), we have \[X_w(b)_K\neq \emptyset \Leftrightarrow \breve{K}w\breve{K}\cap [b]\neq \emptyset\] by definition. By \cite[Theorem 8.1]{zhou isog parahoric}, the map $\Upsilon_K: \ES_0(k)\ra C(\G,\{\mu\})$ is surjective. Let $C^w$ and $C^b$ be the inverse images of $w$ and $[b]$ under the natural surjections $C(\G,\{\mu\})\ra \Admu_K$ and $C(\G,\{\mu\})\ra B(G,\{\mu\})$ respectively. Then the surjectivity of $\Upsilon_K$ implies that \[\ES_0^w(k)\cap \ES_0^b(k)\neq \emptyset\Leftrightarrow C^w\cap C^b\neq \emptyset.\] Finally, one sees easily that \[C^w\cap C^b\neq \emptyset\Leftrightarrow \breve{K}w\breve{K}\cap [b]\neq \emptyset.\]
	
	The proof for (2) is similar.
\end{proof}

Fix a point $x_0\in \ES_{0}^b(k)$ and assume further \emph{either $G$ is residually split or $[b]$ is basic}. Then by \cite[Proposition 6.5]{zhou isog parahoric}, there exists a unique uniformization map
\[\iota_{x_0}: X(\mu,b)_K\ra \ES_0^b(k).\]One can check directly the following local-global compatibilities under the uniformization map:
\begin{corollary}\label{P: uniformization EKOR}
	The above morphism $\iota_{x_0}$ induces
	\begin{enumerate}
		\item for any $w\in \Admu_K$  a map \[X_w(b)_K\ra  \ES_0^w(k)\cap \ES_0^b(k),\]
		\item for any $x\in {}^K\Admu$ a map \[X_{K,x}(b)\ra \ES_0^x(k)\cap \ES_0^b(k).\] 
	\end{enumerate}
\end{corollary}

\section{EKOR strata of Hodge type: global constructions}\label{section global}
In this section, we will give some global constructions of the EKOR strata by adapting and generalizing some ideas of Xiao-Zhu in \cite{XiaoZhu}. More precisely, we will introduce the notions of local $(\G,\mu)$-Shtukas and their truncations in level 1: the so called $(m,1)$-restricted local $(\G,\mu)$-Shtukas, generalizing those in \cite{XiaoZhu} in the case of good reductions. We will construct a smooth morphism from the perfection $\ES_{\K,0}^{pf}$ of $\ES_{\K,0}$ to the moduli stack of $(m,1)$-restricted local $(\G,\mu)$-Shtukas, such that its fibers give the EKOR strata on $\ES_{\K,0}^{pf}$. This global construction will enable us to prove the closure relation for the EKOR strata, which is independent of the He-Rapoport axioms in \cite{He-Rap}. In the hyperspecial case, see \cite{Yan} for a different construction of the EO stratification using classical affine Grassmannians.

\subsection{Local $(\G,\mu)$-Shtukas and their moduli}\label{subsection loc sht}
We return to the setting of subsection \ref{Iwahori Weyl}. 
Let $G$ be a reductive group over $\mathbb{Q}_p$, and $\mathcal{G}$ be a parahoric group scheme over $\INT_p$ with generic fiber $G$. Let $K=\G(\Z_p)\subset G(\Q_p)$ be the associated parahoric subgroup.

\subsubsection[]{Witt vector affine flag varieties}
Let $k=\ov{\mathbb{F}}_p$ and $\Aff_k^{pf}$ be the category of perfect $k$-algebras. We will work with the fpqc topology on $\Aff_k^{pf}$, and we refer the reader to the appendices of \cite{zhu-aff gras in mixed char} and \cite{XiaoZhu} to some generalities on perfect algebraic geometry. Recall that we have the Witt vector loop groups $L^+\G$, $LG$, and the Witt vector affine flag variety $\Gr_\G$, such that
for any $R\in \Aff_k^{pf}$,  $L^+\G(R)=\G(W(R))$, $LG(R)=\G(W(R)[1/p])=G(W(R)[1/p])$, and 
\[\Gr_\G(R)=\{(\El,\beta)|\,\El:  \G\text{-torsor on}\; W(R), \beta: \El[1/p]\stackrel{\sim}{\ra}\El_0[1/p]\}\]	
where $\El_0$ is the trivial $\G$-torsor. 

Let $\mu: \mathbb{G}_{m,\ov{\Q}_p}\ra G_{\ov{\Q}_p}$ be a cocharacter and $\{\mu\}$ be the associated conjugacy class. After choosing a Borel subgroup $B$ of $G_{\ov{\Q}_p}$, we may assume that $\mu$ is dominant with respective to $B$. Recall that we have the finite subset $\Admu_K\subset W_K\backslash \wt{W}/W_K$. For any $w\in W_K\backslash \wt{W}/W_K$, we have the associated affine Schubert cell $\Gr_w\subset \Gr_\G$.
Let \[M^{\loc}=\bigcup_{w\in \Admu_K}\Gr_w\subset \Gr_\G\] be the closed subscheme associated to $\Admu_K$. In subsection \ref{subsection app Shim}, it will be the (perfection of the) special fiber of the Pappas-Zhu local model. 

\subsubsection[]{Local $\G$-Shtukas in mixed characteristic}
Motivated by \cite{SW} Definitions 11.4.1 and 23.1.2 and \cite{Pappas Ober},
we have the following generalization of local Shtukas in mixed characteristic of \cite[Definition 5.2.1]{XiaoZhu}. Recall that we write $\sigma: W(R)\ra W(R)$ for the Frobenius map, as in subsection \ref{subsection deformation}.
\begin{definition}\label{D: shtuka}
	Let	$R\in \Aff_k^{pf}$.
	A local $(\G,\mu)$-Shtuka (or, a local $\G$-Shtuka of type $\mu$) over $R$ is a pair $(\El,\beta)$, where
	\begin{itemize}
		
		\item $\El$ is a $\G$-torsor over $\Spec W(R)$, 
		\item $\beta$ is a modification of the $\G$-torsors $\sigma^\ast\El$ and $\El$ over $\Spec W(R)$, i.e. an isomorphism \[\beta: \sigma^\ast\El|_{\Spec W(R)[1/p]}\stackrel{\sim}{\ra}\El|_{\Spec W(R)[1/p]},\]
	\end{itemize}
	such that for any geometric point $x$ of $\Spec R$, we have $Inv_x(\beta)\in \Admu_K$.
	Sometimes we will also write the modification $\beta$ as $\beta: \sigma^\ast\El\dashrightarrow \El$.
\end{definition}
The last condition may need a few explanations: take any trivializations $\alpha: \El_{0,x} \stackrel{\sim}{\ra}\sigma^\ast\El_x  $ and $\gamma: \El_x\stackrel{\sim}{\ra} \El_{0,x} $, then the composition $\gamma\circ\beta\circ \alpha$ defines an element $g\in G(W(k(x))_\Q)$. For different choices of $\alpha'$ and $\gamma'$, we get $g'=h_1gh_2$ for some $h_1, h_2\in \G(W(k(x)))$. Thus the modification $\beta$ defines a well defined element \[Inv_x(\beta):=[g]\in  \G(W(k(x)))\backslash G(W(k(x))_\Q)/\G(W(k(x)))\simeq \breve{K}\backslash G(\breve{\Q}_p)/\breve{K}\simeq W_K\backslash \wt{W}/W_K.\] The above condition is that for all geometric points $x$ of $\Spec R$, we require the elements $Inv_x(\beta)$ to be in the finite subset $\Admu_K$.

Let $\Sht_{\mu,K}^{\loc}$ be  the prestack of $\G$-Shtukas of type $\mu$. To describe it, first recall that by construction we have
$\Gr_\G=LG/L^+\G$. In the following we will need to consider the $\sigma$-conjugation action of $L^+\G$ on $LG$: for any $R\in \Aff^{pf}_k, g\in LG(R)=G(W(R)_\Q)$, and $h\in L^+\G(R)=\G(W(R))$, \[h\cdot g: =hg\sigma(h^{-1})\in LG(R).\]
Similar to \cite{XiaoZhu} (5.3.7) and \cite{Zhu2} (4.1.1), we can describe $\Sht_{\mu,K}^{\loc}$ as follows.
\begin{lemma}\label{L: pretacks sht}
	We have the following isomorphism of prestacks
	\[\Sht_{\mu,K}^{\loc}\simeq \Big[\frac{M^{\loc,\infty}}{Ad_\sigma L^+\G}\Big],\]where $M^{\loc,\infty}\subset LG$ is the pre-image of $M^{\loc}\subset \Gr_\G=LG/L^+\G$, which is stable by the $\sigma$-conjugation action of $L^+\G$ on $LG$, and the quotient means that we take the $\sigma$-conjugation action of $L^+\G$ on $M^{\loc,\infty}$.
\end{lemma}
\begin{proof}
	This is similar to \cite{Zhu2} 4.1. We briefly sketch the arguments. First we note $M^{\loc,\infty}$ is stable under the $\sigma$-conjugation action of $L^+\G$.
	Let $\Sht_{\mu,K}^{\loc, \square}$ denote the $L^+\G$-torsor over $\Sht_{\mu,K}^{\loc}$ classifying local $(\G,\mu)$-Shtukas together with a trivialization $\epsilon: \El\simeq \El_0$. Then the composition $\epsilon\circ\beta\circ\sigma(\epsilon^{-1})$ defines an element $g\in M^{\loc,\infty}$ and induces an isomorphism $\Sht_{\mu,K}^{\loc, \square}\simeq M^{\loc,\infty}$, under which the $L^+\G$-action on $\Sht_{\mu,K}^{\loc, \square}$ is identified with the $\sigma$-conjugation action of $L^+\G$ on $M^{\loc,\infty}$. This gives the desired isomorphism of prestacks.
\end{proof}
Recall that in \ref{subsubsection CGmu} we have introduced the set $C(\G,\{\mu\})=\breve{K}\Admu\breve{K}/\breve{K}_\sigma$. Since $k=\ov{\mathbb{F}}_p$, we get $L^+\G(k)=\G(W(k))=\breve{K}, M^{\loc,\infty}(k)=\breve{K}\Admu\breve{K}$.
In particular, by Lemma \ref{L: pretacks sht} we have \[\Sht_{\mu,K}^{\loc}(k)=C(\G,\{\mu\}).\]

\subsubsection[]{An alternative description}
Let $(\El,\beta)$ be a local $(\G,\mu)$-Shtuka. Set $\overleftarrow{\El}:=\sigma^\ast\El$, $\overrightarrow{\El}:=\El$, then 
$(\El,\beta)$ is equivalent to the following data $(\gamma:  \overleftarrow{\El}\dashrightarrow \overrightarrow{\El}, \psi)$, where
\begin{itemize}
	\item $\gamma:  \overleftarrow{\El}\dashrightarrow \overrightarrow{\El}$ is a modification of $\G$-torsors over $\Spec W(R)$ of type $\mu$,
	\item $\psi:  \sigma^\ast \overrightarrow{\El}\simeq \overleftarrow{\El} $ is an isomorphism of $\G$-torsors over $\Spec W(R)$.
\end{itemize}
Consider the associated local Hecke stack $Hk^{\loc}_{\mu,K}$: for any $R\in \Aff^{pf}_k$, $Hk^{\loc}_{\mu,K}(R)$ classifies the modifications $\gamma:  \overleftarrow{\El}\dashrightarrow \overrightarrow{\El}$ of $\G$-torsors over $\Spec W(R)$ of type $\mu$. We have the isomorphism of stacks
\[Hk^{\loc}_{\mu,K}\simeq [L^+\G\backslash M^{\loc}].\]
Let $BL^+\G$ be the stack of $L^+\G$-torsors, and  \[\ovla{t}: Hk^{\loc}_{\mu,K}\ra  BL^+\G  \quad (\text{resp.} \ovra{t}: Hk^{\loc}_{\mu,K}\ra  BL^+\G)\] be the functor which sends $\gamma$ to $\overleftarrow{\El}$ (resp. $\overrightarrow{\El}$).
Then we have the following cartesian digram
\[\xymatrixcolsep{5pc}\xymatrix{
	\Sht_{\mu,K}^{\loc}\ar[r]\ar[d]& Hk^{\loc}_{\mu,K}\ar[d]^{\ovla{t}\times \ovra{t}}\\
	BL^+\G\ar[r]^{\sigma\times 1} & BL^+\G\times BL^+\G.
}\]
For later use, we will also view the functors $\ovla{t}$ and $\ovra{t}$ in the following way. Recall that $M^{\loc,\infty}\ra M^{\loc}$ is a $L^+\G$-torsor. Then $\ovra{t}$ is given by the $L^+\G$-torsor \[M^{\loc}\ra [L^+\G\backslash M^{\loc}]=Hk^{\loc}_{\mu,K},\] and $\ovla{t}$ is given by the 
$L^+\G$-torsor \[[L^+\G\backslash M^{\loc,\infty}]\ra [L^+\G\backslash M^{\loc,\infty}/L^+\G]=[L^+\G\backslash M^{\loc}]=Hk^{\loc}_{\mu,K}.\] 

\subsection{Moduli of $(m,1)$-restricted local $(\G,\mu)$-Shtukas}\label{subsection restricted loc sht}
To construct EKOR strata, we need ``truncation in level 1'' of local $(\G,\mu)$-Shtukas. 
\subsubsection[]{$(\infty,1)$-restricted local Shtukas}
Consider the reductive 1-truncation group $L^{1-\rdt}\G$, i.e. for any $R\in \Aff^{pf}_k$, \[L^{1-\rdt}\G(R)=\G_0^{\rdt}(R).\] Let \[L^+\G^{(1)-\rdt}:=\ker (L^+\G\ra L^{1-\rdt}\G)\]
and \[M^{\loc,(1)-\rdt}\subset LG/L^+\G^{(1)-\rdt}\] be the image of $M^{\loc,\infty}\subset LG$ under the projection $LG\ra LG/L^+\G^{(1)-\rdt}$.  Therefore, the natural morphism $LG/L^+\G^{(1)-\rdt}\ra LG/L^+\G$ induces a morphism
\[ M^{\loc,(1)-\rdt}\ra M^{\loc},\]which is a $\G_0^{\rdt}$-torsor. Motivated by Lemma \ref{L: pretacks sht}, we consider the prestack
\[ \Sht_{\mu,K}^{\loc(\infty,1)}= \Big[\frac{M^{\loc,(1)-\rdt}}{Ad_\sigma L^+\G}\Big].\] Let $B\G^{\rdt}_0$ be the stack of $\G_0^{\rdt}$-torsors. To describe $\Sht_{\mu,K}^{\loc(\infty,1)}$, let us first note that there are morphisms \[(\ovla{t}^{1-\rdt}, \ovra{t}): Hk^{\loc}_{\mu,K} \ra  B\G^{\rdt}_0\times BL^+\G, \]
where $\ovra{t}$ is given by the $L^+\G$-torsor \[M^{\loc}\ra [L^+\G\backslash M^{\loc}]=Hk^{\loc}_{\mu,K},\] and $\ovla{t}^{1-\rdt}$ is given by the $\G_0^{\rdt}$-torsor \[[L^+\G\backslash M^{\loc,(1)-\rdt}]\ra [L^+\G\backslash M^{\loc,(1)-\rdt}/\G_0^{\rdt}]=[L^+\G\backslash M^{\loc}]=Hk^{\loc}_{\mu,K}.\]
Then
we have a similar cartesian digram \[\xymatrixcolsep{5pc}\xymatrix{
	\Sht_{\mu,K}^{\loc(\infty,1)}\ar[r]\ar[d]& Hk^{\loc}_{\mu,K}\ar[d]^{\ovla{t}^{1-\rdt}\times res^\infty_1\circ \ovra{t}}\\
	B\G^{\rdt}_0\ar[r]^{\sigma\times 1} & B\G^{\rdt}_0\times B\G^{\rdt}_0,
}\]where \[res^\infty_1: BL^+\G\ra B\G_0^{\rdt}\] is the map induced by the composition of projections $L^+\G\ra L^1G=\G_0\ra \G_0^{\rdt}$.
In particular,
$\Sht_{\mu,K}^{\loc(\infty,1)}$ has the following moduli interpretation. 
For any $R\in \Aff^{pf}_k$, $\Sht_{\mu,K}^{\loc(\infty,1)}(R)$ classifies \begin{itemize}
	\item a modification $\beta: \overleftarrow{\El}\dashrightarrow\overrightarrow{\El}$ of $\G$-torsors over $\Spec W(R)$ of type $\mu$, i.e. an element $\beta\in Hk^{\loc}_{\mu,K}(R)$,
	\item an isomorphism $\psi: \sigma^\ast(\overrightarrow{\El})|^{\rdt}_R\simeq \overleftarrow{\El}|^{\rdt}_R$ of $\G_0^{\rdt}$-torsors over $\Spec R$.
\end{itemize}
Unfortunately, $\Sht_{\mu,K}^{\loc(\infty,1)}$ is not algebraic, since $L^+\G$ is infinite dimensional.

\subsubsection[]{$(m,1)$-restricted local Shtukas}
We need some variants of the above construction. For any integer $m\geq 1$, consider the $m$-truncation group $L^m\G$, i.e. for any $R\in \Aff^{pf}_k$, $L^m\G (R)=\G(W_m(R))$. We have the natural projection morphism \[\pi_{m,1-\rdt}: L^m\G\ra L^{1-\rdt}\G.\] Before proceeding further, let us note that the action of $L^+\G$ on $M^{\loc}$ factors through $L^m\G$ for some sufficiently large integer $m$. Indeed, recall $M^{\loc}=\bigcup_{w\in \Admu_K}\Gr_w$ and for each $w\in \Admu_K$, we have \[\Gr_w\simeq L^+\G/(L^+\G\cap n_w L^+\G n_w^{-1}),\] where $n_w\in LG(k)$ are some representatives for all the $w\in  \Admu_K$.
Consider the intersection $\bigcap_{w\in \Admu_K}L^+\G\cap n_w L^+\G n_w^{-1}$. Let \[L^+\G^{(m)}=\ker (L^+\G\ra L^m\G).\]   Since the descending chain of subgroups $L^+\G^{(m)}$ forms a topological basis of neighborhoods of the identity, we find that for  some sufficiently large integer $m$, \[L^+\G^{(m)}\subset \bigcap_{w\in \Admu_K}L^+\G\cap n_w L^+\G n_w^{-1},\] i.e.  the action of $L^+\G$ on $M^{\loc}$ factors through $L^m\G$. 

Let $m_0$ be the minimal integer satisfying the above. Take any integer $m\geq m_0+1$.
Now we consider the $Ad_\sigma$-action of $L^m\G$ on $M^{\loc,(1)-\rdt}$: for any $R\in \Aff^{pf}$, $g\in L^m\G(R)$ and $x\in M^{\loc,(1)-\rdt}(R)$,
\[g\cdot x=gx\sigma(\pi_{m,1-\rdt}(g))^{-1}.\]
\begin{definition}\label{definition (m,1) sht}
	For any sufficiently large integer $m$, set \[\Sht_{\mu,K}^{\loc(m,1)}= \Big[\frac{M^{\loc,(1)-\rdt}}{Ad_\sigma L^m\G}\Big].\] This is an algebraic stack over $k$.
\end{definition}
Let $m$ be as above and $Hk^{\loc(m)}_{\mu,K}=[L^m\G\backslash M^{\loc}]$, which is called the $m$-restricted local Hecke stack. Similarly as before, we have morphisms \[(\ovla{t}^{1-\rdt}, \ovra{t}): Hk^{\loc(m)}_{\mu,K} \ra  B\G^{\rdt}_0\times BL^m\G, \]
where $\ovra{t}$ is given by the $L^m\G$-torsor \[M^{\loc}\ra [L^m\G\backslash M^{\loc}]=Hk^{\loc}_{\mu,K},\] which we denote by $\overrightarrow{\El}|_{D_m}$,
and $\ovla{t}^{1-\rdt}$ is given by the $\G_0^{\rdt}$-torsor \[[L^m\G\backslash M^{\loc,(1)-\rdt}]\ra [L^m\G\backslash M^{\loc,(1)-\rdt}/\G_0^{\rdt}]=[L^m\G\backslash M^{\loc}]=Hk^{\loc(m)}_{\mu,K},\]which we denote by $\overleftarrow{\El}|^{\rdt}_{D_1}$.
Then we have a similar cartesian digram \[\xymatrixcolsep{5pc}\xymatrix{
	\Sht_{\mu,K}^{\loc(m,1)}\ar[r]\ar[d]& Hk^{\loc(m)}_{\mu,K}\ar[d]^{\ovla{t}^{1-\rdt}\times res^m_1\circ \ovra{t}}\\
	B\G^{\rdt}_0\ar[r]^{\sigma\times 1} & B\G^{\rdt}_0\times B\G^{\rdt}_0,
}\]where $res^m_1: BL^m\G\ra B\G^{\rdt}_0$ is the map induced by the projection $L^m\G\ra \G_0^{\rdt}$. $\Sht_{\mu,K}^{\loc(m,1)}$ has the following moduli interpretation. 
For any $R\in \Aff^{pf}_k$, $\Sht_{\mu,K}^{\loc(m,1)}(R)$ classifies \begin{itemize}
	\item a point of $Hk^{\loc(m)}_{\mu,K}(R)$,
	\item an isomorphism $\psi: \sigma^\ast(\overrightarrow{\El}|_{D_m})|^{\rdt}_R\simeq \overleftarrow{\El}|^{\rdt}_R$ of $\G_0^{\rdt}$-torsors over $\Spec R$.
\end{itemize}

\begin{lemma}\label{L: top restricted sht}
	For any $m$ sufficiently large as above,
	we have homeomorphisms of 
	underlying topological spaces 
	\[ |\Sht_{\mu,K}^{\loc(\infty,1)}|\simeq |\Sht_{\mu,K}^{\loc(m,1)}|\simeq {}^K\Admu.\]
\end{lemma}
\begin{proof}
	As always, set $k=\ov{\mathbb{F}}_p$. 
	Let $\breve{K}_1$ be the pro-unipotent radical of $\breve{K}$, then by definition $L^+\G^{(1)-\rdt}(k)=\breve{K}_1$. Since $M^{\loc,\infty}(k)=\breve{K}\Admu\breve{K}$, we have \[M^{\loc,(1)-\rdt}(k)=\breve{K}\Admu\breve{K}/\breve{K}_1.\]
	Then \[\Sht_{\mu,K}^{\loc(\infty,1)}(k)=\frac{\breve{K}\Admu \breve{K}/\breve{K}_1}{\breve{K}_{\sigma}}=\frac{\breve{K}\Admu \breve{K}}{\breve{K}_{\sigma}(\breve{K}_1\times \breve{K}_1)}.\]
	By Theorem \ref{He's result--decompo} (works of Lusztig and He), we have
	\[\breve{K}\Admu \breve{K}/\breve{K}_{\sigma}(\breve{K}_1\times \breve{K}_1) \simeq \Admu^K\cap  {}^K\wt{W}={}^K\Admu.\]
	This gives a bijiection \[|\Sht_{\mu,K}^{\loc(\infty,1)}|\simeq {}^K\Admu.\]  
	To show it is a homeomorphism of topological spaces, we apply
 \cite[Proposition 3.5]{He} (see also loc. cit. page 3248, paragraph 3.5), which implies that the closure relation on $|\Sht_{\mu,K}^{\loc(\infty,1)}|$ is exactly given by the partial order $\leq_{K,\sigma}$.
	Since for any $m$ sufficiently large, the $Ad_\sigma$-action of $L^+\G$ on $M^{\loc,(1)-\rdt}$ factors through $L^m\G$, we have the homeomorphism $|\Sht_{\mu,K}^{\loc(\infty,1)}|\simeq |\Sht_{\mu,K}^{\loc(m,1)}|$.
\end{proof}

Consider the $1$-restricted local Hecke stack $Hk^{\loc(1)}_{\mu,K}=[\G_0\backslash M^{\loc}]$. Sometimes we also write it as $\Sht_{\mu,K}^{\loc(1,0)}$. The underlying topological space is \[|[\G_0\backslash M^{\loc}]|=\Admu_K.\]
We have natural maps
\[\Sht_{\mu,K}^{\loc}\ra \Sht_{\mu,K}^{\loc(\infty,1)}\ra\Sht_{\mu,K}^{\loc(m,1)}\stackrel{\pi_K}{\lra}[\G_0\backslash M^{\loc}].\]
\begin{proposition}
	All the arrows above are perfectly smooth.
\end{proposition}
\begin{proof}
	The first arrow is perfectly smooth, since $M^{\loc,\infty}\ra M^{\loc, 1-\rdt}$ is a $L^{(1)-\rdt}\G$-torsor; the morphism $\Sht_{\mu,K}^{\loc(\infty,1)}\ra\Sht_{\mu,K}^{\loc(m,1)}$ is a $L^{(m)}\G$-gerbe, thus it is perfectly smooth; the morphism $\Sht_{\mu,K}^{\loc(m,1)}\stackrel{\pi_K}{\lra}[\G_0\backslash M^{\loc}]$ is a composition of a $\G_0^\rdt$-torsor and a $\ker(L^m\G\ra L^1\G=\G_0)$-gerbe, thus it is also perfectly smooth.
\end{proof}
When $K=I$ is an Iwahori subgroup, then by \cite[Corollary 6.2]{He-Rap}, the morphism \[\pi_I: \Sht_{\mu,I}^{\loc(m,1)}\ra[\G_0\backslash M^{\loc}]\] induces a homeomorphism of underlying topological spaces \[|\Sht_{\mu,I}^{\loc(m,1)}|\simeq|[\G_0\backslash M^{\loc}]|\simeq \Admu.\]

The following proposition gives the links between moduli stacks of $(m,1)$-restricted local $(\G,\mu)$-Shtukas and $\G_0^{\rdt}$-zips. For any $w\in \Admu_K$, recall that in subsection \ref{subsection EO in KR} we have the algebraic stack $\G_0^{rdt}\text{-Zip}_{J_w}$ of $\G_0^{rdt}$-Zips of type $J_w$. Let $\G_0^{rdt}\text{-Zip}_{J_w}^{pf}$ be its perfection.
\begin{proposition}\label{P: G zip and shtukas}
	For any $ w\in [\G_0\backslash M^{\loc}]$, there exists a natural perfectly smooth map of algebraic stacks
	\[\pi_K^{-1}(w)\lra \G_0^{rdt}\text{-Zip}_{J_w}^{pf} \]
	inducing a homeomorphism of the underlying topological spaces $|\pi_K^{-1}(w)|\simeq |\G_0^{rdt}\text{-Zip}_{J_w}^{pf}|$.
\end{proposition}
\begin{proof}
	For any $w\in |[\G_0\backslash M^{\loc}]|=\Admu_K$, let $\wt{\Gr_w}\ra \Gr_w$ be the $\G_0^{\rdt}$-torsor induced by $M^{\loc, (1)-\rdt}\ra M^{\loc}$ over $\Gr_w\subset M^{\loc}$. Then we have \[\pi_K^{-1}(w)\simeq \Big[\frac{\wt{\Gr_w}}{Ad_\sigma L^m\G}\Big].\] Let $J_w$ be as above. We get parabolic subgroups $P_{J_w}$ and $P_{\sigma'(J_w)}$ of $\G_0^{\rdt}$ with corresponding unipotent radicals $U_{J_w}$ and $U_{\sigma'(J_w)}$ as in \ref{zip in EKOR}. Let $L_{J_w}$ be the common Levi subgroup of $P_{J_w}$ and $P_{\sigma'(J_w)}$.
	Then similar to \cite[Lemma 5.3.6]{XiaoZhu}, we have a perfectly smooth morphism
	\[\wt{\Gr_w}\ra \Big[(\G_0^{\rdt}/U_{J_w}\times \G_0^{\rdt}/U_{\sigma'(J_w)})/L_{J_w}\Big]^{pf},\]
	which intertwines the $L^m\G$-action on the left hand side to the $\G_0^{\rdt}$-action on the right hand side.
	Hence we get a perfectly smooth morphism
	\[ \Big[\frac{\wt{\Gr_w}}{Ad_\sigma L^m\G}\Big]\ra \G_0^{\rdt}\backslash \Big[(\G_0^{\rdt}/U_{J_w}\times \G_0^{\rdt}/U_{\sigma'(J_w)})/L_{J_w}\Big]^{pf}= \G_0^{rdt}\text{-Zip}_{J_w}^{pf},\]
	which induces a homeomorphism between the underlying topological spaces by \ref{zip in EKOR}. Here the last equality comes from \cite[Theorem 12.7]{zipdata}.
\end{proof}

When $K$ is a hyperspecial subgroup, then $|[\G_0\backslash M^{\loc}]|$ consists of a single point, thus we have $\pi_K^{-1}(w)=\Sht_{\mu,K}^{\loc(m,1)}$, and the above proposition recovers \cite[Lemma 5.3.6]{XiaoZhu}.

\subsection{Application to affine Deligne-Lusztig varieties}\label{subsection ADLV global}
Let $(G, [b], \{\mu\})$ be a triple as in subsection \ref{subsection ADLV}. Fix a representative $b\in G(\breve{\Q}_p)$ and a parahoric subgroup $K\subset G(\Q_p)$. Let $\G$ be the parahoric model over $\Z_p$ of $G$ corresponding to $K$.
We get the associated affine Deligne-Lusztig variety $X(\mu,b)_K$ as in subsection \ref{subsection ADLV}. This is a closed subscheme of $\Gr_\G$ which can be described as follows. For any $R\in \Aff_k^{pf}$, we have \[
\begin{split}
X(\mu,b)_K(R)=\{(\El,\beta)\in \Gr_\G(R)\,&|\,\forall\, \text{geometric point}\, x \,\text{of} \,\Spec R,\\ &Inv_x(\beta^{-1}b\sigma(\beta))\in \Admu_K\}.\end{split}\]
Fix an element $w\in \Admu_K$, then the locally closed subscheme $X_w(b)_K$ introduced in \ref{subsection ADLV} can be described similarly: for any $R\in \Aff_k^{pf}$, we have \[\begin{split}
X_w(b)_K(R)=\{(\El,\beta)\in \Gr_\G(R)\,&|\,\forall\, \text{geometric point}\, x \,\text{of} \,\Spec R,\\ &Inv_x(\beta^{-1}b\sigma(\beta))=w\}.\end{split}\]
The KR stratification in this setting is the following \[X(\mu,b)_K=\coprod_{w\in\Admu_K}X_w(b)_K.\]

The element $b$ defines a local $(\G,\mu)$-Shtuka by the modification \[b: \sigma^\ast\El_0=\El_0\dashrightarrow\El_0.\] For any point $(\El,\beta)\in X(\mu,b)_K(R)$, we get another local $(\G,\mu)$-Shtuka by the modification \[\beta^{-1}b\sigma(\beta): \sigma^\ast\El\dashrightarrow\El.\] In this way
we get a natural morphism of prestacks \[X(\mu,b)_K\ra \Sht_{\mu,K}^{\loc}.\] Fix an integer $m$ which is sufficiently large as in the last subsection.
Composing the above morphism with the restriction map $\Sht_{\mu,K}^{\loc}\ra \Sht_{\mu,K}^{\loc(m,1)}$, we get a morphism of
algebraic stacks
\[\upsilon_K: X(\mu,b)_K\ra \Sht_{\mu,K}^{\loc(m,1)},\]
which interpolates the morphisms $\zeta_w$ in \ref{subsubsec loc zeta} when $w$ varies. The fibers of $\upsilon_K$ are then the EKOR strata of $X(\mu,b)_K$. This gives a geometric meaning of the EKOR decomposition in \cite{coxeter type} 3.4 and \cite{fully H-N decomp} 1.4.

\subsection{Application to Shimura varieties}\label{subsection app Shim}
Now we come back to Shimura varieties and the setting at the beginning of section \ref{section EKOR loc}. Let $(G,X)$ be a Shimura datum of Hodge type, and $\ES_\K=\ES_\K(G,X)$ be the integral model introduced in \ref{subsection integral Hodge} with $\K=K_pK^p, K_p=\G(\Z_p)$ and $\G=\G^\circ$. 
We are interested in the special fibers $\ES_{\K,0}$ of $\ES_\K$, and we consider its perfection \[\Sh_\K:=(\ES_{\K,0})^{pf}=\varprojlim_{\sigma}\ES_{\K,0}.\] In the following, since $K^p$ is fixed, we will simply write $\Sh_\K$ as $\Sh_K$, and the subscript of related morphisms with source $\Sh_K$ will also be written simply as $K$.
By \cite[Corollary 21.6.10]{SW} and \cite[Theorem 2.15]{He-Pap-Rap}, the perfection of the special fiber of the Pappas-Zhu local model can be identified with the
closed subscheme $M^{\loc}$ in subsection \ref{subsection loc sht}.
As usual, we get the conjugacy class of minuscule characters $\{\mu\}$.
\begin{proposition}\label{P: univ sht}
	There exists a $\G$-Shtuka of type $\mu$ over $\Sh_K$. In particular, we get a morphism of prestacks \[\Sh_K \ra \Sht_{\mu,K}^{\loc}.\]
\end{proposition}
\begin{proof}
	Recall that in \ref{subsubsection global cystalline tensor} we have the display $(\underline{M}, \underline{M}_1, \underline{\Psi})$ attached to the $p$-divisible group $\A[p^\infty]\mid_{\ES_\K^\wedge}$, and by Proposition \ref{global cris} there is a tensor $\ul{s}_{\cris}\in \ul{M}^\otimes$.
	For any $R\in \Aff^{pf}_k$, by Corollary \ref{coro--cris torsor} \[\El:=\sigma^{-1,\ast}\IIsom_{W(R)}\big((L^\vee,s)\otimes W(R), (\underline{M}, \underline{s}_\cris)\big)\] is a $\mathcal{G}$-torsor over $W(R)$. Then $\sigma^\ast\El$ is also a $\mathcal{G}$-torsor over $W(R)$, and the linearization of the Frobenius on $\ul{M}$ induces a modification $\beta: \sigma^\ast\El \dashrightarrow \El$.
	The local model diagram over $k$ implies that for any geometric point $x\in \Spec R$, we have $Inv_x(\beta)\in\Admu_K$.
	Thus we get a $\G$-Shtuka of type $\mu$ over $\Sh_K$.
	
\end{proof}

Now consider the moduli stack $\Sht_{\mu,K}^{\loc(m,1)}$ of $(m,1)$-restricted local $(\G,\mu)$-Shtukas as in Definition \ref{definition (m,1) sht} for the current setting.
\begin{lemma}
	The minimal $m$ for the definition of the above $\Sht_{\mu,K}^{\loc(m,1)}$ is $m=2$.
\end{lemma}
\begin{proof}
	We will prove that the left action of $L^+\G$ on $M^{\loc,(1)-\rdt}$ factors through $L^2\G$, and the $Ad_\sigma$-action is similar.

	For any $m\geq 2$, as before, let $L^+\G^{(m)}=\ker (L^+\G\ra L^m\G)$. We need to check that $L^+\G^{(2)}$ acts trivially on $M^{\loc,(1)-\rdt}$. This comes from the fact the $L^+\G^{(1)}$ acts trivially on $M^{\loc}$, since by construction of the Pappas-Zhu local model, the action of $L^+\G$  on $M^{\loc}$ factors through $\G_0$. 
	
	More precisely,	let $k=\ov{\mathbb{F}}_p$ and $\mathfrak{g}$ be the Lie algebra of $\mathcal{G}_{W(k)}$, then $L^+\mathcal{G}^{(n)}(k)=1+p^n\cdot \mathfrak{g}$ for all positive integer $n$. The action of $L^+\mathcal{G}^{(1)}$ on $M^{\mathrm{loc}}$ is trivial, so for any $w\in \mathrm{Adm}(\{\mu\})_K$, \[L^+\mathcal{G}^{(1)}(k)\subset w\cdot L^+\mathcal{G}(k)\cdot w^{-1},\] and hence $p\cdot \mathfrak{g}\subset w\mathfrak{g} w^{-1}$. In particular, we have $$L^+\mathcal{G}^{(n+1)}(k)\subset w\cdot L^+\mathcal{G}^{(n)}(k)\cdot w^{-1}, \text{\ for\ all\ positive\ integer\ } n \text{\ \ and all\ } w\in \mathrm{Adm}(\{\mu\})_K.$$
	We use the notation in the proof of Proposition \ref{P: G zip and shtukas}. For any $w\in \mathrm{Adm}(\{\mu\})_K$, \[\wt{\Gr_w}(k)=\breve{K}w\breve{K}/\breve{K}_1=L^+\G(k)wL^+\G(k)/L^+\G^{(1)-\rdt}(k).\]
	For any $g\in L^+\mathcal{G}^{(2)}(k), w\in \mathrm{Adm}(\{\mu\})_K$ and $g_1,g_2\in L^+\mathcal{G}(k)$, we have \[g\cdot g_1wg_2=g_1wg_2\cdot (g_1wg_2)^{-1}\cdot g\cdot g_1wg_2.\] Noting that $L^+\mathcal{G}^{(2)}\subset L^+\mathcal{G}$ is normal, we see by the above inclusion that $g\cdot g_1wg_2$ is of the form $g_1wg_2 \cdot g_3$ for some $g_3\in L^+\mathcal{G}^{(1)}(k)\subset L^+\G^{(1)-\rdt}(k)$. This proves that the usual action of $L^+\G^{(2)}$ on $M^{\loc,(1)-\rdt}$ is trivial. 
\end{proof}

Fix an integer $m\geq 2$. Composing the morphism $\Sh_K \ra \Sht_{\mu,K}^{\loc}$ in Proposition \ref{P: univ sht} with $
\Sht_{\mu,K}^{\loc}\ra \Sht_{\mu,K}^{\loc(m,1)}$, we get a morphism of stacks
\[\upsilon_K: \Sh_K\ra \Sht_{\mu,K}^{\loc(m,1)}. \]
Recall that by the local model diagram, we have the morphism of stacks \[\lambda_K: \Sh_K\ra [\G_0\backslash M^{\loc}],\] which is perfectly smooth.
\begin{theorem}\label{T: perf smooth}
	The following diagram commutes:
	\[\xymatrixcolsep{5pc}\xymatrix{
		\Sh_K\ar[r]^{\upsilon_K}\ar[rd]_{\lambda_K}&\Sht_{\mu,K}^{\loc(m,1)}\ar[d]^{\pi_K}\\
		& [\G_0\backslash M^{\loc}].
	}\]Moreover, $\upsilon_K$ is perfectly smooth.
\end{theorem}
\begin{proof}
	The above commutative diagram comes from the fact that the perfection of the special fiber of the Pappas-Zhu local model can be identified with the
	closed subscheme $M^{\loc}$ in subsection \ref{subsection loc sht}, by \cite[Corollary 21.6.10]{SW} and \cite[Theorem 2.15]{He-Pap-Rap}.
	
	Now we prove that $\upsilon_K$ is perfectly smooth. The arguments in the proof of \cite[Proposition 7.2.4]{XiaoZhu} apply here. For the reader's convenience, we recall how to adapt the arguments of loc. cit. to our situation as follows. First, similarly as in the proof of Lemma \ref{L: pretacks sht}, let $\Sht_{\mu,K}^{\loc(m,1),\square}$ be the $L^m\G$-torsor over $\Sht_{\mu,K}^{\loc(m,1)}$ such that for any $R\in \Aff_k^{pf}$, $\Sht_{\mu,K}^{\loc(m,1),\square}(R)$ classifies the trivialization $\epsilon: \overrightarrow{\El}|_{D_m(R)}\simeq \El_{0,D_m(R)}$, where $D_m(R)=\Spec W_m(R)$. Then standard arguments show that \[\Sht_{\mu,K}^{\loc(m,1),\square}\simeq M^{\loc,(1)-\rdt}\] such that the $L^m\G$-action on the left hand side corresponds to the $Ad_\sigma L^m\G$-action on the right hand side. On the other hand, we can consider the $L^m\G$-torsor $\Sh_K^{(m,1)\square}$ over $\Sh_K$ such that for any $R\in \Aff_k^{pf}$, $\Sh_K^{(m,1)\square}(R)$ classifies the trivialization $\epsilon: \overrightarrow{\El}|_{D_m(R)}\simeq \El_{0,D_m(R)}$. 
	Clearly we have \[\Sh_K^{(m,1)\square}=\Sh_K\times_{\Sht_{\mu,K}^{\loc(m,1)}}\Sht_{\mu,K}^{\loc(m,1),\square}. \]
	Let $\pi(m,1): \Sh_K^{(m,1)\square}\ra \Sh_K$ be the projection, which is a $L^m\G$-torsor. Consider the composition map
	\[q(m,1): \Sh_K^{(m,1)\square}\ra \Sht_{\mu,K}^{\loc(m,1),\square} \simeq M^{\loc,(1)-\rdt}.\] Then the map $\upsilon_K: \Sh_K\ra \Sht_{\mu,K}^{\loc(m,1)}$ is equivalent to the following diagram
	\[\xymatrix{
		&\Sh_K^{(m,1)\square}\ar[ld]_{\pi(m,1)}\ar[rd]^{q(m,1)}&\\
		\Sh_K& & M^{\loc,(1)-\rdt},
	}\]
	which is $L^m\G$-equivariant for the actions of $L^m\G$ on $\Sh_K^{(m,1)\square}$ and $M^{\loc,(1)-\rdt}$ as above. 
	To prove $\upsilon_K$ is perfectly smooth, we need only to show $q(m,1)$ is perfectly smooth.
	
	Let $x^\square$ be a $k$-point of $\Sh_K^{(m,1)\square}$  with image $x\in 
	\Sh_K$. Then we can find an \'etale neighborhood $a: U\ra \Sh_K$ of $x$ such that the pullback of the $L^m\G$-torsor $\overrightarrow{\El}|_{D_m}$ to $U$ is trivial. Fix such a trivialization  $\epsilon: \overrightarrow{\El}|_{D_m}\simeq \El_{0,D_m}$, which is equivalent to a lifting $a^{(m)}: U\ra \Sh_K^{(m,1)\square}$ of $a$. Then the map
	\[a^\square: U\times L^m\G\ra \Sh_K^{(m,1)\square}, \quad (u,g)\mapsto (a(u), g\epsilon)\]
	is \'etale and gives an \'etale neighborhood of $x^\square$. It suffices to show the composition
	\[a^\square(m,1)=q(m,1)\circ a^\square:  U\times L^m\G\ra \Sh_K^{(m,1)\square}\ra M^{\loc,(1)-\rdt}\]
	is perfectly smooth. We can make the above map more explicit.
	Consider the composition
	\[a(m,1)=q(m,1)\circ a^{(m)}: U\ra \Sh_K^{(m,1)\square}\ra  M^{\loc,(1)-\rdt}.\]
	Then for any $(u,g)\in U\times L^m\G$, we have \[a^\square(m,1)(u,g)=ga(m,1)(u)\pi_{m,1-\rdt}(\sigma(g)^{-1}).\]
	
	The perfection of the local model diagram over $k$ gives us
	\[\xymatrix{
		&\Sh_K^{(1,0)\square}\ar[ld]_{\pi(1,0)}\ar[rd]^{q(1,0)}&\\
		\Sh_K& & M^{\loc},
	}\]
	with $\Sh_K^{(1,0)\square}=\wt{\Sh_K}:=(\wt{\ES}_{K,0})^{pf}, \pi(1,0)=\pi^{pf}$, and $q(1,0)=q^{pf}$. The morphism $a^{(m)}: U\ra \Sh_K^{(m,1)\square}$ naturally induces a morphism
	$a^{(1)}: U\ra \Sh_K^{(1,0)\square}$. The local model diagram implies that the composition
	\[U\ra  \Sh_K^{(1,0)\square}\ra M^{\loc} \] is 
	\'etale. Now we have the following commutative diagram
	\[\xymatrixcolsep{5pc}\xymatrix{
		U\times L^m\G\ar[r]^{a^\square(m,1)}\ar[d]& M^{\loc,(1)-\rdt}\ar[d]\\
		U\ar[r]& M^{\loc},
	}\]
	where $U\times L^m\G\ra U$ is the natural projection. Since the bottom line is \'etale, to show $a^\square(m,1)$ is perfectly smooth, it suffices to show the induced map
	\[\wt{a^\square(m,1)}: U\times L^m\G\ra U\times_{M^{\loc}}M^{\loc,(1)-\rdt}\] is perfectly smooth.  Note that the left hand side is a trivial $L^m\G$-torsor over $U$, the right hand side is a $\G_0^{\rdt}$-torsor over $U$, and the above map is morphism over $U$. 
	
	Let $L^m_p\G$ be the usual Greenberg transform of $\G\otimes{\Z/p^m\Z}$, then $L^m\G$ is the perfection of $L^m_p\G$. The morphism $a: U\ra M^{\loc}$ descends to an \'etale morphism $a': U'\ra M^{\loc}$. Let $U^{'(1)}$ be the trivial $\G_0^{\rdt}$-torsor over $U'$. Then $(U^{'(1)})^{pf}\simeq U\times_{M^{\loc}}M^{\loc,(1)-\rdt}$, and the above morphism $\wt{a^\square(m,1)}$ is the perfection of 
	\[f: U'\times L^m_p\G\ra  U^{'(1)}, \quad (u,g)\mapsto g a(m,1)(u)\pi_{m,1-\rdt}(\sigma(g)^{-1}),\]
	which is a morphism over $U'$. 
	It suffices to show this morphism is smooth over $U'$. 
	Noting that $L_p^m\mathcal{G}$ is smooth and that $U'^{(1)}$ is smooth over $U'$, we only need to show that for each $x\in U'(k)$, the induced map on fibers $f_x:L_p^m\mathcal{G}\rightarrow U'^{(1)}_x$ is smooth. On the level of tangent spaces, we can ignore the action of $\sigma(g)$ and the infinitesimal action induced by $f_x$ is identified with the projection $L_p^m\mathcal{G}\rightarrow \mathcal{G}_0^{\mathrm{rdt}}$. So $f_x$ and hence $f$ is smooth.
	

\end{proof}

Consider the morphism of stacks
\[\upsilon_K: \Sh_K\ra \Sht_{\mu,K}^{\loc(m,1)}. \]
By Lemma \ref{L: top restricted sht} $|\Sht_{\mu,K}^{\loc(m,1)}|\simeq {}^K\Admu$ and by Proposition \ref{P: G zip and shtukas},
the fibers of $\upsilon_K$ are then the  EKOR strata of $\Sh_K$. Note that we have the identification of underling topological spaces \[|\Sh_K|=|\ES_{\K,0}|.\]
Since by Theorem \ref{T: perf smooth} $\upsilon_K$ is perfectly smooth, we get the closure relation for EKOR strata on $\ES_{\K,0}$, which is independent of the axioms of \cite{He-Rap} (compare Corollary \ref{coro--non-emp and closure} (2)):
\begin{corollary}\label{C: closure relation}
	For any $x\in {}^K\Admu$, the Zariski closure of the EKOR stratum $\ES_{\K,0}^x$ is given by
	\[\ov{\ES_{\K,0}^x}=\coprod_{x'\leq_{K,\sigma}x}\ES_{\K,0}^{x'}.\]
\end{corollary}

\section[EKOR abelian type]{EKOR strata of abelian type}\label{section abelian}
In this section, we will extend the construction of the EKOR stratification to Shimura varieties of abelian type.

\subsection{Some group theory}

\subsubsection[]{}\label{subsubsection connected hodge adjoint}
Let $\alpha: G_1\ra G_2$ be a central extension between connected reductive groups over $\Q_p$ with kernel $Z$. Then $\alpha$ induces a canonical $G_1(\Q_p)$-equivariant map \[\alpha_\ast: \B(G_1,\Q_p)\ra \B(G_2,\Q_p).\] Let $x\in \B(G_1,\Q_p)$ with associated group schemes $\G_{1,x}$ and $\G_{1,x}^\circ$. Set $y=\alpha_\ast(x)$. Then $\alpha$ extends to group scheme homomorphisms
\[\alpha: \G_{1,x}\ra \G_{2,y},\quad \alpha: \G_{1,x}^\circ\ra\G_{2,y}^\circ.\]
Let $\mathcal{Z}$ be the schematic closure of $Z$ in $\G_{1,x}^\circ$.
\begin{proposition}{\rm(\cite[Proposition 1.1.4]{Paroh})}
	Suppose that $G_1$ splits over a tamely ramified extension of $\Q_p$ and that $Z$ is either a torus or is finite of rank prime to $p$. Then $\mathcal{Z}$  is smooth over $\Z_p$ and it fits in an exact sequence
	\[1\ra \mathcal{Z}\ra \G_{1,x}^\circ\ra\G_{2,y}^\circ \ra 1\] of group schemes over $\Z_p$. If $Z$ is a torus which is a direct summand of an induced torus, then $\mathcal{Z}=\mathcal{Z}^\circ$ is the connected Neron model of $Z$.  
\end{proposition}

In particular, if $G$ splits over a tamely ramified extension of $\Q_p$, $Z=Z_G$ is either a torus or $Z_{G^{\mathrm{der}}}$ has rank prime to $p$, for $x\in \B(G,\Q_p)$ with associated integral model $\G=\G^\circ$, $\G^\ad=\G/\mathcal{Z}$ is connected, and it can be identified with the parahoric model of $G^\ad_{\Q_p}$ attached to $x^{\ad}$, i.e. $\G^\ad=\G^{\ad\circ}$.

\subsubsection[]{}
Let $G$ be a connected reductive group over $\Q_p$ with a parahoric model $\G$ over $\Z_p$. Let $\{\mu\}$ be the conjugacy class of a cocharacter $\mu: \mathbb{G}_{m,\ov{\Q}_p}\ra G_{\ov{\Q}_p}$. As in \ref{Iwahori Weyl}, we get the associated $\{\mu\}$-admissible set $\Admu\subset \wt{W}$. Let $K=\G(\Z_p)$. We have also the sets ${}^K\Admu, \Admu_K$ together with the surjection ${}^K\Admu\twoheadrightarrow \Admu_K$.

Let $(G^\ad,\{\mu^\ad\})$ be the associated adjoint group with the induced conjugacy class of cocharacter. Let $K^\ad=\G^{\ad\circ}(\Z_p)$, then we have the associated sets $\mathrm{Adm}(\{\mu^\ad\})$, $\Adm(\{\mu^\ad\})_{K^\ad}$ and ${}^{K^\ad}\Adm(\{\mu^\ad\})$.
\begin{lemma}\label{L: adjoint bijections}
	The natural map $(G, \{\mu\}, K)\ra (G^\ad, \{\mu^\ad\}, K^\ad)$ induces bijections
	\begin{enumerate}
		\item $\Admu\stackrel{\sim}{\ra}\mathrm{Adm}(\{\mu^\ad\})$, 
		\item $\Admu_K \stackrel{\sim}{\ra}\Adm(\{\mu^\ad\})_{K^\ad}$, 
		\item ${}^K\Admu \stackrel{\sim}{\ra} {}^{K^\ad}\Adm(\{\mu^\ad\})$.
	\end{enumerate}
\end{lemma}
\begin{proof}
	Let $\wt{W}$ and $\wt{W}^\ad$ be the Iwahori Weyl group of $G$ and $G^\ad$ respectively. Then the natural map $G\ra G^\ad$ induces a map $\wt{W}\ra \wt{W}^\ad$, which restricts to a bijection on the affine Weyl groups $W_a(G)\stackrel{\sim}{\ra}W_a(G^\ad)$. By definition, $\Admu\subset W_a(G)\tau$, $\mathrm{Adm}(\{\mu^\ad\})\subset W_a(G^\ad)\tau^\ad$, with the elements $\tau$ and $\tau^\ad$ attached to $\{\mu\}$ and $\{\mu^\ad\}$ in \ref{tau} respectively. Since $W_0(G_{\breve{\Q}_p})\simeq W_0(G_{\breve{\Q}_p}^\ad)$, one sees that  $W_a(G)\tau\stackrel{\sim}{\ra}W_a(G^\ad)\tau^\ad$ restricts to a bijection
	\[\Admu\stackrel{\sim}{\ra}\mathrm{Adm}(\{\mu^\ad\}).\] On the other hand, the natural map $K\ra K^\ad$ induces a bijection on finite Weyl groups $W_K\stackrel{\sim}{\ra}W_{K^\ad}$. Thus we get a bijection \[W_K\backslash W_K\Admu W_K/W_K \stackrel{\sim}{\ra}W_{K^\ad}\backslash W_{K^\ad} \mathrm{Adm}(\{\mu^\ad\})W_{K^\ad}/W_{K^\ad}.\] The maps $\wt{W}\ra \wt{W}^\ad$ and $W_K\stackrel{\sim}{\ra}W_{K^\ad}$
	induce a map ${}^{K} \wt{W}\ra \, {}^{K^\ad}\wt{W}$. Thus we get a map \[\Admu\cap {}^{K} \wt{W}\ra \mathrm{Adm}(\{\mu^\ad\})\cap {}^{K^\ad}\wt{W},\] which is a bijection since the two sides admit surjections to $\Admu_K$ and $\Admu_{K^\ad}$ respectively, but we have just shown $\Admu_K\simeq \Admu_{K^\ad}$, and the above map induces bijections between the fibers on both hand sides.
	We  can conclude since ${}^K\Admu=\Admu\cap {}^{K} \wt{W}$ and similarly ${}^{K^\ad}\Adm(\{\mu^\ad\})=\mathrm{Adm}(\{\mu^\ad\})\cap {}^{K^\ad}\wt{W}$ by Theorem \ref{T: EKOR set}.
\end{proof}

\subsection{The adjoint group action on KR strata}
In this and the next subsection,
let $(G,X)$ be a Shimura datum of Hodge type such that $G$ splits over a tamely ramified extension of $\Q_p$ and the center $Z=Z_G$ is a torus. 
\subsubsection[]{}\label{subsubsection adjoint local model}
Fix a Siegel  embedding \[i:(G,X)\hookrightarrow (\mathrm{GSp},S^\pm)\] as in subsection \ref{subsection integral Hodge}. Fix a point $x\in \mathcal {B}(G,\mathbb{Q}_p)$. We write $\G=\mathcal{G}_x$ for the Bruhat-Tits group scheme attached to $x$. We assume $\G=\G^\circ$.
Then $\G^\ad=\G^{\ad\circ}$ and $G_{\Z_{(p)}}^\ad=G_{\Z_{(p)}}^{\ad\circ}$ (cf. \ref{subsubsection connected hodge adjoint}).
Let $K_p:=\mathcal{G}(\mathbb{Z}_p)$, for $K^p\subset G(\mathbb{A}_f^p)$ small enough, we set $\K:=K_pK^p$. Let $K_p'\subset \GSp(\Q_p)$ be the stabilizer of the lattice $V_{\Z_p}$ as in \ref{subsection integral Hodge}. For $K^p$ above, we can find an open compact subgroup $K^{'p}\subset \GSp(\Ab_f^p)$  and set $\K'=K'_pK^{'p}$, such that
we have a morphism of schemes over $O_E$
\[\ES_\K(G,X)\ra  \ES_{\K'}(\GSp,S^\pm)_{O_E},\]which induces a closed embedding $\Sh_{\K}(G,X)\hookrightarrow \Sh_{\K'}(\GSp,S^\pm)_E$ on generic fibers.
Taking the limit over $K^p$, we get a morphism
$\ES_{K_p}(G,X)\ra  \ES_{K'_p}(\GSp,S^\pm)_{O_E}$.

By \cite[Lemma 4.5.9]{Paroh}, we have the following diagram
\[\xymatrix{&\widetilde{\ES}_{K_p}^\ad(G,X)\ar[ld]_\pi\ar[rd]^q&\\
	\ES_{K_p}(G,X) & &	\mathrm{M}_{G,X}^{\mathrm{loc}},
}\]
where $\pi$ is a $\G^\ad$-torsor, and $q$ is $\G^\ad$-equivariant. The $\G^\ad$-torsor $\widetilde{\ES}_{K_p}^\ad(G,X)$ is induced from the $\G$-torsor $\widetilde{\ES}_{K_p}(G,X)$ in Theorem \ref{result P-Z and K-P} by the natural map $\G\ra \G^\ad$.
For any sufficiently small $K^p$, the induced map $\widetilde{\ES}_{K_p}^\ad(G,X)/K^p\ra 	\mathrm{M}_{G,X}^{\mathrm{loc}}$ is smooth of relative dimension $\dim G^\ad$. 

\subsubsection[]{}\label{subsub--gamma act}
By \cite{Paroh} subsections 4.4 and 4.5, we can describe the action of $G^{\ad\circ}(\Z_{(p)})^+=G^{\ad}(\Z_{(p)})^+$ on $\ES_{K_p}(G,X)$ as follows. Let $\gamma\in G^{\ad}(\Z_{(p)})^+$ and $\P$ the fibre of $G_{\Z_{(p)}}\ra G_{\Z_{(p)}}^\ad$ over $\gamma$. Then $\P$ is a $Z_{G_{\Z_{(p)}}}$-torsor. The element $\gamma$ induces a morphism on generic fibers \[\gamma: \Sh_{K_p}(G,X)\ra \Sh_{K_p}(G,X).\] Let $T$ be a $O_E$-scheme and $x\in \ES_{K_p}(G,X)(T)$. As in \cite{Paroh} 4.5.1,
we get a triple \[(\A_x,\lambda_x, \varepsilon_x^p)\] by the morphism $\ES_{K_p}(G,X)\ra \ES_{K'_p}(\GSp,S^\pm)_{O_E}$, where $\A_x$ is an abelian scheme over $T$ up to $\Z_{(p)}$-isogeny, equipped with a weak $\Z_{(p)}$-polarization $\lambda_x$, and \[\varepsilon_x^p\in \varprojlim_{K^p}\Gamma(T, \Isom(V_{\Ab_f^p}, \wh{V}^p(\A_x)_\Q)).\] 
Then by \cite{Paroh} Lemmas 4.4.6, 4.4.8 and 4.5.4, we get another triple $(\A_x^\P,\lambda_x^\P, \varepsilon_x^{p,\P})$, and by loc. cit. Lemma 4.5.7, the assignment
\[(\A_x,\lambda_x, \varepsilon_x^p)\mapsto (\A_x^\P,\lambda_x^\P, \varepsilon_x^{p,\P}) \]induces a map \[\gamma: \ES_{K_p}(G,X)\ra \ES_{K_p}(G,X) \]whose generic fibre agrees with the map induced by conjugation by $\gamma$.

Combining the $G^{\ad}(\Z_{(p)})^+$-action with the natural action of $G(\Ab_f^p)$ on $\ES_{K_p}(G,X)$ induces an action of $\mathscr{A}(G_{\Z_{(p)}})$ (cf. \ref{subsubsection connected adjoint}) on $\ES_{K_p}(G,X)$.

Now, following \cite[Lemma 4.5.9]{Paroh}, we explain how to lift the $G^{\ad}(\Z_{(p)})^+$-action on $\ES_{K_p}(G,X)$ to an action on $\widetilde{\ES}_{K_p}^\ad(G,X)$. Fix a Galois extension $F|\Q$ such that $\P$ admits an $O_{F,(p)}=O_F\otimes \Z_{(p)}$-point $\tilde{\gamma}$.
We have an isomorphism of abelian schemes
\[\alpha_{\tilde{\gamma}}: \A_x^\P\otimes O_F\ra \A_x\otimes O_F,\]
which is $O_F$-linear for the natural $O_F$-actions on both side. Passing to the de Rham cohomology, we get an $O_F$-linear isomorphism
\[\alpha_{\tilde{\gamma}}^{-1}: H^1_{\dr}(\A_x/T)\otimes O_F\stackrel{\sim}{\ra}H^1_{\dr}(\A_x^\P/T)\otimes O_F.\]
Let $(x,f)\in \wt{\ES}_{K_p}(G,X)(T)$ be a point which lifts $x\in \ES_{K_p}(G,X)(T)$. Then \[f: V_{\Z_p}^\vee\otimes O_T\ra H^1_{\dr}(\A_x/T)\] is an isomorphism with $f^\otimes (s_\alpha)=s_{\alpha,\dr}$. The composition
\[\alpha_{\tilde{\gamma}}^{-1}\circ (f\otimes 1): V_{\Z_p}^\vee\otimes O_T\otimes O_F\stackrel{f\otimes 1}{\ra} H^1_{\dr}(\A_x/T)\otimes O_F\stackrel{\alpha_{\tilde{\gamma}}^{-1}}{\ra} H^1_{\dr}(\A_x^\P/T)\otimes O_F\]
induces a well defined element in $\wt{\ES}^\ad_{K_p}(G,X)(T)$ (cf. the proof of \cite[Lemma 4.5.9]{Paroh}),
which depends only on the image of $(x,f)$ in $\wt{\ES}^\ad_{K_p}(G,X)(T)$ and on $\gamma$.

\begin{proposition}\label{P: adjoint action on KR}
	Let $\ES_0$ be the special fiber of $\ES_{K_p}(G,X)$ over $k=\ov{\mathbb{F}}_p$ and $K=K_p$.
	Consider the KR stratification $\ES_{0}=\coprod_{w\in \Admu_K} \ES_0^w$. Then each KR stratum $\ES_0^w$ is stable under the action of $G^{\ad}(\Z_{(p)})^+$ on $\ES_0$.
\end{proposition}
\begin{proof}
	Let $\widetilde{\ES}^\ad_0$ and $\mathrm{M}_0$ be the special fibers of $\widetilde{\ES}_{K_p}^\ad(G,X)$ and $\mathrm{M}_{G,X}^{\mathrm{loc}}$ respectively.
	Consider the following diagram
	\[\xymatrix{&\widetilde{\ES}^\ad_0\ar[ld]_\pi\ar[rd]^q&\\
		\ES_0 & &	\mathrm{M}_{0}
	}\] induced from that in \ref{subsubsection adjoint local model}
	on the special fibers. By the above, the $G^{\ad}(\Z_{(p)})^+$-action on $\ES_0$ can be lifted to an action on $\widetilde{\ES}^\ad_0$. Let $G^{\ad}(\Z_{(p)})^+$ act trivially on $	\mathrm{M}_{0}$. Then $\pi$ and $q$ in the above diagram are $G^{\ad}(\Z_{(p)})^+$ -equivariant. Therefore each KR stratum $\ES_0^w$ is stable under the action of $G^{\ad}(\Z_{(p)})^+$ on $\ES_0$.
\end{proof}

\subsection{The adjoint group action on Zips}
As before, let $\ES_0$ be the special fiber of $\ES_{K_p}(G,X)$ over $k$.
We change our notations slightly and write $\tilde{\I}=\widetilde{\ES}_0$ and $\tilde{\I}^\ad=\widetilde{\ES}^\ad_0$. 
Consider the following diagram
\[\xymatrix{
	&\widetilde{\I}^\ad\ar[ld]\ar[rd]&\\
	\ES_0 & &	\mathrm{M}_{0}
}\]introduced in \ref{subsection conjugate loc mod} and in the proof of Proposition \ref{P: adjoint action on KR}. As $\G_0^\ad$-torsors over $\ES_0$, we have $\tilde{\I}^\ad=\wt{\I}\times^{\G_0}\G_0/Z_k$.

We write $K=K_p$.
Fix an element \[w\in \Admu_K=W_K\backslash W_K\Admu W_K/W_K=[\G_0\backslash M_0](k),\] and consider the associated KR stratum $\ES_0^w\subset \ES_0$.
Let \[\tilde{\I}^w\ra \ES_0^w, \quad \mathrm{and}\quad \tilde{\I}^{\ad,w}\ra \ES_0^w\]  be the pullbacks of $\tilde{\I}\ra \ES_0$ and $\tilde{\I}^\ad\ra \ES_0$ respectively under the inclusion of the KR stratum $\ES_0^w\subset \ES_0$. As in \ref{subsection EO in KR}, we get the map
\[\tilde{\I}_+^w\ra \ES_0^w,\] which is a $\G_{0,w}$-torsor. Set $\G_{0,w}^\ad=\G_{0,w}/Z_k$ and \[\tilde{\I}^{\ad,w}_+=\wt{\I}^w_+\times^{\G_{0,w}}\G_{0,w}^\ad.\] Then $\tilde{\I}_+^{\ad,w}\ra \ES_0^w$ is a $\G_{0,w}^\ad$-torsor .

By Proposition \ref{P: adjoint action on KR} we have a $G^{\ad}(\Z_{(p)})^+$-action on $ \ES_0^w$. We get an induced $G^{\ad}(\Z_{(p)})^+$-action on $\tilde{\I}^{\ad,w}$, which is by definition $f\mapsto \alpha_{\tilde{\gamma}}^{-1}\circ (f\otimes 1)$, for $f\in \tilde{\I}^{\ad,w}$.

Similarly, we can
consider the conjugate local model diagram on special fibers (cf. \ref{subsection conjugate loc mod})
\[\xymatrix{ 
	&\widetilde{\I}^\ad\ar[ld]\ar[rd]&\\
	\ES_0 & &	\mathrm{M}_{0}^{c},
}\]
which is $G^{\ad}(\Z_{(p)})^+$-equivariant (where $G^{\ad}(\Z_{(p)})^+$ acts trivially on $\mathrm{M}_{0}^{c}$). Similarly as above, we get \[\tilde{\I}_{-}^w\ra \ES_0^w,\quad \mathrm{and}\quad \tilde{\I}_{-}^{\ad,w}\ra \ES_0^w.\]

Consider the universal abelian scheme $\A\ra \ES_0^w$ and its de Rham cohomology $H^1_{\dr}(\A/\ES_0^w)$. The $F$-$V$-module structure on $H^1_{\dr}(\A/\ES_0^w)$ induces an isomorphism
\[\wt{\iota}: \wt{\I}_+^{w,(p)}/\G_{0,w}^{U,(p)}\ra \wt{\I}_{-}^w/\G_{0,\sigma(w)^{-1}}^U, \]which is equivariant with respective to the isomorphism \[\G_{0,w}/\G_{0,w}^{U,(p)} \stackrel{\sim}{\lra} \G_{0,\sigma(w)^{-1}}/\G_{0,\sigma(w)^{-1}}^U,\]where $\G_{0,w}^{U}$ and $\G_{0,\sigma(w)^{-1}}^U$ are the groups defined in Proposition \ref{prop--lift G-zip at point} (3). Passing to $\G_0^\rdt$ and $\G_0^{\rdt,\ad}$ we get a $\G_0^\rdt$-zip of type $J_w$ over $\ES_0^w$ (cf. Corollary \ref{corollary--G-zip on KR}): \[(\I^w,\I^w_+,\I^w_{-}, \iota),\] and a $\G_0^{\rdt,\ad}$-zip of type $J_w$ over $\ES_0^w$: \[(\I^{\ad,w},\I^{\ad,w}_+,\I^{\ad,w}_{-}, \iota).\] In particular,
as in the proof of Theorem \ref{sm of zeta}, we get a diagram
\[\xymatrix{ \mathbb{E}^{\ad,w}\ar[r]\ar[d]& \G_0^{\rdt,\ad}\ar[d] \\
	\ES_0^w \ar[r]& [E_{\mathcal{Z}_w^\ad}\backslash \G_0^{\rdt,\ad}],
}\]
with $\mathbb{E}^{\ad,w}\ra \ES_0^w$ an $E_{\mathcal{Z}_w^\ad}$-torsor, where $\mathcal{Z}_w^\ad$ is the algebraic zip datum induced from $\mathcal{Z}_w$ by $\G_0^\rdt\ra \G_0^{\rdt,\ad}$. By Proposition \ref{P: adjoint action on KR}, $G^{\ad}(\Z_{(p)})^+$ acts on $\ES_0^w$.
\begin{proposition}\label{P: adjoint action on zip cover}
	The $G^{\ad}(\Z_{(p)})^+$-action on $\ES_0^w$ lifts to an action on $\mathbb{E}^{\ad,w}$. Moreover, the above diagram is $G^{\ad}(\Z_{(p)})^+$-equivariant, where $G^{\ad}(\Z_{(p)})^+$ acts trivially on $\G_0^{\rdt,\ad}$.
\end{proposition}
\begin{proof}
	Notations as in Lemma \ref{lemma--global lift zip}, we work with the tuple $(\widetilde{\I}^w,\widetilde{\I}_{+}^{w},\widetilde{\I}_{-}^w, \widetilde{\iota})$. Let $\widetilde{\mathbb{E}}^w$ be formed by the following cartesian diagram.
	\[\xymatrix{\widetilde{\mathbb{E}}^w\ar[d]\ar[rrr] & & & \widetilde{\I}_{-}^w\ar[d]\\
		\widetilde{\I}_{+}^{w}\ar[r]&\widetilde{\I}_{+}^{w,(p)}\ar[r]&\widetilde{\I}_{+}^{w,(p)}/\mathcal{G}_{0,w}^{U,(p)}\ar[r]^{\widetilde{\iota}}&\widetilde{\I}_{-}^w/\mathcal{G}_{0,\sigma(w)^{-1}}^U}\]
	There is a morphism $\widetilde{\mathbb{E}}^w\rightarrow \G_0$ induced by $(f_+,f_-)\mapsto g$ where $g$ is the unique element in $\G_0$ mapping $f_+$ to $f_-$. Passing to quotients by $Z_k$, we have $(\widetilde{\I}^{\ad,w},\widetilde{\I}_{+}^{\ad,w},\widetilde{\I}_{-}^{\ad,w},\widetilde{\iota}^\ad)$, and hence $\widetilde{\mathbb{E}}^{\ad,w}$ with a morphism to $\G_0^\ad:=\G_0/Z_k$.
	
	It suffices to show that the $G^{\ad}(\Z_{(p)})^+$-action on $\ES_0^w$ lifts to an action on $\widetilde{\mathbb{E}}^{\ad,w}$, and the above morphism is $G^{\ad}(\Z_{(p)})^+$-equivariant, where $G^{\ad}(\Z_{(p)})^+$ acts trivially on $\G_0^{\ad}$.
	
	For $\gamma\in G^{\ad}(\Z_{(p)})^+$, $T$ a scheme over $\ES_0^w$ and $(x,f)\in \widetilde{\I}^w(T)$ with $x\in \ES^w_0(T)$, notations as in \ref{subsub--gamma act}, we have an isomorphism of abelian schemes
	\[\alpha_{\tilde{\gamma}}: \A_x^\P\otimes O_F\ra \A_x\otimes O_F,\]which is $O_F$-linear for the natural $O_F$-actions on both sides. Passing to the de Rham cohomology, we get an $O_F$-linear isomorphism respecting Hodge-Tate tensors
	\[\alpha_{\tilde{\gamma}}^{-1}: H^1_{\dr}(\A_x/T)\otimes O_F\stackrel{\sim}{\ra}H^1_{\dr}(\A_x^\P/T)\otimes O_F.\]
	Moreover, $\big(\gamma(x), \alpha_{\tilde{\gamma}}^{-1}\circ (f\otimes 1)\big)\in \widetilde{\I}^w(T\otimes O_F)$ descends to an element in $\widetilde{\I}^{\ad,w}(T)$ which depends only on the image of $(x,f)$ in $\widetilde{\I}^{\ad,w}(T)$ and on $\gamma$.
	
	The isomorphism $\alpha_{\tilde{\gamma}}^{-1}$ is induced by an isomorphism of abelian schemes, and hence commutes with Frobenius and Verschiebung. So, if $(x,f)\in \widetilde{\I}^{w}_+(T)$, then \[\big(\gamma(x), \alpha_{\tilde{\gamma}}^{-1}\circ (f\otimes 1)\big)\in \widetilde{\I}^w(T\otimes O_F)\] descends to an element in $\widetilde{\I}^{\ad,w}_+(T)$. The same holds if we change $+$ to $-$. Moreover, as $\widetilde{\iota}$ is induced by Frobenius and Verschiebung, we have a commutative diagram
	\[\xymatrix{\widetilde{\I}_{+}^{w,\ad,(p)}/\mathcal{G}_{0,w}^{U,(p)}\ar[r]^{\widetilde{\iota}^\ad}\ar[d]^{\alpha_{\tilde{\gamma}}^{-1}}& \widetilde{\I}_{-}^{\ad,w}/\mathcal{G}_{0,\sigma(w)^{-1}}^U\ar[d]^{\alpha_{\tilde{\gamma}}^{-1}}\\
		\widetilde{\I}_{+}^{w,\ad,(p)}/\mathcal{G}_{0,w}^{U,(p)}\ar[r]^{\widetilde{\iota}^\ad}& \widetilde{\I}_{-}^{\ad,w}/\mathcal{G}_{0,\sigma(w)^{-1}}^U,}\]
	and hence a $G^{\ad}(\Z_{(p)})^+$-action on $\widetilde{\mathbb{E}}^{\ad,w}$. Noting that $$\widetilde{\mathbb{E}}^{w}(T)\ni(x,f_+,f_-)\mapsto (\gamma(x),\ \alpha_{\tilde{\gamma}}^{-1}\circ (f_+\otimes 1),\  \alpha_{\tilde{\gamma}}^{-1}\circ (f_-\otimes 1))$$
	via $\gamma$, 
	they have the same image in $\G_0^\ad$, as it is the unique $g\in \G_0^\ad(T)$ such that $f_+\circ g=f_-\circ g$ in $\widetilde{\I}^{\ad}$.
	
\end{proof}

\subsection{EKOR strata of abelian type}
We return to the setting of subsection \ref{subsection integral abelian}. Let $(G,X)$ be a Shimura datum of abelian type such that
\begin{itemize}
	\item either
	$(G^\ad, X^\ad)$ has no factors of type $D^\mathbb{H}$,
	\item or $G$ is unramified over $\Q_p$ and $K_p$ is contained in some hyperspecial subgroup of $G(\Q_p)$.
\end{itemize}
We take an associated Hodge type datum $(G_1,X_1)$ as in Theorem \ref{T: KP abelian} (3) or (4) according to the above cases. 
We will apply the constructions in the last two subsections to $(G_1,X_1)$. 

\subsubsection[]{}
Let $x\in \B(G,\Q_p)$ and $x_1\in \B(G_1,\Q_p)$ such that $x^{\ad}=x_1^\ad\in \B(G^\ad,\Q_p)$. We denote the model of $G$ (resp. $G_1$) defined as the stabilizer of $x$ (resp. $x_1$) by $\G$ (resp. $\G_1$), with connected model $\G^\circ$ (resp. $\G_1^\circ$). As in the proof of \cite[Theorem 4.6.23]{Paroh}, we can and we do choose $(G_1,X_1)$ and $x_1$ such that $Z_{G_1}$ is a torus and $\G_1=\G_1^\circ$.
We have group schemes $G_{\Z_{(p)}}$, $G_{\Z_{(p)}}^\circ$  and $G_{1,\Z_{(p)}}$ over $\Z_{(p)}$ corresponding to $\G$, $\G^\circ$ and $\G_1$ respectively.
Write $K_p=\G(\Z_p), K_p^\circ=\G^\circ(\Z_p)$ and $K_{1,p}=\G_1(\Z_p)$.
By the discussion in \ref{subsubsection connected hodge adjoint},
we have \[G^{\ad\circ}_{\Z_{(p)}}=G^\ad_{1,\Z_{(p)}} \] as  group schemes over $\Z_{(p)}$. In particular, we have $G^{\ad\circ}(\Z_{(p)})^+=G_1^\ad(\Z_{(p)})^+$, and \[\G_0^{\circ,\rdt,\ad}=\G_0^{\ad\circ,\rdt}=\G_{1,0}^{\ad,\rdt}=\G_{1,0}^{\rdt,\ad}\]
as reductive adjoint groups over $k$.

Let $\ES_{K_p^\circ,0}$ (resp. $\ES_{K_{1,p},0}$ ) be the special fiber of $\ES_{K^\circ_p}(G,X)$ (resp. $\ES_{K_{1,p}}(G_1,X_1)$) over $k=\ov{\mathbb{F}}_p$.
Recall that we have (cf. \ref{subsubsection Hodge to abelian})	\[\ES_{K^\circ_p,0}=\Big[[\ES_{K_{1,p},0}^+\times\mathscr{A}(G_{\Z_{(p)}})]/\mathscr{A}(G_{1,\Z_{(p)}})^\circ \Big]^{|J|}, \] 
and the following diagram (cf. Theorem \ref{T: KP abelian} (3))
\[\xymatrix{
	&\wt{\ES}^\ad_{K^\circ_p,0}\ar[ld]_\pi\ar[rd]^q&\\
	\ES_{K^\circ_p,0}& &\mathrm{M}^{\mathrm{loc}}_{G_1,X_1,0},
}\]
where $\pi$ is a $\G_0^{\ad\circ}=\G_{1,0}^\ad$-torsor and $q$ is $\G_0^{\ad\circ}$-equivariant.

\begin{corollary}\label{C: KR abelian type}
	We have the KR stratification $\ES_{K^\circ_p,0}=\coprod_{w\in \Admu_K}\ES_{K^\circ_p,0}^w$, such that for each $w$, the stratum $\ES_{K^\circ_p,0}^w$ is non empty, smooth, equi-dimensional of dimension $\dim \ES_{K^\circ_p,0}^w=\ell({}^Kw_K)$.
\end{corollary}
\begin{proof}
	Since 	$\mathrm{M}^{\mathrm{loc}}_{G_1,X_1,0}$ only depends on the associated adjoint datum, and the $\G_{1,0}^\ad$-orbits on $\mathrm{M}^{\mathrm{loc}}_{G_1,X_1,0}$ are parametrized by $\Adm(\{\mu_1^\ad\})_{K_{1,p}^\ad}=\Admu_K$,
	the above diagram gives a KR stratification 
	$\ES_{K^\circ_p,0}=\coprod_{w\in \Admu_K}\ES_{K^\circ_p,0}^w$, with each stratum $\ES_{K^\circ_p,0}^w$ smooth, of equi-dimension with $\dim \ES_{K^\circ_p,0}^w=\ell({}^Kw_K)$.
	
	Moreover, for $w\in \Admu_K=\Adm(\{\mu_1^\ad\})_{K_{1,p}^\ad}$,
	we can make the link between the KR strata $\ES_{K^\circ_p,0}^w$ and $\ES_{K_{1,p},0}^w$. Let $\ES_{K_{1,p},0}^{w,+}\subset \ES_{K_{1,p},0}^{+}$ be the pullback of $\ES_{K_{1,p},0}^{w}$ to the connected component $\ES_{K_{1,p},0}^{+}\subset \ES_{K_{1,p},0}$. By Proposition \ref{P: adjoint action on KR}, the $G_1^\ad(\Z_{(p)})^+$-action on $\ES_{K_{1,p},0}^{+}$  stabilizes $\ES_{K_{1,p},0}^{w,+}$, we get an extended action of $\mathscr{A}(G_{1,\Z_{(p)}})^\circ$ on it. The construction in \ref{subsubsection Hodge to abelian} gives us
	\[\ES_{K^\circ_p,0}^w=\Big[[\ES_{K_{1,p},0}^{w,+}\times\mathscr{A}(G_{\Z_{(p)}})]/\mathscr{A}(G_{1,\Z_{(p)}})^\circ \Big]^{|J|}. \] 
	We get the non emptiness since each
	$\ES_{K_{1,p},0}^{w,+}$ is non empty.
\end{proof}

\subsubsection[]{}
Consider the diagram
\[\xymatrix{ \mathbb{E}^{\ad,w}_{K_{1,p}}\ar[r]\ar[d]& \G_{1,0}^{\rdt,\ad} \\
	\ES_{K_{1,p},0}^w &
}\]
for the KR stratum $\ES_{K_{1,p},0}^w$ as in the paragraph above Proposition \ref{P: adjoint action on zip cover} (see also
the proof of Theorem \ref{sm of zeta}). Here $\mathbb{E}^{\ad,w}_{K_{1,p}}\ra \ES_{K_{1,p},0}^w$ is an $E_{\mathcal{Z}_{1,w}^\ad}$-torsor, with $\mathcal{Z}_{1,w}^\ad$ the algebraic zip datum induced from $\mathcal{Z}_{1,w}$ by $\G_{1,0}^\rdt\ra \G_{1,0}^{\rdt,\ad}$.
We get \[\mathbb{E}^{\ad,w,+}_{K_{1,p}}\ra \ES_{K_{1,p},0}^{w,+}\] by pulling back $\mathbb{E}^{\ad,w}_{K_{1,p}}\ra \ES_{K_{1,p},0}^w$ along the inclusion $\ES_{K_{1,p},0}^{w,+}\subset \ES_{K_{1,p},0}^w$. Since the $G_1^\ad(\Z_{(p)})^+$-action stabilizes $\ES_{K_{1,p},0}^{w,+}$, by Proposition \ref{P: adjoint action on zip cover}, we have a lift of this action to $\mathbb{E}^{\ad,w,+}_{K_{1,p}}$.
Then we get an induced action of $\mathscr{A}(G_{1,\Z_{(p)}})^\circ$ on $\mathbb{E}^{\ad,w,+}_{K_{1,p}}$.
Set
\[ \mathbb{E}^{\ad,w}_{K_{p}^\circ}=\Big[[\mathbb{E}^{\ad,w,+}_{K_{1,p}}\times\mathscr{A}(G_{\Z_{(p)}})]/\mathscr{A}(G_{1,\Z_{(p)}})^\circ \Big]^{|J|}.\]  By (the proof of) the above Corollary \ref{C: KR abelian type}, we get a diagram
\[\xymatrix{ \mathbb{E}^{\ad,w}_{K_p^\circ}\ar[r]\ar[d]& \G_{1,0}^{\rdt,\ad} \\
	\ES_{K_p^\circ,0}^w. &
}\]
Here $\mathbb{E}^{\ad,w}_{K_{p}^\circ}\ra \ES_{K_{p}^\circ,0}^w$ is an $E_{\mathcal{Z}_{w}^\ad}$-torsor, with $\mathcal{Z}_{w}^\ad$ the algebraic zip datum induced from $\mathcal{Z}_{w}$ by $G_1\ra G_1^{\ad}$. We have $\G_{0}^{\circ,\rdt,\ad}=\G_{1,0}^{\rdt,\ad}$, $\mathcal{Z}_{w}^\ad=\mathcal{Z}_{1,w}^\ad$ and thus $E_{\mathcal{Z}_{w}^\ad}=E_{\mathcal{Z}_{1,w}^\ad}$.
In particular, as in \ref{subsection EO in KR} we have a morphism of stacks
\[\zeta_w:\ES_{K_p^\circ,0}^w\rightarrow [E_{\mathcal {Z}_w^\ad}\backslash\mathcal{G}_{1,0}^{\mathrm{rdt},\ad}],\] 
from which we
can define an EO stratification to get the EKOR strata in $\ES_{K_p^\circ,0}^w$. Letting $w\in\Admu_K$ vary, we get the EKOR stratification on $\ES_{K^\circ_p,0}$.

\subsubsection[]{}
By Lemma \ref{L: adjoint bijections} we can identify the sets $\Admu_K=\Adm(\{\mu_1^\ad\})_{K_{1,p}^\ad}$ and ${}^K\Admu=$ $ {}^{K_{1,p}^\circ}\Adm(\{\mu_1\})$. For $w\in \Admu_K$,
we have the following diagram
\[\xymatrix{
	\ES_{K_{1,p},0}^{w,+} \ar[d]\ar[r]& \ES_{K_{1,p},0}^w\ar[r]& [E_{\mathcal{Z}_{1,w}}\backslash \G_{1,0}^{\rdt}]\ar[d]\\
	\ES_{K_{p}^\circ,0}^{w,+}\ar[r]& \ES_{K_{p}^\circ,0}^w\ar[r]& [E_{\mathcal{Z}_{w}^\ad}\backslash \G_{1,0}^{\rdt,\ad}],
}\]
where $\ES_{K_{1,p},0}^{w,+}\ra \ES_{K_{p}^\circ,0}^{w,+}$ is a pro-\'etale cover and \[\ES_{K_{p}^\circ,0}^{w,+}=\ES_{K_{1,p},0}^{w,+}/\Delta,\] with \[\Delta=\ker (\mathscr{A}(G_{1,\Z_{(p)}})^\circ\ra \mathscr{A}(G_{\Z_{(p)}})^\circ).\] The map $[E_{\mathcal{Z}_{1,w}}\backslash \G_{1,0}^{\rdt}]\ra [E_{\mathcal{Z}_{w}^\ad}\backslash \G_{1,0}^{\rdt,\ad}]$ is a homeomorphism. 
For $x\in {}^K\Admu$, we have also a pro-\'etale cover \[\ES_{K_{1,p},0}^{x,+}\ra \ES_{K_{p}^\circ,0}^{x,+},\quad \ES_{K_{p}^\circ,0}^{x,+}=\ES_{K_{1,p},0}^{x,+}/\Delta,\] and 
\[\ES_{K^\circ_p,0}^x=\Big[[\ES_{K_{1,p},0}^{x,+}\times\mathscr{A}(G_{\Z_{(p)}})]/\mathscr{A}(G_{1,\Z_{(p)}})^\circ \Big]^{|J|}. \] 
Therefore, we can deduce results for EKOR strata of abelian type from those in the Hodge type case. We summarize the results as follows.
\begin{theorem}\label{T: EKOR abelian}
	Let $(G,X)$ be a Shimura datum of abelian type such that $(G^\ad, X^\ad)$ has no factors of type $D^\mathbb{H}$. Let $\K=K_p^\circ K^p$ and $\ES_0=\ES_{\K,0}$.
	\begin{enumerate}
		\item We have the EKOR stratification
		\[\ES_0=\coprod_{x\in {}^K\Admu}\ES_0^x,\]
		where for each $x\in {}^K\Admu$, the stratum $\ES_0^x$ is a non-empty, locally closed smooth subscheme of $\ES_0$, which is equi-dimensional of dimension $\ell(x)$.
		Moreover, we have the closure relation
		\[\ov{\ES_0^x}=\coprod_{x'\leq_{K,\sigma}x}\ES_0^{x'}.\]
		
		\item Every KR stratum in $\ES_{I,0}$ is quasi-affine. If in addition axiom 4 (c) is satisfied, every EKOR in $\ES_{K,0}$ is quasi-affine.
		\item For $x\in {}^K\Admu$ viewed as an element of $\Admu$, the morphism \[\pi_{I,K}^x:\ES_{I,0}^x\rightarrow \ES_{K,0}^x\] induced by $\pi_{I,K}$ is finite \'{e}tale. If in addition axiom 4 (c) is satisfied, $\pi_{I,K}^x$ is a finite \'{e}tale covering.
		
	\end{enumerate}
\end{theorem}
We believe that our global constructions in section \ref{section global} can be generalized to the abelian type case, except we should work with the associated $\G^{\ad\circ}$-torsors instead, as \cite{Paroh} subsection 4.6. In case of good reductions, see \cite{Loverig 2} and \cite{stra abelian-good redu} for some related constructions. We leave the general case to the interested readers.

\subsection{Relations with central leaves and Newton strata}
We continue the assumptions and notations of the last subsection.
\subsubsection[]{}
Recall that in subsection \ref{subsection stratifications hodge} we have defined \[\Upsilon_{K_{1,p}}: \ES_{K_{1,p}}(k)\rightarrow C(\G_1,\{\mu_1\})\] and \[\delta_{K_{1,p}}: \ES_{K_{1,p}}(k)\rightarrow B(G_1,\{\mu_1\}).\]
The fibers of $\Upsilon_{K_{1,p}}$ and $\delta_{K_{1,p}}$ are called central leaves and Newton strata respectively, both of which are the sets of $k$-valued points of some locally closed reduced subschemes of $\ES_{K_{1,p},0}$.

By the method of \cite{stra abelian-good redu} sections 3 and 4, we can define central leaves and Newton strata for the abelian type case $\ES_{K_{p}^\circ,0}$. More precisely, consider the composition of $\Upsilon_{K_{1,p}}$ with the natural map \[C(\G_1,\{\mu_1\})\ra C(\G^\ad_1,\{\mu^\ad_1\}).\] We call the fibers of \[\ES_{K_{1,p}}(k)\rightarrow  C(\G^\ad_1,\{\mu^\ad_1\})\] \emph{adjoint central leaves}, which are finite (set theoretically) disjoint unions of the  central leaves defined above, by the following lemma. Therefore, we may consider them as locally closed reduced subschemes of $\ES_{K_{1,p},0}$, which we will also call adjoint central leaves.
\begin{lemma}
	We write simply $G=G_1$, $\G=\G_1$ and $\G^\ad=\G^\ad_1$.
	The natural map $C(\G)\ra C(\G^\ad)$ is finite to one. If $Z_{\Z_{p}}$, the closure of $Z_G$ in $\G$, has connected fibers, then it is bijective.
\end{lemma}
\begin{proof}
	Recall that we always assume that $Z_{\Z_{(p)}}$ is smooth.
	Let $h\in G(\breve{\Q}_p)$ with image $h^\ad\in G^\ad(\breve{\Q}_p)$. Denote by $[h]$ and $[h^\ad]$ the associated classes in $C(\G)$ and $C(\G^\ad)$. 
	Up to $\sigma\text{-}\breve{K}$-conjugacy, we can assume that $h=gw$ with $g\in \breve{K}$ and $w\in \wt{W}$. Then $h^\ad=g^\ad w^\ad$. The preimage of $[h^\ad]$ under the map $C(\G)\ra C(\G^\ad)$  is the set \[\{[zh]|\,z\in Z_{\Z_{p}}(\breve{\Z}_p)\}.\] If $Z_{\Z_{p}}$ has connected special fiber, for any $m\geq 1$, consider the level-$m$ Greenberg transformation $Z_m$ of $Z_{\Z_{p}}$, which is a connected smooth scheme over $\mathbb{F}_p$. Noting that $Z_m$ is connected, the morphism \[\phi_m: Z_m\ra Z_m,\quad x\mapsto x^{-1}\sigma(x)\] is a finite \'etale cover.
	Let $\phi_m': Z_m'\ra Z_{m+1}$ be the pullback of $\phi_m: Z_m\ra Z_m$ under the natural projection $Z_{m+1}\ra Z_m$. Since the following diagram
	\[\xymatrix{Z_{m+1}\ar[d]_{\phi_{m+1}}\ar[r]& Z_m\ar[d]^{\phi_m}\\
		Z_{m+1}\ar[r]& Z_m
	}\]
	is commutative, we get a morphism $\phi_{m+1,m}: Z_{m+1}\ra Z_m'$ such that $\phi_{m+1}=\phi_m'\circ\phi_{m+1,m}$. The morphisms $\phi_{m+1}$ and $\phi_m'$ are both finite \'etale covers, thus so is $\phi_{m+1,m}$. We deduce that for any $z\in Z_{\Z_{p}}(\breve{\Z}_p)$, there exists $x\in Z_{\Z_{p}}(\breve{\Z}_p)$ such that $z=x^{-1}\sigma(x)$. Therefore the set $\{[zh]|\,z\in Z_{\Z_{p}}(\breve{\Z}_p)\}$ consists of one class. On the other hand, by Steinberg's theorem $G(\breve{\Q}_p)\ra G^\ad(\breve{\Q}_p)$ is surjective if $Z_{\Z_{p}}$ has connected generic fiber, and $\breve{K}\ra \breve{K}^\ad$ is surjective since $Z_{\Z_{p}}$ is smooth.
	We conclude that if $Z_{\Z_{p}}$ has connected fibers,  $C(\G)\ra C(\G^\ad)$ is bijective.
	
	In general, let $Z_{\Z_{p}}^\circ\subset Z_{\Z_{p}}$ be the neutral connected component. Then $Z_{\Z_{p}}^\circ(\breve{\Z}_p)$ is of finite index in $Z_{\Z_{p}}(\breve{\Z}_p)$. From the above paragraph we deduce that $C(\G)\ra C(\G^\ad)$ is finite to one.
\end{proof}

\subsubsection[]{Adjoint central leaves and Newton strata of abelian type}
Consider the map \[\ES_{K_{1,p}}(k)\rightarrow  C(\G^\ad_1,\{\mu^\ad_1\})\]as above.
For each $c\in C(\G^\ad_1,\{\mu^\ad_1\})$, let $\ES_{K_{1,p},0}^{c}$ be the fiber of the above map at $c$, considered as a subscheme of $\ES_{K_{1,p},0}$.
Consider the connected component $\ES_{K_{1,p},0}^+\subset \ES_{K_{1,p},0}$.
Let \[\ES_{K_{1,p},0}^{c,+}\subset \ES_{K_{1,p},0}^{c}\] be the pullback of the adjoint leaf $\ES_{K_{1,p},0}^{c}$  under the inclusion $\ES_{K_{1,p},0}^+\subset \ES_{K_{1,p},0}$.
\begin{proposition}\label{P: adjoint action on leaves}
	The $G^{\ad}_1(\Z_{(p)})^+$-action on $\ES_{K_{1,p},0}^{+}$ stabilizes 
	$\ES_{K_{1,p},0}^{c,+}$. 
\end{proposition}
\begin{proof}
	For $x\in \ES_{K_{1,p},0}(k)$, we denote by $D_x$ the Dieudonn\'{e} module of the $p$-divisible group at $x$, and $I_x$ the \emph{set} of trivializations respecting the crystalline tensors. In particular, $I_x$ is a $\breve{K}_1$-torsor. Let $\xi:V^\vee_{\breve{\INT}_p}\rightarrow V^{\vee,\sigma}_{\breve{\INT}_p}$ be the the isomorphism given by $v\otimes k\mapsto v\otimes 1\otimes k$, then the assignment $$t\mapsto t^*(\varphi_x^\mathrm{lin}):=(V^\vee_{\breve{\INT}_p}\stackrel{\xi}{\longrightarrow} V^{\vee,(\sigma)}_{\breve{\INT}_p}\stackrel{t^{(\sigma)}}{\longrightarrow} D_x^{(\sigma)}\stackrel{\varphi_x^\mathrm{lin}}{\longrightarrow}D_x\stackrel{t^{-1}}{\longrightarrow}V^\vee_{\breve{\INT}_p})$$
	induces a $\breve{K}_1$-equivariant map $I_x\rightarrow \breve{G}_1$. Here $\varphi_x^\mathrm{lin}$ is the linearization of $\varphi_x$, and the $\breve{K}_1$-action on $\breve{G}_1$ is via $\sigma$-conjugacy.
	
	When $x$ varies, we have a diagram \[\xymatrix{\mathcal{S}:=\coprod\limits_{x\in \ES_{K_{1,p},0}(k)} I_x\ar[r]^(0.7){\widetilde{\Upsilon}}\ar[d] &\breve{G}_1\\
		\ES_{K_{1,p},0}(k) &
	}\]
	which, after passing to adjoint, induces
	\[\xymatrix{\mathcal{S}^\ad:=\coprod\limits_{x\in \ES_{K_{1,p},0}(k)} I_x^\ad\ar[r]^(0.7){\widetilde{\Upsilon}^\ad}\ar[d] &\breve{G}^\ad_1\\
		\ES_{K_{1,p},0}(k). &
	}\]
	
	We claim that the $G^{\ad}_1(\Z_{(p)})^+$-action on $\ES_{K_{1,p},0}(k)$ lifts to $\mathcal{S}^\ad$, and $\widetilde{\Upsilon}^\ad$ is $G^{\ad}_1(\Z_{(p)})^+$-equivariant. Here $\breve{G}^\ad_1$ is equipped with the trivial $G^{\ad}_1(\Z_{(p)})^+$-action.
	
	Notations as in \ref{subsub--gamma act}, we have an isomorphism of abelian schemes
	\[\alpha_{\tilde{\gamma}}: \A_x^\P\otimes O_F\ra \A_x\otimes O_F,\]which is $O_F$-linear for the natural $O_F$-actions on both sides. Passing to Dieudonn\'{e} modules, we get an $O_F$-linear isomorphism respecting Hodge-Tate tensors
	\[\alpha_{\tilde{\gamma}}^{-1}: D_x\otimes O_F\stackrel{\sim}{\ra}D_{\gamma(x)}\otimes O_F.\]
	
	Noting that $\alpha_{\tilde{\gamma}}^{-1}$ commutes with Frobenius, we have a commutative diagram
	\[\xymatrix{V^\vee_{\breve{\INT}_p}\otimes O_F\ar[r]^{\xi} & V^{\vee,(\sigma)}_{\breve{\INT}_p}\otimes O_F\ar[r]^{t^{(\sigma)}} & D_x^{(\sigma)}\otimes O_F\ar[r]^{\varphi_x^{\mathrm{lin}}}\ar[d]^{\alpha_{\tilde{\gamma}}^{-1}} & D_x\otimes O_F\ar[r]^{t^{-1}}\ar[d]^{\alpha_{\tilde{\gamma}}^{-1}} & V^\vee_{\breve{\INT}_p}\otimes O_F\\
		& & D_x^{(\sigma)}\otimes O_F\ar[r]^{\varphi_{\gamma(x)}^{\mathrm{lin}}} & D_x\otimes O_F.
	}\]
	The images in $\breve{G}^\ad_1$ of $t$ and $\gamma(t)$ are precisely the images of the two compositions of isomorphisms from left to right, and hence coincide. This proves the claim, and the proposition follows formally as before.
\end{proof}
Therefore, we can extend the $G^{\ad}_1(\Z_{(p)})^+$-action to an action of $\mathscr{A}(G_{1,\Z_{(p)}})^\circ$ on $\ES_{K_{1,p},0}^{c,+}$. 
Recall that $\G^{\ad\circ}=\G_1^\ad$, $\{\mu^\ad\}=\{\mu_1^\ad\}$ and thus $C(\G^{\ad\circ},\{\mu^\ad\})=C(\G^\ad_1,\{\mu^\ad_1\})$. For any $c\in C(\G^{\ad\circ},\{\mu^\ad\})$,
we
define the associated adjoint central leaf of abelian type
\[ \ES_{K^\circ_p,0}^c=\Big[[\ES_{K_{1,p},0}^{c,+}\times\mathscr{A}(G_{\Z_{(p)}})]/\mathscr{A}(G_{1,\Z_{(p)}})^\circ \Big]^{|J|}, \]
with \[\ES_{K_{p}^\circ,0}^{c,+}=\ES_{K_{1,p},0}^{c,+}/\Delta,\]where as before $\Delta=\ker (\mathscr{A}(G_{1,\Z_{(p)}})^\circ\ra \mathscr{A}(G_{\Z_{(p)}})^\circ).$ 
In particular, we have a map 
\[\ES_{K^\circ_p}(k)\ra C(\G^{\ad\circ},\{\mu^\ad\}),\]which is surjective.

Since $B(G,\{\mu\}) \simeq B(G^\ad, \{\mu^\ad\})\simeq B(G_1,\{\mu_1\})$ (cf. \cite{isocys with addi 2} 6.5.1), for each $b\in B(G,\{\mu\})$ (which we identify with an element of $B(G_1,\{\mu_1\})$), we consider the associated Newton strata of Hodge type on the connected component $\ES_{K_{1,p}^\circ,0}^{b,+}\subset \ES_{K_{1,p}^\circ,0}^{+}$. By the above Proposition \ref{P: adjoint action on leaves}, the $G^{\ad}_1(\Z_{(p)})^+$-action on $\ES_{K_{1,p}^\circ,0}^{+}$ stabilizes 
$\ES_{K_{1,p}^\circ,0}^{b,+}$. 
We can then extend it to an action of $\mathscr{A}(G_{1,\Z_{(p)}})^\circ$ on $\ES_{K_{1,p}^\circ,0}^{b,+}$, and
define similarly Newton strata of abelian type
\[ \ES_{K^\circ_p,0}^b=\Big[[\ES_{K_{1,p},0}^{b,+}\times\mathscr{A}(G_{\Z_{(p)}})]/\mathscr{A}(G_{1,\Z_{(p)}})^\circ \Big]^{|J|}, \]
with \[\ES_{K_{p}^\circ,0}^{b,+}=\ES_{K_{1,p},0}^{b,+}/\Delta,\]for the above $\Delta$.
We have then  a map 
\[\delta_K: \ES_{K^\circ_p}(k)\ra B(G,\{\mu\}), \]
which factors through the above $\ES_{K^\circ_p}(k)\twoheadrightarrow C(\G^{\ad\circ},\{\mu^\ad\})$ under the projection \[C(\G^{\ad\circ},\{\mu^\ad\})\twoheadrightarrow B(G^\ad, \{\mu^\ad\})\simeq B(G,\{\mu\}).\]
We get the following decomposition
\[\ES_{K_p^\circ,0}=\coprod_{b\in \BGmu}\ES_{K_p^\circ,0}^b,\]
which we call the Newton stratification of $\ES_{K_p^\circ,0}$.
\subsubsection[]{}
Let $K=K_p^\circ$.
By Lemma \ref{L: adjoint bijections} we have ${}^K\Admu \stackrel{\sim}{\ra} {}^{K^\ad}\Adm(\{\mu^\ad\})$.
The map \[\upsilon_K:\ES_{K_p^\circ}(k)\rightarrow {}^K\Admu\] factors through $C(\G^{\ad\circ},\{\mu^\ad\})$.  We get the following commutative diagram
\[ \xymatrix{ & & B(G,\{\mu\})\\
	\ES_{K^\circ_p}(k)\ar@{->>}[r]\ar@{->>}[urr]^{\delta_K}\ar@{->>}[drr]_{\upsilon_K}& C(\G^{\ad\circ},\{\mu^\ad\})\ar@{->>}[ur]\ar@{->>}[dr]\\
	& & {}^K\Admu,} \]
and the composition \[\lambda_K:	\ES_{K^\circ_p}(k)\ra {}^K\Admu\twoheadrightarrow \Admu_K\] gives the KR stratification. Recall the subset of $\sigma$-straight elements \[{}^K\Admu_{\sigma\text{-str}}=\Admu_{\sigma\text{-str}}\cap  {}^K\wt{W} \subset {}^K\Admu\] was introduced in \ref{recall sigm straight}. By Theorem \ref{He's result--straight vs fundmtl}, we have the following analogue of Corollary \ref{coro--first properties}.
\begin{corollary}\label{C: straight EKOR}
	\begin{enumerate}
		\item For $x\in {}^K\Admu_{\sigma\text{-str}}$, $\ES_{K_p^\circ,0}^x$ is an adjoint central leaf.
		\item For any $b\in \BGmu$, the Newton stratum $\ES_{K_p^\circ,0}^b$ contains an EKOR stratum $\ES_{K_p^\circ,0}^x$ such that $x$ is $\sigma$-straight.
	\end{enumerate}
\end{corollary}

\subsubsection[]{Fully Hodge-Newton decomposable Shimura varieties}

We will change our setting to include also our results in section 3. Let $(G,X)$ be a Shimura datum of abelian type, $\K=K_pK^p\subset G(\Ab_f)$ an open compact subgroup with $K^p\subset G(\Ab_f^p)$ sufficently small and $K_p\subset G(\Q_p)$ a parahoric subgroup. 
Let $x$ be a point of the Bruhat-Tits building $\B(G,\Q_p)$, with the attached Bruhat-Tits stabilizer group scheme $\G=\G_x$ and its neutral connected component $\G^\circ=\G_x^\circ$, such that $K_p=\G^\circ(\Z_p)$. We will consider the following cases:
\begin{itemize}
	\item $(G,X)$ is of Hodge type and $\G=\G^\circ$.
	\item $(G,X)$ is of abelian type such that $(G^\ad,X^\ad)$ has no factors of type $D^\mathbb{H}$.
	\item $(G,X)$ is of abelian type such that $G$ is unramified over $\Q_p$ and $K_p$ is contained in some hyperspecial subgroup of $G(\Q_p)$.
\end{itemize}
We write $\ES_0=\ES_{\K,0}$ for the special fiber of the associated Kisin-Pappas integral model. 

Recall that the notion of fully Hodge-Newton decomposable pairs $(G,\{\mu\})$ is introduced in \cite{fully H-N decomp}. Roughly speaking, it says that for any non basic element $[b']\in B(G,\{\mu\})$, the pair $([b'],\{\mu\})$ satisfies the Hodge-Newton condition. We refer the readers to \cite{fully H-N decomp} Definition 2.1 for the precise definition of a fully Hodge-Newton decomposable pair $(G,\{\mu\})$  and loc. cit. Theorem 2.5 for a complete classification of all such pairs.
Under the assumption that the pair attached to the Shimura datum is fully Hodge-Newton decomposable, with our geometric constructions at hand, we have the following results (see also \cite{fully H-N decomp} section 6, where the results are conditional on the He-Rapoport axioms), which generalizes the corresponding results of \cite{stra abelian-good redu} in the good reduction case.

\begin{theorem}\label{EKOR vs NP}
	Let the notations be as above.  Assume that the attached pair $(G,\{\mu\})$ is fully Hodge-Newton decomposable. Then
	\begin{enumerate}
		\item each Newton stratum of $\ES_{0}$ is a union of EKOR strata;
		
		\item each EKOR stratum in a non-basic Newton stratum is an adjoint central leaf, and it is open and closed in the Newton stratum, in particular, non-basic Newton strata are smooth;
		
		\item the basic Newton stratum is a union of certain Deligne-Lusztig varieties.
	\end{enumerate}
\end{theorem}
\begin{proof}
	With our geometric constructions of EKOR strata at hand, the above statements (1) and (2) follow from \cite[Theorem 2.3]{fully H-N decomp}, see \cite{stra abelian-good redu} section 6 for example (which also works in the parahoric level case). 
	
	The assertion (3) also follows from \cite[Theorem 2.3]{fully H-N decomp}, but we use the informal version (5) as in loc. cit. Theorem B and the paragraph above 4.11 there. 
	Indeed,
	in the Hodge type case, we can use the uniformization morphism constructed in \cite[Proposition 6.4]{zhou isog parahoric} (see also our Proposition \ref{P: uniformization EKOR}) to prove that all irreducible components of the basic Newton stratum are certain Deligne-Lusztig varieties. For the general abelian type case, one can use arguments as in \cite{Sh} section 6 and \cite{Paroh} subsection 4.6 to deduce it from the Hodge type case.
\end{proof}
In fact, by \cite{fully H-N decomp} Theorems 2.3 and 6.4, the above assertions (1)-(3) are equivalent to each other, and any of them characterizes the condition that $(G,\{\mu\})$ is fully Hodge-Newton decomposable.

\section{EKOR strata for Siegel modular varieties}\label{section example}

In this section, we shall discuss the case of Siegel modular varieties in more details. Namely, we consider the Shimura variety attached to
$(\GSp_{2g},S^\pm, \K)$ with $\K=KK^p$ and $K\subset \GSp_{2g}(\Q_p)$ a parahoric subgroup.  In the case $g=2$, we describe explicitly the geometry.

\subsection{Moduli spaces of polarized abelian varieties with parahoric level structure}
\label{sec:Siegel.1}
Let $g\ge 1$ be a positive integer, $p$ a rational prime, and
$N\ge 3$ an integer with $(p,N)=1$. 
For a primitive
$N$th root of unity $\zeta_N\in \Qbar$, let $\calA_{g,1,\zeta_N}$ be the moduli space of $g$-dimensional principally polarized abelian varieties with a symplectic level-$N$ structure over $\Fpbar$. This is a smooth connected scheme over $\Fpbar$. We have the following decomposition \[\ES_{K(N), 0}(\GSp_{2g}, S^\pm)=\coprod_{\zeta_N}\calA_{g,1,\zeta_N},\]
where $\ES_{K(N), 0}(\GSp_{2g}, S^\pm)$ is the special fiber (over $\Fpbar$) of the canonical integral model of the Shimura varieties associated to $(\GSp_{2g}, S^\pm, K(N))$ with $K(N)$ the principal level-$N$ congruence subgroup, and $\zeta_N$ runs through the set of primitive $N$th root of unity. From now on, fix a choice of primitive
$N$th root of unity $\zeta_N\in \Qbar \subset \C$ and 
fix an embedding $\Qbar\hookrightarrow \Qbar_p$. We simply write $\calA_{g,1,N}=\calA_{g,1,\zeta_N}$. More generally, for any integer $d\geq 0$, we can consider $\calA_{g,p^d,N}$, the moduli space of $g$-dimensional abelian varieties with a degree $p^d$ polarization and a symplectic level-$N$ structure over $\Fpbar$.

Let $I:=\{0,1,\dots,g\}$. 
Let $\calA_I$ denote the Siegel moduli space 
with Iwahori level structure  
over $\Fpbar$ with respect to $\zeta_N$. It parametrizes 
the equivalence classes of objects
\[ (A_0\stackrel{\alpha}{\to} A_1\stackrel{\alpha}{\to}\dots
\stackrel{\alpha}{\to} A_g, \lambda_0,\lambda_g,\eta), \]
where 
\begin{itemize}
	\item each $A_i$ is a $g$-dimensional abelian variety,
	\item $\alpha$ is an isogeny of degree $p$,
	\item $\lambda_0$ and $\lambda_g$ are principal polarizations on $A_0$
	and $A_g$, respectively, such that $(\alpha^g)^* \lambda_g=p\lambda_0$.
	\item $\eta$ is a symplectic level-$N$ structure on $A_0$
	with respect to $\zeta_N$. 
\end{itemize}
Put $\eta_0:=\eta$, $\eta_i:=\alpha_* \eta_{i-1}$ for $i=1,\dots, g$,
and $\lambda_{i-1}:=\alpha^* \lambda_i$ for $i=g,\dots, 2$. Let $\ul
A_i:=(A_i,\lambda_i,\eta_i)$. Then $\calA_I$ parametrizes equivalence
classes of  objects
\[ (\ul A_0\stackrel{\alpha}{\to} \ul A_1\stackrel{\alpha}{\to}\dots
\stackrel{\alpha}{\to} \ul A_g), \]
where $\ul A_0\in \calA_{g,1,N}$, and for $i\neq 0$, 
\[ \ul A_i\in \calA'_{g,p^{g-i},N}:=\{\ul A\in \calA_{g,p^{g-i},N}\,
|\, \ker \lambda\subset A[p]\, \}. \]

For any non-empty subset $J=\{i_0,\dots, i_r\}\subset I$ with
$i_0<\dots <i_r$, let $\calA_J$
be the moduli space over $\Fpbar$ parameterizing equivalence classes of 
objects
\[  (\ul A_{i_0} \stackrel{\alpha}{\to} \ul
A_{i_1}\stackrel{\alpha}{\to}\dots 
\stackrel{\alpha}{\to} \ul A_{i_r}), \]
where $\ul A_{i_0}\in \calA_{g,1,N}$ if $i_0=0$, and $\ul A_{i_j}\in
\calA'_{g,p^{g-i_j},N}$ for others which satisfies the natural
compatibility condition. The moduli space $\calA_J$ is the
Siegel moduli space (over $\Fpbar$) with parahoric level structure of
type $J$. For $J_2\subset J_1$, let \[\pi_{J_1,J_2}:\calA_{J_1}\to \calA_{J_2}\]
be the natural projection, which is proper and surjective.  

We also write $\calA_J$ as $\calA_{J,\zeta_N}$ if we want to emphasize that it is relative to the choice of $\zeta_N$. When $\zeta_N$ varies, we get a similar decomposition as above
\[\ES_{K_JK^p(N), 0}(\GSp_{2g}, S^\pm)=\coprod_{\zeta_N}\calA_{J,\zeta_N}, \]
where $K_J\subset \GSp_{2g}(\Qp)$ is the
parahoric subgroup corresponding to the lattice chains of type
$J$, $K^p(N)\subset \GSp_{2g}(\mathbb{A}_f^p)$ is the principal level-$N$ congruence subgroup outside $p$, and $\ES_{K_JK^p(N), 0}(\GSp_{2g}, S^\pm)$ is the special fiber (over $\Fpbar$) of the integral model of the Shimura varieties associated to $(\GSp_{2g}, S^\pm, K_JK^p(N))$ defined by the moduli problem as \cite[chapter 6]{RZ} (see also the following appendix).
\\

In the following we fix our choice of $\zeta_N$ as before. We will study the geometry of $\calA_J=\calA_{J,\zeta_N}$.





\begin{theorem} \label{Siegel.1}\
	\begin{itemize}
		\item[(1)] The ordinary locus $\calA_J^{\rm ord}\subset \calA_J$ is dense.
		
		\item[(2)] $\calA_J$ is equi-dimensional of dimension $g(g+1)/2$. 
		
		\item [(3)] $\calA_J$ is irreducible if $|J|=1$, and for
		$|J|\ge 2$, $\calA_J$ has $(k_1+1)\dots(k_r+1)$ irreducible
		components, where $k_j:=i_j-i_{j-1}$. 
		
		\item [(4)] $\calA_J$ is connected.   
	\end{itemize}
\end{theorem}

\begin{proof}
	For statements (1)-(3), see
	\cite{Goertz2,  NG, yu:gamma} and also see \cite[Theorem
	2.1]{yu:KR}. (4) It suffices to show that $\calA_I$ is connected
	because the map $\pi_{I,J}$ is surjective. The moduli space
	$\calA_I$ is a union of $2^g$ irreducible components and each
	irreducible component is the closure of a maximal KR stratum, because every maximal KR stratum is irreducible \cite{yu:gamma}.
	Since the closure of every maximal KR stratum contains the the minimal
	KR stratum, every two irreducible components intersect. 
	It follows that $\calA_I$ is connected. 
\end{proof}


\subsection{The KR and EKOR stratifications}
\label{sec:Siegel.2}

Let $(V=\Q_p^{2g},\psi)$ a symplectic space of dimension $2g$, where
the alternating pairing $\psi :V\times V\to \Qp$ is represented by 
$
\begin{bmatrix}
0 & \wt I_g \\ -\wt I_g & 0
\end{bmatrix}$ and $\wt I_g=\text{anti-diag}(1,\dots,1)$. 
Let
$G:=\GSp_{2g}\subset \GL_{2g}$ and $T\subset G$ be the diagonal 
maximal torus. Let $\wt
W:=N_G(T)(\Qp)/T(\Zp)$ be the Iwahori-Weyl group of $\GSp_{2g}$ with
respect to $T$. The torus $T$ is contained in the diagonal maximal
torus $T_0$ of $\GL_{2g}$ and the cocharacter group $X_*(T)$ is
contained in $X_*(T_0)$, which is equal to $\Z^{2g}$.  
Then
$X_*(T)_\mathbb{R}=\{(u_i)\in X_*(T_0)_\mathbb{R}=\mathbb{R}^{2g} \mid u_1+u_{2g}=\dots
=u_{g}+u_{g+1}\, \}$.
We fix a base point in the apartment corresponding to the maximal
torus $T$ and identify it with $X_*(T)_\mathbb{R}$. Then $\wt W=X_*(T) \rtimes
W$, where $W$ is the Weyl group of $G$.
We fix the base alcove 
\[ \gra:=\{(u_i)\in X_*(T)_\mathbb{R}
\mid 1+u_1>u_{2g}>\dots>u_{g+1}> u_g\, \}, \]
and let $s_0,\dots, s_g$ be the simple reflections with respect to
the facets of
$\gra$. The affine Weyl group $W_a$ then is the Coxeter group
generated by $s_0,\dots, s_g$ and we have $\wt W=W_a \rtimes\Omega$,
where $\Omega\subset \wt W$ is the stabilizer of $\gra$. The length
function and Bruhat order on $W_a$ are naturally extended to those on
$\wt W$. 

Let $\mu=(1^{(g)},0^{(g)})\in X_*(T)$ be the standard minuscule coweight.
We denote by $\Adm(\mu):=\Admu\subset \wt W$ the set of $\mu$-admissible
elements. Let $\tau$ be the unique minimal element in $\Adm(\mu)$. We
have 
\[ s_i=(i,i+1) (2g+1-i, 2g-i), \quad i=1,\dots,g-1,\]
\[ s_g=(g,g+1),\quad s_0=\big((-1,0,\dots,0,1), (1,2g)\big), \]
\[ \tau=\big((0,\dots,0,1,\dots, 1), (1,g+1)(2,g+2)\dots(g,2g)\big). \]   

For any non-empty subset $J\subset I$, let 
\[ \calA_J=\coprod_{y \in \Adm(\mu)_J} \calA_{J,y} \]
be the KR stratification, where $\Adm(\mu)_J:= W_{J^c}\backslash
W_{J^c}\Adm(\mu)W_{J^c}/W_{J^c}$ is the image of $\Adm(\mu)$ in $W_{J^c}\backslash
\wt W/W_{J^c}$ under the map $\wt W\to W_{J^c}\backslash
\wt W/W_{J^c}$  and $W_{J^c}\subset \wt W$ is the subgroup generated by
$s_i$ for $i\in J^c:=I-J$. Put ${}^J\Adm(\mu):=\Adm(\mu)\cap
{}^{J^c} \wt W$, where ${}^{J^c} \wt W\subset \wt W$ is the set of minimal
length coset representatives for $W_{J^c}\backslash \wt W$. Let
\begin{equation}
\label{eq:Siegel.1}
\calA_J=\coprod_{x\in {}^J\Adm(\mu)} \calA_{J}^x
\end{equation}
be the EKOR stratification. If $J=I$, then $\calA_I^x=\calA_{I,x}$ and
\eqref{eq:Siegel.1} is the KR stratification.

\begin{theorem}\label{Siegel.2} For all $x\in  {}^J\Adm(\mu)$, one has the following statements. \ 
	
	{\rm (1)} The corresponding EKOR stratum $\calA_{J}^x$ is quasi-affine, smooth and
	equi-dimensional of 
	dimension $\ell(x)$. Every point in $\calA_{J}^x$ has the same
	$p$-rank.  
	
	{\rm (2)} If the EKOR stratum $\calA_{J}^x$ is not 
	supersingular (not contained in the
	supersingular locus of $\calA_J$), then it is irreducible. 
	
	{\rm (3)} Every non-supersingular KR stratum $\calA_{J,y}$, where $y\in \Adm(\mu)_J$, is
	irreducible. 
\end{theorem}

\begin{proof}  
	(1) The case where $J=I$ is proved  
	in \cite{NG} and \cite[Theorem 1.5]{Iwahoric Siegel-Gor-Yu}. 
	Since $\pi_{I,J}:
	\calA_{I,x}\to \calA_{J}^x$ is finite, \'etale and surjective, the
	stratum $\calA_{J}^x$ also share the same properties.
	
	
	(2) For the case where $J=I$ the irreducibility of $\calA_{I,x}$ 
	is proved 
	in \cite[Theorem 1.4 and Theorem
	1.5]{Iwahoric Siegel-Gor-Yu}. For
	arbitrary $J$, since $\calA_{I,x}$ is irreducible and $\pi_{I,J}:
	\calA_{I,x}\to \calA_J^x$ is surjective, $\calA_J^x$ is
	irreducible.
	
	(3) This is \cite[Proposition 4.4]{GoertzHoeve}. 
	We give a different proof using EKOR
	strata. The KR stratum $\calA_{J,y}$ is a union of EKOR strata, and by
	Theorem \ref{thm--first properties} (1), 
	it contains a unique maximal EKOR stratum, called the $y$-ordinary locus, which is
	open and dense in the KR stratum $\calA_{J,y}$. Since this
	EKOR stratum is non-supersingular, it is irreducible by (2).
	Therefore, the KR stratum $\calA_{J,y}$ is irreducible. 
\end{proof}


\subsection{Geometry of Siegel threefolds}
\label{sec:Siegel.3}

We restrict ourselves to the case $g=2$. Let $K_J\subset G(\Qp)$ denote the
parahoric subgroup corresponding to the lattice chains of type
$J$. The group $K_J$ is conjugate to $K_{J^\vee}$ in $G(\Qp)$, where
$J^\vee:=\{g-j |j\in J\}$. Thus, one can only consider 
the cases
$J=\{0\}, \{1\}, \{0,1\}, \{0,2\},\{0,1,2\}$, which correspond to
hyperspecial, paramodular, Klingen parahoric, Siegel parahoric 
and Iwahori open compact subgroups, respectively.

Recall that the $p$-rank function is constant on each KR stratum of
$\calA_I$, so it induces a map $\text{$p$-rank}: \Adm(\mu)\to \Z_{\ge
	0}$. For each integer $0\le f \le g=2$, 
set 
\[ \Adm^f(\mu):=\{x\in \Adm(\mu) \mid
\text{$p$-rank}(x)=f\}. \] 
We denote by $\bar x=(\bar x)_J$ the image of $x\in \Adm(\mu)$
in $\Adm(\mu)_J$, and 
\[ [x]=[x]_J:=\{z\in \Adm(\mu) \mid \bar z=\bar
x\}=W_{J^c} x W_{J^c} \cap \Adm(\mu).\]
By abuse of notation, we also write $\calA_{[x]_J}$ for the KR
stratum $\calA_{J,\bar x}$.
For each $0\le f\le g=2$, let $\calA_J^{f}\subset \calA_J$ 
(resp. $\calA_J^{\le f}\subset \calA_J$) be the
subvariety consisting of objects with $p$-rank $f$ (resp. $p$-rank
less than or equal to $f$). For any locally closed subvariety $X\subset
\calA_J$, denote by $\ol X$ the Zariski
closure of $X$ in $\calA_J$. 
Let $\calS_2=\calS_{\{0\}}\subset \calA_2 =\calA_{\{0\}}$ 
be the supersingular locus of $\calA_2$.
Note that for $g=2$, the Newton strata
coincide with the $p$-rank strata. The relationship between the Newton
strata and EKOR strata can be described by that of the $p$-rank strata
and EKOR strata. \\

{\bf (1)} {\bf Case $J=I=\{0,1,2\}$ (Iwahori level).} 
In this case the EKOR strata coincide with the KR strata. 
The following are the elements in the set
${\rm Adm}(\mu)$ together with the Bruhat order 
(cf.~\cite[Section 6]{yu:prank} and \cite[Section 4.1]{yu:KR}):

\begin{figure}[h]
	\begin{center}
		\scalebox{0.6}{\includegraphics{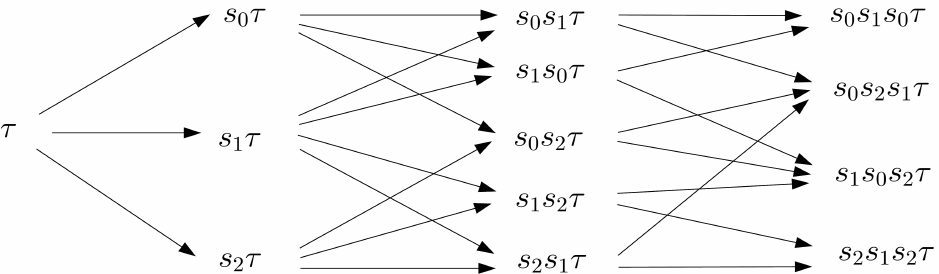}}.  
	\end{center}
\end{figure}
Here we write $x\to y$ for two elements $x, y \in W_a \tau$ if $x\le y$
in the Bruhat order. Using the $p$-rank formula \cite{NG} we obtain
\begin{equation}\label{eq:Siegel.2}
\begin{split}
&{\rm Adm}^2(\mu)=\{s_0s_1s_0\tau, \ s_1s_0s_2\tau,\  
s_2s_1s_2\tau,\ s_0s_2s_1\tau \},\\
&{\rm Adm}^1(\mu)=\{ s_0s_1\tau, \  s_1s_2\tau, \  
s_2s_1\tau,\ s_1s_0\tau\}, \\
&{\rm Adm}^0(\mu)=\{\tau,\ s_1 \tau,\ s_0 \tau, \ 
s_2\tau,\ s_0s_2\tau  \}. \\
\end{split}
\end{equation}
In the following we shall write $s_{j_1 j_2 \dots j_r}$ for the
element $s_{j_1} s_{j_2} \cdots s_{j_r}$ in the affine Weyl group $W_a$. 



\begin{proposition}\label{Siegel.3} \ 
	
	{\rm (1)} The $p$-rank two stratum $\calA_I^{2}$ is smooth of pure
	dimension $3$. It is a disjoint union of $4$ irreducible components
	indexed by $s_{010}\tau$, $s_{102}\tau$, $s_{212}\tau$ and
	$s_{021}\tau$. It is the smooth locus of $\calA_I$. 
	The closure $\ol {\calA_I^{2}}$ is equal to $\calA_I$ and is connected. 
	
	{\rm (2)} The $p$-rank one stratum $\calA_I^{1}$ is smooth of pure
	dimension 
	$2$. It is a disjoint union of $4$ irreducible components
	indexed by $s_{01}\tau$, $s_{12}\tau$, $s_{21}\tau$ and
	$s_{10}\tau$. The
	closure $\ol {\calA_{I}^{1}}$ of $\calA_{I}^{1}$ is connected. 
	One has the
	decomposition  
	$\calA_I^{\le 1}=\ol {\calA_{I}^{1}} \coprod \calA_{I, s_{02}\tau}$. 
	In particular, the $p$-rank one stratum $\calA_I^{1}\subset
	\calA_I^{\le 1}$ is not dense. 
	
	{\rm (3)} The supersingular locus $\calS_I=\calA_I^0\subset \calA_I$ 
	consists of
	one-dimensional components (the closure $\ol {\calA_{s_1\tau}}$) 
	and two-dimensional components (the closure $\ol
	{\calA_{s_{02}\tau}}$). Each connected component of $\ol
	{\calA_{s_1\tau}}$ is isomorphic to $\bfP^1$, and each connected
	component of $\ol 
	{\calA_{s_{02}\tau}}$ is isomorphic to $\bfP^1\times \bfP^1$. The
	supersingular locus $\calS_I$ is connected.
	
	{\rm (4)} The projection morphism $\calS_I\to \calS_{2}\subset \calA_2$ 
	is not finite. Therefore, the morphism $\calA_I\to \calA_2$ is
	proper but not finite.  
	
\end{proposition}

\begin{proof}
	Statements (1), (2), (4) and the first part of (3) 
	follow directly from Theorems \ref{Siegel.1} and
	\ref{Siegel.2}, \eqref{eq:Siegel.2} and the Bruhat order. 
	(3) The second
	part follows from the description of $\calS_I$; see 
	\cite[Theorem 8.1]{yu:prank}. We show the connectedness of
	$\calS_I$. 
	By \cite[Theorem 7.3]{Iwahoric Siegel-Gor-Yu}, the union
	$\calA_{I,\le 1}:=\cup_{\ell(x)\le 1} \calA_{I,x}$ of KR strata 
	of dimension $\le
	1$ is connected. The supersingular locus $\calS_I$ is the union of
	$\calA_{I,\le 1}$ and $\calS_{I, s_{20}\tau}$. 
	Since each connected component of $\calA_{I, s_{20}\tau}$ 
	is quasi-affine,
	its closure in $\calA_I$ is proper and intersects with 
	$\calA_{I,\le 1}$. Thus, $\calS_I$ is connected. 
\end{proof}

\begin{remark}
	Proposition~\ref{Siegel.3} (3) rules out the possibility of
	equi-dimensionality of $p$-rank strata and of Newton strata. 
	Proposition~\ref{Siegel.3} (2)
	shows that the $p$-rank strata (and the Newton strata) 
	do not form a stratification on $\calA_I$,
	that is, the closure of each stratum is a union of strata. \\  
\end{remark}

{\bf (2)} {\bf Case $J=\{0\}$ (hyperspecial level) and 
	$W_{J^c}=\<s_1, s_2\>$.} In this case the EKOR strata coincide
with the EO strata and the whole moduli space is a KR stratum. Using
$\tau s_2=s_0 \tau$ and $\tau s_1=s_1 \tau$, one easily computes the
set ${}^J\Adm(\mu)$ with closure relation
\[ {}^J\Adm(\mu)=\{\tau, s_0\tau, s_{01}\tau, s_{010}\tau\}, \quad
\tau\to s_0\tau\to s_{01}\tau\to s_{010}\tau. \]  

\begin{proposition} \label{Siegel.4} Let $J=\{0\}$. 
	
	{\rm (1)} There are $4$ EKOR strata. The $p$-prank two stratum
	$\calA_2^2$ is smooth and
	irreducible. The $p$-prank one stratum $\calA_2^1$ is smooth and
	irreducible. The supersingular locus $\calS_J$ is connected
	and each 
	irreducible component is isomorphic to $\bfP^1$.
	
	{\rm (2)} We have
	\begin{equation}
	\label{eq:Siegel.3}
	\pi_{I,J}(\calA_{I,x})=
	\begin{cases}
	\calA_J^{s_{010}\tau} & \text{for $x\in \Adm^{2}(\mu)$}; \\
	\calA_J^{s_{01}\tau} & \text{for $x\in \Adm^{1}(\mu)$}, 
	\end{cases}
	\end{equation}
	and 
	\begin{equation}
	\label{eq:Siegel.4}
	\begin{split}
	& \pi_{I,J}(\calA_{I,s_0\tau}\coprod
	\calA_{I,s_2\tau})=\calA_J^{s_{0}\tau}, \\ 
	& \pi_{I,J}(\calA_{I,\tau}\coprod
	\calA_{I,s_1\tau})=\calA_J^{\tau}, \\
	& \pi_{I,J}(\calA_{I,s_{02}\tau})=\calA_J^{\tau}\coprod 
	\calA_J^{s_0 \tau}.
	\end{split}
	\end{equation}
\end{proposition}
\begin{proof}
	(1) The connectedness of $S_J$ follows from
	Proposition~\ref{Siegel.3} (3). The other results are
	well-known. See Katsura-Oort \cite{katsura-oort:surface} 
	for a detailed description
	of $\calS_J$. 
	
	
	(2) Since the map $\pi_{I,J}$ preserves $p$-ranks, the first
	relation follows. The second relation follows from 
	the explicit description of the supersingular locus $S_I$;
	see \cite[Theorem 8.1]{yu:prank}.  
\end{proof}


\begin{remark}
	By Proposition~\ref{Siegel.4}, the intersection
	$\pi_{I,J}^{-1}(\calA_J^{\tau})\cap \calA_{I,s_{02}\tau}$ 
	is nonempty and $\pi_{I,J}^{-1}(\calA_J^{\tau})$ does not contain the
	KR stratum $\calA_{I,s_{02}\tau}$. Therefore, the preimage 
	$\pi_{I,J}^{-1}(\calA_J^{\tau})$  of $\calA_J^{\tau}$ is not a
	union of KR strata.  \\
\end{remark}

{\bf (3)} {\bf Case $J=\{0,1\}$ (Klingen parahoric level) and $W_{J^c}=\<s_2\>$.} 
Using  $\tau s_2=s_0 \tau$,
we get \\
\begin{center}
	\begin{tabular}{clll}
		KR types           & EKOR &  dim & $p$-rank \\   
		$[\tau]_J=[s_0\tau]_J  =\{\tau, s_2\tau, s_0\tau, s_{02}\tau\,\}$, 
		& $\tau$, $s_0\tau$ & 0, 1 & 0, 0 \\
		$[s_1\tau]_J=[s_{10}\tau]_J =\{s_1\tau, s_{10}\tau, s_{21}\tau\,\}$, &
		$s_1\tau$, $s_{10}\tau$ & 1, 2 & 0, 1 \\
		$[s_{12} \tau]_J=[s_{120}\tau]_J  =\{s_{12}\tau, s_{120}\tau,
		s_{212}\tau\,\}$,& $s_{12} \tau$,  $s_{120}\tau$ & 2, 3 & 1, 2 \\ 
		$[s_{01}\tau]_J=[s_{010}\tau]_J =\{s_{01}\tau,
		s_{010}\tau,s_{201}\tau\,\}$,  
		& $s_{01}\tau$, $s_{010}\tau$ & 2, 3 & 1, 2. 
	\end{tabular}\
\end{center}

By \cite[Theorem 6.15]{He-Rap}, we obtain the closure relation for
EKOR strata: 
\begin{equation}\label{eq:Siegel.5}
\xymatrix{
	& s_0\tau \ar[r] \ar[rd]\ar[rdd]  & 
	s_{01}\tau \ar[r]  & s_{010}\tau  \\
	\tau\ar[r]\ar[ru] & s_1\tau \ar[rd]\ar[ru]\ar[r]   
	& s_{10} \ar[ru] \ar[rd] \tau  & \\
	&          & s_{12}\tau \ar[r]  & s_{120}\tau
}
\end{equation}
Notice the new EKOR order $s_0\tau \to s_{12}\tau$ because of 
$s_2 (s_0\tau)s_2^{-1}=s_2 \tau \le s_{12}\tau$ in the Bruhat order. 
The KR closure relation is as follows:     
\begin{equation}
\label{eq:Siegel.6}
\xymatrix{
	& &[s_{120}\tau]_J\\
[s_0\tau]_J \ar[r] & [s_{10}\tau]_J \ar[ur] \ar[dr] &\\
& &[s_{010}\tau]_J.}
\end{equation}
There are 8 EKOR strata and 4 KR strata.

\begin{proposition}\label{Siegel.7}\
	\begin{itemize}
		\item[(1)] The $p$-rank two stratum $\calA_J^2$ is smooth and 
		has two irreducible components which are
		the EKOR strata $\calA_J^{s_{010}\tau}$ and 
		$\calA_J^{s_{120}\tau}$; they are
		properly contained 
		in the KR strata $\calA_{[s_{010}\tau]_J}$ and 
		$\calA_{[s_{120}\tau]_J}$,
		respectively. 
		\item[(2)] The $p$-rank one stratum $\calA_J^1$ 
		is smooth and has three irreducible components which are
		the EKOR strata $\calA_J^{s_{01}\tau},\calA_J^{s_{10}\tau}$ and
		$\calA_J^{s_{12}\tau}$; they are
		properly contained in the KR strata
		$\calA_{[s_{01}\tau]_J}$, $\calA_{[s_{10}\tau]_J}$ and
		$\calA_{[s_{12}\tau]_J}$, respectively.

		\item[(3)] The supersingular stratum $\calS_J$ is connected and it
		consists of two types of irreducible components: those in $\ol 
		{\calA_J^{s_0\tau}}=\calA_{[\tau]_J}$ (``horizontal'' ones) 
		and those in  
		$\ol {\calA_J^{s_1\tau}}$ (``vertical'' ones)\footnote{We refer
			to \cite[Remark 1.3]{yu:ss_siegel} for a detailed description
			of the role of horizontal and vertical components of $\calS_J$.}.
		The intersection  $\calS_J\cap \calA_{[s_{10}\tau]_J}$ is
		equal to the EKOR stratum $\calA_J^{s_1\tau}$, which
		consists of open 
		``vertical'' components of $\calS_J$.

		\item[(4)] The union $\calA_{[s_{010}\tau]_J}\cup
		\calA_{[s_{120}\tau]_J}$ is the 
		smooth locus of $\calA_J$. 
		\item[(5)] Two irreducible components $\ol {\calA_{[s_{010}\tau]_J}}$
		and $\ol {\calA_{[s_{120}\tau]_J}}$ are smooth and they intersect
		transversally at an irreducible smooth surface, which is equal to 
		the closure $\ol {\calA_{[s_{10}\tau]_J}}$ of the KR stratum 
		${\calA_{[s_{10}\tau]_J}}$.
		\item[(6)] (The transition relation) We have
		$\pi_{I,J}(\calA_I^x)=\calA_J^x$ for all $x\in {}^J\Adm(\mu)$, 
		\begin{equation}
		\label{eq:Siegel.7}
		\pi_{I,J}(\calA_I^{s_{212}\tau})=\calA_J^{s_{120}\tau}, \quad
		\pi_{I,J}(\calA_I^{s_{201}\tau})=\calA_J^{s_{010}\tau}, \quad 
		\pi_{I,J}(\calA_I^{s_{21}\tau})=\calA_J^{s_{10}\tau}, 
		\end{equation}
		and 
		\begin{equation}
		\label{eq:Siegel.8}
		\begin{split}
		\pi_{I,J}(\calA_I^{s_{20}\tau})=\calA_J^{\tau}\coprod
		\calA_J^{s_0\tau}  , \quad
		& \pi_{I,J}(\calA_I^{s_{2}\tau})=\calA_J^{s_{0}\tau}, \\  
		\pi_{J,\{0\}}(\calA_J^{\tau}\coprod
		\calA_J^{s_{1}\tau})=\calA_{\{0\}}^{\tau}, \quad
		& \pi_{J,\{0\}}(\calA_J^{s_{0}\tau})=\calA_{\{0\}}^{s_{0}\tau}. \\  
		\end{split}       
		\end{equation}
	\end{itemize}
\end{proposition}
\begin{proof}
	Statements (1), (2) and (3) follow directly from the EKOR stratification and their
	relation with $p$-rank strata; see also
	\cite[p.~2346]{yu:KR} and \cite[Proposition 4.5]{yu:ss_siegel} 
	and for  
	the description of $\calS_J$. (4) From the closure relation of KR
	strata \eqref{eq:Siegel.6}, 
	the complement $\ol {\calA_{[s_{10}\tau]_J}}$ of
	$\calA_{[s_{010}\tau]_J}\cup
	\calA_{[s_{120}\tau]_J}$ is contained in both
	irreducible components, and hence it is the singular locus of the
	moduli space $\calA_J$. 
	(5) This follows from Theorem 3 of \cite{tilouine:coates}. 
	(6) The transition relation \eqref{eq:Siegel.7} follows from 
	$\pi_{I,J}(\calA_I^{s_{212}\tau})=\calA_{[s_{120}\tau]_J}$, 
	$\pi_{I,J}(\calA_I^{s_{201}\tau})=\calA_{[s_{010}\tau]_J}$ and
	$\pi_{I,J}(\calA_I^{s_{21}\tau})=\calA_{[s_{10}\tau]_J}$. The
	relation \eqref{eq:Siegel.8} follows from the description of the
	supersingular locus $\calS_I$ and $\calS_J$; see \cite[Theorem
	8.1]{yu:prank} and \cite[Proposition 4.5]{yu:ss_siegel}. \\
\end{proof}

{\bf (4)} {\bf Case $J=\{0,2\}$ (Siegel parahoric level) and 
	$W_{J^c}=\<s_1\>$.} 
Using  $\tau s_1=s_1 \tau$,
we get \\
\begin{center}
	\begin{tabular}{clll}
		KR types           & EKOR &  dim & $p$-rank \\   
		$[\tau]_J  =\{\tau, s_1\tau, s_0\tau, s_{02}\tau\,\}$, 
		& $\tau$, $s_0\tau$ & 0  & 0  \\
		$[s_{21}\tau]_J=[s_2\tau]_J= \{s_2\tau, s_{21}\tau,
		s_{12}\tau \,\}$, &
		$s_2\tau$, $s_{21}\tau$ & 1, 2 & 0, 1 \\
		$[s_{01} \tau]_J=[s_{0}\tau]_J  =\{s_{0}\tau, s_{10}\tau, s_{01}\tau
		\,\}$,& $s_{0} \tau$,  $s_{01}\tau$ & 1, 2 & 0, 1 \\ 
		$[s_{021} \tau]_J =[s_{02} \tau]_J=\{s_{02}\tau, s_{021}\tau,
		s_{102 }\tau
		\,\}$,& $s_{02} \tau$,  $s_{021}\tau$ & 2, 3 & 1, 2 \\ 
		$[s_{212} \tau]_J =\{s_{212}\tau
		\,\}$,& $s_{212} \tau$ & 3 & 2 \\ 
		$[s_{010}\tau]_J =\{s_{010}\tau
		\,\}$, & $s_{010}\tau$ & 3 & 2. 
	\end{tabular}\
\end{center}

By \cite[Theorem 6.15]{He-Rap}, we obtain the closure relation for
EKOR strata: 
\begin{equation}\label{eq:Siegel.9}
\xymatrix{
	& s_0\tau \ar[r] \ar[rd]  & 
	s_{01}\tau \ar[r] \ar[rd] & s_{010}\tau  \\
	\tau\ar[rd]\ar[ru] &    
	& s_{02} \ar[r]\tau  & s_{021}\tau \\
	& s_2\tau \ar[ru]\ar[r] & s_{21}\tau \ar[r] \ar[ru] & s_{212}\tau
}
\end{equation}
and the KR closure relation:    
\begin{equation}
\label{eq:Siegel.10}
\xymatrix{
	&  [s_{01}\tau]_J \ar[r] \ar[rd] & [s_{010}\tau]_J  \\
	[\tau]_J\ar[rd]\ar[ru] &    
	& [s_{021}\tau]_J \\
	&  [s_{21}\tau]_J \ar[r] \ar[ru] & [s_{212}\tau]_J.
}
\end{equation}
There are 9 EKOR strata and 6 KR strata. 

\begin{proposition}
	\label{Siegel.8} Let $J=\{0,2\}$. 
	\begin{itemize}
		\item[(1)] The $p$-rank two stratum $\calA_J^2$ is smooth and has
		three irreducible components. Two of them are
		$\calA_J^{s_{212}\tau}=\calA_{[s_{212}\tau]_J}$ and 
		$\calA_J^{s_{010}\tau}=\calA_{[s_{010}\tau]_J}$, and the
		other is $\calA_J^{s_{021}\tau}$, which is properly contained in
		the KR stratum $\calA_{[s_{021}\tau]_J}$.
		\item[(2)] The $p$-rank one stratum $\calA_J^1$ is smooth and has
		two irreducible components. 
		They are EKOR strata 
		$\calA_J^{s_{01}\tau}$ and 
		$\calA_J^{s_{21}\tau}$, which are 
		properly contained in  $\calA_{[s_{01}\tau]_J}$ and
		$\calA_{[s_{21}\tau]_J}$, respectively. 
		\item[(3)] The supersingular locus $\calS_J$ has pure dimension
		$2$. It is contained in the 3-dimensional closed KR stratum $\ol
		{\calA_{[s_{021}\tau]_J}}$. Each irreducible component of $\calS_J$
		is isomorphic to $\bfP^1\times \bfP^1$. 
		\item[(4)] The zero dimensional stratum $\calA_{[\tau]_J}$ consists
		of  points $(\ul A_0\stackrel{F}{\to} \ul A_0^{(p)})$, where $A_0$
		is   superspecial. 
		\item[(5)] The union $\calA_{[s_{212}\tau]_J}\cup
		\calA_{[s_{010}\tau]_J}\cup\calA_{[s_{02}\tau]_J}$ is the smooth
		locus of $\calA_J$. \\
		\item[(6)] We have
		$\pi_{I,J}(\calA_I^x)=\calA_J^x$ for all $x\in {}^J\Adm(\mu)$ and 
		\begin{equation}
		\label{eq:Siegel.11}
		\begin{split}
		\pi_{I,J}(\calA_I^{s_{102}\tau})=\calA_J^{s_{021}\tau}, \quad
		& \pi_{I,J}(\calA_I^{s_{12}\tau})=\calA_J^{s_{21}\tau}, \\
		\pi_{I,J}(\calA_I^{s_{10}\tau})=\calA_J^{s_{01}\tau}, \quad 
		& \pi_{I,J}(\calA_I^{s_{1}\tau})=\calA_J^{\tau}.  
		\end{split}
		\end{equation}
	\end{itemize}
\end{proposition}
\begin{proof}
	Statements (1), (2) and (5) follow directly from the closure
	relations \eqref{eq:Siegel.9} and \eqref{eq:Siegel.10}. The
	supersingular locus $\calS_I$ is the union of
	$\ol{\calA_{I}^{s_1\tau}}$ and $\ol{\calA_{I}^{s_{20}\tau}}$. The
	transition map $\tau_{I,J}$ gives an isomorphism 
	$\ol{\calA_{I}^{s_{20}\tau}}\isoto \calS_J$. Thus, we can describe 
	$\calS_J$ by $\ol{\calA_{I}^{s_{20}\tau}}$, whose description
	is included in
	\cite[Theorem 8.1]{yu:prank}. From this, statements (3) and (4)
	follow. (6) The first three relations follow from the relations
	\[ \pi_{I,J}(\calA_I^{s_{102}\tau})\subset \calA_{[s_{021}\tau]_J}, \quad
	\pi_{I,J}(\calA_I^{s_{12}\tau})=\calA_{[s_{21}\tau]_J}, \quad 
	\pi_{I,J}(\calA_I^{s_{10}\tau})=\calA_{[s_{01}\tau]_J}. \]
	The last one follows from the description of $\calS_J$ via (3) and (4). 
\end{proof}

In fact, in the moduli space $\calA_I$ with Iwahori level structure,
we have 
\[ \calS_I=\ol{\calA_{s_{021}\tau}} \cap \ol {\calA_{s_{102}\tau}}. \]
These two components are mapped, through the transition map $\pi_{I,J}$,
onto the component $\ol  {\calA_{[s_{102}\tau]_J}}$. 
\\

{\bf (5)} {\bf Case $J=\{1\}$ (paramodular level) and 
	$W_{J^c}=\<s_0, s_2\>$.} 
Using  $\tau s_2=s_0 \tau$ and $\tau s_0=s_2 \tau$,
we get \\
\begin{center}
	\begin{tabular}{clll}
		KR types           & EKOR &  dim & $p$-rank \\   
		$[\tau]_J =\{\tau, s_2\tau, s_0\tau, s_{02}\tau\,\}=W_{\{0,2\}}
		\tau$  
		& $\tau$  & 0 & 0 \\ 
		\\
		$[s_{102}\tau]_J=$\  \parbox{1.5in}{$W_{\{0,2\}} s_1\tau\cup W_{\{0\}}
			s_{10}\tau$ \\
			$\cup W_{\{2\}} s_{12}\tau\cup s_{102}\tau$} &
		\parbox{1in}{$s_1\tau$, $s_{10}\tau$\\ $s_{12}\tau$,
			$s_{120}\tau$} & \parbox{0.5in}{1, 2\\ 2, 3} &
		\parbox{0.5in}{0, 1  \\ 1, 2} \\ 
	\end{tabular}\
\end{center}
and the closure relation for
EKOR strata and KR strata: 
\begin{equation}\label{eq:Siegel.12}
\tau \to s_1\tau \to s_{10}\tau, s_{12}\tau \to s_{120}\tau, \quad
[\tau]_J \to [s_{120}\tau]_J. 
\end{equation} 
There are 5 EKOR strata and 2 KR strata.
We have the follow result (cf.~\cite[Theorem 4.4]{yu:KR}). 

\begin{proposition}
	\label{Siegel.9} Let $J=\{1\}$. 
	\begin{itemize}
		\item[(1)] The $p$-rank two stratum $\calA_J^2$ is smooth and has one
		irreducible component. This component is the EKOR stratum
		$\calA_J^{s_{102}\tau}$ and it is
		properly contained in the maximal KR stratum $\calA_{[s_{102}\tau]_J}$.
		
		\item[(2)] The $p$-rank one stratum $\calA_J^1$ is smooth and has
		two irreducible components. 
		They are the EKOR strata 
		$\calA_J^{s_{10}\tau}$ and $\calA_J^{s_{12}\tau}$ properly
		contained in the maximal KR stratum $\calA_{[s_{102}\tau]_J}$. 
		\item[(3)] The supersingular locus has pure dimension $1$. Each
	irreducible	component is isomorphic to $\bfP^1$.  The intersection $S_J\cap
		\calA_{[s_{102}\tau]_J}$ is the smooth locus of $S_J$.
		\item[(4)] The zero dimensional stratum $\calA_{[\tau]_J}$ is the
		singular locus of $\calA_J$, and it is also the singular locus of
		$\calS_J$. 
		\item[(5)] The stratum $\calA_{[s_{102}\tau]_J}$ is the smooth locus. 
		\item[(6)] We have
		\begin{equation}\label{eq:Siegel.13}
		\pi_{I,J}(\calA_I^{x})=
		\begin{cases}
		\calA_J^{s_{120}\tau} & \text{for $x\in \Adm^2(\mu)$;} \\ 
		\calA_J^{s_{10}\tau} & \text{for $x\in \{ s_{10}\tau,s_{21}\tau\} $;} \\ 
		\calA_J^{s_{12}\tau} & \text{for $x\in \{s_{12}\tau,s_{01}\tau \}$;} \\ 
		\calA_J^{s_1\tau} & \text{for $x=s_1\tau $;} \\ 
		\calA_J^{\tau} & \text{for $x\in W_{\{0,2\}} \tau $.} \\ 
		\end{cases}
		\end{equation}
	\end{itemize}
\end{proposition}
\begin{proof}
	Statements (1)-(5) follow from the closure relation  and the description of the supersingular
	locus $\calS_J$ \cite{yu:ss_siegel}. (6) We only need to show the case
	of $p$-rank one strata. Using the geometric
	characterization of KR types \cite[Section 4.2]{yu:KR}, for a $p$-rank
	one point $(A,\lambda)$ in $\calA_J^1$, the kernel $\ker \lambda$ is
	isomorphic to $\mu_p\times \Z/p\Z$ if $(A,\lambda)\in
	\calA_J^{s_{12}\tau}$, or is local-local  if $(A,\lambda)\in
	\calA_J^{s_{10}\tau}$. From this, we obtain
	$\pi_{I,J}(\calA_I^{s_{21}\tau})=\calA_J^{s_{10}\tau}$ and $\pi_{I,J}(\calA_I^{s_{01}\tau})=\calA_J^{s_{12}\tau}$.
\end{proof}

\appendix

\section{He-Rapoport axiom 4 (c) for Shimura varieties of PEL-type}
\label{sec:axiom-4-c}

In this appendix we verify He-Rapoport's axiom 4 (c) (\cite{He-Rap} 3.3) for Shimura varieties of
PEL-type in the case where the parahoric level subgroup $K$ at $p$ is
the stabilizer group, i.e. $K=K^\circ$. 
This extends earlier results of Zhou \cite{zhou isog parahoric} and He-Zhou \cite{HZ} mainly in the
ramified cases and the case which contains a simple factor of type
$D$.  This also improves our main
results (Theorem \ref{Thm C} (2) and Theorem \ref{intro thm--etale}) for the PEL-type case.

We follow Rapoport-Zink \cite{RZ} for the construction of the ``naive'' integral
model with ``parahoric'' level structure.  
\subsection{Moduli spaces of PEL-type}
\label{sec:PEL.1}

\begin{defn}
	A \emph{PEL-datum} is a tuple $(B,*,V,\psi,h)$, where
	\begin{itemize}
		\item $(B,*)$ is a finite dimensional semi-simple $\Q$-algebra with
		a positive involution,
		\item $(V,\psi)$ is a finite faithful non-degenerate 
		$\Q$-valued skew-Hermitian $B$-module, and 
		\item $h:\C\to \End_{B_\mathbb{R}} (V_\mathbb{R})$ is an $\mathbb{R}$-algebra homomorphism
		such that $h(z)'=h(\bar z)$ and $(x,y):=\psi(h(i)x,y)$ is a
		positive definite symmetric form, where $B_\mathbb{R}:=B\otimes_\Q \mathbb{R}$,
		$V_\mathbb{R}:=V\otimes_\Q \mathbb{R}$ and $'$ is the adjoint with respect to the
		pairing $\psi$.
	\end{itemize}
\end{defn}

For each PEL-datum, we associate an algebraic $\Q$-group
$G=GU_B(V,\psi)$ of unitary similitudes on $(V,\psi)$. Let $X$ be the
$G(\mathbb{R})$-conjugacy class of $h:\bbS:=\Res_{\C/\mathbb{R}} \Gm_\C \to G_\mathbb{R}$. 
For each open compact subgroup $\K\subset G(\Ab_f)$, denote by
$\Sh_\K(G,X)$ the Shimura variety\footnote{Since the group $G$ may be non connected, we have to modify a little the formalism of Shimura varieties in \cite{varideshi} to include the current setting.} defined over the reflex field $\E$,
whose $\C$-points are given by
\[ \Sh_\K(G,X)(\C)=G(\Q)\backslash X\times G(\Ab_f) /\K, \]
even $G$ may not be connected, i.e. if its adjoint group contains a
$\Q$-simple factor of type $D$.  When there is no confusion, we simply write $\Sh_\K$ for $\Sh_\K(G,X)$.

Recall that the canonical model $\Sh_\K(G,X)$ is constructed using a moduli interpretation. Indeed, one first constructs an $\E$-scheme $M_\K$ of finite type using a moduli interpretation (see below) and obtains the inclusion $\Sh_\K(G,X)(\C)\subset M_\K(\C)$ (the inclusion may be strict due to the failure of the Hasse principle, see \cite{kottwitz:jams92} section 8, which holds in the current setting; see also the proof of Proposition \ref{PEL.8}). Then one defines $\Sh_\K(G,X)$ to be the corresponding $\E$-subscheme of $M_\K$. Using the main theorem of complex multiplication, one shows that $\Sh_\K(G,X)$ is the canonical model of the complex Shimura variety $\Sh_\K(G,X)(\C)$; see \cite{Deligne} for more details.

Let $p> 2$ be a prime. Assume that there exists an order $O_B\subset B$
which is stable under $*$ and maximal at $p$,
i.e. $O_{B_p}:=O_{B}\otimes \Zp$ is a maximal order in $B_p:=B\otimes
\Qp$, and we fix such an order.

Suppose that there exists a self-dual $O_{B_p}$-lattice in $V_p:=V\otimes
\Qp$. Choose an $O_B$-lattice $\Lambda$ in $V$ such that
$\Lambda_p=\Lambda\otimes \Zp$ is self-dual. Let 
\[ K=\Stab_{G(\Qp)} (\Lambda_p),\quad \text{and} \quad 
K^p\subset \Stab_{G(\Ab_f^p)} (\Lambda\otimes \hat \Z^{(p)}) \]
be a sufficiently small open compact subgroup, and put $\K=KK^p$.   

Let $v|p$ be a place of $\E$ over $p$, and $O_{E}$ the ring of
integers of the $v$-adic completion $E=\E_v$ of $\E$. Using the modular
interpretation, we construct the naive integral model over $O_{E}$
denoted by $\bfM^{\rm naiv}_\K$. It classifies the isomorphism classes of
tuples $(A,\lambda,\iota, \bar \eta)$, where 
\begin{itemize}
	\item $(A,\lambda, \iota)$ is a prime-to-$p$ degree polarized
	$O_B$-abelian variety of dimension $\frac{1}{2} \dim_\Q V$,
	\item $\bar \eta$ is a $K^p$-orbit of $O_B\otimes \hat
	\Z^{(p)}$-linear isomorphisms 
	\[ \eta: \Lambda \otimes \hat
	\Z^{(p)} \isoto T^{(p)}(A)=\prod_{\ell\neq p} T_\ell(A) \]
	which preserve the pairings for a suitable isomorphism
	$\Z^{(p)}\simeq  \Z^{(p)}(1)$. 
\end{itemize}
We also require that objects $(A,\lambda,\iota, \bar \eta)$ in
$\bfM^{\rm naiv}_\K$ satisfy the Kottwitz determinant condition, cf. \cite{RZ} chapter 6 or \cite{kottwitz:jams92}.

We have defined an $O_E$-scheme $\bfM^{\rm naiv}_\K$. By construction, it is an integral model of $M_\K$ over $O_E$. In particular, this gives an inclusion 
\[\Sh_\K(G,X)_{E}:=\Sh_\K(G,X)\otimes_\E E \subset \bfM^{\rm naiv}_\K \otimes_{O_E}E.\] 
Define $\ES_\K$ as the flat closure of $\Sh_\K(G,X)_{E}$ in
$\bfM^{\rm naiv}_\K$.

\subsection{Parahoric data at $p$}
\label{sec:PEL.2}

We define the notion of self-dual multichains of
lattices $\calL$ in $V_p$ following \cite{RZ}. This allows us to extend the above level subgroup $K$ at $p$ to more general forms.

To simplify the notation, we write $(B,*,V,\psi,O_B)$ for $(B_p,
*,V_p,\psi_p, O_{B_p})$ in this subsection. 

\begin{defn}
	Suppose that $B$ is a simple $\Q_p$-algebra. \emph {A chain of
		$O_B$-lattices} in $V$ is a set of totally ordered $O_B$-lattices
	$\calL$ such that for every element $x\in B^\times$ which 
	normalizes $O_B$, one has 
	\[ \Lambda \in \calL \implies x \Lambda \in \calL. \] 
\end{defn}

Explicitly, suppose we write $B=M_n(D)$ and let $O_B=M_n(O_D)$, 
where $D$ is a division $\Qp$-algebra and $O_D$ is the maximal order
of $D$. 
The group $N_{B^\times}(O_B)$ of normalizers of
$O_B$ in $B^\times$ is equal to $\Pi^{\Z} \GL_n(O_D)$, where $\Pi$ is
a prime element of $O_D$. Then $\calL$ is a chain of
$O_B$-lattices in $V$ if and only if $\calL$ is totally ordered and
for any member $\Lambda\in \calL$, one has $\Pi^{\pm 1}\Lambda \in
\calL$. If we fix a member $\Lambda_0\in \calL$, then there is
an unique integer $r\ge 1$, called the \emph{period}, and a
sequence of lattices $\{\Lambda_i\}_{0\le i\le r}$ such that
\[ \Lambda_0\subsetneq \Lambda_1 \subsetneq \dots \Lambda_r=\Pi^{-1} \Lambda_0, \] and every member $\Lambda\in \calL$ is $\Pi^n\Lambda_i$ for some $n\in \Z$ and some $i$ with $0\leq i\leq r$.
We extend the set of lattices by putting $\Lambda_{i+r}=\Pi^{-1}
\Lambda_i$ for $i\in \Z$ and we have $\calL=\{\Lambda_i\}_{i\in \Z}$. 
So $\calL$ is determined by the period $r$ and the sequence 
$\{\Lambda_i\}_{0\le i\le r}$ of $r+1$ $O_B$-lattices. 

In general, let $B=B_1\times \dots \times B_m$ and $O_B=O_{B_1}\times
\dots \times O_{B_m}$, where $B_i=M_{n_i}(D_i)$ and
$O_{B_i}=M_{n_i}(O_{D_i})$ for some division $\Qp$-algebra
$D_i$ whose maximal order is denoted by $O_{D_i}$. With respect to
this decomposition, one has a decomposition 
$V=V_1\oplus \dots \oplus V_m$ and for
each $O_B$-lattice $\Lambda$ in $V$ a decomposition
$\Lambda=\Lambda_1\oplus \dots \oplus \Lambda_m$. Let $e_1,\dots ,e_m$
be the central primitive idempotents of $B$. Write $\pr_{i}
\Lambda=e_i \Lambda=\Lambda_i$. 

\begin{defn}\label{PEL.3}
	\begin{enumerate}
		\item A set $\calL$ of $O_B$-lattices in $V$ is said to be
	\emph{a multichain of $O_B$-lattices} if there exist a chain of
	$O_{B_i}$-lattices $\calL_i$ in $V_i$ for each $i=1,\dots, m$ 
	such that for any member $\Lambda\in \calL$ one has $\pr_i
	\Lambda\in \calL_i$ for all $i=1,\dots, m$. 
	
	\item A multichain of $O_B$-lattices $\calL$ is called
	\emph{self-dual} if for every member $\Lambda\in \calL$, its dual
	lattice $\Lambda^\vee$ also belongs to $\calL$, where 
	\[ \Lambda^\vee:=\{x\in V \mid \psi(x,\Lambda)\subset \Zp\}. \] 
\end{enumerate} 
\end{defn}

For a multichain of $O_B$-lattices $\calL$, let 
$\calG=\calG_\calL \subset \prod_{\Lambda\in \calL} \calG_\Lambda$ be
the group scheme over $\Zp$ whose $S$-valued points, for each
$\Zp$-scheme $S$, are the $S$-points $(g_\Lambda)\in
\calG_\Lambda(S)$ which are compatible with all transition maps, where 
$\calG_\Lambda$ is the integral model of $G_{\Qp}$ defined by the
lattice $\Lambda$. Put $K=K_\calL:=\calG(\Zp)=\cap_{\Lambda\in \calL}
K_\Lambda$, where $K_\Lambda=\calG_{\Lambda}(\Zp)\subset G(\Qp)$.  
Denote by $\calG^\circ$ the connected component of $\calG$, which is the
maximal open subscheme whose fibers are all connected. Put
$K^\circ=K^\circ_\calL:=\calG^\circ(\Zp)$.

Note that $\calL$ is not determined by 
$\calL_1,\dots , \calL_m$. Therefore, the input datum $\calL$ is not
determined by $K_\calL$ in general.

\subsection{Moduli spaces with parahoric type $\calL$}
\label{sec:PEL.3}

We retain the notation 
as in subsection~\ref{sec:PEL.1}.
Let $AV$ denote the category of $O_B$-abelian varieties up to
prime-to-$p$ isogeny. The objects of $AV$ consist of pairs $(A,\iota)$, 
where $A$ is an abelian variety and \[\iota: O_{B}\otimes \Z_{(p)} \to
\End(A)\otimes \Z_{(p)}\] is a ring monomorphism. 
For any two objects $\ul A_1=(A_1,\iota_1)$ and 
$\ul A_2=(A_2,\iota_2)$, the
set of morphisms from $\ul A_1$ to $\ul A_2$ is 
$\Hom_{AV}(\ul A_1,\ul A_2):=\Hom_{O_B}(A_1,A_2)\otimes \Z_{(p)}$.   
If there is no risk of confusion, we write $A$ for $\ul A$. 

Suppose $\rho: A_1 \to A_2$ is an isogeny in $AV$. Then $\rho_p:
A_1[p^\infty]\to A_2[p^\infty]$ 
is an $O_{B_p}$-linear
isogeny. With respect to the decomposition $O_{B_p}=O_{B_1}\times
\dots \times O_{B_m}$, 
one has the decomposition 
\begin{equation}
\label{eq:PEL.2}
\ker \rho_p=H_1\times \dots \times H_m 
\end{equation}
as the product of finite flat group schemes $H_i$ with
$O_{B_i}$-action. The height of $\rho$ is the (height) function 
$h=h_H:\{1,\dots, m\}\to \bbN$ given by \[h(i)=h_H(i):=\log_p ( \ord
(H_i)),\] where $\ord (H_i)$ denotes the order of the finite group
scheme $H_i$.      

Let $A=(A,\iota)$ be an object in $AV$.
The dual abelian variety
is defined to be $A^t=(A,\iota)^t:=(A^t,\iota^t)$, where
$\iota^t(b):=\iota(b^*)^t$ for $b\in O_B\otimes \Z_{(p)}$. 
For any element $a\in N_{B^\times}(O_B\otimes \Z_{(p)})$, the
normalizer of $O_B\otimes \Z_{(p)}$ in $B^\times$, define 
$A^a=(A,\iota^a)$, where $\iota^a(b)=\iota(a^{-1} b a)$ for all $b\in
O_B\otimes \Z_{(p)}$. 
The multiplication by $a$ gives a quasi-isogeny 
$a: A^a \to A$ in $AV$. A polarization on $A$ in $AV$ is a quasi-isogeny
$\lambda: A\to A^t$ in $AV$ such that the morphism $n\lambda$ comes
from an ample line bundle on $A$ for some $n\in \bbN$. 

\begin{defn}
	Let $\calL$ be a multichain of $O_{B_p}$-lattices in $V_p$.
	A $\calL$-set of abelian varieties over a $\Z_{(p)}$-scheme $S$ 
	is a functor  
	\begin{equation*}
	\begin{split}
	\calL &\to AV, \\
	\Lambda &\mapsto A_{\Lambda}, \\
	\Lambda\subset \Lambda' &\mapsto  \rho_{\Lambda, \Lambda'}:
	A_\Lambda \to A_{\Lambda'}   
	\end{split}
	\end{equation*}
	satisfying 
	\begin{itemize}
		\item[(a)] For any $\Lambda\subset \Lambda'$ in $\calL$, write
		$\rho_{\Lambda,\Lambda'}: H^{DR}_1(A_\Lambda/S)\to
		H^{DR}_1(A_{\Lambda'}/S)$ 
		for the induced map on de Rham homologies,
		then locally on $S$ for the Zariski topology, the $\calO_S$-module 
		$H^{DR}_1(A_{\Lambda'}/S)/
		\rho_{\Lambda,\Lambda'} H^{DR}_1(A_{\Lambda}/S)$ 
		is isomorphic to
		$(\Lambda'/\Lambda)\otimes \calO_S$ as $O_B\otimes \calO_S$-modules. \\ 
		In particular, the morphism $\rho_{\Lambda, \Lambda'}:
		A_\Lambda \to A_{\Lambda'}$ is an isogeny of height \[h(i)=\log_p
		|\pr_i\Lambda'/\pr_i \Lambda|.\]
		\item[(b)] For any element $a\in B^\times \cap (O_B\otimes \Z_{(p)})$
		which normalizes $O_B\otimes \Z_{(p)}$ and any member 
		$\Lambda\in \calL$ (so $a\Lambda\subset
		\Lambda$), 
		there exists an isomorphism $\theta_{a,\Lambda}: 
		A_\Lambda^a \isoto A_{a\Lambda}$ such that the following diagram
\begin{displaymath}
\xymatrix{
	A^a_\Lambda \ar[r]^{\theta_{a,\Lambda}}_{\sim} \ar[rd]^{a}
	& A_{a\Lambda} \ar[d]^{\rho_{a\Lambda,\Lambda}} \\ 
	& A_\Lambda \\   
}    
\end{displaymath}
		commutes. 
	\end{itemize}
\end{defn}

The map $\theta_{a,\Lambda}$ is unique if it exists, and it is
functorial in $\Lambda$.

\begin{defn}\label{PEL.5}
	Let $\calL$ be a self-dual multichain of $O_{B_p}$-lattices in
	$V_p$, and \[
	A_\calL=((A_\Lambda)_{\Lambda\in \calL},\rho_{\Lambda,\Lambda'})\]
	be an
	$\calL$-set of abelian varieties over $S$ in $AV$. Define an
	$\calL$-set of abelian varieties $\wt A_\calL=((\wt A_\Lambda), \wt \rho_{\Lambda,\Lambda'})$ over $S$ in $AV$ as follows: for each
	$\Lambda$ and $\Lambda\subset \Lambda'$ in $\calL$,
	\begin{itemize}
		\item $\wt A_\Lambda:=(A_{\Lambda^\vee})^t$;
		\item $\wt\rho_{\Lambda,\Lambda'}:=(\rho_{\Lambda'^\vee,\Lambda^\vee})^t:\wt A_{\Lambda}=
		A_{\Lambda^\vee}^t\to \wt A_{\Lambda'}=A_{\Lambda'^\vee}^t$;
		\item $\wt\theta_{a,\Lambda}:=[\theta_{(a^*)^{-1},\Lambda^\vee}^t]^{-1}: (\wt A_\Lambda)^a \isoto \wt A_{a \Lambda}$.  
	\end{itemize} 
\end{defn}
We explain the last item. One easily computes 
$(a\Lambda)^\vee=(a^*)^{-1} \Lambda^\vee$. Also,
\[ (\iota^t)^a(b)=\iota^t(a^{-1} b a)=\iota(a^* b^* (a^*)^{-1}
)^t=[\iota^{(a^*)^{-1}}(b^*)]^t=[\iota^{(a^*)^{-1}}]^t(b). \] 
This shows that
$(A_{\Lambda^\vee}^{(a^*)^{-1}})^t=(A_{\Lambda^\vee}^t)^a=(\wt
A_{\Lambda})^a$ and $\wt A_{a\Lambda}=(A_{{a^*}^{-1} \Lambda^\vee})^t$.
By definition, $\wt \theta_{a,\Lambda}$ is the inverse of the isomorphism 
$\theta_{(a^*)^{-1},\Lambda^\vee}^t:
(A_{(a^*)^{-1} \Lambda^\vee})^t \ra (A_{\Lambda^\vee}^{(a^*)^{-1}})^t$. 

\begin{defn}\label{PEL.6}
	Let $\calL$, $A=A_\calL$ and $\wt A=\wt A_\calL$ be as in
	Definition~\ref{PEL.5}. 
	\begin{enumerate}
		\item  A \emph{polarization} on $A=(A_\Lambda)$ is a
	quasi-isogeny $\lambda: A\to \wt A$, (i.e. there exists $n\in
	\bbN_{>0}$  
	such that $n\lambda_\Lambda:
	A_\Lambda\to \wt A_\Lambda$ is an isogeny in $AV$ for all
	$\Lambda\in \calL$) such that for all $\Lambda\in \calL$, the composition
	\[ 
	\begin{CD}
	\lambda'_\Lambda: A_\Lambda @>{\lambda_\Lambda} >> \wt
	A_\Lambda=A_{\Lambda^\vee}^t @>{(\rho_{\Lambda,\Lambda^\vee})^t}>>
	A_\Lambda^t  
	\end{CD} \]
	is a polarization on $A_\Lambda$. 
	
	\item A polarization $\lambda$ on $A$ in $AV$ is called
	\emph{principal} if $\lambda_\Lambda$ is an isomorphism in $AV$
	for all $\Lambda\in \calL$.  
\end{enumerate}
\end{defn}


Suppose $\calL$ contains a self-dual member $\Lambda_0$, then
$\lambda_{\Lambda_0}: A_{\Lambda_0} \to A_{\Lambda_0}^t$ is 
an isomorphism in
$AV$, i.e. there exists $n\in \bbN_{>0}$, $(n,p)=1$ such that $n\lambda_
{\Lambda_0}$ is a polarization on $A_\Lambda$ in the usual sense. 
For $\Lambda,\Lambda'\in \calL$, we have the following commutative
diagram
\[ 
\begin{CD}
A_\Lambda @>\rho_{\Lambda,\Lambda'}>> A_{\Lambda'} \\
@VV{\lambda_\Lambda}V @VV{\lambda_{\Lambda'}}V \\
A_{\Lambda^\vee}^t @>{(\rho_{\Lambda'^\vee,\Lambda^\vee})^t}>> 
A_{\Lambda'^\vee}^t. 
\end{CD} \]
Therefore, a polarization $\lambda$ is uniquely determined by
$\lambda_\Lambda$ for one member $\Lambda\in \calL$.  

Let $K^p\subset G(\Ab_f^p)$ be a sufficiently small open compact
subgroup. 
Let $\bfM^{\rm naiv}_\calL$ denote the moduli scheme over $O_{E}$
classifying the objects $(A_\calL,\bar \lambda,\bar \eta)_S$, where 
\begin{itemize}
	\item $A_\calL=(A_\Lambda)_{\Lambda\in \calL}$ is an $\calL$-set of
	abelian varieties over $S$ in $AV$;
	\item $\bar \lambda=\Q^\times \cdot \lambda$ is a $\Q$-homogeneous 
	principal polarization on $A_\calL$; 
	\item $\bar \eta$ is a $\pi_1(S,\bar s)$-invariant $K^p$-orbit of
	isomorphisms 
	\[ \eta: V\otimes \Ab_f^p \isoto V^{(p)}(A_{\bar s}) \]
	which preserve the pairings to up a scalar in
	$(\Ab_f^p)^\times$. Here for simplicity we assume that $S$ is
	connected.   
\end{itemize}

Put $K=K_\calL, \K=K\cdot K^p$ and $\K^\circ=K^\circ K^p$.  
Similarly as before, $\bfM^{\rm naiv}_\calL$ is an integral model over $O_E$ of the rational moduli scheme $M_\K$ (which is defined over $\E$), and one has
$\bfM^{\rm naiv}_\calL \otimes_{O_E} E \supset
\Sh_{\K,E}$. Let $\ES_\calL=\ES_\calL^{\rm can}$ be the flat closure of
$\Sh_{\K,E}$ in $\bfM^{\rm naiv}_\calL$. As $\K^\circ\subset \K$, we have
the natural cover $\Sh_{\K^\circ}\to \Sh_{\K}$ over $E$. 
We define $\ES_{\K^\circ}$ as the normalization of $\ES_\calL$ in
$\Sh_{\K^\circ}$. If $G$ splits over a tamely ramified extension of $\Q_p$, then the main results of Pappas-Zhu \cite{local model P-Z} subsection 8.2 imply that $\ES_{\K^\circ}$ fits into a local model diagram.

Let $\calL'$ be another set of self-dual multichain of
$O_{B_p}$-lattices in 
$V_p$ containing $\calL$. Put $K':=K_{\calL'}$ and
$K'^\circ:=K^\circ_{\calL'}$. Then $K'\subset K$, $\K' \subset \K$,
$K'^\circ \subset K^\circ$ and $\K'^\circ\subset \K^\circ$. 
Since $\Sh_{\K'}$ and $\Sh_{\K}$ are dense in $\ES_{\calL'}$ and
$\ES_{\calL}$, respectively, the morphism $\pi_{\K',\K}: \bfM^{\rm
	naiv}_{\calL'} \to \bfM^{\rm naiv}_{\calL}$ maps $\ES_{\calL'}$
into $\ES_{\calL}$, and we have the following commutative diagram:
\begin{displaymath}
\xymatrix{
	\Sh_{\K'} \ar[r] \ar[d]^{\pi_{\K',\K}} & \ES_{\calL'}
	\ar[d]^{\pi_{\K',\K}} \ar[r] & \bfM^{\rm 
		naiv}_{\calL'} \ar[d]^{\pi_{\K',\K}}\\ 
	\Sh_{\K}  \ar[r]  & \ES_{\calL} \ar[r] & \bfM^{\rm naiv}_{\calL}.} 
\end{displaymath}






\begin{defn}\label{PEL.7}
	Let $k$ be an \ac field containing the residue field $k(v)$ of $E$. 
	Let $x\in \bfM^{\rm naiv}_{\calL}(k)$ be a $k$-point. Denote by
	$\calC^{\rm naiv}_\calL(x)\subset \bfM^{\rm naiv}_{\calL}(k)$ the subset
	consisting of the points $y\in \bfM^{\rm naiv}_{\calL}(k)$ such that
	the associated $\calL$-set of $p$-divisible groups
	$A_{y,\calL}[p^\infty]$ is isomorphic to $A_{x,\calL}[p^\infty]$,
	where $A_{x,\calL}$ (resp.~$A_{y,\calL}$) is the $\calL$-set of
	abelian varieties in $AV$ corresponding to the point $x$ (resp.~$y$)
	in $\bfM^{\rm naiv}_{\calL}(k)$. The set 
	$\calC^{\rm naiv}_{\calL}(x)$ is locally
	closed in $\bfM^{\rm naiv}_{\calL}\otimes_{O_{E}} k$ and we
	regard it as the locally closed subscheme with the induced reduced scheme
	structure, and call it the \emph{central leaf passing through $x$} in 
	$\bfM^{\rm naiv}_{\calL}\otimes_{O_{E}} k$.  
	
	If $x\in \ES_{\calL}(k)$, then write $\calC_\calL(x)$ for
	the intersection $\calC^{\rm naiv}_\calL(x)\cap 
	(\ES_{\calL}\otimes_{O_{E}} k)$ and call it the 
	\emph{central leaf passing through $x$} in 
	$\ES_{\calL}\otimes_{O_{E}} k$.
\end{defn}

It is known that $\calC^{\rm naiv}_\calL(x)$ is smooth and
quasi-affine. 
The smoothness follows from the properties that 
$\calC^{\rm naiv}_\calL(x)$ 
is reduced and the completions 
$\calC^{\rm naiv}_\calL(x)^\wedge_y$ at all $y 
\in \calC^{\rm naiv}_\calL(x)(k)$ are isomorphic. 
The quasi-affineness
follows from that of central leaves in the Siegel modular varieties. 
It follows that if $x$ lies in $\ES_{\calL}(k)$, then the subscheme 
$\calC_\calL(x)\subset \ES_{\calL}\otimes_{O_{E}} k$ 
is smooth and quasi-affine.   

\begin{proposition}\label{PEL.8}
	Let $k$ be as in Definition~\ref{PEL.7}, 
	$x'\in \ES_{\calL'}(k)$ and 
	$x$ be its image in $\ES_{\calL}(k)$. Then the morphism
	$\pi_{\K',\K}:\calC_{\calL'}(x')\to \calC_{\calL}(x)$ is surjective
	with finite fibers.    
\end{proposition}
\begin{proof}
	We first show this statement for the morphism
	$\pi_{\K',\K}:\calC^{\rm naiv}_{\calL'}(x')\to \calC^{\rm
		naiv}_{\calL}(x)$. Let $x'=[((A_{\Lambda'})_{\Lambda'\in \calL'}, 
	\bar \lambda, \bar \eta)]$, where $\lambda$ is a principal
	representative.  
	Then \[x=[((A_{\Lambda})_{\Lambda\in\calL}, \bar \lambda,
	\bar \eta)].\] 
	Let 
	$y=[((B_\Lambda)_\Lambda, \bar \lambda_{B}, \bar \eta_B)] \in 
	\calC^{\rm naiv}_\calL(x)$, 
		where $\lambda_B$ is a principal representative. 
	We fix an isomorphism 
	\[((B_\Lambda)_\Lambda,\lambda_B)[p^\infty]\simeq 
	((A_\Lambda)_\Lambda,\lambda)[p^\infty]\] of $\calL$-sets of
	$p$-divisible groups. Then there is an extension
	$((B_{\Lambda'})_{\Lambda'},\lambda_B)$ such that
	\[((B_{\Lambda'})_{\Lambda'},\lambda_B)[p^\infty]\simeq 
	((A_{\Lambda'})_{\Lambda'},\lambda)[p^\infty].\] It is enough to
	extend $(B_{\Lambda'})_{\Lambda'\in \calL'}$ 
	for $\Lambda'\in \calL'^\square$,  where
	\[\calL'^\square:=\{\Lambda\in \calL' \mid \Lambda\subset
	\Lambda^\vee \subset p^{-1} \Lambda \},\] and use the periodic
	property to extend $(B_{\Lambda'})_{\Lambda'}$ 
	for all $\Lambda'\in \calL'$. Choose $\Lambda\in \calL$ such 
	$\Lambda\subset \Lambda'$ for all $\Lambda'\in \calL'^\square$. 
	Suppose $\Lambda'\not\in \calL$. We define $B_{\Lambda'}$ as the
	quotient $B_\Lambda/H_{\Lambda'}$, where $H_{\Lambda'}\subset
	B_\Lambda[p^\infty]$ is the finite flat subgroup scheme which corresponds
	to the finite subgroup scheme $\ker \left (
	\rho_{\Lambda,\Lambda'}[p^\infty]:
	A_\Lambda[p^\infty]\to A_{\Lambda'}[p^\infty]\right )$ 
	via the isomorphism
	$B_\Lambda[p^\infty]\simeq A_\Lambda[p^\infty]$. As explained, the
	polarization $\lambda_B$ is uniquely determined by
	$\lambda_{B,\Lambda}$ for one member $\Lambda\in \calL'$. 
	Therefore, there
	is a unique extension of the polarization $\lambda_B$ on 
	$(B_{\Lambda'})_{\Lambda'\in \calL'}$. So we construct a member
	$((B_{\Lambda'})_{\Lambda'\in \calL'}, \bar \lambda_B,\bar \eta_B)$
	such that there is an isomorphism
	$((B_{\Lambda'}),\lambda_B)[p^\infty]\simeq
	((A_{\Lambda'}),\lambda)[p^\infty]$. This shows the surjectivity of
	$\pi_{\K',\K}:\calC^{\rm naiv}_{\calL'}(x')\to \calC^{\rm
		naiv}_{\calL}(x)$. 
	
	Let $\calL_{\GSp}$ denote the set of $\Zp$-lattices in $\calL$ by
	ignoring the $O_{B_p}$-module structure, and let
	$\bfA_{\calL_{\GSp}}$ denote the Siegel moduli space over $O_{E}$
	with parahoric level $\calL_{\GSp}$ at $p$. For any $z\in
	\bfA_{\calL_{\GSp}} (k)$, we denote by $\calC_{\calL_{\GSp}}(z)$
	the central leaves passing through $z$ in
	$\bfA_{\calL_{\GSp}}\otimes k$. Forgetting the
	endomorphism structure gives a finite morphism \[f_\calL: \bfM^{\rm
		naiv}_{\calL}\to \bfA_{\calL_{\GSp}}\] \cite[Proposition
	1.1]{Yu:lift}. Similarly, we have a finite morphism 
	$f_{\calL'}: \bfM^{\rm naiv}_{\calL'}\to \bfA_{{\calL'}_{\GSp}}$ 
	and a commutative diagram
	\begin{displaymath}
	\xymatrix{
		\bfM^{\rm naiv}_{\calL'} \ar[r]^{f_{\calL'}} \ar[d]^{\pi_{\K',\K}} 
		&  \bfA_{{\calL'}_{\GSp}}
		\ar[d]^{\pi^{\GSp}_{\calL',\calL}}\\ 
		\bfM^{\rm naiv}_{\calL} \ar[r]^{f_\calL} & \bfA_{{\calL}_{\GSp}}. 
	}
	\end{displaymath}
	Restricting the forgetting
	morphisms $f_{\calL'}$ and $f_\calL$ to the central leaves we have
	the following commutative diagram 
	\begin{displaymath}
	\xymatrix{
		\calC^{\rm naiv}_{\calL'}(x') \ar[r]^{f_{\calL'}} \ar[d]^{\pi_{\K',\K}} 
		& \calC_{\calL'_{\GSp}}(f_{\calL'}(x')) 
		\ar[d]^{\pi^{\GSp}_{\calL',\calL}}\\ 
		\calC^{\rm
			naiv}_{\calL}(x) \ar[r]^{f_\calL} &
		\calC_{\calL_{\GSp}}(f_{\calL}(x)).
	} 
	\end{displaymath}
	
	Note that the morphism $\pi^{\GSp}_{\calL',\calL}$ is finite
	\cite[Section 7]{He-Rap} and so is the composition 
	$\pi^{\GSp}_{\calL',\calL}\circ f_{\calL'}$. 
	From the above commutative diagram, any fiber of $\pi_{\K',\K}$ is
	contained in a fiber of the finite morphism $\pi^{\GSp}_{\calL',\calL}\circ
	f_{\calL'}$, 
	which has only finitely many elements. Thus,   
	the morphism
	$\pi_{\K',\K}:\calC^{\rm naiv}_{\calL'}(x')\to \calC^{\rm
		naiv}_{\calL}(x)$ 
	has finite fibers. It follows that the morphism 
	$\calC_{\calL'}(x')\to \calC_{\calL}(x)$ has finite
	fibers. It remains to show the surjectivity.
	
	Note that the generic fiber of $\bfM^{\rm naiv}_{\calL}$ is the
	disjoint union of $\Sh^{(1)}_{\K,E}=\Sh_{\K,E},\dots ,
	\Sh^{(m)}_{\K,E}$, where $m$ is the cardinality of the finite set $\ker^1(\Q,G):=\ker\big(H^1(\Q, G)\ra \prod_vH^1(\Q_v,G)\big)$, for each $i$, $\Sh_{\K,E}^{(i)}=\Sh_{\K}^{(i)}\otimes_\E E$, and $\Sh^{(i)}_{\K}$ is the corresponding model of $\Sh^{(i)}_\K(\C)$ in the decomposition $\bfM^{\rm naiv}_{\calL}(\C)=\coprod_{i=1}^m\Sh_\K^{(i)}(\C)$.
	This is
	described in \cite[Section 8]{kottwitz:jams92} 
	under the unramified setting but because this
	is a statement in characteristic zero and we have assumed $p\neq 2$, 
	it also holds for the present setting. Similarly, $\bfM^{\rm naiv}_{\calL'}\otimes
	E=\coprod_{i=1}^m \Sh^{(i)}_{\K',E}$ and the transition map
	sends $\Sh^{(i)}_{\K',E} \onto \Sh^{(i)}_{\K,E}$. 
	
	Let $y=[((B_\Lambda)_\Lambda, \bar \lambda_{B}, \bar \eta_B)] \in 
	\calC_\calL(x)$, and
	 $y'=[((B_{\Lambda'})_{\Lambda'}, \bar \lambda_{B}, \bar
	\eta_B)] \in  \calC^{\rm naiv}_{\calL'}(x')$ be a point 
	we constructed over
	$y$ as before. 
	As $x'\in \ES_{\calL'}(k)$, one fixes a lifting $\wt x'\in
	\ES_{\calL'}(R)$ of $x'$ 
	over a complete DVR $R$ with residue field
	$k$; the base change $\wt x'_{\Frac(R)}$ lands in 
	$\Sh_{\K'}(\Frac
	(R))$. By the Serre-Tate theorem, deforming abelian varieties is the
	same as deforming the associated $p$-divisible groups. Thus, the way
	of lifting $x'$ to $\wt x'$ produces a lifting 
	$\wt y'\in
	\bfM^{\rm naiv}_{\calL'}(R)$ of $y'$ over
	$R$. 
	Since its image $\pi_{\K',\K}(\wt y')=:\wt y$ is a lifting of
	$y$ over $R$, the point $\wt y_{\Frac(R)}$ lands in
	$\Sh_{\K}$. Therefore, the point 
	$\wt y'_{\Frac(R)}$ lands in $\Sh_{\K'}$ and
	$y'\in \calC_{\calL'}(x')$.       
\end{proof}

Suppose that $K'=K'^\circ$ and $K=K^\circ$, which are parahoric subgroups by
definition. 
Put ${\ES}_\K:=\ES_\calL$ and $\ES_{\K'}:=\ES_{\calL'}$. Let $k=\ov{\mathbb{F}}_p$. For each
point $x\in \ES_\K(k)$, there exists an 
$O_{B_p}\otimes_{\Zp} \breve{\Q}_p$-linear isomorphism
\[M(A_{x,\Lambda})\otimes \breve{\Q}_p\simeq V\otimes \breve{\Q}_p\] 
which preserves the
pairings up to a scalar in $\breve{\Q}_p^\times$, where $M(A_{x,\Lambda})$
is the covariant Dieudonn\'e module of $A_{x,\Lambda}$, 
for one $\Lambda\in \calL$. 
Transporting the Frobenius
map on $M(A_{x,\Lambda})$ to $V\otimes \breve{\Q}_p$, 
we obtain an element $G(\breve{\Q}_p)$,
which is well-defined up to $\sigma$-$\breve{K}$-conjugate. Thus, we have
defined a map \[\Upsilon_\K: \ES_\K(k)\to G(\breve{\Q}_p)/\breve{K}_\sigma,\] whose fibers are the central leaves introduced in Definition \ref{PEL.7}.
Similarly, we also have a map
$\Upsilon_{\K'}: \ES_{\K'}(k)\to G(\breve{\Q}_p)/\breve{K}'_\sigma$ and have 
the following commutative diagram:
\[ 
\begin{CD}
\ES_{\K'}(k) @>\Upsilon_{\K'}>> G(\breve{\Q}_p)/\breve{K}'_\sigma \\
@VV\pi_{K',K}V @VV\pi^{G}V \\
\ES_{\K}(k) @>\Upsilon_{\K}>>   G(\breve{\Q}_p)/\breve{K}_\sigma.
\end{CD} \] 
Proposition~\ref{PEL.8} confirms axiom 4 (c) of \cite{He-Rap} for the
canonical Rapoport-Zink integral models $\ES_\calL$ under the
assumption $p>2$, $K'=K'^\circ$ and $K=K^\circ$.

\end{document}